\documentclass{article}

\usepackage[a4paper, total={6in, 8in}]{geometry}
\usepackage{emptypage}
\usepackage[utf8]{inputenc}
\usepackage[T1]{fontenc}
\usepackage{lmodern}
\usepackage[english]{babel}
\usepackage{amsmath, amssymb}
\usepackage{amsthm}
\usepackage{longtable}
\usepackage{placeins}
\usepackage[font=small,labelfont=bf]{caption}

\usepackage{enumerate}
\usepackage{xcolor}
\RequirePackage[colorlinks,allcolors=blue]{hyperref}
\RequirePackage{graphicx}
\usepackage[nottoc]{tocbibind}
\usepackage{svg}

\usepackage{titlesec}
\titleformat{\section}{\large\bfseries}{\thesection}{1em}{}

\usepackage[round]{natbib} 

\numberwithin{equation}{section}

\newtheorem{theorem}{Theorem}
\numberwithin{theorem}{section}

\newtheorem{lemma}{Lemma}
\numberwithin{lemma}{section}
\newtheorem{corollary}{Corollary}
\numberwithin{corollary}{section}

\numberwithin{proposition}{section}

\theoremstyle{definition}

\newtheorem{definition}{Definition}

\newtheorem*{remark*}{Remark}

\usepackage{tocloft}
  \title{Extension of Process Convergence With Application to Chatterjee's Rank Correlation\footnote{This article was previously entitled `Asymptotic Normality of Chatterjee's Rank Correlation'.}}
\date{\today}
\author{Marius Kroll\footnote{Faculty of Mathematics, Ruhr-Universität Bochum, \href{mailto:marius.kroll@rub.de}{marius.kroll@rub.de}}}
\begin{document}
\begin{titlepage}
  \centering
  {\huge Extension of Process Convergence With Application to Chatterjee's Rank Correlation\footnote{This article was previously entitled `Asymptotic Normality of Chatterjee's Rank Correlation'.} \par}
  \vspace{2mm}

  {\large Marius Kroll\footnote{Faculty of Mathematics, Ruhr University Bochum, \href{mailto:marius.kroll@rub.de}{marius.kroll@rub.de}} \par}
  \vspace{2mm}

  {\today \par}

  \begin{abstract}
    We give conditions under which weak convergence of a stochastic process indexed in the class of $d$-dimensional hyperrectangles is sufficient to ensure convergence in the larger class of functions of uniformly bounded Hardy-Krause variation. When applied to the empirical process, this can further be extended to derive weak convergence of V-processes indexed in the class of kernel functions which are coordinate-wise of uniformly bounded Hardy-Krause variation. Our proofs use a generalisation of the Koksma-Hlawka inequality for linear operators, allowing us to establish our results without any continuity assumptions on the functions involved. Our theory is complemented by two separate applications: First, we establish asymptotic normality of Chatterjee's rank correlation in the fully general setting. Second, we present new limit theorems for U- and V-processes of strongly mixing data.\\[2mm]
  \end{abstract}

  \tableofcontents
\end{titlepage}

\section{Introduction}
Let $d \in \mathbb{N}$ and $\mathcal{G}$ be some class of functions $g : [0,1]^d \to \mathbb{R}$. Suppose that $X_n$, $n \in \mathbb{N}$, is a sequence of stochastic processes in $\ell^\infty(\mathcal{G})$, the space of all bounded real-valued functions defined on $\mathcal{G}$, for which we want to show weak convergence to some limiting process $X$, i.e.\@ $X_n \rightsquigarrow X$ in $\ell^\infty(\mathcal{G})$. Whether this is feasible depends quite strongly on the nature of both the processes $X_n$, and the function class $\mathcal{G}$. For instance, a common case is that of the empirical process $X_n = r_n(P_n - \mu_n)$, where $r_n$ is a rate, $P_n$ is the empirical measure based on some observations, and $\mu_n$ is a sequence of deterministic centring measures. If the observations which constitute the empirical measures obey suitable regularity conditions, such as i.i.d.\@ or weakly dependent data, then this problem is well-understood, and verifying the weak convergence $X_n \rightsquigarrow X$ often reduces to analysing the complexity of the function class $\mathcal{G}$ in some sense.

However, when the underlying processes $X_n$ are less regular, establishing $X_n \rightsquigarrow X$ in $\ell^\infty(\mathcal{G})$ can quickly become very challenging. For example, consider once again the empirical process $X_n = r_n(P_n - \mu_n)$, but now assume that the data which constitute the empirical measures $P_n$ are highly dependent or irregular in some other way which breaks the assumptions of the usual empirical process results. In such cases, it can be a fruitful approach to first establish the weak convergence $X_n \rightsquigarrow X$ in $\ell^\infty(\mathcal{F})$ for some smaller and easier to handle class $\mathcal{F}$, and to then show that this convergence extends to the larger class of interest $\mathcal{G}$.

In this article, we take care of the second part of this argument for specific function classes $\mathcal{F}$ and $\mathcal{G}$. $\mathcal{F}$ will be the class of all indicator functions $\textbf{1}_{[0,x]}$, $x = (x_1, \ldots, x_d) \in [0,1]^d$, where $[0,x]$ is simply notation for the $d$-dimensional hypercube $[0,x] = [0,x_1] \times \ldots \times [0,x_d]$. This is perhaps the most canonical function class in the study of random processes, and as such there are many tools available to prove the initial weak convergence $X_n \rightsquigarrow X$ in $\ell^\infty(\mathcal{F})$. The class $\mathcal{G}$ will consist of functions $h : [0,1]^d \to \mathbb{R}$ with uniformly bounded Hardy-Krause variation, which is a certain multivariate generalisation of the total variation of a univariate function. Loosely speaking, the Hardy-Krause variation (like all notions of variation) measures the roughness of a multivariate function. In the next section, we will give a formal definition of the Hardy-Krause variation, together with a simple bound for the Hardy-Krause variation applicable to suitably differentiable functions (although a function does not need to be differentiable or even continuous to be of bounded Hardy-Krause variation).

One of the main results of this article is the following: If $X_n \rightsquigarrow X$ in $\ell^\infty(\mathcal{F})$, where $\mathcal{F}$ is now the class of $d$-dimensional hyperrectangles in $[0,1]^d$, then the convergence can, under some assumptions, be extended to also hold on the space of all functions $h : [0,1]^d \to \mathbb{R}$ which are of uniformly bounded Hardy-Krause variation. The assumptions placed on the processes $X_n$ are only in terms of their sample path properties. In particular, they do not require any information about the dependence structure between the $X_n$. Our results are built on a generalisation of the so-called Koksma-Hlawka inequality. This not only allows us to work in a more general setting when compared to previous works in this area \citep[in partiuclar, we do not require any continuity assumptions on the functions of bounded Hardy-Krause variation. Compare with this][]{radulovic_et_al:2017,berghaus_et_al:2017,beutner_zaehle:2023}, but it also makes the extension of the limit process accessible. We discuss this new proof technique in more detail in Section \ref{sec:mains}.

There is another problem which may also be described as extending the convergence of a stochastic process: If we can show weak convergence of the univariate empirical process $X_n = r_n(P_n - \mu_n)$ in some function class $\mathcal{G}$, can we derive convergence of the associated product measure process $r_n(P_n^m - \mu_n^m)$ for some fixed $m \in \mathbb{N}$? We answer this question in the affirmative and describe the classes $\mathcal{G}_m$ in which the associated product measure process converges. These results should be understood in combination with the previously described extensions to functions of uniformly bounded Hardy-Krause variation. Taken together like this, our article implies that showing convergence of the empirical process $r_n(P_n - \mu_n)$ in $\ell^\infty(\mathcal{F})$, $\mathcal{F}$ being the class of $d$-dimensional hyperrectangles in $[0,1]^d$, is sufficient to guarantee convergence of the product measure process $r_n(P_n^m - \mu_n^m)$ indexed in a large class of functions $h : ([0,1]^d)^m \to \mathbb{R}$. The precise definition of this class is given in Section \ref{sec:mains}; essentially it consists of functions which are required to be of bounded Hardy-Krause variation only coordinate-wise. Product measure processes of this kind may be called V-processes, since any evaluation $r_n[P_n^m(h) - \mu_n^m (h)]$ is a V-statistic with kernel function $h$ (though perhaps more common are their close relatives, U-processes, where every evaluation is a U-statistic). Since our results only rely on sample path properties, one particular application of them is to derive weak convergence of U- and V-statistics and -processes for non-standard data, e.g.\@ in the presence of erratic dependence structures. This is essentially a continuous mapping type argument, and the general idea of expressing a V-statistic as a smooth functional of the empirical measure is not new. Most relevant to this article are the works by \cite{beutner_zaehle:2012,beutner_zaehle:2014}; for details, see Section \ref{sec:mains}.

The second part of this article consists of applications of our theoretical results to specific processes, mostly those connected with order statistics and their concomitants. For an i.i.d.\@ sample $(X_1, Y_1), \ldots, (X_n, Y_n) \in [0,1]^2$, let us denote by $(X_{n,1}', Y_{n,1}'), \ldots, (X_{n,n}', Y_{n,n}')$ a reordering of the original sample such that $X_{n,k}'$ is the $k$-th order statistic of $X_1, \ldots, X_n$ with ties broken at random. The random variables $Y_{n,1}', \ldots, Y_{n,n}'$ are called the concomitants of the order statistics $X_{n,1}', \ldots, X_{n,n}'$. Letting the empirical measure based on $(Y_{n,1}', Y_{n,2}'), \ldots, (Y_{n,n-1}', Y_{n,n}')$ be denoted by $P_n$, we will show that $\sqrt{n}(P_n - \mathbb{E}P_n)$ converges in $\ell^\infty(\mathcal{F})$, where $\mathcal{F}$ is again the class of all indicator functions $\textbf{1}_{[0,x]}$, $x \in [0,1]^2$. Our interest in these processes is motivated by the study of Chatterjee's rank correlation \citep{chatterjee:2021}, defined as
$$
  \xi_n = \xi_n(X,Y) = 1 - \frac{n \sum_{i=1}^{n-1} |r_{i+1} - r_i|}{2 \sum_{i=1}^n l_i (n-l_i)},
$$
where $r_i = \sum_{j=1}^n \textbf{1}\{Y_{n,j}' \leq Y_{n,i}'\}$ is the rank of $Y_{n,i}'$ among the $Y_1, \ldots, Y_n$, and $l_i = \sum_{j=1}^n \textbf{1}\{Y_{n,j}' \geq Y_{n,i}'\}$. It estimates the Dette-Siburg-Stoimenov measure of dependence \citep{dette_siburg_stoimenov:2013} which can be written as
$$
  \xi = \xi(X,Y) = \frac{\int \mathrm{Var}\left(\mathbb{E}\left[\textbf{1}_{[y, \infty)}(Y) ~|~ X\right]\right) ~\mathrm{d}\mathbb{P}^Y(y)}{\int \mathrm{Var}\left(\textbf{1}_{[y, \infty)}(Y)\right) ~\mathrm{d}\mathbb{P}^Y(y)},
$$
where $\mathbb{P}^Y$ denotes the distribution of $Y$. This measure of dependence is $0$ if and only if $X$ and $Y$ are independent and $1$ if and only if $Y$ is a measurable function of $X$ almost surely. By inserting the representation of $r_i$ and $l_i$ as sums of indicator functions in the definition of $\xi_n$ and collecting sums, it is straightforward to show that $\xi_n = 1 - V_1/(2V_2) + \mathcal{O}(1/n)$, where $V_1$ and $V_2$ are certain V-statistics based on $(Y_{n,1}', Y_{n,2}'), \ldots, (Y_{n,n-1}', Y_{n,n}')$. Using the general extension results described above, convergence of $\sqrt{n}(P_n - \mathbb{E}P_n)$ indexed in indicator functions $\textbf{1}_{[0,x]}$, $x \in [0,1]^2$, also gives us convergence of the $m$-fold product measure process $\sqrt{n}(P_n^m - \mathbb{E}P_n^m)$, indexed in suitable classes of functions. We have already observed that evaluations of these processes are V-statistics -- in this case, V-statistics based on the data $(Y_{n,1}', Y_{n,2}'), \ldots, (Y_{n,n-1}', Y_{n,n}')$, which is precisely what we need for the analysis of Chatterjee's rank correlation. It should be appreciated that these data are highly erratic due to their complex dependence structure, so that more classical methods of dealing with V-statistics do not seem to be useful here.

The interest in Chatterjee's rank correlation is considerable, as is the body of literature on this simple yet elegant measure of dependence \citep[to give just a few references:][]{
  azadkia_chatterjee:2021,
  lin_han:2022,
  lin_han:2023,
  lin_han:2024,
  dette_kroll:2025,
  han_huang:2024,
  deb2020measuringassociationtopologicalspaces,
  auddy_deb_nandy:2024,
  shi_drton_han:2021,
  shi_drton_han:2024,
  ansari2025continuitychatterjeesrankcorrelation,
  ansari2025directextensionazadkia,
  bücher2024lackweakcontinuitychatterjees,
  hörmann2025azadkiachatterjeesdependencecoefficientinfinite}. Despite these efforts, the limiting behaviour of Chatterjee's rank correlation was, until now, only partly understood. \cite{chatterjee:2021} in his original paper derives the limiting distribution under the assumption that $X$ and $Y$ are independent, and \cite{lin_han:2022} obtain the weak limit for random vectors $(X,Y)$ with continuous joint distribution function. This article is the first to establish asymptotic normality in the fully general setting, which in particular allows for the application of the $m$ out of $n$ bootstrap for Chatterjee's rank correlation to construct confidence intervals for $\xi$ \citep[as suggested by][]{dette_kroll:2025}.

As a second application, we use our results to establish weak convergence of U- and V-processes based on strongly mixing data. Statistical literature usually focuses on U-processes instead of V-processes, but under suitable moment conditions they are equivalent. While certainly less obscure than the concomitants $(Y_{n,i}', Y_{n,i+1}')$ from the construction of Chatterjee's rank correlation, proving convergence of even a single U- or V-statistic based on strongly mixing data can be challenging (in fact, the literature on this topic is surprisingly scant, as most articles focus on stronger mixing conditions such as absolute regularity). In particular, our method does not rely on the so-called $P$-Lipschitz continuity assumption, which is a central condition in one of the most commonly used results for U-statistics of strongly mixing data \citep{dehling_wendler:2010}, and which can easily fail if the underlying data do not follow a continuous distribution. Two specific statistical quantities which we investigate using our technique are the Grassberger-Procaccia estimator \citep{grassberger_procaccia:1983} and Kendall's $\tau$ in the presence of ties for strongly mixing data.

The structure of the remaining article is as follows: In Section \ref{sec:preliminaries} we collect some basic definitions and notations that will be used throughout the paper. This includes a definition of the Hardy-Krause variation. Section \ref{sec:mains} contains our main results on the convergence of stochastic processes. Sections \ref{sec:mains_chatterjee} and \ref{sec:mixing} contain our results on Chatterjee's rank correlation and related processes, and on U- and V-processes of strongly mixing data, respectively. The remaining sections contain all proofs of this paper, sorted by topic: proofs for our theoretical main results (Section \ref{sec:proofs_empirical_processes}), proofs for processes arising from order statistics and their concomitants (Section \ref{sec:proofs_order_concomitants}), proofs for Chatterjee's rank correlation (Section \ref{sec:proofs_chatterjee}), and proofs for U- and V-processes of strongly mixing data (Section \ref{sec:mixing_proofs}).

\section{Preliminaries}
\label{sec:preliminaries}
\subsection{Basic Concepts and Notation}
\label{sec:notation}
The most important reference for this article is \cite{van_der_vaart_wellner:weak_convergence}. We collect some of their definitions in this section.

For an index set $T$, $\ell^\infty(T)$ denotes the linear space of all bounded functions from $T$ to $\mathbb{R}$, equipped with the supremum norm, which we denote by $\|\cdot\|_T$. When we understand $\ell^\infty(T)$ as a measurable space, we implicitly equip it with its Borel $\sigma$-algebra generated by the supremum norm. We usually take $T$ as some function class $\mathcal{F}$, and often we also consider the supremum norm defined on $\mathcal{F}$. To differentiate between the supremum norms on $\mathcal{F}$ and $\ell^\infty(\mathcal{F})$, we usually write $\|f\|_\infty$ for the supremum norm of some $f \in \mathcal{F}$, and reserve the notation $\|X\|_\mathcal{F}$ for the supremum norm of some $X \in \ell^\infty(\mathcal{F})$.

The notation $\bar{\mathbb{R}}$ means the extended real line, i.e.\@ $\bar{\mathbb{R}} = \mathbb{R} \cup \{\pm \infty\}$, equipped with its standard topology. Similarly, we always equip $\mathbb{R}$ with its standard topology. Whenever measurability is relevant, we understand both $\mathbb{R}$ and $\bar{\mathbb{R}}$ to be equipped with the Borel $\sigma$-algebras generated by their standard topologies.

Let $(\Omega, \mathcal{A}, \mathbb{P})$ be a probability space and $T : \Omega \to \bar{\mathbb{R}}$ a possibly non-measurable function. The outer expectation $\mathbb{E}^* T$ is defined as
$$
  \mathbb{E}^* T = \inf \left\{\mathbb{E}U ~|~ U : (\Omega, \mathcal{A}, \mathbb{P}) \to \bar{\mathbb{R}} \textrm{ is Borel measurable and integrable}\right\}.
$$
The outer probability $\mathbb{P}^*(A)$ of a possibly non-measurable subset $A \subseteq \Omega$ is defined as $\mathbb{P}^*(A) = \mathbb{E}^* \textbf{1}_A$ (there are other definitions but this equality always holds). A minimal measurable cover $T^*$ of $T$ is a measurable function $T : (\Omega, \mathcal{A}) \to \bar{\mathbb{R}}$ with the property that $T^* \geq T$ always, and $T^* \leq U$ almost surely for any measurable $U : (\Omega, \mathcal{A}) \to \bar{\mathbb{R}}$ satisfying $U \geq T$ almost surely. Minimal measurable covers always exist and are unique up to $\mathbb{P}$-null sets. If $\mathbb{E}^* T < \infty$, then $\mathbb{E}^* T = \mathbb{E} T^*$.

If $X_n$, $n \in \mathbb{N}$, is a sequence of not necessarily measurable random elements $X_n : (\Omega_n \mathcal{A}_n, P_n) \to \ell^\infty(T)$ defined on probability spaces $(\Omega_n \mathcal{A}_n, P_n)$, we say that $X_n$ converges weakly to some Borel measurable $X : (\Omega, \mathcal{A}, P) \to \ell^\infty(T)$ if
$$
  \mathbb{E}^*f(X_n) \xrightarrow[n \to \infty]{} \mathbb{E}f(X)
$$
for every bounded and continuous function $f : \ell^\infty(T) \to \mathbb{R}$. We denote this weak convergence by $X_n \rightsquigarrow X$. While the random elements $X_n$ may be non-measurable, the weak limit $X$ is always Borel measurable by definition. If the random elements $X_n$ are measurable, then this concept coincides with the usual definition of convergence in distribution. Other modes of convergence considered in this article are outer almost sure convergence, almost uniform convergence, and convergence in outer probability, denoted by
$$
  X_n \xrightarrow[n \to \infty]{as*} X, \quad X_n \xrightarrow[n \to \infty]{au} X, \quad X_n \xrightarrow[n \to \infty]{\mathbb{P}^*} X,
$$
respectively. Outer almost sure convergence holds by definition if $\|X_n - X\|_T^* \to 0$ almost surely for some versions of $\|X_n - X\|_T^*$. Almost uniform convergence holds by definition if for any $\varepsilon > 0$ there exists a measurable set $A$ with $\mathbb{P}(A) \geq 1-\varepsilon$ and $\|X_n - X\|_T \to 0$ uniformly on $A$. Convergence in outer probability holds by definition if $\|X_n - X\|_T^* \to 0$ in probability; equivalently, if $\mathbb{P}^*(\|X_n - X\|_T > \varepsilon) \to 0$ for any $\varepsilon > 0$. In contrast to weak convergence, we do not assume that the limit $X$ is Borel measurable in these three types of convergence; but if it is, then outer almost sure convergence and almost uniform convergence are equivalent.

In this article, we consider stochastic processes, which are random elements $X \in \ell^\infty(T)$ with the additional property that for any fixed $t \in T$, $X(t)$ is a Borel measurable random variable. For such a process $X$, we can define under appropriate integrability assumptions a non-random function on $T \ni t \mapsto \mathbb{E}X(t) \in \mathbb{R}$. We denote this new function by $\mathbb{E}X$, agreeing that this is to be understood only as notation for the pointwise expectation $t \mapsto \mathbb{E}X(t)$, as opposed to a Banach-space valued integral such as in the sense of Bochner or Dunford.

Let us now collect some miscellaneous notation which may or may not be standard. If $V$ is a linear space and $S \subset V$ a subset, we write $\mathrm{span}(S)$ for the linear span of $S$. If $V$ is normed, then $\overline{\mathrm{span}}(S)$ denotes the closed linear span of $S$. $\mathcal{U}[0,1]$ is the continuous uniform distribution on the unit interval. $\# A$ is the cardinality of a set $A$, $\bar{A}$ its closure, $A^C$ its complement, and $\textbf{1}_A$ its indicator function. We may move complicated sets from the index to the main text; e.g.\@ $\textbf{1}\{X \leq a\}$ for the indicator function of the set $\{X \leq a\}$. $f \lor g$ and $f \land g$ may denote the maximum and minimum, respectively, of $f$ and $g$. For any $p \geq 1$, $L_p(Q)$ denotes the $L_p$-norm with respect to $Q$, if $Q$ is a measure, or with respect to the distribution of $Q$, if $Q$ is a random variable.

For a function class $\mathcal{F}$, a norm $\|\cdot\|$ on $\mathcal{F}$ and two functions $l$ and $u$, $[l,u] = \{f \in \mathcal{F} ~|~ l \leq f \leq u\}$ is called a bracket, and an $\varepsilon$-bracket (with respect to $\|\cdot\|$) if additionally $\|u-l\| < \varepsilon$. $l$ and $u$ need not be elements of $\mathcal{F}$, but we require $\|u\|, \|l\| < \infty$. The bracketing number $N_{[\,]}(\varepsilon, \mathcal{F}, \|\cdot\|)$ is defined as the minimum number of $\varepsilon$-brackets needed to cover $\mathcal{F}$.

We often consider hyperrectangles in this article and use an interval notation for them. For $a,b \in \mathbb{R}^d$ we use the notation $[a,b] = \prod_{i=1}^d [a_i, b_i]$. $0 \in \mathbb{R}^d$ is the $d$-dimensional $0$-vector, and so if $x \in [0,1]^d$ is a vector, then as a special case of our multidimensional interval notation, $[0,x]$ is the closed hyperrectangle with one vertex at $0$ and its opposite vertex at $x$.

\subsection{The Hardy-Krause Variation}
\label{sec:HK_introduction}
The Hardy-Krause variation \citep{hardy:1906,krause:1903} extends the concept of bounded variation from the univariate case to multivariate functions $h : [0,1]^d \to \mathbb{R}$. We give two equivalent definitions of the Hardy-Krause variation: One which is constructive, and one which is more abstract. The constructive definition is very `hands-on', while the abstract one makes the connection to the indicator functions $\textbf{1}_{[0,x]}$, $x \in [0,1]^d$, more obvious.

Let us begin with the constructive definition. We follow the expositions in \cite{leonov:1998} and \cite{aistleitner_dick:2015}. For a function $h : [0,1]^d \to \mathbb{R}$ and a hyperrectangle $B = \prod_{i=1}^d [a_i, b_i] \subseteq [0,1]^d$, $a_i < b_i$, define the $d$-dimensional quasi-volume of the hyperrectangle $B$ with respect to $h$ by
\begin{align}
  \begin{split}
    \label{eq:definition_quasivolume}
    \Delta(h;B) & = \Delta^{(d)}(h;B)                                                                                                      \\
                & = \sum_{j_1, \ldots, j_d \in \{0,1\}} (-1)^{j_1 + \cdots + j_d} h\{b_1 + j_1(a_1 - b_1), \ldots, b_d + j_d(a_d - b_d)\}.
  \end{split}
\end{align}
If $\Delta^{(d)}(h;B) \geq 0$ for any hyperrectangle $B$, then $h$ is called completely monotone. Next, let $\mathcal{P}$ be the collection of all finite partitions of $[0,1]^d$ into hyperrectangles. The Vitali variation of $h$ on $[0,1]^d$ is defined as
$$
  V^{(d)}(h) = \sup_{P \in \mathcal{P}} \sum_{B \in P}\left|\Delta^{(d)}(h;\bar{B})\right|,
$$
where $\bar{B}$ is the closure of $B$ (simply because we have only defined the quasi-volume for closed hyperrectangles). For $1 \leq i_1 < \ldots < i_m \leq d$, $m = 1, \ldots, d$, let
\begin{equation}
  \label{eq:def_hk_boxes}
  B_{i_1, \ldots, i_m} = \left\{x \in [0,1]^d ~|~ x_j = 1 \textrm{ for } j \notin \{i_1, \ldots, i_m\}\right\}.
\end{equation}
By restricting $h$ to the face $B_{i_1, \ldots, i_m}$ and identifying it with $[0,1]^m$ in the natural way, we obtain functions $f_{i_1, \ldots, i_m} : [0,1]^m \to \mathbb{R}$. The Hardy-Krause variation of $h$ anchored at $1$ (or just Hardy-Krause variation) is defined as
\begin{equation}
  \label{eq:hk_vitali_sum}
  \|h\|_{\mathrm{HK}} = \sum_{m = 1}^d \sum_{1 \leq i_1 < \ldots < i_m \leq d} V^{(m)}(f_{i_1, \ldots, i_m}).
\end{equation}
Described in words, the Hardy-Krause variation of $h$ is obtained by summing the Vitali variations of the restrictions of $h$ to all faces of $[0,1]^d$ which contain $(1, \ldots, 1)$ as a vertex, plus the Vitali variation of $h$ itself. In the special case $d = 1$, the Vitali variation is the usual total variation of $h : [0,1] \to \mathbb{R}$, which we denote by
$$
  \|h\|_{\mathrm{TV}} = \sup \sum_{k=1}^{r} |h(t_k) - h(t_{k-1})|,
$$
where the supremum is taken over all partitions $0 = t_0 < t_1 < \ldots < t_r = 1$. Since the one-dimensional hypercube $[0,1]$ has no lower-dimensional faces, the Hardy-Krause variation is equal to the Vitali variation in this case. Hence, for a univariate function $h : [0,1] \to \mathbb{R}$, we have $\|h\|_{\mathrm{HK}} = \|h\|_{\mathrm{TV}}$.

For differentiable functions, the variations of both Vitali and Hardy-Krause can be controlled in terms of their partial derivatives. More precisely, if $h : [0,1]^d \to \mathbb{R}$ is a function whose mixed partial derivative $(\partial^d / \partial x_1 \cdots \partial x_d) h(x)$ exists, then it holds that
$$
  V^{(d)}(h) \leq \int_{[0,1]^d} \left|\frac{\partial^d}{\partial x_1 \cdots \partial x_d} h(x)\right| ~\mathrm{d}x,
$$
with equality if the mixed partial derivative $(\partial^d / \partial x_1 \cdots \partial x_d) h(x)$ is continuous on $[0,1]^d$ \citep[see][]{frechet:1910,owen:2005}. It then immediately follows from Eq.\@ \eqref{eq:hk_vitali_sum} that
$$
  \|h\|_{\mathrm{HK}} \leq \sum_{m = 1}^d \sum_{1 \leq i_1 < \ldots < i_m \leq d} \int_{[0,1]^d} \left|\frac{\partial^m}{\partial x_1 \cdots \partial x_m} f_{i_1, \ldots, i_m}(x)\right| ~\mathrm{d}x,
$$
with equality if the mixed partial derivatives are all continuous on their respective lower-dimensional hypercubes.

If $h : [0,1]^d \to \mathbb{R}$ is of bounded Hardy-Krause variation, then for any Borel probability measure $\mu$ on $[0,1]^d$, it holds that, for any $x_1, \ldots, x_n \in [0,1]^d$,
\begin{equation}
  \label{eq:koksma_hlawka_original}
  \left|\int h ~\mathrm{d}\mu - \frac{1}{n}\sum_{i=1}^n h(x_i)\right| \leq \|h\|_{\mathrm{HK}} \sup_{a \in [0,1]^d} \left|\frac{1}{n} \sum_{i=1}^n \textbf{1}_{[0,a]}(x_i) - \mu([0,a])\right|.
\end{equation}
This is a general form of the so-called Koksma-Hlawka inequality, proved in this form by \cite{aistleitner_dick:2015}. In that reference, the authors include Borel measurability of $h$ as an assumption; however, that assumption is superfluous, as \cite{aistleitner_etal:2017} later proved that every function of bounded Hardy-Krause variation is Borel measurable (Corollary 4.4 in that reference). By choosing $\mu$ as a Dirac measure and considering a single point $x_1$, we see that $\|h\|_\infty \leq \inf_x |h(x)| + \|h\|_{\mathrm{HK}}$, and so every function of bounded Hardy-Krause variation is bounded. A more general version of this inequality is given by Proposition 3.3 in \cite{pausinger_svane:2015}. The Koksma-Hlawka inequality goes back to \cite{koksma:1942} and \cite{hlawka:1961}, who proved it for the case where $\mu$ is the uniform distribution on $[0,1]$ (Koksma) and $[0,1]^d$ (Hlawka).

By Theorem 3 in \cite{aistleitner_dick:2015}, every right-continuous (in every argument) function $h : [0,1]^d \to \mathbb{R}$ which is of bounded Hardy-Krause variation can be written as $h(x) = \nu([0,x])$, where $\nu$ is a signed Borel measure uniquely determined by $h$. For a measurable $g : [0,1]^d \to \mathbb{R}$, we can now define
\begin{equation}
  \label{eq:definition_hk_integral}
  \int g ~\mathrm{d}h = \int g ~\mathrm{d}\nu.
\end{equation}
One major advantage of the Hardy-Krause variation is that this integral allows for an integration by parts formula similar to the usual Lebesgue-Stieltjes integral in the univariate case \citep[Theorem 15 in][]{radulovic_et_al:2017}.

We now turn to the more abstract definition. As proven by \cite{aistleitner_etal:2017}, the Hardy-Krause variation can be considered a special case of the $\mathcal{D}$-variation introduced by \cite{pausinger_svane:2015}. Let $\mathcal{D}$ be a family of measurable subsets of $[0,1]^d$ with $\emptyset, [0,1]^d \in \mathcal{D}$. Let $\mathcal{S}(\mathcal{D})$ denote the space of $\mathcal{D}$-simple functions, i.e.\@ the set of all $h : [0,1]^d \to \mathbb{R}$ for which there are $\alpha_1, \ldots, \alpha_m \in \mathbb{R}$ and $A_1, \ldots, A_m \in \mathcal{D}$ such that
$$
  h = \sum_{i=1}^m \alpha_i \textbf{1}_{A_i}.
$$
Letting $h_\mathcal{D}(A_i) = 0$ if $A_i \in \{\emptyset, [0,1]^d\}$ and $h_\mathcal{D}(A_i) = 1$ for all other $A_i \in \mathcal{D}$, the $\mathcal{D}$-variation of a $\mathcal{D}$-simple function $h$ is
$$
  V_{\mathcal{S}, \mathcal{D}}(h) = \inf \left\{\sum_{i=1}^m |\alpha_i| h_\mathcal{D}(A_i) ~\Big|~ h = \sum_{i=1}^m \alpha_i \textbf{1}_{A_i}, \alpha_i \in \mathbb{R}, A_i \in \mathcal{D}\right\}.
$$
More generally, if $\mathcal{V}_\infty(\mathcal{D})$ is the space of all functions $h : [0,1]^d \to \mathbb{R}$ which can be uniformly approximated by functions from $\mathcal{S}(\mathcal{D})$, \cite{pausinger_svane:2015} introduce the $\mathcal{D}$-variation of $h \in \mathcal{V}_\infty(\mathcal{D})$ as
$$
  V_\mathcal{D}(h) = \inf \left\{\liminf_{i \to \infty} V_{\mathcal{S}, \mathcal{D}}(f_i) ~|~ f_i \in \mathcal{S}(\mathcal{D}), \|f_i - h\|_\infty \to 0\right\}.
$$
Let $\mathcal{R}^*$ be the set of all hyperrectangles containing $0$ as a vertex, including the empty set. More precisely, set
\begin{equation}
  \label{eq:R_star}
  \mathcal{R}^* = \left\{[0,x]^v ~|~ x \in [0,1]^d, v \subseteq \{1, \ldots, d\}\right\},
\end{equation}
where $[0,x]^v = \{y \in [0,1]^d ~|~ \forall i \in v : 0 \leq y \leq x_i \land \forall i \notin v : 0 \leq y < x_i\}$. \cite{aistleitner_etal:2017} show that the Hardy-Krause variation is identical to the $\mathcal{R}^*$-variation (Corollary 4.3 in that reference). More precisely, they prove that a function $h : [0,1]^d \to \mathbb{R}$ is of bounded Hardy-Krause variation if and only if it is of bounded $\mathcal{R}^*$-variation, and it holds that
\begin{equation}
  \label{eq:hk_Dvar_equivalent}
  V_{\mathcal{R}^*}(h) = \|h\|_{\mathrm{HK}}
\end{equation}
for all functions $h : [0,1]^d \to \mathbb{R}$. The Hardy-Krause variation is therefore a special case of the more general $\mathcal{D}$-variation. Thus, the linear space of functions of bounded Hardy-Krause variation is dense in $\mathcal{V}_\infty(\mathcal{R}^*)$, which is just the closed (with respect to the supremum norm) linear span of the class of all indicator functions $\textbf{1}_R$, $R \in \mathcal{R}^*$. This class is of course strongly related the the class of all indicator functions $\textbf{1}_{[0,x]}$, $x \in [0,1]^d$. It is this connection which we leverage for our results.

\section{Extension of Process Convergence}
\label{sec:mains}
The idea of extending convergence of processes indexed in $\mathcal{F} = \{\textbf{1}_{[0,x]} ~|~ x \in [0,1]^d\}$, to a larger class of functions of bounded Hardy-Krause variation is not new. Related results are derived by \cite{radulovic_et_al:2017} \citep[][also use similar methods, though they work under different assumptions]{berghaus_et_al:2017,beutner_zaehle:2023}. Their main ingredient is a certain integration by parts formula for functions of bounded variation. An informal description of this technique is as follows: If $X_n$ is a sequence of random processes in $\ell^\infty(\mathcal{F})$ whose sample paths satisfy some regularity assumptions -- most importantly, bounded Hardy-Krause variation and right-continuity --, then
$$
  \int h ~\mathrm{d}X_n = \int \phi(X_n) ~\mathrm{d}h
$$
for any $h$ which is sufficiently regular, where $\phi(X_n)$ is a certain continuous transformation of $X_n$ \citep[Theorem 15 in][]{radulovic_et_al:2017}. The left-hand side is the appropriate extension of $X_n$ from the class $\mathcal{F}$ to a larger class of functions of bounded Hardy-Krause variation. Most importantly, if $X_n = r_n(P_n - \mu_n)$ is an empirical process, then the left-hand side is just equal to $X_n(h)$. Under some technical assumptions, and if $X_n \rightsquigarrow X$ in $\ell^\infty(\mathcal{F})$, the right-hand side can be shown to converge in $\ell^\infty(\mathcal{H})$ for suitable function classes $\mathcal{H}$ by a continuous mapping type argument.

This technique has two drawbacks: First, the transformation $\phi(X_n)$ has a rather complicated form. In the limit, $\phi(X_n)$ gets replaced by $\phi(X)$, and while this expression can be simplified when the sample paths of $X$ are continuous almost surely, this of course restricts the applicability of the result. It is possible to make this technique work if the limiting process $X$ does not have continuous sample paths, but the cost for this is that the limiting process $\int \phi(X) ~\mathrm{d}h$, $h \in \mathcal{H}$, will be challenging to analyse further. The second drawback is more fundamental: The function class $\mathcal{H}$, to which the process convergence $X_n \rightsquigarrow X$ is extended, cannot contain all functions of uniformly bounded Hardy-Krause variation. In particular, elements of $\mathcal{H}$ must always satisfy at least some continuity assumptions. This is because integration with respect to a function of bounded Hardy-Krause variation is defined in a Lebesgue-Stieltjes type sense; see its definition \eqref{eq:definition_hk_integral}. Thus, for the integral $\int \phi(X_n) ~\mathrm{d}h$ to be defined, $h$ must be right-continuous since this allows the identification of $h$ with the signed Borel measure $\nu = \mathrm{d}h$. This measure is uniquely determined by $h$ through the identity $h(x) = \nu([0,x])$; $h$ is the distribution function of the signed measure $\nu$. Again, it is possible to weaken this condition slightly, for instance by allowing $h$ to be left-continuous instead of right-continuous. This corresponds to identifying $h$ not with the distribution function $\nu([0,x])$, but instead with the function $x \mapsto \nu([0,x))$. Since this still uniquely determines the signed Borel measure $\nu$, the underlying integration theory remains untouched. If, however, $h$ is neither left- nor right-continuous, then the the set function $\nu$ from the identity $h(x) = \nu([0,x])$ no longer extends to a signed measure on the Borel $\sigma$-algebra. This means that the integral with respect to $\mathrm{d}h$ cannot be defined in a Lebesgue-Stieltjes sense. It might be possible to salvage the integration theory, e.g.\@ by considering integration with respect to finitely additive set functions, but in any case it is unclear whether the integration by parts formula from \cite{radulovic_et_al:2017} would remain applicable to such a generalised integral.

We therefore choose a different approach which circumvents the integration theory altogether. This allows us to extend the convergence $X_n \rightsquigarrow X$ to classes of functions which have no continuity restrictions. Specifically, we will show convergence in the class $\mathcal{HK}(c)$, as defined below.

\begin{definition}
  \label{def:HKc_not_multivariate}
  For fixed dimension $d \in \mathbb{N}$ and some number $0 < c \leq \infty$, we define $\mathcal{HK}(c)$ to be the set of all functions $h : [0,1]^d \to \mathbb{R}$ with $\|h\|_{\mathrm{HK}} < c$ and $\|h\|_\infty \leq c$. The dimension $d$ is suppressed in this notation.
\end{definition}

The reason why the norms $\|h\|_{\mathrm{HK}}$ and $\|h\|_\infty$ are controlled with different inequality signs (strict vs.\@ non-strict) in Definition \ref{def:HKc_not_multivariate} is purely technical and bears no special significance. As a special case of this definition, $\mathcal{HK}(\infty)$ is the linear space of all functions $h : [0,1]^d$ which are of bounded Hardy-Krause variation. Any element of $\mathcal{HK}(\infty)$ actually satisfies the strict inequality $\|h\|_\infty < \infty$, as we have already seen in Section \ref{sec:HK_introduction} that functions of finite Hardy-Krause variation are always bounded.

We can now present our first main result. The obvious application is for the case where $X_n = r_n(P_n - \mu_n)$ is an empirical process, $r_n$ being a rate, $P_n$ random probability measures, and $\mu_n$ deterministic probability measures. But our result is also valid for more general linear processes $X_n$, provided they satisfy some mild continuity assumptions. The theorem has two important aspects: Extension of sample paths, and extension of convergence. First, extension of the sample paths of $X_n$. This is irrelevant if $X_n$ is the empirical process, since then $X_n(h)$ is already well-defined for any bounded and Borel measurable function $h$; but in general, we only require $X_n$ to be defined on a certain subset of $\mathcal{HK}(\infty)$. Our theorem then ensures the existence of a unique extension to the entire space. Second, extension of the process convergence $X_n \rightsquigarrow X$. This is the main part of our theorem: If $X_n \rightsquigarrow X$ in $\ell^\infty(\mathcal{F})$, with $\mathcal{F}$ being the class of $d$-dimensional hyperrectangles, then our theorem guarantees that this convergence can be extended to the richer class $\mathcal{HK}(c)$. Similarly to before, the original limit $X$ need not be defined on $\mathcal{HK}(c)$ a priori, since our theorem also yields the existence of a suitable extension.

\begin{theorem}
  \label{thm:extension_hk_simple}
  Let $\mathcal{F} = \{\textbf{1}_{[0,x]} ~|~ x \in [0,1]^d\}$, and let $\mathcal{G}$ be the closure of $\mathcal{F}$ under monotone convergence, i.e.\@ $\mathcal{G}$ consists of all functions $g : [0,1]^d \to \mathbb{R}$ for which there exists a sequence $f_k \in \mathcal{F}$, $k \in \mathbb{N}$, with $f_k \uparrow g$ as $k \to \infty$. Suppose $X_n$, $n \in \mathbb{N}$, is a sequence of real-valued linear processes on $\mathrm{span}(\mathcal{G})$ such that every $X_n$ satisfies the following assumptions:
  \begin{enumerate}
    \item[(a)] All sample paths of $X_n$ are continuous with respect to monotone convergence in $\mathcal{F}$, i.e.\@ if $f_k$, $k \in \mathbb{N}$, is a sequence in $\mathcal{F}$ such that $f_k \uparrow g$ for some $g \in \mathcal{G}$, then $X_n (f_k) \to X_n (g)$ as $k \to \infty$.
    \item[(b)] All sample paths of $X_n$ are continuous with respect to the supremum norm on the entire space $\mathrm{span}(\mathcal{G})$.
    \item[(c)] $X_n (\textbf{1}_{[0,1]^d}) = 0$.
  \end{enumerate}
  Then every $X_n$ can be uniquely extended to a continuous (with respect to the supremum norm) linear operator $\bar{X}_n$ on the space $\mathcal{HK}(\infty)$ of functions of bounded Hardy-Krause variation. Moreover, if the restrictions of $X_n$ to $\mathcal{F}$ weakly converge to some tight process $X$, i.e.\@ $X_n \rightsquigarrow X$ in $\ell^\infty(\mathcal{F})$, then there exists a copy $X'$ of $X$ with the following properties:
  \begin{enumerate}
    \item $X'$ can be extended to some Borel measurable and tight $\bar{X}' \in \ell^\infty[\mathcal{HK}(c)]$.
    \item $\bar{X}_n \rightsquigarrow \bar{X}'$ in $\ell^\infty[\mathcal{HK}(c)]$ as $n \to \infty$.
    \item\label{it:extension_sample_paths} Almost all sample paths of the extension $\bar{X}'$ are continuous with respect to monotone convergence in $\mathcal{F}$, and uniformly continuous with respect to the supremum norm on $\mathcal{HK}(c)$.
    \item $\bar{X}'$ is obtained by applying a continuous linear operator to $X'$. In particular, if $X$ and thus $X'$ is mean-zero and/or Gaussian, then so is its extension $\bar{X}'$.
  \end{enumerate}
\end{theorem}

It is worth pointing out that we make no assumptions on the original limiting process $X$, apart from it being tight. If the covariance function of $X$ satisfies some mild assumptions, we can deduce from this the covariance function of the extension $\bar{X}'$. If $X$ is Gaussian, this completely determines the limiting process $\bar{X}'$, since it, too, is then Gaussian.

\begin{theorem}
  \label{thm:extension_covariance}
  Let $\mathcal{F}$ and $\mathcal{G}$ be as in Theorem \ref{thm:extension_hk_simple}. Let $\Gamma : \mathcal{HK}(\infty) \times \mathcal{HK}(\infty) \to \mathbb{R}$ be a bilinear function which is (in both arguments) continuous with respect to monotone convergence in $\mathcal{F}$, as defined in Theorem \ref{thm:extension_hk_simple}, and with respect to the supremum norm on $\mathcal{HK}(c)$. Suppose that the original limit $X$ in Theorem \ref{thm:extension_hk_simple} satisfies $\mathbb{E}\|X\|_\mathcal{F}^2 < \infty$ and $\mathrm{Cov}\left[X(f), X(f')\right] = \Gamma(f,f')$ for all $f,f' \in \mathcal{F}$. Then the extended process $\bar{X}'$ has the same covariance function $\Gamma$, i.e.\@
  $$
    \mathrm{Cov}\left[\bar{X}'(h), \bar{X}'(h')\right] = \Gamma(h,h')
  $$
  for all $h,h' \in \mathcal{HK}(c)$. If the original limit $X$ is Gaussian, then the condition $\mathbb{E}\|X\|_\mathcal{F}^2 < \infty$ is always satisfied.
\end{theorem}

Situations to which Theorem \ref{thm:extension_covariance} is applicable are not uncommon. For example, a $P$-Brownian bridge has covariance function $\Gamma(f,g) = P(fg) - Pf Pg$, and this satisfies the assumptions of Theorem \ref{thm:extension_covariance}. More generally, whenever $\Gamma$ is defined in terms of certain integrals, it is worth checking if it satisfies the assumptions of Theorem \ref{thm:extension_covariance}.

We have already explained that the proof of Theorem \ref{thm:extension_hk_simple} does not rely on the integration theory for integrator functions which are of bounded Hardy-Krause variation. Instead, we use a certain generalisation of the Koksma-Hlawka inequality \eqref{eq:koksma_hlawka_original}. One way to state the original Koksma-Hlawka inequality is as follows: If we define the linear operators $X, Y \in \ell^\infty[\mathcal{HK}(\infty)]$ by $X(h) = \int h ~\mathrm{d}\mu$ and $Y(h) = n^{-1} \sum_i h(x_i)$, where $\mu$ is a Borel probability measure, and $x_1, \ldots, x_n$ is any collection of points in $[0,1]^d$, then
$$
  |X(h) - Y(h)| \leq \|h\|_{\mathrm{HK}} \|X-Y\|_\mathcal{F},
$$
where $\mathcal{F}$ is again the class of all indicator functions $\textbf{1}_{[0,x]}$, $x \in [0,1]^d$. Taking the supremum over all $h \in \mathcal{HK}(c)$ for some fixed $c < \infty$ gives us
\begin{equation}
  \label{eq:koksma_hlawka_original_2}
  \|X - Y\|_{\mathcal{HK}(c)} \leq c \|X-Y\|_\mathcal{F}.
\end{equation}
The main application of the Koksma-Hlawka inequality seems to be in (quasi-)Monte Carlo methods, since $|X(h) - Y(h)|$ is the error made in approximating the integral $\int h ~\mathrm{d}\mu$ with the arithmetic mean $n^{-1}\sum_i h(x_i)$. But it is clear from the form \eqref{eq:koksma_hlawka_original_2} that an appropriate version of the Koksma-Hlawka inequality can be be useful in empirical process theory as well. If $X$ and $Y$ were allowed to be more general linear operators -- say, the sample paths of two stochastic processes --, then we could use such an inequality to extend convergence in $\ell^\infty(\mathcal{F})$ to $\ell^\infty[\mathcal{HK}(c)]$. This is precisely what we do to prove Theorem \ref{thm:extension_hk_simple}. The following result is the required Koksma-Hlawka inequality for general linear operators.

\begin{theorem}
  \label{thm:koksma-hlawka_variant}
  Let $\mathcal{F} = \{\textbf{1}_{[0,x]} ~|~ x \in [0,1]^d\}$, and let $\mathcal{G}$ be the closure of $\mathcal{F}$ under monotone convergence (as in Theorem \ref{thm:extension_hk_simple}). Suppose that $X, Y : \mathrm{span}(\mathcal{G}) \to \mathbb{R}$ are two linear maps which satisfy the following conditions:
  \begin{enumerate}
    \item[(a)] Both $X$ and $Y$ are continuous with respect to monotone convergence in $\mathcal{F}$ (as defined in Theorem \ref{thm:extension_hk_simple}).
    \item[(b)] It holds that $X(\textbf{1}_{[0,1]^d}) = Y(\textbf{1}_{[0,1]^d})$.
  \end{enumerate}
  Then it holds that
  \begin{equation}
    \label{eq:koksma-hlawka_variant_inequality_only_span}
    |X(g) - Y(g)| \leq \|g\|_{\mathrm{HK}} \|X - Y\|_{\mathcal{F}}
  \end{equation}
  for all $g \in \mathrm{span}(\mathcal{G})$ (and $\|g\|_{\mathrm{HK}} < \infty$ for all such $g$).

  If we furthermore assume (in addition to the previous conditions) that $X$ and $Y$ are defined on $\mathcal{HK}(c)$ for some $0 < c < \infty$ and continuous with respect to uniform convergence on $\mathcal{HK}(c)$, then it holds that
  \begin{equation}
    \label{eq:koksma-hlawka_variant_inequality}
    \left|X(h) - Y(h)\right| \leq \|h\|_{\mathrm{HK}} \left\|X - Y\right\|_\mathcal{F}
  \end{equation}
  for any $h \in \mathcal{HK}(c)$.
\end{theorem}

Theorem \ref{thm:koksma-hlawka_variant} requires knowledge about continuity properties of both $X$ and $Y$. These assumptions should be compared with statement \ref{it:extension_sample_paths} of Theorem \ref{thm:extension_hk_simple}, which ensures that our version of the Koksma-Hlawka inequality is applicable to the case where one of the linear maps is defined as a weak limit of suitable processes.

While we state Theorem \ref{thm:koksma-hlawka_variant} only for the Hardy-Krause variation, both its statement and its proof carry over almost verbatim to the more general $\mathcal{D}$-variation introduced by \cite{pausinger_svane:2015}. From this, a more general version of Theorem \ref{thm:extension_hk_simple} can be established. This allows for a flexible extension of the concept of this article -- use convergence in a small function class to get convergence in a large one --, which is not limited to the indicator functions $\textbf{1}_{[0,x]}$, $x \in [0,1]^d$, on the one side, and functions of bounded Hardy-Krause variation on the other side.

The remaining results in this section specifically concern empirical processes $r_n(P_n - \mu_n)$, where $r_n \in \mathbb{R}$ is a rate, each $P_n$ is a random probability measure (e.g.\@ an empirical measure), and the $\mu_n$ are deterministic probability measures. For instance, choosing $r_n = \sqrt{n}$ and $\mu_n = P$ for some fixed measure $P$ yields the usual empirical process $\sqrt{n}(P_n - P)$. Our next theorem takes the weak convergence of the empirical process $r_n(P_n - \mu_n)$ and extends it to convergence of the $m$-fold product measure process $r_n(P_n^m - \mu_n^m)$, where $m \in \mathbb{N}$ is some fixed parameter.

Before we state our next theorem, we need to define a suitable function class in which we will establish convergence, and the form of the resulting limiting process.

\begin{definition}
  \label{def:Fm_class}
  Suppose that $(\Omega, \mathcal{A})$ is a measurable space, $\mathcal{F}$ is a class of real-valued and measurable functions defined on $(\Omega, \mathcal{A})$, and $m \in \mathbb{N}$ is an integer.
  \begin{enumerate}
    \item We define the class $\mathcal{F}_m$ by setting $\mathcal{F}_1 = \mathcal{F}$, and, for $m \geq 2$, letting $\mathcal{F}_m$ consist of all functions $h : (\Omega^m, \mathcal{A}^m) \to \mathbb{R}$ with the following properties:
          \begin{enumerate}
            \item $h$ is bounded and measurable.
            \item\label{it:def_Fm_integral} For any product probability measure $Q = Q_1 \otimes \ldots \otimes Q_{m-1}$ on $(\Omega^{m-1}, \mathcal{A}^{m-1})$ and any $i = 1, \ldots, m$, the function
                  $$
                    x \mapsto \int h(x_1, \ldots, x_{i-1}, x, x_{i+1}, \ldots, x_m) ~\mathrm{d}Q(x_1, \ldots, x_{i-1}, x_{i+1}, \ldots, x_m)
                  $$
                  is an element of $\mathcal{F}$.
          \end{enumerate}
    \item For any $G \in \ell^\infty(\mathcal{F})$ and any probability measure $\mu$ on $(\Omega, \mathcal{A})$, we define the process $G_\mu^{(m)} \in \ell^\infty(\mathcal{F}_m)$ by setting $G_\mu^{(1)} = G$ and $G_{\mu}^{(m)}(h) = \sum_{i=1}^m G(h_{\mu, i})$ with
          $$
            h_{\mu, i}(x) = \int h(x_1, \ldots, x_{i-1}, x, x_{i+1}, \ldots, x_m) ~\mathrm{d}\mu^{m-1}(x_1, \ldots, x_{i-1}, x_{i+1}, \ldots, x_m)
          $$
          for $m \geq 2$..
  \end{enumerate}
\end{definition}

We point one important consequence of Property \ref{it:def_Fm_integral}, which is the following: For any $x_1, \ldots, x_m \in \Omega$ and any $i = 1, \ldots, m$, the univariate function
$$
  x \mapsto h(x_1, \ldots, x_{i-1}, x, x_{i+1}, \ldots, x_m)
$$
is an element of $\mathcal{F}$. To see this, just choose the measures $Q_1, \ldots, Q_{m-1}$ in Property \ref{it:def_Fm_integral} as the Dirac measures concentrated on $x_1, \ldots, x_{i-1}, x_{i+1}, \ldots, x_m$.

Let us now state our main result on V-processes.

\begin{theorem}
  \label{thm:v_statistik_from_empirical_process}
  Let $P_n$ and $\mu_n$, $n \in \mathbb{N}$, be sequences of probability measures on some measurable space $(\Omega, \mathcal{A})$, where the $P_n$ are random and the $\mu_n$ are deterministic. Let $\mathcal{F}$ be a class of measurable real-valued functions defined on $(\Omega, \mathcal{A})$. Assume that there exist a real-valued sequence $r_n \uparrow \infty$, a tight random process $G \in \ell^\infty(\mathcal{F})$ and a probability measure $\mu$ such that $r_n(P_n - \mu_n) \rightsquigarrow G$ and $\mu_n \to \mu$ in $\ell^\infty(\mathcal{F})$. Assume that almost all sample paths of $G$ are uniformly continuous with respect to the supremum norm. Then for any $m \in \mathbb{N}$, it holds that $r_n(P_n^m - \mu_n^m) \rightsquigarrow G_\mu^{(m)}$ in $\ell^\infty(\mathcal{F}_m)$, where $G_\mu^{(m)}$ is the process from Definition \ref{def:Fm_class}. The transformation $G \mapsto G_\mu^{(m)}$ is linear and continuous for all $m \in \mathbb{N}$; therefore, if $G$ is Borel measurable and/or tight and/or Gaussian and/or centred, then so is $G_\mu^{(m)}$. If $G$ has covariance function $\Gamma$, then $G_\mu^{(m)}$ has covariance function
  \begin{equation}
    \label{eq:V_kovarianz_statement}
    \Gamma_\mu^{(m)}(h,h') = \sum_{i,j=1}^m \Gamma(h_{\mu,i}, h_{\mu,j}'),
  \end{equation}
  where $h_{\mu,i}$ and $h_{\mu,j}'$ are constructed as in Definition \ref{def:Fm_class}.
\end{theorem}

Theorem \ref{thm:v_statistik_from_empirical_process} seems to be especially useful in combination with Theorem \ref{thm:extension_hk_simple}. By that theorem, if we can establish convergence of the empirical process $r_n(P_n - \mu_n)$ indexed in the class of indicator functions $\textbf{1}_{[0,x]}$, $x \in [0,1]^d$, we may apply Theorem \ref{thm:v_statistik_from_empirical_process} with the choice $\mathcal{F} = \mathcal{HK}(c)$. The result is then convergence of the $m$-fold product measure process $r_n (P_n^m - \mu_n^m)$ indexed in the class $\mathcal{HK}(c)_m$. For this class, condition \ref{it:def_Fm_integral} from Definition \ref{def:Fm_class} can be expressed in a simpler form, leading to the following characterisation.

\begin{theorem}
  \label{thm:characterisation_HKcm}
  Let $0 < c < \infty$ and $d,m \in \mathbb{N}$. A function $h : ([0,1]^d)^m \to \mathbb{R}$ is an element of $\mathcal{HK}(c)_m$ if and only if the following statements hold: $h$ is Borel measurable, and the function
  $$
    x \mapsto h(x_1, \ldots, x_{i-1}, x, x_{i+1}, \ldots, x_m)
  $$
  is an element of $\mathcal{HK}(c)$ for any fixed $x_1, \ldots, x_m \in [0,1]^d$ and $i = 1, \ldots, m$.
\end{theorem}

The assumption that $h$ is Borel measurable is necessary, as it does not follow from the other assumption that $h$ is of uniformly bounded Hardy-Krause variation in each coordinate separately. As an example, take a non-measurable subset $F \subseteq [0,1]^d$ and define the function $h$ by $h(x_1, \ldots, x_m) = 0$ except when $x_1 = \ldots = x_m \in F$, in which case $h(x_1, \ldots, x_m) = 1$.

There is one subtlety which should be kept in mind when using Theorem \ref{thm:v_statistik_from_empirical_process}. One natural choice for $\mu_n$ is the mean function of $P_n$, i.e.\@ $\mathbb{E}P_n$ -- recall that this is defined pointwise by $\mathbb{E}P_n(f) = \mathbb{E}[P_n(f)]$. Theorem \ref{thm:v_statistik_from_empirical_process} then makes a statement about the convergence of the processes $r_n[P_n^m - (\mathbb{E}P_n)^m]$. However, $(\mathbb{E}P_n)^m$ is not the natural centring measure for $P_n^m$, which would be $\mathbb{E}\left[P_n^m\right]$. As an example, suppose that $P_n$ is the empirical measure based on some observations $X_1, \ldots, X_n$. Then $\mathbb{E}\left[P_n^m\right]$ is given by
$$
  \mathbb{E}\left[P_n^m\right](h) = n^{-m} \sum_{1 \leq i_1, \ldots, i_m \leq n} \mathbb{E} h(X_{i_1}, \ldots, X_{i_m}).
$$
On the other hand, $(\mathbb{E}P_n)^m$ is given by
$$
  (\mathbb{E}P_n)^m (h) = n^{-m} \sum_{1 \leq i_1, \ldots, i_m \leq n} \mathbb{E} h\left(X_{i_1}^{(1)}, \ldots, X_{i_m}^{(m)}\right),
$$
where $(X_1^{(i)}, \ldots, X_n^{(i)})$ are independent copies of $(X_1, \ldots, X_n)$. This is perhaps a somewhat odd centring. Observe that
\begin{equation}
  \label{eq:moment_convergence_identity}
  r_n\left\{\mathbb{E}\left[P_n^m\right] - (\mathbb{E}P_n)^m\right\} = \mathbb{E}\left[r_n\left\{P_n^m - (\mathbb{E}P_n)^m\right\}\right].
\end{equation}
Theorem \ref{thm:v_statistik_from_empirical_process} ensures weak convergence of the integrand on the right-hand side. If its limiting process is mean-zero, then assuming a suitable kind of uniform integrability condition will allow us to replace the measure $(\mathbb{E}P_n)^m$ by the more natural choice $\mathbb{E}\left[P_n^m\right]$. This is formalised in the following theorem.

\begin{theorem}
  \label{thm:product_uniformly_integrable}
  Let the assumptions of Theorem \ref{thm:v_statistik_from_empirical_process} be satisfied. If
  \begin{equation}
    \label{eq:moment_convergence_assumption}
    \sup_{n \in \mathbb{N}}\mathbb{E}^* \left\|r_n(P_n - \mu_n)\right\|_\mathcal{F}^p < \infty,
  \end{equation}
  for some $p > 1$, then
  \begin{equation}
    \label{eq:moment_convergence_conclusion_1}
    \sup_{n \in \mathbb{N}}\mathbb{E}^* \left\|r_n(P_n^m - \mu_n^m)\right\|_{\mathcal{F}_m}^p < \infty,
  \end{equation}
  for any $m \in \mathbb{N}$. If we furthermore assume that the original limit process $G$ from the statement of Theorem \ref{thm:v_statistik_from_empirical_process} is tight and mean-zero, then it also holds that
  \begin{equation}
    \label{eq:moment_convergence_conclusion_2}
    \left\|\mathbb{E} r_n(P_n^m - \mu_n^m)\right\|_{\mathcal{F}_m} \xrightarrow[n \to \infty]{} 0
  \end{equation}
  for any $m \in \mathbb{N}$. In this case, if $\mu_n = \mathbb{E}P_n$, the conclusion of Theorem \ref{thm:v_statistik_from_empirical_process} remains valid if we replace $\mu_n^m = (\mathbb{E}P_n)^m$ by $\mathbb{E}[P_n^m]$.
\end{theorem}

Theorem \ref{thm:koksma-hlawka_variant} may be useful for verifying the assumption of Theorem \ref{thm:product_uniformly_integrable}, as it can be used to immediately extend condition \eqref{eq:moment_convergence_assumption} from the class of indicator functions $\textbf{1}_{[0,x]}$, $x \in [0,1]^d$, to the richer class $\mathcal{HK}(c)$. In fact, if $P_n$ is the usual empirical measure, then this can be done using the normal Koksma-Hlawka inequality \eqref{eq:koksma_hlawka_original}.

For the special case $d = 1$, $m = 2$, and constant measures $\mu_n = \mu$ for all $n$, \cite{beutner_zaehle:2012,beutner_zaehle:2014} give results which are similar to the technique described here. The authors of these papers only consider single U- or V-statistics, as opposed to the V-processes from Theorem \ref{thm:v_statistik_from_empirical_process}, and use a different notion of variation; namely, functions of locally bounded variation. Importantly, these results are also based on an integration-by-parts approach, similar to the one outlined in the beginning of this section.

We close this section by pointing out a small but important detail. The central condition in Theorem \ref{thm:characterisation_HKcm} is that $h$ has uniformly bounded Hardy-Krause variation in every coordinate separately. This is not the same as requiring that $h$, defined as a function on $([0,1]^d)^m = [0,1]^{dm}$ is of bounded Hardy-Krause variation. For instance, take $d = 1$, $m = 2$, and $h(x,y) = \textbf{1}\{x \leq y\}$. Then $x \mapsto h(x,y)$ has Hardy-Krause variation $1$ for any fixed $y \in [0,1]$, as does $y \mapsto h(x,y)$ for any fixed $x \in [0,1]$. Yet the full function $h$ is not of bounded Hardy-Krause variation; this follows from \cite{pausinger_svane:2015}, Proposition 3.10, and \cite{aistleitner_etal:2017}, Corollary 4.3. This is of course only a toy example, but in our study of Chatterjee's rank correlation in Section \ref{sec:mains_chatterjee}, we will encounter kernels with similar properties. The results in that section would be impossible if Theorem \ref{thm:v_statistik_from_empirical_process} were valid only for kernels $h$ which are of bounded Hardy-Krause variation as full functions on the $dm$-dimensional hypercube $[0,1]^{dm}$.

\section{Applications}
\subsection{Chatterjee's Rank Correlation}
\label{sec:mains_chatterjee}
Throughout this section, we assume that $(X_k,Y_k)$, $k \in \mathbb{N}$, are i.i.d.\@ copies of some random vector $(X,Y)$ with values in $[0,1]^2$. This induces no loss of generality for our results on Chatterjee's rank correlation, since we can always scale down both $X$ and $Y$ in some strictly monotonic way to the unit interval, and Chatterjee's rank correlation is invariant under strictly monotone transformations.

Recall from the introduction that, for an i.i.d.\@ sample $(X_1, Y_1), \ldots, (X_n, Y_n)$, we denote by $X_{n,1}', \ldots, X_{n,n}'$ the order statistics of $X_1, \ldots, X_n$ with ties broken at random, $Y_{n,1}', \ldots, Y_{n,n}'$ are their concomitants, $r_i = \sum_j \textbf{1}\{Y_{n,j}' \leq Y_{n,i}'\}$ and $l_i = \sum_j \textbf{1}\{Y_{n,j}' \geq Y_{n,i}'\}$. For our main theorem on Chatterjee's rank correlation, we need to extend its definition slightly. Recall that
$$
  \xi_n(X,Y) = 1 - \frac{n \sum_{i=1}^{n-1} |r_{i+1} - r_i|}{2 \sum_{i=1}^n l_i (n-l_i)}.
$$
If $Y_1 = \ldots = Y_n$, then this expression is not defined, since then $l_i = n$. In order to turn $\xi_n$ into an integrable random variable, we need to extend the definition of $\xi_n$ to this case. Let us therefore set $\xi_n(X,Y) = 1$ if $Y_1 = \ldots = Y_ n$. The precise value used here is not important; what is important is that it is some finite real number to ensure integrability of $\xi_n$.

\begin{theorem}
  \label{thm:asymptotik}
  If $Y$ is not almost surely constant, then
  $$
    \sqrt{n}\left(\xi_n - \mathbb{E}\xi_n\right) \rightsquigarrow \mathcal{N}\left(0, \sigma^2\right).
  $$
  A formula for the limiting variance $\sigma^2$ is given below in Corollary \ref{cor:chatterjee_variance}.
\end{theorem}

It should be pointed out that our result only concerns Chatterjee's original rank correlation, and not the multivariate extension given by \cite{azadkia_chatterjee:2021}. We believe that it is in principle possible to use our methods from Section \ref{sec:mains} to derive a corresponding result for the Azadkia-Chatterjee measure of dependence. However, at this point, our methods used to derive the initial weak convergence indexed in $\textbf{1}_{[0,x]}$, $x \in [0,1]^2$, do not immediately carry over to the case where the $X_1, \ldots, X_n$ are multidimensional, in which case the order statistics are replaced by nearest neighbours.

Let us give an overview of the proof idea. Our approach is to describe the ratio $1 - \xi_n$ in terms of V-statistics. Consider first the sum in the numerator. Writing $\mathrm{sgn}(x)$ for the sign of a real number $x$, we have
\begin{align*}
  \sum_{i=1}^{n-1} |r_{i+1} - r_i| & = \sum_{i=1}^{n-1} \mathrm{sgn}(r_{i+1} - r_i)(r_{i+1} - r_i) = \sum_{i=1}^{n-1} \mathrm{sgn}(Y'_{n,i+1} - Y'_{n,i})(r_{i+1} - r_i)                                   \\
                                   & = \sum_{i=1}^{n-1} \sum_{j=1}^{n} \mathrm{sgn}(Y'_{n,i+1} - Y'_{n,i})\left(\textbf{1}\{Y_{n,j}' \leq Y_{n,i+1}'\} - \textbf{1}\{Y_{n,j}' \leq Y_{n,i}'\}\right)       \\
                                   & = \sum_{i,j=1}^{n-1} \mathrm{sgn}(Y'_{n,i+1} - Y'_{n,i})\left(\textbf{1}\{Y_{n,j}' \leq Y_{n,i+1}'\} - \textbf{1}\{Y_{n,j}' \leq Y_{n,i}'\}\right) + \mathcal{O}(n-1)
\end{align*}
The denominator can be similarly written as
$$
  \sum_{i=1}^n l_i (n-l_i) = \sum_{i,j,k = 1}^n \textbf{1}\{Y_{n,j}' \geq Y_{n,i}'\} \textbf{1}\{Y_{n,k}' < Y_{n,i}'\}.
$$
With the notation
\begin{align*}
  h_1((s_1, s_2), (t_1, t_2)) & = \mathrm{sgn}(s_2 - s_1)\left(\textbf{1}\{t_1 \leq s_2\} - \textbf{1}\{t_1 \leq s_1\}\right), \\
  h_2(s,t,u)                  & = \textbf{1}\{t \geq s\} \textbf{1}\{u < s\},
\end{align*}
we can therefore write
\begin{align*}
  \frac{n \sum_{i=1}^{n-1} |r_{i+1} - r_i|}{2 \sum_{i=1}^n l_i (n-l_i)} = \frac{n^3 V_1 + \mathcal{O}\left(n^2\right)}{2 n^3 V_2} = \frac{V_1}{2 V_2} + \mathcal{O}\left(\frac{1}{n}\right),
\end{align*}
where $V_1$ denotes the V-statistic with kernels $h_1$ on $(Y_{n,1}', Y_{n,2}'), \ldots, (Y_{n,n-1}', Y_{n,n}')$, and $V_2$ the V-statistic with kernel $h_2$ on $Y_1, \ldots, Y_n$. We will establish the joint convergence of $V_1$ and $V_2$ with our methods from Section \ref{sec:mains}. For these methods to be applicable in a formal sense, we need to redefine the kernels $h_1$ and $h_2$ in such a way that they have the same degree $m$ and their data have the same dimension $d$. Thus, let us write
\begin{align}
  \begin{split}
    \label{eq:h_stern_definition}
    h_1^*\{(s_1, s_2), (t_1, t_2), (u_1, u_2)\} & = \mathrm{sgn}(s_2 - s_1)\left(\textbf{1}\{t_1 \leq s_2\} - \textbf{1}\{t_1 \leq s_1\}\right), \\
    h_2^*\{(s_1, s_2), (t_1, t_2), (u_1, u_2)\} & = \textbf{1}\{t_1 \geq s_1\} \textbf{1}\{u_1 < s_1\}.
  \end{split}
\end{align}
With this notation, $V_1$ and $V_2$ are also the V-statistics with kernels $h_1^*$ and $h_2^*$, respectively, based on the data $(Y_{n,1}', Y_{n,2}'), \ldots, (Y_{n,n-1}', Y_{n,n}')$, and we are in a position to employ the machinery from Section \ref{sec:mains}.

Let $P_n$ denote the empirical measure of $(Y_{n,1}', Y_{n,2}'), \ldots, (Y_{n,n-1}', Y_{n,n}')$. Before we can establish convergence of the V-process associated with $P_n$, we first need to introduce a certain measure which will show up as the limit of the random sequence $P_n$. We will reference the following definition whenever necessary.

\begin{definition}
  \label{def:centring_P}
  Define a measure $P$ on $\mathbb{R}^2$ in the following way: If $(X_1, Y_1)$ and $(X_2, Y_2)$ are two independent random vectors with the same distribution and $Y_1, Y_2 \in \mathbb{R}$, then we set
  \begin{equation}
    \label{eq:conditional_measure}
    P(A) = \int \mathbb{E}\left[\textbf{1}_A(Y_1, Y_2) ~|~ (X_1, X_2) = (x,x)\right] ~\mathrm{d}\mathbb{P}^X(x).
  \end{equation}
\end{definition}

\begin{theorem}
  \label{thm:prozesskonvergenz_klammer}
  Let $P_n$ be the empirical measure of $(Y_{n,1}', Y_{n,2}'), \ldots, (Y_{n,n-1}', Y_{n,n}')$, and let $P$ be the measure from Definition \ref{def:centring_P}. Then, for any $m \in \mathbb{N}$ and $0 < c < \infty$,
  $$
    \sqrt{n}\left(P_n^m - \mathbb{E}P_n^m\right) \rightsquigarrow G_P^{(m)}
  $$
  in $\ell^\infty[\mathcal{HK}(c)_m]$ for some tight mean-zero Gaussian process $G_P^{(m)}$ depending on $m$ and $P$. Finally, it holds that
  \begin{equation}
    \label{eq:concomitants_p_moments}
    \sup_{n \in \mathbb{N}} \mathbb{E}^* \left\|\sqrt{n}\left(P_n^m - \mathbb{E}P_n^m\right)\right\|_{\mathcal{HK}(c)_m}^p < \infty
  \end{equation}
  for sufficiently large $p$.
\end{theorem}
By considering the evaluations of the process $\sqrt{n}(P_n^m - \mathbb{E}P_n^m)$ at a specific collection of functions $h_1, \ldots, h_k \in \mathcal{HK}(c)_m$, we obtain joint convergence in distribution of the corresponding V-statistics. The class $\mathcal{HK}(c)_m$ for $m = 3$ contains the kernels associated with Chatterjee's rank correlation. Unfortunately, the covariance function of the process $G_P^{(m)}$ is rather complicated, which also translates into a complicated limiting variance of Chatterjee's rank correlation.

\begin{theorem}
  \label{thm:kovarianzfunktion_konkomitantenprozess}
  The process $G_P^{(m)}$ from Theorem \ref{thm:prozesskonvergenz_klammer} has covariance function
  $$
    \Lambda_P^{(m)}(h,h') =  \Gamma\left(\sum_{i=1}^m h_{P,i}, \sum_{j=1}^m h_{P,j}'\right) + \beta\left(\sum_{i=1}^m h_{P,i}, \sum_{j=1}^m h_{P,j}'\right),
  $$
  where $h_{P,i}$ and $h'_{P,j}$ are constructed from $h$ and $h'$ as in Definition \ref{def:Fm_class}, and the functions $\Gamma$ and $\beta$ are given by
  \begin{align*}
    \Gamma(h, h') & = \int \mathbb{E}[h(Y_1, Y_2) h'(Y_1, Y_2) ~|~ (X_1, X_2) = (x,x)] ~\mathrm{d}\mathbb{P}^X(x)                     \\
                  & \quad + \int \mathbb{E}[h(Y_1, Y_2) h'(Y_3, Y_1) ~|~ (X_1, X_2, X_3) = (x,x,x)] ~\mathrm{d}\mathbb{P}^X(x)        \\
                  & \quad + \int \mathbb{E}[h(Y_1, Y_2) h'(Y_2, Y_3) ~|~ (X_1, X_2, X_3) = (x,x,x)] ~\mathrm{d}\mathbb{P}^X(x)        \\
                  & \quad - 3 \int \mathbb{E}[h(Y_1, Y_2) h'(Y_3, Y_4) ~|~ (X_1, \ldots, X_4) = (x,x,x,x)] ~\mathrm{d}\mathbb{P}^X(x)
  \end{align*}
  and
  \begin{align*}
    \beta(h, h')
     & = \int \mathbb{E}[h(Y_1, Y_2) h'(Y_3, Y_4) ~|~ (X_1, \ldots, X_4) = (x,x,x,x)]~\mathrm{d}\mathbb{P}^X(x)          \\
     & \quad - \left\{\int \mathbb{E}[h(Y_1, Y_2) ~|~ (X_1, X_2) = (x,x)] ~\mathrm{d}\mathbb{P}^X(x)   \right.           \\
     & \qquad\qquad \left.\cdot \int \mathbb{E}[h'(Y_1, Y_2) ~|~ (X_1, X_2) = (x,x)] ~\mathrm{d}\mathbb{P}^X(x)\right\}.
  \end{align*}
\end{theorem}

\begin{corollary}
  \label{cor:chatterjee_variance}
  The variance $\sigma^2$ in Theorem \ref{thm:asymptotik} is given by
  $$
    \sigma^2 = \frac{1}{4\mu_2^2}\left\{\sigma_{11} - 2\frac{\mu_1}{\mu_2} \sigma_{12} + \left(\frac{\mu_1}{\mu_2}\right)^2 \sigma_{22}\right\},
  $$
  where, with $\Lambda_P^{(3)}$ being the covariance function from Theorem \ref{thm:kovarianzfunktion_konkomitantenprozess}, $\sigma_{ij} = \Lambda_P^{(3)}(h_i^*, h_j^*)$, and
  \begin{align*}
    \mu_{1} & = \int \mathbb{P}\left[Y_1 \land Y_2 < Y_3 \leq Y_1 \lor Y_2 ~|~ (X_1, X_2) = (x,x)\right] ~\mathrm{d}\mathbb{P}^X(x), \\
    \mu_{2} & = \mathbb{P}(Y_1 < Y_2 \leq Y_3).
  \end{align*}
\end{corollary}

If the joint distribution of $X$ and $Y$ is continuous, we know that $\sigma^2 = 0$ if and only if $Y = f(X)$ almost surely for some Borel measurable function $f$ \citep{lin_han:2022}. We do not have a corresponding result in the general case. An alternative expression for $\sigma^2$ is
$$
  \sigma^2 = \frac{1}{4\mu_2^2} \left\langle v_\mu, \Sigma v_\mu\right\rangle = \frac{1}{4\mu_2^2}\left\|\Sigma^{1/2}v_\mu\right\|_2^2,
$$
where $v_\mu = (-1, \mu_1/\mu_2)^\top$ and $\Sigma = (\sigma_{ij})_{1 \leq i,j \leq 2}$. From this it is clear that $\sigma^2 = 0$ if and only if $v_\mu$ is an element of the null space $\mathrm{ker}(\Sigma^{1/2}) = \mathrm{ker}(\Sigma)$, or equivalently if $\Sigma v_\mu = 0$. It is not difficult to see that, if $Y = f(X)$ almost surely, then $\mu_1 = 0$; furthermore, in this case, $\Gamma(h,h') = 0$ and $\beta(h,h') = \mathrm{Cov}[h(Y,Y), h'(Y,Y)]$ for all $h,h'$. By direct calculations, one then finds that $\sigma_{11} = \sigma_{12} = \sigma_{21} = 0$, and so
$$
  \Sigma v_\mu = \begin{pmatrix}
    0 & 0 \\ 0 & \sigma_{22}
  \end{pmatrix} \begin{pmatrix}
    -1 \\ 0
  \end{pmatrix} = \begin{pmatrix}
    0 \\ 0
  \end{pmatrix},
$$
which implies that $\sigma^2 = 0$. Thus, if $Y = f(X)$ almost surely for some Borel measurable $f$, it must hold that $\sigma^2 = 0$. We do not know how to prove the reverse implication (or indeed if it is always true).

In theory, the formula in Corollary \ref{cor:chatterjee_variance} might be used to construct a consistent estimator for $\sigma^2$. However, for statistical purposes, we recommend estimating $\sigma^2$ via the $m$ out of $n$ bootstrap described in \cite{dette_kroll:2025}. In that paper, the performance of the bootstrap estimator was evaluated (and found to be convincing) in finite samples for both continuous and discrete distributions.

The statistics in Theorems \ref{thm:asymptotik} and \ref{thm:prozesskonvergenz_klammer} are both mean-zero; their centring terms are $\mathbb{E}\xi_n$ and $\mathbb{E}P_n^m$, respectively. Sometimes this is not desirable; for instance, if Theorem \ref{thm:asymptotik} is used to construct confidence intervals, we want these to be centred around $\xi$ and not necessarily around $\mathbb{E}\xi_n$. We now give some sufficient conditions under which the centring terms $\mathbb{E}\xi_n$ and $\mathbb{E}P_n^m$ can be replaced by their population versions $\xi$ and $P^m$, respectively, where $P$ is the measure introduced in Definition \ref{def:centring_P}.

To state our conditions, we define the functions $f_a$, $a \in [0,1]$, by
$$
  f_a(x) = \mathbb{P}(Y \leq a ~|~ X = x).
$$
We denote the total variation of a univariate function $f : [0,1] \to \mathbb{R}$ by $\|f\|_{\mathrm{TV}}$. Finally, for a discrete distribution with infinitely many points of mass and associated probabilities $p_1 \geq p_2 \geq \ldots$, we define
$$
  \alpha(x) = \max \{j \in \mathbb{N} ~|~ p_j \geq 1/x\}.
$$

\begin{theorem}
  \label{thm:bias_conditions}
  Let us write $\delta_n = \mathbb{E}\xi_n - \xi $ and $\Delta_n = \mathbb{E}P_n^m - P^m$. Under any of the following conditions, it holds that $\sqrt{n}\Delta_n \to 0$ in $\ell^\infty[\mathcal{HK}(c)_m]$ and $\delta_n \to 0$ in $\mathbb{R}$:
  \begin{enumerate}
    \item \label{it:bounded_tv} The total variations $\|f_a\|_{\mathrm{TV}}$ are bounded uniformly in $a \in [0,1]$.
    \item \label{it:discrete_finite} $X$ has a discrete distribution with finitely many points of mass.
    \item \label{it:discrete_infinite} $X$ has a discrete distribution with infinitely many points of mass and $\alpha(x) = x^\gamma L(x)$ for some $\gamma < 1/2$ and a slowly varying function $L$.
  \end{enumerate}
\end{theorem}

Neither of the conditions in Theorem \ref{thm:bias_conditions} is necessary for the biases to be negligible. On the other hand, if no conditions at all are imposed on the joint distribution of $X$ and $Y$, then the biases may in fact decay too slowly. This is illustrated by the following result.

\begin{theorem}
  \label{thm:bias_conditions_large}
  There exists a probability distribution $F$ on $\mathbb{R}^2$ with the following property: If $(X,Y)$ is distributed according to $F$, then
  $$
    |\sqrt{n}\delta_n| \xrightarrow[n \to \infty]{} \infty \quad \textrm{and} \quad \|\sqrt{n}\Delta_n\|_{\mathcal{HK}(c)_m} \xrightarrow[n \to \infty]{} \infty.
  $$
\end{theorem}

The fact that $\xi_n$ is not an unbiased estimator for $\xi$ was also encountered by \cite{lin_han:2022}, who prove convergence of $\sqrt{n}(\xi_n - \mathbb{E}\xi_n)$ and then use results by \cite{azadkia_chatterjee:2021} to control the bias $\mathbb{E}\xi_n - \xi$. Using this method, \cite{lin_han:2022} are able to obtain weak convergence of $\sqrt{n}(\xi_n - \xi)$ under two conditions: First, $X$ must satisfy a certain tail bound. Second, the functions $x \mapsto \mathbb{P}(Y \geq t ~|~ X = x)$ must be locally Lipschitz in the sense that there are constants $C, \beta > 0$ such that
$$
  \left|\mathbb{P}(Y \geq t ~|~ X = x) - \mathbb{P}(Y \geq t ~|~ X = x')\right| \leq C \left(1 + |x|^\beta + |x'|^\beta\right)|x - x'|
$$
for all $x,x',t \in \mathbb{R}$. Because \cite{lin_han:2022} consider continuous data, this is equivalent to the functions $f_a$ satisfying the same local Lipschitz condition. Since $\xi$ and $\xi_n$ are invariant under strictly monotone transformations, we can assume without loss of generality that $|x|, |x'| \leq 1$, in which case the local Lipschitz condition implies regular Lipschitz continuity. Lipschitz continuity on $[0,1]$ implies bounded variation on $[0,1]$, but not vice-versa, and so Theorem \ref{thm:bias_conditions} (\ref{it:bounded_tv}) is a relaxation of the local Lipschitz assumption.

\subsection{U- and V-Processes of Strongly Mixing Data}
\label{sec:mixing}
The application in Section \ref{sec:mains_chatterjee} highlighted the strength of our method when non-standard data, such as the concomitants of the order statistics used in the construction of Chatterjee's rank correlation, are involved. In this section, we consider a different application to illustrate that our method can also lead to stronger results in more standard cases.

Namely, we consider V-processes based on strongly mixing data. These are processes of the form $V_n(h)$, $h \in \mathcal{H}$, where for each fixed $h \in \mathcal{H}$, $V_n(h)$ is a V-statistic with kernel $h$. Closely related are U-processes, i.e.\@ processes of the form $U_n(h)$, $h \in \mathcal{H}$, where each $U_n(h)$ is a U-statistic. U-processes often occur in statistical application, such as dimension estimation, goodness-of-fit testing or density estimation. Under suitable moment conditions, V-processes and U-processes are equivalent. For more information, we refer to \cite{arcones_gine:1993}, \cite{nolan_pollard:1987,nolan_pollard:1988} for the i.i.d.\@ case, and to \cite{arcones_yu:1994}, \cite{borovkova_burton_dehling:2001} for the weakly dependent case. A nice introduction with regards to the weakly dependent case is given in Section 5 in \cite{dehling:emirical_process_techniques}. These references also include some concrete examples for statistical applications of U-processes.

For two sigma algebras $\mathcal{A}$ and $\mathcal{B}$, define
$$
  \alpha(\mathcal{A}, \mathcal{B}) = \sup_{A \in \mathcal{A}, B \in \mathcal{B}} | \mathbb{P}(A \cap B) - \mathbb{P}(A) \mathbb{P}(B)|.
$$
For a strictly stationary sequence of random variables $(X_k)_{k \in \mathbb{N}}$, define
$$
  \alpha(n) = \sup_{j \in \mathbb{N}} \alpha\left(\mathcal{F}_1^j, \mathcal{F}_{j+n}^\infty\right),
$$
where for $a \leq b \leq \infty$, $\mathcal{F}_a^b = \sigma(X_k ~|~ a \leq k \leq b)$. We call $(X_k)_{k \in \mathbb{N}}$ strongly mixing or $\alpha$-mixing if $\alpha(n) \to 0$ for $n \to \infty$. The standard reference for mixing conditions such as $\alpha$-mixing is \cite{bradley:mixing_vol1}.

\begin{theorem}
  \label{thm:v_statistik_alpha_mixing}
  Fix some $m \in \mathbb{N}$ and $0 < c < \infty$. Let $(X_k)_{k \in \mathbb{N}}$ be a strictly stationary process of $\mathbb{R}^d$-valued bounded random variables. Let $P$ denote the distribution of $X_1$ and define the processes $V_n = (V_n(h))_{h \in \mathcal{HK}(c)_m}$ and $U_n = (U_n(h))_{h \in \mathcal{HK}(c)_m}$ by
  $$
    V_n(h) = \frac{1}{n^m} \sum_{1 \leq i_1, \ldots, i_m \leq n} h(X_{i_1}, \ldots, X_{i_m}) - P(h)
  $$
  and
  $$
    U_n(h) = \frac{1}{n (n-1) \cdots (n-m+1)} \sum_{1 \leq i_1, \ldots, i_m \leq n : i_k \neq i_l \forall k \neq l} h(X_{i_1}, \ldots, X_{i_m}) - P(h).
  $$
  If $(X_k)_{k \in \mathbb{N}}$ is strongly mixing with a mixing rate satisfying $\alpha(n) = \mathcal{O}\left(n^{-r}\right)$ for some $r > 1$, then $\sqrt{n} V_n \rightsquigarrow G$ in $\ell^\infty[\mathcal{HK}(c)_m]$ for some tight mean-zero Gaussian process $G$, and $\sqrt{n}|U_n - V_n| \to 0$ in $\ell^\infty[\mathcal{HK}(c)_m]$.
\end{theorem}
Of course an immediate consequence of Theorem \ref{thm:v_statistik_alpha_mixing} is the weak convergence of the U-process, $U_n \rightsquigarrow G$ in $\ell^\infty[\mathcal{HK}(c)_m]$.

While there are many results for U-statistics and U-processes based on weakly dependent data, most of these assume dependence structures which are often stronger than strong mixing, e.g.\@ absolute regularity (also called $\beta$-mixing). This is also true for the U-process references cited above. In the context of single U-statistics, \cite{yoshihara:1992} proves a sequential U-process result under strong mixing, but his assumptions are difficult to verify. Arguably the most popular result for U-statistics of strongly mixing (real-valued) data is given by \cite{dehling_wendler:2010}, who investigate non-degenerate U-statistics of order $m = 2$ under the so-called $P$-Lipschitz condition. For a strictly stationary and real-valued process $(X_k)_{k \in \mathbb{N}}$, a kernel $h : \mathbb{R}^2 \to \mathbb{R}$ is called $P$-Lipschitz continuous if there is some constant $L > 0$ such that
$$
  \mathbb{E}\left[\left|h(X,Y) - h(X',Y)\right| \textbf{1}\{|X-X'| \leq \varepsilon\}\right] \leq L\varepsilon
$$
for all $\varepsilon > 0$ and all random variables $X,X',Y$ for which the following hold:
\begin{enumerate}
  \item $(X,Y) \sim \mathcal{L}(X_1, X_k)$ for some $k \in \mathbb{N}$ or $(X,Y) \sim \mathcal{L}(X_1) \otimes \mathcal{L}(X_1)$, and
  \item $X' \sim \mathcal{L}(X_1)$.
\end{enumerate}
\cite{dehling_wendler:2010} prove that, under some further technical conditions, a U-statistic of strongly mixing data is asymptotically normal, provided its kernel $h$ is $P$-Lipschitz continuous. Clearly, $P$-Lipschitz continuity depends on the sample generating process $(X_k)_{k \in \mathbb{N}}$ through the distributions $\mathcal{L}(X_1, X_k)$, $k \in \mathbb{N}$. If not all of these joint distributions are continuous, then the $P$-Lipschitz continuity assumption can fail or be more difficult to verify. In comparison, Theorem \ref{thm:v_statistik_alpha_mixing} works for multivariate data, and the conditions on the kernel do not depend on the distribution of the sample generating process.

Consider, for instance, the kernel functions $h_r : \mathbb{R}^{2d} \to \mathbb{R}$ defined by
$$
  h_r(x,y) = \textbf{1}\{\|x-y\| \leq r\},
$$
with $\|\cdot\|$ denoting the maximum norm. The U-process indexed in $\{h_r ~|~ 0 < r < r_0\}$, where $r_0$ is a fixed parameter, plays a central role in estimating the correlation dimension of an attractor of a dynamical system via the Grassberger-Procaccia estimator \citep{grassberger_procaccia:1983}, which can ultimately be used to test whether a time series is random, or if it is merely chaotic but deterministic; see also Section 5.1 in \cite{dehling:emirical_process_techniques}. For $d = 1$, Example 1.6 in \cite{dehling_wendler:2010} gives a sufficient condition for the kernels $h_r$ to be $P$-Lipschitz continuous, but this sufficient condition is easily violated if the marginal distribution of the sample generating process has a discrete component. On the other hand, the kernel functions $h_r$ are elements of $\mathcal{HK}(c)_2$ for some sufficiently large $c$ depending only on $d$. One way to prove this is by induction over $d$, using the identity
$$
  \textbf{1}\{\|x-y\| \leq r\} = \textbf{1}\left\{\max_{j=1, \ldots, d} |x_j - y_j| \leq r\right\} \textbf{1}\{|x_{d+1} - x_{d+1}| \leq r\},
$$
and the fact that $\mathcal{HK}(\infty)$ is closed under pointwise multiplication \citep[with an explicit bound on the Hardy-Krause variation of the product; cf.\@ Theorem 3.7 in][]{pausinger_svane:2015}. Theorem \ref{thm:v_statistik_alpha_mixing} therefore gives us convergence of the process $U_n(h_r)$, $0 < r < r_0$, in $\ell^\infty[\mathcal{HK}(c)_2]$ regardless of the distribution of the process $(X_k)_{k \in \mathbb{N}}$.

The final statistic on which we want to illustrate our method is Kendall's $\tau$. Assume that $(X_1, Y_1), \ldots, (X_n, Y_n)$ are not necessarily independent copies of some generic random vector $(X,Y) \in \mathbb{R}^2$. In the presence of ties, Kendall's $\tau$ is commonly defined as
$$
  \hat{\tau}_b = \frac{C_n - D_n}{\sqrt{[n(n-1)/2 - T_n] [n(n-1)/2 - U_n]}},
$$
where $C_n$ and $D_n$ are the numbers of concordant and discordant pairs in $(X_1, Y_1), \ldots, (X_n, Y_n)$, respectively, and $T_n$ and $U_n$ are the numbers of ties in $X_1, \ldots, X_n$ and $Y_1, \ldots, Y_n$, respectively; see Chapter 3 in \cite{kendall:rank_correlation_methods}. If there are no ties, then $\hat{\tau}_b$ reduces to the usual definition of Kendall's $\tau$. $\hat{\tau}_b$ is a consistent estimator for the population version
$$
  \tau_b = \frac{\mathbb{P}\{(X - \tilde{X})(Y - \tilde{Y}) > 0\} - \mathbb{P}\{(X - \tilde{X})(Y - \tilde{Y}) < 0\}}{\sqrt{\mathbb{P}(X \neq \tilde{X})\mathbb{P}(Y \neq \tilde{Y})}},
$$
where $(\tilde{X}, \tilde{Y})$ is an independent copy of $(X,Y)$. If $X$ and $Y$ are independent, then $\tau_b = 0$. All of the quantities $C_n$, $D_n$, $T_n$ and $U_n$ can be expressed as V-statistics up to some scaling. Specifically,
\begin{align*}
   & C_n = \frac{1}{2} \sum_{1 \leq i,j \leq n} \textbf{1}\{(X_i - X_j)(Y_i - Y_j) > 0\}, \quad  &  & T_n = \frac{1}{2} \sum_{1 \leq i,j \leq n} \textbf{1}\{X_i = X_j\} + \mathcal{O}(n), \\
   & D_n = \frac{1}{2} \sum_{1 \leq i,j \leq n} \textbf{1}\{(X_i - X_j)(Y_i - Y_j) < 0\},  \quad &  & U_n = \frac{1}{2} \sum_{1 \leq i,j \leq n} \textbf{1}\{Y_i = Y_j\} + \mathcal{O}(n).
\end{align*}
These kernel functions all lie in $\mathcal{HK}(c)_2$ for some sufficiently large $c > 0$, which gives rise to the following result.

\begin{corollary}
  \label{cor:kendalls_tau}
  Let $(X_k, Y_k)$, $k \in \mathbb{N}$, be strictly stationary but not necessarily independent copies of some random vector $(X,Y)$. Assume that $(X_k, Y_k)_{k \in \mathbb{N}}$ is strongly mixing with mixing rate $\alpha(n) = \mathcal{O}\left(n^{-r}\right)$ for some $r > 1$ and that neither $X$ nor $Y$ are almost surely constant. If $\hat{\tau}_b = \hat{\tau}_b(n)$ is calculated on the data $(X_1, Y_1), \ldots, (X_n, Y_n)$, then
  $$
    \sqrt{n}\left(\hat{\tau}_b - \tau_b\right) \rightsquigarrow \mathcal{N}\left(0, \sigma^2\right)
  $$
  for some $\sigma^2 \geq 0$ depending on the distribution of $(X_k, Y_k)_{k \in \mathbb{N}}$.
\end{corollary}

\section{Proofs for Process Convergence Extensions}
\label{sec:proofs_empirical_processes}

\subsection{Functions of Bounded Hardy-Krause Variation}
\label{sec:hk_proofs}
We begin with the proof of Theorem \ref{thm:koksma-hlawka_variant}, the generalised Koksma-Hlawka inequality, which is the backbone of our extension results. The proof idea is simple: We will first prove our claim for functions in $\mathrm{span}(\mathcal{G})$ by elementary means. Next, we will use the continuity assumptions on the operators $X$ and $Y$, together with the fact that $\mathrm{span}(\mathcal{G})$ is dense in $\mathcal{HK}(\infty)$, to prove the final claim of the theorem.

\begin{proof}[Proof of Theorem \ref{thm:koksma-hlawka_variant}]
  Our proof closely follows that of Theorem 4.1 in \cite{pausinger_svane:2015}. We begin by representing the space $\mathrm{span}(\mathcal{G})$ in a slightly different form. This is done for the purposes of making this proof aligned with the notation from the cited reference.

  Let $\mathcal{R}^*$ be the set from Eq.\@ \eqref{eq:R_star}. Let $\mathcal{S}(\mathcal{R}^*)$ be the set of all functions of the form $g = \sum_{i=1}^m \alpha_i \textbf{1}_{A_i}$, $m \in \mathbb{N}$, $\alpha_1, \ldots, \alpha_m \in \mathbb{R}$ and $A_1, \ldots, A_m \in \mathcal{R}^*$. It is easy to see that $\mathcal{R}^* = \mathcal{G} \cup \{\textbf{1}_\emptyset\}$ -- simply fill the interior of any non-empty hyperrectangle $[0,x]^v$, $v \subseteq \{1, \ldots, d\}$, with larger and larger \textit{closed} hyperrectangles $[0,y_1] \times \ldots \times [0,y_d]$, the indicator functions of which lie in $\mathcal{F}$. It follows that $\mathcal{S}(\mathcal{R}^*) = \mathrm{span}(\mathcal{G})$. We use the notation $\mathcal{S}(\mathcal{R}^*)$, because this is the notation which the references cited in this proof use. By Corollary 4.3 in \cite{aistleitner_etal:2017} we find that the Hardy-Krause variation is equal to what \cite{pausinger_svane:2015} call the $\mathcal{R}^*$-variation. By definition of the $\mathcal{R}^*$-variation (see also Section \ref{sec:HK_introduction}) it holds for any $g \in \mathcal{S}(\mathcal{R}^*) = \mathrm{span}(\mathcal{G})$ that
  $$
    \|g\|_{\mathrm{HK}} = \inf\left\{\sum_{i=1}^m |\alpha_i| h_{\mathcal{R}^*}(A_i) ~\Big|~ g = \sum_{i=1}^m \alpha_i \textbf{1}_{A_i}, m \in \mathbb{N}, \alpha_i \in \mathbb{R}, A_i \in \mathcal{R}^*\right\},
  $$
  where $h_{\mathcal{R}^*}(A_i) = 0$ if $A_i \in \{\emptyset, [0,1]^d\}$ and $1$ otherwise. A first implication of this is that any $g \in \mathcal{S}(\mathcal{R}^*) = \mathrm{span}(\mathcal{G})$ has finite Hardy-Krause variation. Let $g \in \mathcal{S}(\mathcal{R}^*) = \mathrm{span}(\mathcal{G})$ be fixed and assume without loss of generality that $g = \alpha_0 + \sum_{i=1}^m \alpha_i \textbf{1}_{A_i}$ with $A_i \notin \{\emptyset, [0,1]^d\}$ for all $i = 1, \ldots, m$ and $\sum_{i=1}^m |\alpha_i| \leq \|g\|_{\mathrm{HK}} + \varepsilon$ for some fixed but arbitrary $\varepsilon > 0$. Then
  \begin{align}
    \begin{split}
      \label{eq:koksma_hlawka_simple_1}
      \left|X(g) - Y(g)\right| & = \left|\alpha_0 \left[X(\textbf{1}_{[0,1]^d}) - Y(\textbf{1}_{[0,1]^d})\right] + \sum_{i=1}^m \alpha_i \left[X(\textbf{1}_{A_i}) - Y(\textbf{1}_{A_i})\right]\right| \\
                               & \leq \sum_{i=1}^m |\alpha_i| \left|X(\textbf{1}_{A_i}) - Y(\textbf{1}_{A_i})\right|,
    \end{split}
  \end{align}
  by linearity of $X$ and $Y$, and because $X(\textbf{1}_{[0,1]^d}) = Y(\textbf{1}_{[0,1]^d})$ by assumption. We have already noted that any $\textbf{1}_{A_i}$ can be monotonically approximated by functions from $\mathcal{F}$. Because $X$ and $Y$ are continuous with respect to monotone convergence, this implies
  $$
    \left|X(\textbf{1}_{A_i}) - Y(\textbf{1}_{A_i})\right| \leq \|X - Y\|_\mathcal{F}.
  $$
  Insert this into Eq.\@ \eqref{eq:koksma_hlawka_simple_1} to see that
  $$
    \left|X(g) - Y(g)\right| \leq \sum_{i=1}^m |\alpha_i| \|X-Y\|_{\mathcal{F}} \leq (\|g\|_{\mathrm{HK}} + \varepsilon) \|X-Y\|_{\mathcal{F}}
  $$
  for any $g \in \mathcal{S}(\mathcal{R}^*) = \mathrm{span}(\mathcal{G})$. Since $\varepsilon > 0$ was arbitrary, this establishes Eq.\@ \eqref{eq:koksma-hlawka_variant_inequality_only_span}.

  Now let $0 < c < \infty$ and assume that $X$ and $Y$ are defined and continuous with respect to uniform convergence on $\mathcal{HK}(c)$. Consider an arbitrary function $h \in \mathcal{HK}(c)$. By the comment following Definition 3.2 in \cite{pausinger_svane:2015}, there exists a sequence $g_k \in \mathcal{S}(\mathcal{R}^*) = \mathrm{span}(\mathcal{G})$ such that $\|h-g_k\|_\infty \to 0$ and $\|g_k\|_{\mathrm{HK}} \to \|h\|_{\mathrm{HK}}$ as $k \to \infty$. Because $\|h\|_{\mathrm{HK}} < c$ by definition of the class $\mathcal{HK}(c)$, it must hold that $\|g_k\|_{\mathrm{HK}} < c$ for almost all $k \in \mathbb{N}$, i.e.\@ $g_k \in \mathrm{span}(\mathcal{G}) \cap \mathcal{HK}(c)$ for sufficiently large $k$. Therefore, for sufficiently large $k \in \mathbb{N}$,
  \begin{align*}
    |X(h) - Y(h)| & \leq |X(h) - X(g_k)| + |Y(h) - Y(g_k)| + |X(g_k) - Y(g_k)|                           \\
                  & \leq |X(h) - X(g_k)| + |Y(h) - Y(g_k)| + \|g_k\|_{\mathrm{HK}} \|X-Y\|_{\mathcal{F}}
  \end{align*}
  by Eq.\@ \eqref{eq:koksma-hlawka_variant_inequality_only_span}, which we have just established. Letting $k$ tend to infinity proves our claim because $X$ and $Y$ are continuous on $\mathcal{HK}(c)$ with respect to uniform convergence, and we have $\|h - g_k\|_\infty \to 0$ as well as $\|g_k\|_{\mathrm{HK}} \to \|h\|_{\mathrm{HK}}$ by choice of the sequence $g_k$, $k \in \mathbb{N}$.
\end{proof}

The next lemma is central for the proof of Theorem \ref{thm:extension_hk_simple}. In the statement of that theorem, we only assume that the processes $X_n$ are defined on $\mathrm{span}(\mathcal{G})$, and their limit $X$ only needs to be defined a priori on the smaller class $\mathcal{F}$. The following result concerns continuous extensions of $X_n$ and $X$ to the classes $\mathcal{HK}(c)$. Its statement may seem somewhat complicated at first glance, but the underlying idea is not: Use the algebraic connection between $\mathrm{span}(\mathcal{G})$ and $\mathcal{HK}(c)$, which we have already encountered in the proof of Theorem \ref{thm:koksma-hlawka_variant}, together with the continuity properties of $X_n$ (and, as we will show, of $X$). The lemma concerns deterministic linear operators as opposed to random linear processes. Later, in the proof of Theorem \ref{thm:extension_hk_simple}, we will combine it with an almost sure coupling to apply it pointwise (in $\omega \in \Omega$, the elements of the underlying probability space) to the sample paths of the random processes in question.

\begin{lemma}
  \label{lem:extension_of_limit}
  Let $\mathcal{F} = \{\textbf{1}_{[0,x]} ~|~ x \in [0,1]^d\}$, and let $\mathcal{G}$ be the closure of $\mathcal{F}$ under monotone convergence (as in Theorem \ref{thm:extension_hk_simple}). Define the linear space $D$ to consist of all linear operators $T : \mathrm{span}(\mathcal{G}) \to \mathbb{R}$ which satisfy the following assumptions:
  \begin{enumerate}
    \item[(a)] $T$ is continuous with respect to monotone convergence in $\mathcal{F}$, i.e.\@ if $f_k$, $k \in \mathbb{N}$, is a sequence in $\mathcal{F}$ such that $f_k \uparrow g$ for some $g \in \mathcal{G}$, then $T f_k \to T g$ as $k \to \infty$.
    \item[(b)] $T$ is continuous with respect to the supremum norm on the entire space $\mathrm{span}(\mathcal{G})$.
    \item[(c)] $T \textbf{1}_{[0,1]^d} = 0$.
  \end{enumerate}
  Furthermore, let $E \subseteq \ell^\infty(\mathcal{F})$ be the linear space of all $T \in \ell^\infty(\mathcal{F})$ for which there exists a sequence $T_n \in D$, $n \in \mathbb{N}$, such that $\|T_n - T\|_\mathcal{F} \to 0$ as $n \to \infty$. (Note that $E$ is not quite equal to the closure of $D$ under uniform convergence, since the domain of elements of $D$ is different from the domain of elements of $E$.) Then the following statements hold:
  \begin{enumerate}
    \item\label{it:Tn_extension} Every $T \in D$ can be uniquely extended to a continuous (with respect to the supremum norm) linear operator $\bar{T} : \mathcal{HK}(\infty) \to \mathbb{R}$.
    \item\label{it:T_extension} If $T \in E$ and $T_n$, $n \in \mathbb{N}$, is a sequence in $D$ such that $\|T_n - T\|_\mathcal{F}$ as $n \to \infty$, then the equation $\bar{T} h = \lim_{n \to \infty} \bar{T}_n h$ gives a well-defined linear operator $\bar{T} : \mathcal{HK}(\infty) \to \mathbb{R}$. In particular, $\bar{T}$ is a linear extension of $T$, and this construction of $\bar{T}$ is independent of the specific sequence $T_n$. If $T$ is not only an element of $E$ but also of the smaller set $D$, then this extension agrees with that of Claim \ref{it:Tn_extension}.
    \item\label{it:T_monotone_continuity} The extension $\bar{T}$ of any $T \in E$ is continuous with respect to monotone convergence in $\mathcal{F}$ as defined above.
    \item\label{it:T_uniform_continuity} For any $0 < c < \infty$, the restriction of $\bar{T}$ to $\mathcal{HK}(c)$ is uniformly continuous with respect to the supremum norm.
  \end{enumerate}
\end{lemma}
\begin{proof}
  By assumption, every $T \in D$ is a continuous linear operator on $\mathrm{span}(\mathcal{G})$ equipped with the supremum norm. $\mathrm{span}(\mathcal{G})$ is a dense linear subspace of its own closure $\overline{\mathrm{span}}(\mathcal{G})$, which itself is a linear space \citep[Lemma II.1.3 in][]{dunford_schwartz:1958}. Here and in the remainder of the proof, the closure of a set is meant with respect to the supremum norm. By Theorem II.3.11 in \cite{dunford_schwartz:1958}, $T$ can be extended to a continuous linear operator $\bar{T}$ on $\overline{\mathrm{span}}(\mathcal{G})$, equipped with the supremum norm. Furthermore, the extension must be unique among all continuous linear extensions of $T$ because the original linear space $\mathrm{span}(\mathcal{G})$ is dense in $\overline{\mathrm{span}}(\mathcal{G})$. We have seen in the proof of Theorem \ref{thm:koksma-hlawka_variant} that functions of bounded Hardy-Krause variation can be uniformly approximated by functions from $\mathrm{span}(\mathcal{G})$; thus, $\mathcal{HK}(\infty)$ is a subspace of $\overline{\mathrm{span}}(\mathcal{G})$. This proves Claim \ref{it:Tn_extension}.

  Now consider some $T \in E$. Let $T_n$, $n \in \mathbb{N}$, be a sequence in $D$ such that $\|T_n - T\|_\mathcal{F} \to 0$ as $n \to \infty$. Fix some $h \in \mathcal{HK}(\infty)$. Since every extension $\bar{T}_n$ is continuous on $\mathcal{HK}(\infty)$ with respect to uniform convergence, we may use Theorem \ref{thm:koksma-hlawka_variant} to see that
  \begin{align}
    \begin{split}
      \label{eq:Tn_bar_cauchy}
      |\bar{T}_n h - \bar{T}_m h| & \leq \|h\|_{\mathrm{HK}} \|\bar{T}_n - \bar{T}_m\|_\mathcal{F}                          \\
                                  & = \|h\|_{\mathrm{HK}} \|T_n - T_m\|_\mathcal{F}                                         \\
                                  & \leq \|h\|_{\mathrm{HK}} \left(\|T_n - T\|_\mathcal{F} + \|T_m - T\|_\mathcal{F}\right)
    \end{split}
  \end{align}
  for any $n,m \in \mathbb{N}$. Since $\|T_n - T\|_\mathcal{F} \to 0$ by assumption, the last bound can be made arbitrarily small by choosing $n,m \in \mathbb{N}$ sufficiently large. Hence, the sequence $\bar{T}_n h$, $n \in \mathbb{N}$, is Cauchy in $\mathbb{R}$, and so $\bar{T}h = \lim_n \bar{T}_n h$ is well-defined. It is clear that $\bar{T}$ is an extension of $T$, since $\lim_n T_n f = Tf$ for any $f \in \mathcal{F}$ by assumption. Furthermore, $\bar{T}$ is linear because it is the pointwise limit of the linear operators $\bar{T}_n$. Finally, if $S_n$, $n \in \mathbb{N}$, is a second sequence in $D$ with the property that $\|S_n - T\|_\mathcal{F} \to 0$ as $n \to \infty$, then Theorem \ref{thm:koksma-hlawka_variant} implies
  $$
    |\bar{T}_n h - \bar{S}_n h| \leq \|h\|_{\mathrm{HK}} \|\bar{T}_n - \bar{S}_n\|_\mathcal{F} = \|h\|_{\mathrm{HK}} \|T_n - S_n\|_\mathcal{F}
  $$
  for any $h \in \mathcal{HK}(\infty)$. The right-hand side converges to $0$ as $n \to \infty$, since both $T_n$ and $S_n$ converge to the same limit in $\ell^\infty(\mathcal{F})$. Therefore,
  $$
    \bar{T} h = \lim_{n \to \infty} \bar{T}_n h = \lim_{n \to \infty} \bar{S}_n h.
  $$
  This proves Claim \ref{it:T_extension}.

  For Claim \ref{it:T_monotone_continuity}, let $T_n$, $n \in \mathbb{N}$, be a sequence in $D$ such that $\|T_n - T\|_\mathcal{F} \to 0$. Fix some $g \in \mathcal{G}$ and $f_k \in \mathcal{F}$, $k \in \mathbb{N}$, such that $f_k \uparrow g$. Then
  \begin{equation}
    \label{eq:extension_monotone}
    \bar{T} g = \lim_{n \to \infty} \bar{T}_n g = \lim_{n \to \infty} \lim_{k \to \infty} \bar{T}_n f_k = \lim_{k \to \infty} \lim_{n \to \infty} \bar{T}_n f_k = \lim_{k \to \infty} \bar{T} f_k.
  \end{equation}
  The switching of the limits in the third equality is valid because
  $$
    \|\bar{T}_n - \bar{T}\|_\mathcal{F} = \|T_n - T\|_\mathcal{F} \to 0
  $$
  by assumption, which in particular implies that $\bar{T}_n f_k \to \bar{T} f_k$ uniformly in $k \in \mathbb{N}$.

  Finally, let $0 < c < \infty$. Choose again an approximating sequence $T_n \in D$ such that $\|T_n - T\|_\mathcal{F} \to 0$. From Eq.\@ \eqref{eq:Tn_bar_cauchy}, we find that
  $$
    \|\bar{T}_n - \bar{T}_m\|_{\mathcal{HK}(c)} \leq c \left(\|T_n - T\|_\mathcal{F} + \|T_m - T\|_\mathcal{F}\right),
  $$
  and so $\bar{T}_n$, $n \in \mathbb{N}$, is Cauchy in $\ell^\infty[\mathcal{HK}(c)]$. Since this is a Banach space \citep[p.\@ 258 in][]{dunford_schwartz:1958}, we find that the sequence $\bar{T}_n$, $n \in \mathbb{N}$, must converge in $\ell^\infty[\mathcal{HK}(c)]$. But $\bar{T}$ is the pointwise limit of $\bar{T}_n$ by construction, and so the limit in $\ell^\infty[\mathcal{HK}(c)]$ of the restrictions of $\bar{T}_n$ to $\mathcal{HK}(c)$ must be the restriction of $\bar{T}$ to $\mathcal{HK}(c)$. Every $\bar{T}_n$ is linear and continuous, hence uniformly continuous \citep[Lemma II.3.4 in][]{dunford_schwartz:1958}. Since $\ell^\infty[\mathcal{HK}(c)]$ is equipped with the supremum norm, this reveals that the restriction of $\bar{T}$ to $\mathcal{HK}(c)$ is the uniform limit of uniformly continuous functions, meaning that it must be uniformly continuous itself. This proves Claim \ref{it:T_uniform_continuity}.
\end{proof}

Next, let us make a simple observation about the continuity properties of the extension $T \mapsto \bar{T}$ from Lemma \ref{lem:extension_of_limit}. We need this because we want the extension of the limiting process $X$ from Theorem \ref{thm:extension_hk_simple} to inherit certain properties from $X$, most importantly Borel measurability and, if applicable, Gaussianity.

\begin{lemma}
  \label{lem:extension_is_continuous}
  Let the set $E$ be as in Lemma \ref{lem:extension_of_limit}. The functions $\phi_c : E \to \ell^\infty[\mathcal{HK}(c)]$, which map an element $T \in E$ to its extension $\bar{T}$ described in Lemma \ref{lem:extension_of_limit}, are linear and Lipschitz continuous for all $0 < c < \infty$.
\end{lemma}
\begin{proof}
  The linearity of $\phi_c$ follows directly from the definition of $\bar{T}$ as a pointwise limit. It remains to prove continuity. Fix some $S, T \in E$. By Lemma \ref{lem:extension_of_limit}, their extensions $\bar{S}$ and $\bar{T}$ satisfy the assumptions of Theorem \ref{thm:koksma-hlawka_variant} for any fixed $0 < c < \infty$. Therefore, by that theorem,
  $$
    |\bar{S}h - \bar{T}h| \leq \|h\|_{\mathrm{HK}} \|\bar{S} - \bar{T}\|_\mathcal{F} = \|h\|_{\mathrm{HK}}\|S-T\|_\mathcal{F}
  $$
  for all $h \in \mathcal{HK}(c)$. Taking the supremum over all $h \in \mathcal{HK}(c)$ gives us
  $$
    \|\bar{S} - \bar{T}\|_{\mathcal{HK}(c)} \leq c\|S-T\|_\mathcal{F},
  $$
  and so $\phi_c$ is Lipschitz continuous.
\end{proof}

We are now in a position to prove Theorem \ref{thm:extension_hk_simple}. The claims in that theorem are already contained in Lemmas \ref{lem:extension_of_limit} and \ref{lem:extension_is_continuous} in a pointwise sense, i.e.\@ for any fixed sample path. To make these lemmas applicable to the weak convergence sense, we use a certain almost sure representation theorem akin to Skorohod's coupling theorem, allowing us to go from weak convergence to outer almost sure convergence.

\begin{proof}[Proof of Theorem \ref{thm:extension_hk_simple}]
  Let us denote the probability spaces underlying $X_n$ and $X$ by $(\Omega_n, \mathcal{A}_n, P_n)$ and $(\Omega, \mathcal{A}, P)$, respectively. $X$ is Borel measurable and tight, hence separable \citep[Lemma 1.3.2 in][]{van_der_vaart_wellner:weak_convergence}. By Theorem 1.10.4 and Addendum 1.10.5 in \cite{van_der_vaart_wellner:weak_convergence}, there exists a probability space $(\Omega', \mathcal{A}', P')$ and measurable maps $\phi_n : (\Omega', \mathcal{A}') \to (\Omega_n, \mathcal{A}_n)$ as well as $\phi : (\Omega', \mathcal{A}') \to (\Omega, \mathcal{A})$ with the following properties: If we set $X'_n = X_n \circ \phi_n$ and $X' = X \circ \phi$, it holds for any bounded function $f : \ell^\infty(\mathcal{F}) \to \mathbb{R}$ that
  \begin{equation}
    \label{eq:Xn_copy_1}
    \mathbb{E}^* f(X'_n) = \mathbb{E}^* f(X_n)
  \end{equation}
  as well as
  \begin{equation}
    \label{eq:Xn_copy_2}
    \mathbb{E}^* f(X') = \mathbb{E}^* f(X).
  \end{equation}
  Furthermore, it holds that $X'_n \to X'$ in $\ell^\infty(\mathcal{F})$ almost uniformly. Almost uniform convergence of a sequence to a Borel measurable limit is equivalent to outer almost sure convergence \citep[Lemma 1.9.2 in][]{van_der_vaart_wellner:weak_convergence}, and so $X'_n \to X'$ in $\ell^\infty(\mathcal{F})$ outer almost surely. By definition \citep[cf.\@ Definition 1.9.1 in][]{van_der_vaart_wellner:weak_convergence}, this means that $\|X'_n - X'\|_\mathcal{F}^* \to 0$ almost surely for some version of the minimal measurable cover $\|X'_n - X'\|_\mathcal{F}^*$ of $\|X'_n - X'\|_\mathcal{F}$. A minimal measurable cover of a possibly non-measurable element $U$ is a measurable random variable $U^*$ with values in $\mathbb{R} \cup \{\pm \infty\}$ and the property that $U^* \geq U$ always and $U^* \leq V$ almost surely for any measurable random variable $V$ satisfying $V \geq U$ almost surely. Minimal measurable covers always exist and are unique up to null sets \citep[specifically in our case, $P'$-null sets; for all of this, see Lemma 1.2.1 and the following paragraphs in][]{van_der_vaart_wellner:weak_convergence}. Thus, there exists a measurable $P'$-null set $N \in \mathcal{A}'$ with
  \begin{equation}
    \label{eq:as_in_F}
    \|X'_n - X'\|_\mathcal{F}(\omega) \leq \|X'_n - X'\|_\mathcal{F}^*(\omega) \xrightarrow[n \to \infty]{} 0
  \end{equation}
  as $n \to \infty$ for any $\omega \notin N$. The first inequality follows from the definition of a minimal measurable cover, and the convergence is just the definition of almost sure convergence.

  Recall the sets $D$ and $E$ from Lemma \ref{lem:extension_of_limit}. Because $X'_n = X_n \circ \phi_n$, $X'_n$ takes no other values than $X_n$. In particular, since $X_n \in D$ for any $n$, it holds that $X_n' \in D$ for all $n$, and so $X' \in E$ on $\Omega'\setminus N$ by Eq.\@ \eqref{eq:as_in_F}. More precisely: For any $\omega \notin N$, the map $X'(\omega) : \mathcal{F} \to \mathbb{R}$ is an element of $E$. We may change $X'$ arbitrarily on the $P'$-null set $N$ without invalidating the outer almost sure convergence $X'_n \to X'$; for instance, we may define $X''$ by $X''(\omega) = X'(\omega)$ for $\omega \in \Omega' \setminus N$, and $X''(\omega) = 0$ for $\omega \in N$. Then, while $\|X_n' - X'\|_\mathcal{F}^*$ is not necessarily a minimal measurable cover of $\|X_n' - X''\|_\mathcal{F}$, it is certainly still a measurable random variable which dominates $\|X_n' - X''\|_\mathcal{F}$ outside the $P'$-null set $N$. Therefore, any minimal measurable cover of $\|X_n - X''\|_\mathcal{F}$ must be smaller than $\|X_n' - X'\|_\mathcal{F}^*$ $P'$-almost surely. Since $\|X_n' - X'\|_\mathcal{F}^*$ converges to $0$ $P'$-almost surely, we can find versions of $\|X_n' - X''\|_\mathcal{F}^*$ which also converge to $0$ $P'$-almost surely (though the $P'$-null set involved here is not necessarily the set $N$). Therefore, $X_n' \to X''$ in $\ell^\infty(\mathcal{F})$ outer almost surely, and $X''$ inherits the Borel measurability and separability from $X'$. Importantly, $X'' \in E$ holds on the entire space $\Omega'$; and we have already seen that, for any $n \in \mathbb{N}$, $X_n' \in D \subseteq E$ on the entire space $\Omega'$ as well. Hence, by Lemma \ref{lem:extension_is_continuous} and the extended continuous mapping theorem \citep[Theorem 1.11.1 in][]{van_der_vaart_wellner:weak_convergence}, we get $\phi_c(X_n') \to \phi_c(X'')$ in $\ell^\infty[\mathcal{HK}(c)]$ outer almost surely, where $\phi_c$ is the map from Lemma \ref{lem:extension_is_continuous}. This implies $\phi_c(X_n') \rightsquigarrow \phi_c(X'')$ in $\ell^\infty[\mathcal{HK}(c)]$ by Lemmas 1.9.2 and 1.10.2 in \cite{van_der_vaart_wellner:weak_convergence}. Now if $g : \ell^\infty[\mathcal{HK}(c)] \to \mathbb{R}$ is bounded and continuous, apply Eq.\@ \eqref{eq:Xn_copy_1} to the bounded and continuous function $f = g \circ \phi_c$ to deduce from this that
  $$
    \mathbb{E}^* g\circ\phi_c(X_n) = \mathbb{E}^* g\circ\phi_c(X_n') \xrightarrow[n \to \infty]{} \mathbb{E} g \circ \phi_c(X'') = \mathbb{E} g \circ \phi_c(X').
  $$
  The final equality sign holds because $X'$ and $X''$ agree outside the $P'$-null set $N$. But this is just the definition of the weak convergence
  $$
    \bar{X}_n = \phi_c(X_n) \rightsquigarrow \phi_c(X') = \bar{X}'.
  $$
  The continuity properties of the sample paths of $\bar{X}'$ follows from Lemma \ref{lem:extension_of_limit}. Finally, $X'$ is a copy of $X$ by Eq.\@ \eqref{eq:Xn_copy_2}, and so we have proven our claim.
\end{proof}

The final proof in this section is that of Theorem \ref{thm:extension_covariance}. It contains no surprises: We use the algebraic connection between $\mathcal{F}$ and $\mathcal{HK}(\infty)$, which we have already exploited several times in this section, together with the assumed continuity properties of the covariance function $\Gamma$, to establish the claim.

\begin{proof}[Proof of Theorem \ref{thm:extension_covariance}]
  Since $X$ and $\bar{X}'$ are Borel measurable, so are $\|X\|_\mathcal{F}$ and $\|\bar{X}'\|_{\mathcal{HK}(c)}$. Apply Theorem \ref{thm:koksma-hlawka_variant} to $\bar{X}'$ with $Y$ taken as the constant $0$-process to see that
  \begin{equation}
    \label{eq:bar_X_prime_sup}
    \|\bar{X}'\|_{\mathcal{HK}(c)} \leq c \|X\|_\mathcal{F}.
  \end{equation}
  Since the right-hand side is square-integrable by assumption, so is $\|\bar{X}'\|_{\mathcal{HK}(c)}$. Let $g,g'$ be elements of $\mathcal{G}$ and pick two sequences $f_k, f_k' \in \mathcal{F}$ with $f_k \uparrow g$ and $f_k' \uparrow g'$. Since the sample paths of $\bar{X}'$ are continuous with respect to monotone convergence in $\mathcal{F}$, we get $\bar{X}'(f_k) \to \bar{X}'(g)$ and $\bar{X}'(f_k') \to \bar{X}'(g')$. Since $\bar{X}'(f_k) \lor \bar{X}'(f'_k) \leq \|\bar{X}'\|_{\mathcal{HK}(c)}$, and the right-hand side is integrable by Eq.\@ \eqref{eq:bar_X_prime_sup}, it follows that $\mathbb{E}\bar{X}'(f_k) \to \mathbb{E}\bar{X}'(g)$ and $\mathbb{E}\bar{X}'(f'_k) \to \mathbb{E}\bar{X}'(g')$ by the dominated convergence theorem. Similarly, we have for every $k \in \mathbb{N}$ that
  $$
    \left[\bar{X}'(f_k) - \mathbb{E}\bar{X}'(f_k)\right] \left[\bar{X}'(f_k') - \mathbb{E}\bar{X}'(f_k')\right] \leq 4 \left(\|\bar{X}'\|_{\mathcal{HK}(c)} \lor \mathbb{E}\|\bar{X}'\|_{\mathcal{HK}(c)}\right)^2,
  $$
  and the right-hand side is integrable. We can use the dominated convergence theorem once more to establish
  $$
    \mathrm{Cov}\left[\bar{X}'(g), \bar{X}'(g')\right] = \lim_{k \to \infty} \lim_{l \to \infty}\mathrm{Cov}\left[\bar{X}'(f_k), \bar{X}'(f_l')\right].
  $$
  Now recall that $\bar{X}'$ is an extension of $X'$, which in turn is a copy of $X$, and so
  $$
    \mathrm{Cov}\left[\bar{X}'(f_k), \bar{X}'(f_l')\right] = \mathrm{Cov}\left[X'(f_k), X'(f_l')\right] = \Gamma(f_k, f_l').
  $$
  Since $\Gamma$ is continuous in each argument with respect to monotone convergence in $\mathcal{F}$, we therefore get
  \begin{equation}
    \label{eq:cov_barX_spanG}
    \mathrm{Cov}\left[\bar{X}'(g), \bar{X}'(g')\right] = \lim_{k \to \infty} \lim_{l \to \infty}\Gamma(f_k, f_l') = \Gamma(g,g').
  \end{equation}
  Now we can use the linearity of $\bar{X}'$ and the bilinearity of both the covariance operator and the function $\Gamma$ to see that Eq.\@ \eqref{eq:cov_barX_spanG} is actually valid not only for $g,g' \in \mathcal{G}$, but even for all $g, g' \in \mathrm{span}(\mathcal{G}) \cap \mathcal{HK}(c)$.

  The extension to $\mathcal{HK}(c)$ can be performed in almost the same way: Fix $h, h' \in \mathcal{HK}(c)$ and pick two sequences $g_k, g_k' \in \mathrm{span}(\mathcal{G}) \cap \mathcal{HK}(c)$ such that $\|h-g_k\|_\infty \to 0$ and $\|h' - g_k'\|_\infty \to 0$ as $k \to \infty$. Since almost all sample paths of $\bar{X}'$ are continuous with respect to the supremum norm on $\mathcal{HK}(c)$, it holds that $\bar{X}'(g_k) \to \bar{X}'(h)$ and $\bar{X}'(g_k') \to \bar{X}'(h')$ almost surely as $k \to \infty$. By Eq.\@ \eqref{eq:cov_barX_spanG} and the same dominated convergence argument as before, we see that
  $$
    \mathrm{Cov}\left[\bar{X}'(h), \bar{X}'(h')\right] = \lim_{k \to \infty} \lim_{l \to \infty}\mathrm{Cov}\left[\bar{X}'(g_k), \bar{X}'(g_l')\right] = \lim_{k \to \infty} \lim_{l \to \infty}\Gamma(g_k, g_l') = \Gamma(h,h').
  $$
  This proves our claim about the covariance structure of $\bar{X}'$. Now assume that the original process $X$ is Gaussian. As an element of $\ell^\infty(\mathcal{F})$, it has bounded sample paths on $\mathcal{F}$ by definition. Borell's inequality \citep[Proposition A.2.1 in][]{van_der_vaart_wellner:weak_convergence} therefore implies that $\|X\|_{\mathcal{F}}$ has finite moments of all orders. In particular, it holds that $\mathbb{E}\|X\|_\mathcal{F}^2 < \infty$.
\end{proof}

\subsection{Product Measure Processes}
\label{sec:product_measures}
In this section, we prove Theorem \ref{thm:v_statistik_from_empirical_process} and the results relating to it. We do not use a Delta-method type argument for the proof of Theorem \ref{thm:v_statistik_from_empirical_process}, and instead prove our claim directly via induction over the degree $m$ (to apply the Delta-method, one has to verify that the map $P \mapsto P^m$ is Hadamard differentiable, and we suspect that the steps necessary for this are very similar to what we do in our proof anyway). There is a series of auxiliary results preceding the actual proof of Theorem \ref{thm:v_statistik_from_empirical_process}, all of which concern sequences which converge outer almost surely or almost uniformly. The reason for this is that, in proving Theorem \ref{thm:v_statistik_from_empirical_process}, we will resort to an almost sure representation technique, just like we did in Section \ref{sec:hk_proofs}.

We begin with a short technical lemma. Its statement is perhaps not very surprising, and in fact it would be trivial if the sequence $\delta_n$ were not random itself, and potentially dependent on $X$.

\begin{lemma}
  \label{lem:sup_delta_n_au}
  Let $\mathcal{F}$ be a function class and $X \in \ell^\infty(\mathcal{F})$ a Borel measurable process which has continuous sample paths with respect to the supremum norm. For $\delta \geq 0$, write
  $$
    \varepsilon(\delta) = \sup_{f,f' \in \mathcal{F} : \|f - f'\|_\infty \leq \delta} |X(f) - X(f')|.
  $$
  If $\delta_n$, $n \in \mathbb{N}$, is a sequence of $[0,\infty)$-valued, not necessarily measurable random elements and $\delta_n \to 0$ almost uniformly, then $\varepsilon(\delta_n) \to 0$ almost uniformly.
\end{lemma}
\begin{proof}
  Almost uniform convergence is equivalent to outer almost sure convergence for sequences with Borel measurable limits \citep[Lemma 1.9.2 in][]{van_der_vaart_wellner:weak_convergence}. By definition of outer almost sure convergence \citep[Definition 1.9.1 in][]{van_der_vaart_wellner:weak_convergence}, this means that $\delta_n^* \to 0$ almost surely for some versions of the minimal measurable covers $\delta_n^*$ of $\delta_n$, $n \in \mathbb{N}$. Recall from Section 1.2 in \cite{van_der_vaart_wellner:weak_convergence} that for a possibly non-measurable random element $T : (\Omega, \mathcal{A}, \mathbb{P}) \to \bar{\mathbb{R}}$, where $\bar{\mathbb{R}}$ denotes the extended real line, a minimal measurable cover $T^*$ is a Borel measurable random variable $T^* : (\Omega, \mathcal{A}, \mathbb{P}) \to \bar{\mathbb{R}}$ such that $T^* \geq T$ always and $T^* \leq U$ almost surely for every Borel measurable $U : (\Omega, \mathcal{A}, \mathbb{P}) \to \bar{\mathbb{R}}$ with $U \geq T$ almost surely. The minimal measurable cover is unique up to $\mathbb{P}$-null sets.

  Define the random elements $E_n = \varepsilon(\delta_n^*)$. We claim that every $E_n$ is Borel measurable. To see this, define
  \begin{equation}
    \label{eq:Enk_sum}
    E_{n,k} = \varepsilon\left(\frac{\lfloor k \delta_n^*\rfloor + 1}{k}\right) = \sum_{l=1}^\infty \varepsilon\left(\frac{l}{k}\right) \textbf{1}\left\{\delta_n^* \in \left[\frac{l-1}{k}, \frac{l}{k}\right)\right\}
  \end{equation}
  for any $n,k \in \mathbb{N}$. Because the process $X$ is a Borel measurable $\ell^\infty(\mathcal{F})$-valued random variable by assumption and the map
  \begin{align*}
    \ell^\infty(\mathcal{F}) & \to \mathbb{R},                                                                    \\
    T                        & \mapsto \sup_{f,f' \in \mathcal{F} : \|f - f'\|_\infty \leq \delta} |T(f) - T(f')|
  \end{align*}
  is continuous for any fixed $\delta$, the quantities $\varepsilon(l/k)$ are Borel measurable random variables for every fixed $l,k \in \mathbb{N}$. Furthermore, every $\delta_n^*$ is Borel measurable by definition, since it is a measurable cover of $\delta_n$. Therefore, each summand on the right-hand side of Eq.\@ \eqref{eq:Enk_sum} is Borel measurable, and so every $E_{n,k}$ is Borel measurable by Theorem 13.4 in \cite{billingsley:prob_and_measure}. By construction, every sample path $\delta \mapsto \varepsilon(\delta) = \varepsilon(\omega, \delta)$ is right-continuous. Therefore, $E_{n,k}(\omega) \to E_n(\omega)$ for every fixed $\omega$. Thus, $E_n$ is the pointwise limit of Borel measurable random variables, and hence Borel measurable, again by Theorem 13.4 in \cite{billingsley:prob_and_measure}.

  Having established measurability of every $E_n$, we immediately get $E_n \to 0$ almost surely, since all sample paths of $X$ are continuous with respect to the supremum norm, and $\delta_n^* \to 0$ almost surely. Now consider a minimal measurable cover $\varepsilon(\delta_n)^*$ of $\varepsilon(\delta_n)$. Since $\delta_n^* \geq \delta_n$, it holds that $E_n = \varepsilon(\delta_n^*) \geq \varepsilon(\delta_n)$. But we have just seen that $E_n$ is Borel measurable, and so by definition of the minimal measurable cover, this means that $\varepsilon(\delta_n)^* \leq E_n$ almost surely. Since $E_n \to 0$ almost surely, this implies $\varepsilon(\delta_n)^* \to 0$ almost surely. Hence, by definition of outer almost sure convergence, we get $\varepsilon(\delta_n) \to 0$ outer almost surely. Our claim now follows from another application of Lemma 1.9.2 in \cite{van_der_vaart_wellner:weak_convergence}.
\end{proof}

The next lemma is the main ingredient in the proof of Theorem \ref{thm:v_statistik_from_empirical_process}. It contains both the almost sure representation argument as well as an approximation of the process of interest, $r_n(P_n^m - \mu_n^m)$, via the original limiting process $G$ (or more precisely, a coupling thereof). The approximation is not quite the limiting process $G_\mu^{(m)}$ from the statement of Theorem \ref{thm:v_statistik_from_empirical_process} -- this gap will be closed in another lemma --, but it has a very similar structure already. The main challange in the proof of this lemma is to bound the difference between the couplings of $r_n(P_n^m - \mu_n^m)$ and the approximating process which we have just mentioned. It is at this point where the assumptions on the class $\mathcal{F}_m$ are used. In short, they are necessary because we can control the behaviour of $r_n(P_n - \mu_n)$ only in the class $\mathcal{F}$, and the definition of the class $\mathcal{F}_m$ allows us to relate the behaviour of $r_n(P_n^m - \mu_n^m)$ in $\ell^\infty(\mathcal{F}_m)$, to the behaviour of the original process $r_n(P_n - \mu_n)$ in $\ell^\infty(\mathcal{F})$. The critical step in the following proof where this is carried out is Eqs.\@ \eqref{eq:v-1} and \eqref{eq:Rn_au_0}.

\begin{lemma}
  \label{lem:V_approximation}
  Let $P_n$ and $\mu_n$, $n \in \mathbb{N}$, be sequences of probability measures on some measurable space $(\Omega, \mathcal{A})$, where the $P_n$ are random and the $\mu_n$ are deterministic. Let $\mathcal{F}$ be a class of measurable real-valued functions defined on $(\Omega, \mathcal{A})$. Assume that there exist a real-valued sequence $r_n \uparrow \infty$ and a Borel measurable and tight random process $G \in \ell^\infty(\mathcal{F})$ such that $r_n(P_n - \mu_n) \rightsquigarrow G$. Finally, assume that, almost surely, all sample paths of this process $G$ are uniformly continuous with respect to the supremum norm.

  Then there exist random elements $\tilde{P}_n, \tilde{G} \in \ell^\infty(\mathcal{F})$, $\tilde{G}$ being Borel measurable and tight, such that the following statements hold: First, for any bounded function $f : \ell^\infty(\mathcal{F}) \to \mathbb{R}$ and any $n \in \mathbb{N}$, we have
  \begin{equation}
    \label{eq:statement_Pn_copy}
    \mathbb{E}^* f(\tilde{P}_n) = \mathbb{E}^* f(P_n),
  \end{equation}
  as well as
  \begin{equation}
    \label{eq:statement_G_copy}
    \mathbb{E}^* f(\tilde{G}) = \mathbb{E}^* f(G).
  \end{equation}
  Second, all sample paths of $\tilde{G}$ are uniformly continuous with respect to the supremum norm. Third, it holds for any $m \in \mathbb{N}$ that
  \begin{equation}
    \label{eq:statement_V_convergence}
    r_n(\tilde{P}_n^m - \mu_n^m) - \tilde{G}_{\mu,n}^{(m)} \xrightarrow[n \to \infty]{au}0,
  \end{equation}
  in $\ell^\infty(\mathcal{F}_m)$, where we define $\tilde{G}_{\mu,n}^{(1)} = \tilde{G}$ and $\tilde{G}_{\mu,n}^{(m)}(h) = \sum_{i=1}^m \tilde{G}(h_{n,i})$ with
  $$
    h_{n,i}(x) = \int h(x_1, \ldots, x_{i-1}, x, x_{i+1}, \ldots, x_m) ~\mathrm{d}\mu_n^{m-1}(x_1, \ldots, x_{i-1}, x_{i+1}, \ldots, x_m)
  $$
  for $m \geq 2$.
\end{lemma}
\begin{proof}
  By changing $G$ on a null set, we may assume without loss of generality that all sample paths of $G$ are continuous with respect to the supremum norm. Since $r_n(P_n - \mu_n) \rightsquigarrow G$ in $\ell^\infty(\mathcal{F})$ by assumption and $G$ is Borel measurable and tight \citep[hence separable, cf.\@ Lemma 1.3.2 in][]{van_der_vaart_wellner:weak_convergence}, there exist by Theorem 1.10.4 in \cite{van_der_vaart_wellner:weak_convergence} random elements $\tilde{X}_n$ and $\tilde{G}$ such that $\tilde{X}_n \to \tilde{G}$ almost uniformly in $\ell^\infty(\mathcal{F})$ and
  \begin{equation}
    \label{eq:nonmeasurable_copy}
    \mathbb{E}^* f(\tilde{X}_n) = \mathbb{E}^* f[r_n(P_n - \mu_n)]
  \end{equation}
  as well as
  $$
    \mathbb{E}^* f(\tilde{G}) = \mathbb{E}^* f(G)
  $$
  for any bounded function $f : \ell^\infty(\mathcal{F}) \to \mathbb{R}$. This establishes Eq.\@ \eqref{eq:statement_G_copy}. Moreover, by Addendum 1.10.5 in \cite{van_der_vaart_wellner:weak_convergence}, we can assume that $\tilde{X}_n = r_n[P_n \circ \phi_n - \mu_n]$ for a certain sequence of measurable functions $\phi_n : \Omega_0 \to \Omega_n$, where $\Omega_0$ is some probability space and the $\Omega_n$ are the probability spaces underlying the random elements $P_n : \Omega_n \to \ell^\infty(\mathcal{F})$. Set $\tilde{P}_n = P_n \circ \phi_n$. Then all $\tilde{P}_n$ are probability measures, because $\tilde{P}_n = P_n \circ \phi_n$ takes the same values as $P_n$ for all $n \in \mathbb{N}$. Furthermore, it holds that
  \begin{equation}
    \label{eq:copies_au}
    r_n(\tilde{P}_n - \mu_n) = \tilde{X}_n \xrightarrow[n \to \infty]{au} \tilde{G}
  \end{equation}
  in $\ell^\infty(\mathcal{F})$, and
  $$
    \mathbb{E}^* f(\tilde{P}_n) = \mathbb{E}^* f(\mu_n + r_n^{-1} \tilde{X}_n) = \mathbb{E}^* f(P_n)
  $$
  for any bounded function $f : \ell^\infty(\mathcal{F})$. This follows by applying Eq.\@ \eqref{eq:nonmeasurable_copy} to the function $x \mapsto f(\mu_n + r_n^{-1} x)$. This is exactly the claim from Eq.\@ \eqref{eq:statement_Pn_copy}. Similarly, we can assume \citep[also by Addendum 1.10.5 in][]{van_der_vaart_wellner:weak_convergence} that $\tilde{G} = G \circ \phi$ for some measurable $\phi : \Omega_0 \to \Omega_G$, where $\Omega_G$ is the space on which the random variable $G$ is defined. Since we assumed that all sample paths of $G$ are uniformly continuous with respect to the supremum norm, this implies that all sample paths of $\tilde{G}$ have the same continuity property, and that $\tilde{G}$ inherits Borel measurability and separability from $G$.

  We now prove the statement from Eq.\@ \eqref{eq:statement_V_convergence} by induction over $m$. For $m = 1$, our claim is an immediate consequence of Eq.\@ \eqref{eq:copies_au}:
  $$
    r_n(\tilde{P}_n^m - \mu_n^m) - \tilde{G}_{\mu,n}^{(m)} = r_n(\tilde{P}_n - \mu_n) - \tilde{G} \xrightarrow[n \to \infty]{au} 0.
  $$
  Now assume that we have established our claim for $m$. To extend it to $m+1$, consider a fixed function $h \in \mathcal{F}_{m+1}$ and write
  \begin{align}
    \begin{split}
      \label{eq:v-1}
       & r_n[\tilde{P}_n^{m+1}(h) - \mu_n^{m+1}(h)]                                                                                                   \\
       & = r_n\iint h(x_1, \ldots, x_{m+1}) ~\mathrm{d}\tilde{P}_n(x_{m+1}) ~\mathrm{d}\tilde{P}_n^m(x_1, \ldots, x_m)                                \\
       & \quad - r_n\iint h(x_1, \ldots, x_{m+1}) ~\mathrm{d}\mu_n(x_{m+1}) ~\mathrm{d}\mu_n^m(x_1, \ldots, x_m)                                      \\
       & = \iint h(x_1, \ldots, x_{m+1}) ~\mathrm{d}\mu_n(x_{m+1}) ~\mathrm{d}\left[r_n(\tilde{P}_n^m - \mu_n^m)\right](x_1, \ldots, x_m)             \\
       & \quad + \iint h(x_1, \ldots, x_{m+1}) ~\mathrm{d}\mu_n^m(x_1, \ldots, x_m) ~\mathrm{d}\left[r_n(\tilde{P}_n - \mu_n)\right](x_{m+1})         \\
       & \quad + \iint h(x_1, \ldots, x_{m+1}) ~\mathrm{d}\tilde{P}_n(x_{m+1}) ~\mathrm{d}\left[r_n(\tilde{P}_n^m - \mu_n^m)\right](x_1, \ldots, x_m) \\
       & \quad - \iint h(x_1, \ldots, x_{m+1}) ~\mathrm{d}\mu_n(x_{m+1}) ~\mathrm{d}\left[r_n(\tilde{P}_n^m - \mu_n^m)\right](x_1, \ldots, x_m)       \\
       & = \tilde{G}_{\mu,n}^{(m)}\left[(x_1, \ldots, x_m) \mapsto \int h(x_1, \ldots, x_{m+1}) ~\mathrm{d}\mu_n(x_{m+1})\right]                      \\
       & \quad + \tilde{G}_{\mu,n}^{(1)}\left[x_{m+1} \mapsto \int h(x_1, \ldots, x_{m+1}) ~\mathrm{d}\mu_n^m(x_1, \ldots, x_m)\right]                \\
       & \quad + \tilde{G}_{\mu,n}^{(m)}\left[(x_1, \ldots, x_m) \mapsto \int h(x_1, \ldots, x_{m+1}) ~\mathrm{d}\tilde{P}_n(x_{m+1})\right]          \\
       & \quad - \tilde{G}_{\mu,n}^{(m)}\left[(x_1, \ldots, x_m) \mapsto \int h(x_1, \ldots, x_{m+1}) ~\mathrm{d}\mu_n(x_{m+1})\right]                \\
       & \quad + R_n(h)
    \end{split}
  \end{align}
  for some remainder term $R_n(h)$, which is just defined by the last equality. This last expression is somewhat artificial, as two of the terms actually cancel out; the reason for choosing this particular representation is that it somewhat simplifies the following arguments. Observe that the three functions
  \begin{align*}
    (x_1, \ldots, x_m) & \mapsto \int h(x_1, \ldots, x_{m+1}) ~\mathrm{d}\mu_n(x_{m+1}),            \\
    x_{m+1}            & \mapsto \int h(x_1, \ldots, x_{m+1}) ~\mathrm{d}\mu_n^m(x_1, \ldots, x_m), \\
    (x_1, \ldots, x_m) & \mapsto \int h(x_1, \ldots, x_{m+1}) ~\mathrm{d}\tilde{P}_n(x_{m+1}),
  \end{align*}
  lie in $\mathcal{F}_m$, $\mathcal{F}$ and $\mathcal{F}_m$, respectively. This is a consequence of the fact that the original function $h$ was chosen from $\mathcal{F}_{m+1}$ (the third function is a random element of $\mathcal{F}_{m}$ because it is constructed by integration with respect to the random measure $\tilde{P}_n$, but this does not complicate the following arguments). Since we assume that we haven proven our claim for $m$, it follows that
  \begin{align}
    \begin{split}
      \label{eq:Rn_au_0}
       & \sup_{h \in \mathcal{F}_{m+1}}|R_n(h)|                                                                                                                                                                                     \\
       & \leq 3 \sup_{h \in \mathcal{F}_m} \left|r_n\left[\tilde{P}_n^m(h) - \mu_n^m(h)\right] - \tilde{G}_{\mu,n}^{(m)}(h)\right| + \sup_{f \in \mathcal{F}} \left|r_n\left[\tilde{P}_n(f) - \mu_n(f)\right] - \tilde{G}(f)\right| \\
       & \xrightarrow[n \to \infty]{au} 0.
    \end{split}
  \end{align}

  Let us now consider the other terms on the right-hand side of Eq.\@ \eqref{eq:v-1}. By definition of $\tilde{G}_{\mu,n}^{(m)}$ we have that
  \begin{align*}
     & \tilde{G}_{\mu,n}^{(m)}\left[(x_1, \ldots, x_m) \mapsto \int h(x_1, \ldots, x_{m+1}) ~\mathrm{d}\mu_n(x_{m+1})\right]                            \\
     & = \sum_{i=1}^m \tilde{G}\left[x_i \mapsto \int h(x_1, \ldots, x_{m+1}) ~\mathrm{d}\mu_n^m(x_1, \ldots, x_{i-1}, x_{i+1}, \ldots, x_{m+1})\right] \\
     & = \sum_{i=1}^m \tilde{G}(h_{n,i}).
  \end{align*}
  Similarly, it holds that
  \begin{align*}
     & \tilde{G}_{\mu,n}^{(1)}\left[x_{m+1} \mapsto \int h(x_1, \ldots, x_{m+1}) ~\mathrm{d}\mu_n^m(x_1, \ldots, x_m)\right] \\
     & = \tilde{G}\left[x_{m+1} \mapsto \int h(x_1, \ldots, x_{m+1}) ~\mathrm{d}\mu_n^m(x_1, \ldots, x_m)\right]             \\
     & = \tilde{G}(h_{n,m+1}).
  \end{align*}
  These two identities combined imply
  \begin{align}
    \begin{split}
      \label{eq:v-2}
       & \tilde{G}_{\mu,n}^{(m)}\left[(x_1, \ldots, x_m) \mapsto \int h(x_1, \ldots, x_{m+1}) ~\mathrm{d}\mu_n(x_{m+1})\right]         \\
       & \quad + \tilde{G}_{\mu,n}^{(1)}\left[x_{m+1} \mapsto \int h(x_1, \ldots, x_{m+1}) ~\mathrm{d}\mu_n^m(x_1, \ldots, x_m)\right] \\
       & = \sum_{i=1}^m \tilde{G}(h_{n,i}) + \tilde{G}(h_{n,m+1})                                                                      \\
       & = \tilde{G}_{\mu,n}^{(m+1)}(h).
    \end{split}
  \end{align}

  Now let $\delta > 0$ and consider two functions $h, h' \in \mathcal{F}_m$ such that $\|h-h'\|_\infty \leq \delta$. Then for any $i = 1, \ldots, m$,
  $$
    \|h_{n,i} - h'_{n,i}\|_\infty \leq \int \|h-h'\|_\infty ~\mathrm{d}\mu_n^{m-1} \leq \delta,
  $$
  and therefore, since $h_{n,i}, h_{n,i}' \in \mathcal{F}$,
  \begin{align}
    \begin{split}
      \label{eq:G_tilde_uniformly_continuous}
      \left|\tilde{G}_{\mu,n}^{(m)}(h) - \tilde{G}_{\mu,n}^{(m)}(h')\right| \leq \sum_{i=1}^m \left|\tilde{G}(h_{n,i}) - \tilde{G}(h_{n,i}')\right| \leq m \sup_{f,f' \in \mathcal{F} : \|f-f'\|_\infty \leq \delta} \left|\tilde{G}(f) - \tilde{G}(f')\right|.
    \end{split}
  \end{align}
  For any $h \in \mathcal{F}_{m+1}$, we have
  \begin{align}
    \begin{split}
      \label{eq:h_diff_Pn_mun}
       & \sup_{x_1, \ldots, x_m \in \Omega}\left|\int h(x_1, \ldots, x_{m+1}) ~\mathrm{d}\tilde{P}_n(x_{m+1}) - \int h(x_1, \ldots, x_{m+1}) ~\mathrm{d}\mu_n(x_{m+1})\right| \\
       & \leq \sup_{f \in \mathcal{F}}|\tilde{P}_n(f) - \mu_n(f)|,
    \end{split}
  \end{align}
  because $x_{m+1} \mapsto h(x_1, \ldots, x_m, x_{m+1})$ is an element of $\mathcal{F}$ for any $x_1, \ldots, x_m$. Eqs.\@ \eqref{eq:G_tilde_uniformly_continuous} and \eqref{eq:h_diff_Pn_mun} combined yield
  \begin{align}
    \begin{split}
      \label{eq:v-3}
       & \sup_{h \in \mathcal{F}_{m+1}} \left|\tilde{G}_{\mu,n}^{(m)}\left[(x_1, \ldots, x_m) \mapsto \int h(x_1, \ldots, x_{m+1}) ~\mathrm{d}\tilde{P}_n(x_{m+1})\right]\right. \\
       & \qquad\qquad - \left.\tilde{G}_{\mu,n}^{(m)}\left[(x_1, \ldots, x_m) \mapsto \int h(x_1, \ldots, x_{m+1}) ~\mathrm{d}\mu_n(x_{m+1})\right]\right|                       \\
       & \leq m \sup_{f,f' \in \mathcal{F} : \|f-f'\|_\infty \leq \delta_n} \left|\tilde{G}(f) - \tilde{G}(f')\right|,
    \end{split}
  \end{align}
  where $\delta_n = \sup_{f \in \mathcal{F}}|\tilde{P}_n(f) - \mu_n(f)|$.

  Now, by Eqs.\@ \eqref{eq:v-1}, \eqref{eq:v-2} and \eqref{eq:v-3},
  \begin{align}
    \begin{split}
      \label{eq:v-4}
       & \sup_{h \in \mathcal{F}_{m+1}}\left|r_n[\tilde{P}_n^{m+1}(h) - \mu_n^{m+1}(h)] - \tilde{G}_{\mu,n}^{(m+1)}(h)\right|                               \\
       & \leq \sup_{h \in \mathcal{F}_m}|R_n(h)| + m \sup_{f,f' \in \mathcal{F} : \|f-f'\|_\infty \leq \delta_n} \left|\tilde{G}(f) - \tilde{G}(f')\right|.
    \end{split}
  \end{align}
  We have already seen in Eq.\@ \eqref{eq:Rn_au_0} that $\sup_{h \in \mathcal{F}_{m+1}}|R_n(h)| \to 0$ almost uniformly. On the other hand, we can see from Eq.\@ \eqref{eq:copies_au} that
  $$
    \delta_n = \sup_{f \in \mathcal{F}}|\tilde{P}_n(f) - \mu_n(f)| \xrightarrow[n \to \infty]{au} 0,
  $$
  which also implies
  $$
    \sup_{f,f' \in \mathcal{F} : \|f-f'\|_\infty \leq \delta_n} \left|\tilde{G}(f) - \tilde{G}(f')\right| \xrightarrow[n \to \infty]{au} 0
  $$
  by Lemma \ref{lem:sup_delta_n_au}. Therefore, the right-hand side of Eq.\@ \eqref{eq:v-4} converges to $0$ almost uniformly. This is what we wanted to show.
\end{proof}

Next, a quick technical lemma showing that the difference between $\tilde{G}_{\mu, n}^{(m)}$ from the previous lemma and $\tilde{G}_\mu^{(m)}$ -- the coupling of the limiting process from Theorem \ref{thm:v_statistik_from_empirical_process} -- is small. It is not very complicated: Use the sample path continuity of the process $\tilde{G}$ together with the fact that $\mu_n$ is close to $\mu$ in $\ell^\infty(\mathcal{F})$. Next, use the definition of the class $\mathcal{F}_m$ to show that this also implies that $\mu_n^m$ must be close to $\mu^m$ in $\ell^\infty(\mathcal{F}_m)$. This part of the argument is similar to the technique from the previous lemma.

\begin{lemma}
  \label{lem:V-difference}
  Assume that the conditions of Lemma \ref{lem:V_approximation} are satisfied. We use the notation from that lemma. Suppose that the sequence $\mu_n$, $n \in \mathbb{N}$, fulfils the additional assumption that $\mu_n \to \mu$ in $\ell^\infty(\mathcal{F})$ for some $\mu \in \ell^\infty(\mathcal{F})$. Then, for any $m \in \mathbb{N}$,
  $$
    \tilde{G}_{\mu,n}^{(m)} \xrightarrow[n \to \infty]{} \tilde{G}_\mu^{(m)}
  $$
  in $\ell^\infty(\mathcal{F}_m)$, where $\tilde{G}_{\mu}^{(m)}$ is defined as in Definition \ref{def:Fm_class}.
\end{lemma}
\begin{proof}
  This is a straightforward consequence of the fact that the process $\tilde{G}$ has uniformly continuous sample paths with respect to the supremum norm. If $m = 1$, the claim is trivial since then $\tilde{G}_{\mu,n}^{(1)} = \tilde{G} = \tilde{G}_{\mu}^{(1)}$. Assume therefore that $m \geq 2$. For any $h \in \mathcal{F}_m$, we have that
  \begin{align*}
    h_{n,i}(x) & = \int h(x_1, \ldots, x_{i-1}, x, x_{i+1}, \ldots, x_m) ~\mathrm{d}\mu_n^{m-1}(x_1, \ldots, x_{i-1}, x_{i+1}, \ldots, x_m)                      \\
               & = \iint h(x_1, \ldots, x_{i-1}, x, x_{i+1}, \ldots, x_m) ~\mathrm{d}\mu(x_1) ~\mathrm{d}\mu_n^{m-2}(x_2, \ldots, x_{i-1}, x_{i+1}, \ldots, x_m) \\
               & \quad + R_{n,1}(h,x),
  \end{align*}
  where the remainder term $R_{n,1}(h,x)$ is equal to
  $$
    \iint h(x_1, \ldots, x_{i-1}, x, x_{i+1}, \ldots, x_m) ~\mathrm{d}(\mu_n - \mu)(x_1) ~\mathrm{d}\mu_n^{m-2}(x_2, \ldots, x_{i-1}, x_{i+1}, \ldots, x_m).
  $$
  Hence,
  \begin{align*}
    |R_{n,1}(h,x)| & \leq \sup_{x_2, \ldots, x_m \in \Omega} \left|\int h(x_1, \ldots, x_m) ~\mathrm{d}(\mu_n - \mu)(x_1)\right| \leq \sup_{f \in \mathcal{F}}|\mu_n(f) - \mu(f)|,
  \end{align*}
  since $x_1 \mapsto h(x_1, \ldots, x_m)$ is an element of $\mathcal{F}$ for any $x_2, \ldots, x_m$. We can iterate this argument to find that
  \begin{align*}
     & h_{n,i}(x)                                                                                                                          \\
     & = \int h(x_1, \ldots, x_{i-1}, x, x_{i+1}, \ldots, x_m) ~\mathrm{d}\mu^{m-1}(x_1, \ldots, x_{i-1}, x_{i+1}, \ldots, x_m) + R_n(h,x) \\
     & = h_i(x) + R_n(h,x),
  \end{align*}
  for some remainder term $R_n(h,x)$ satisfying
  $$
    |R_n(h,x)| \leq (m-1) \sup_{f \in \mathcal{F}}|\mu_n(f) - \mu(f)|.
  $$
  Let us denote the right-hand side of this inequality by $\delta_n$. Thus, for any $h \in \mathcal{F}_m$, we get $\|h_{n,i} - h_i\|_\infty \leq \delta_n$, and so
  $$
    \left|\tilde{G}_{\mu,n}^{(m)}(h) - \tilde{G}_{\mu}^{(m)}(h)\right| \leq \sum_{i=1}^m \left|\tilde{G}(h_{n,i}) - \tilde{G}(h_i)\right| \leq m \sup_{f,f' \in \mathcal{F} : \|f-f'\|_\infty \leq \delta_n} \left|\tilde{G}(f) - \tilde{G}(f')\right|,
  $$
  since $h_{n,i}, h_i \in \mathcal{F}$. The right-hand side does not depend on $h$, and so
  $$
    \sup_{h \in \mathcal{F}_m}\left|\tilde{G}_{\mu,n}^{(m)}(h) - \tilde{G}_{\mu}^{(m)}(h)\right| \leq m \sup_{f,f' \in \mathcal{F} : \|f-f'\|_\infty \leq \delta_n} \left|\tilde{G}(f) - \tilde{G}(f')\right|.
  $$
  Because $\mu_n \to \mu$ in $\ell^\infty(\mathcal{F})$ by assumption, we get that $\delta_n \downarrow 0$. Since $\tilde{G}$ has uniformly continuous sample paths with respect to the supremum norm, our claim follows.
\end{proof}

We can now prove Theorem \ref{thm:v_statistik_from_empirical_process}. As in Section \ref{sec:hk_proofs}, there is not much left to do except to apply an appropriate almost uniform representation argument. The pointwise behaviour is handled by the previous results.

\begin{proof}[Proof of Theorem \ref{thm:v_statistik_from_empirical_process}]
  By Lemmas \ref{lem:V_approximation} and \ref{lem:V-difference}, there exist random processes $\tilde{P}_n, \tilde{G} \in \ell^\infty(\mathcal{F})$ such that
  \begin{equation}
    \label{eq:copy_1}
    \mathbb{E}^* f(\tilde{P}_n) = \mathbb{E}^* f(P_n),
  \end{equation}
  and
  \begin{equation}
    \label{eq:copy_2}
    \mathbb{E}^* f(\tilde{G}) = \mathbb{E}^* f(G).
  \end{equation}
  for any bounded function $f : \ell^\infty(\mathcal{F}) \to \mathbb{R}$, and
  $$
    r_n(\tilde{P}_n^m - \mu_n^m) \xrightarrow[n \to \infty]{au} \tilde{G}_{\mu}^{(m)}
  $$
  in $\ell^\infty(\mathcal{F}_m)$. Since $\tilde{G}_\mu^{(m)}$ arises as a continuous transformation of the Borel measurable and tight random variable $\tilde{G}$, it too is tight and Borel measurable. Almost uniform convergence implies weak convergence for Borel measurable limits \citep[Lemmas 1.9.3 and 1.10.2 in][]{van_der_vaart_wellner:weak_convergence}, and so we get that
  \begin{equation}
    \label{eq:copy_3}
    \mathbb{E}^* f[r_n(\tilde{P}_n^m - \mu_n^m)] \xrightarrow[n \to \infty]{} \mathbb{E}f\left(\tilde{G}_\mu^{(m)}\right)
  \end{equation}
  for every bounded and continuous $f : \ell^\infty(\mathcal{F}_m) \to \mathbb{R}$. The objects $r_n(\tilde{P}_n^m - \mu_n^m)$ and $\tilde{G}_\mu^{(m)}$ can be written as functions of $\tilde{P}_n$ and $\tilde{G}$, i.e.\@ there exist functions $g, g_n : \ell^\infty(\mathcal{F}) \to \ell^\infty(\mathcal{F}_m)$ ($g$ continuous) such that $g_n(\tilde{P}_n) = r_n(\tilde{P}_n^m - \mu_n^m)$ and $g(\tilde{G}) = \tilde{G}_\mu^{(m)}$. We can use this representation on both sides of Eq.\@ \eqref{eq:copy_3} and combine it with Eqs.\@ \eqref{eq:copy_1} and \eqref{eq:copy_2} to find that
  \begin{align*}
    \mathbb{E}^* f[r_n(P_n^m - \mu_n^m)] & = \mathbb{E}^* f \circ g_n (P_n) = \mathbb{E}^* f \circ g_n (\tilde{P}_n) = \mathbb{E}^* f[r_n(\tilde{P}_n^m - \mu_n^m)]                                                     \\
                                         & \xrightarrow[n \to \infty]{} \mathbb{E}f\left(\tilde{G}_\mu^{(m)}\right) = \mathbb{E} f \circ g (\tilde{G}) = \mathbb{E} f \circ g (G) = \mathbb{E}f\left(G_\mu^{(m)}\right)
  \end{align*}
  for every bounded and continuous $f : \ell^\infty(\mathcal{F}_m) \to \mathbb{R}$. But this is just the definition of the weak convergence $r_n(P_n^m - \mu_n^m) \rightsquigarrow G_\mu^{(m)}$.

  The fact that the transformation $G \mapsto G_\mu^{(m)}$ is continuous and linear, as well as the statement about the covariance function of $G_\mu^{(m)}$, follow immediately from the definition of $G_\mu^{(m)}$ given in Definition \ref{def:Fm_class}.
\end{proof}

We now turn to Theorem \ref{thm:characterisation_HKcm}. Before we prove that theorem, we first need a technical lemma about functions of bounded Hardy-Krause variation. Recall that the definition of the class $\mathcal{F}_m$ contains a certain integration operation applied to every $h \in \mathcal{F}_m$. When $\mathcal{F} = \mathcal{HK}(c)$, we need to be able to relate the Hardy-Krause variation of the result of this integral operation to the original function. This is done in the following lemma. Its conclusion is perhaps not entirely surprising, as the Hardy-Krause variation (or indeed any type of variation) of a function $h$ is connected to its smoothness, and integrating usually makes functions smoother.

If $h : ([0,1]^d)^m \to \mathbb{R}$ is a function and $x = (x_1, \ldots, x_{m-1}) \in ([0,1]^d)^m$ a point, let $h_{i,x}$, $i = 1, \ldots, m$, be the functions defined by
$$
  [0,1]^d \ni t \mapsto h_{i,x}(t) = h(x_1, \ldots, x_{i-1}, t, x_i, \ldots, x_{m-1}).
$$

\begin{lemma}
  \label{lem:HK_integration_contraction}
  Let $m \in \mathbb{N}$, and let $Q$ be a Borel probability measure on $([0,1]^d)^{m-1}$ and $h : ([0,1]^d)^m \to \mathbb{R}$ a Borel measurable function. Let $i = 1, \ldots, m$ be fixed and define $h_Q : [0,1]^d \to \mathbb{R}$ by
  $$
    h_Q(t) = \int h(x_1, \ldots, x_{i-1}, t, x_{i+1}, \ldots, x_m) ~\mathrm{d}Q(x_1, \ldots, x_{i-1}, x_{i+1}, \ldots, x_m).
  $$
  Then
  $$
    \|h_Q\|_{\mathrm{HK}} \leq \int \|h_{i,x}\|_{\mathrm{HK}} ~\mathrm{d}Q(x).
  $$
\end{lemma}
\begin{proof}
  Assume without loss of generality that $\|h_{i,x}\|_{\mathrm{HK}} < \infty$ for all $i = 1, \ldots, m$ and $Q$-almost all $x \in ([0,1]^d)^{m-1}$. Otherwise, the right-hand side in the statement of the lemma is infinite, and the claim is trivial.

  We can write
  $$
    h_Q(t) = \int h_{i,x}(t) ~\mathrm{d}Q(x).
  $$
  Consider a collection of indices $1 \leq i_1 < \ldots < i_r \leq d$ and consider the restrictions $h_{Q;i_1, \ldots, i_r}$ and $h_{i,x ; i_1, \ldots, i_r}$ as defined in Eq.\@ \eqref{eq:hk_vitali_sum}. To make the notation less cluttered, we will write $h_Q^{i_1, \ldots, i_r}$ and $h_{i,x}^{i_1, \ldots, i_r}$ instead of $h_{Q; i_1, \ldots, i_r}$ and $h_{i,x; i_1, \ldots, i_r}$, respectively. It is clear frome the above representation that
  $$
    h_Q^{i_1, \ldots, i_r}(t) = \int h_{i,x}^{i_1, \ldots, i_r}(t) ~\mathrm{d}Q(x).
  $$
  Combining this with Jensen's inequality in the definition of the quasi-volume from Eq.\@ \eqref{eq:definition_quasivolume}, we find that
  \begin{equation}
    \label{eq:quasivolume_bound}
    \left|\Delta\left(h_Q^{i_1, \ldots, i_r};\bar{B}\right)\right| \leq \int \left|\Delta\left(h_{i,x}^{i_1, \ldots, i_r};\bar{B}\right)\right| ~\mathrm{d}Q(x).
  \end{equation}
  Recall from Section \ref{sec:HK_introduction} that $\mathcal{P}$ is the collection of all finite partitions of $[0,1]^d$ into hyperrectangles. It holds that
  \begin{equation}
    \label{eq:supremum_quasivolume}
    \sup_{P \in \mathcal{P}} \int \sum_{B \in P} \left|\Delta\left(h_{i,x}^{i_1, \ldots, i_r}; \bar{B}\right)\right| ~\mathrm{d}Q(x) \leq \int \sup_{P \in \mathcal{P}} \sum_{B \in P} \left|\Delta\left(h_{i,x}^{i_1, \ldots, i_r}; \bar{B}\right)\right| ~\mathrm{d}Q(x).
  \end{equation}
  Eqs.\@ \eqref{eq:quasivolume_bound} and \eqref{eq:supremum_quasivolume} can be combined to yield
  \begin{align*}
    V^{(d)}\left(h_Q^{i_1, \ldots, i_r}\right) & = \sup_{P \in \mathcal{P}} \sum_{B \in P}\left|\Delta\left(h_Q^{i_1, \ldots, i_r};\bar{B}\right)\right|                         \\
                                               & \leq \int \sup_{P \in \mathcal{P}} \sum_{B \in P} \left|\Delta\left(h_{i,x}^{i_1, \ldots, i_r}; B\right)\right| ~\mathrm{d}Q(x) \\
                                               & = \int V^{(d)}\left(h_{i,x}^{i_1, \ldots, i_r}\right) ~\mathrm{d}Q(x).
  \end{align*}
  Therefore,
  \begin{align*}
    \|h_Q\|_{\mathrm{HK}} & = \sum_{r=1}^d \sum_{1 \leq i_1 < \ldots < i_r \leq d} V^{(r)}\left(h_{Q}^{i_1, \ldots, i_r}\right)                           \\
                          & \leq \int \sum_{r=1}^d \sum_{1 \leq i_1 < \ldots < i_r \leq d} V^{(r)}\left(h_{i,x}^{i_1, \ldots, i_r}\right) ~\mathrm{d}Q(x) \\
                          & = \int \|h_{i,x}\|_{\mathrm{HK}} ~\mathrm{d}Q(x).
  \end{align*}
  This is what we wanted to prove.
\end{proof}

\begin{proof}[Proof of Theorem \ref{thm:characterisation_HKcm}]
  If $h \in \mathcal{HK}(c)_m$, then it trivially satisfies the conditions from the statement of this theorem. Let us therefore assume that $h$ is a function with the properties in the statement of the theorem, from which we will prove that $h \in \mathcal{HK}(c)_m$.

  Since $h$ is an element of $\mathcal{HK}(c)$ coordinate-wise, and elements of $\mathcal{HK}(c)$ are bounded in the supremum norm by $c$, it must hold that $\|h\|_\infty \leq c$. In particular, $h$ is bounded (and Borel measurable by assumption). It remains to verify that the function
  $$
    x \mapsto \int h(x_1, \ldots, x_{i-1}, x, x_{i+1}, \ldots, x_m) ~\mathrm{d}Q(x_1, \ldots, x_{i-1}, x_{i+1}, \ldots, x_m)
  $$
  is an element of $\mathcal{HK}(c)$ for any Borel product probability measure $Q = Q_1 \otimes \ldots \otimes Q_{m-1}$ on $([0,1]^d)^m$. But this is an immediate consequence of the already established inequality $\|h\|_\infty \leq c$, as well as Lemma \ref{lem:HK_integration_contraction}.
\end{proof}

The next lemma is a preparation for the proof of Theorem \ref{thm:product_uniformly_integrable}. It is a version for potentially non-measurable stochastic processes of the well-known fact that weak convergence plus uniform $p$-integrability implies convergence of the $p$-th moments. We wish to prove that, if $X_n$ and $X$ are processes indexed in some set $T$, with corresponding mean functions $\mu_n$ and $\mu$, then $X_n \rightsquigarrow X$ plus some appropriate moment condition implies $\mu_n \to \mu$ in $\ell^\infty(T)$. There is some added difficulty here because we require convergence in the supremum norm. Thus, our proof consists of two steps: First, pointwise convergence of $\mu_n$ to $\mu$, which is just a consequence of the previously mentioned well-known fact for convergence of the $p$-th moments. We then show that $\mu_n$, $n \in \mathbb{N}$, is a Cauchy sequence in $\ell^\infty(T)$, and since $\ell^\infty(T)$ is a Banach space when equipped with the supremum norm, this implies that $\mu_n \to \mu$ must hold not only in a pointwise sense, but also in $\ell^\infty(T)$. The proof that $\mu_n$, $n \in \mathbb{N}$, is Cauchy in $\ell^\infty(T)$ is itself not obvious. The key insight here is to establish a suitable connection between the condition of asymptotic tightness -- which is one part of weak convergence -- to a certain characterisation of relative compactness in the Banach space $\ell^\infty(T)$.

\begin{lemma}
  \label{lem:uniformly_integrable}
  Let $T$ be an arbitrary set and $X_n$, $n \in \mathbb{N}$, be a sequence in $\ell^\infty(T)$ such that $X_n \rightsquigarrow X$ for some tight limit $X$. Define the functions $\mu_n$ and $\mu$ by $\mu_n(t) = \mathbb{E} X_n(t)$ and $\mu(t) = \mathbb{E}X(t)$. If $\sup_n \mathbb{E}^* \|X_n\|_T^p < \infty$ for some $p > 1$, then $\mu \in \ell^\infty(T)$, and $\|\mu_n - \mu\|_T \to 0$ as $n \to \infty$.
\end{lemma}
\begin{proof}
  For any fixed $t \in T$, it holds that $X_n(t) \rightsquigarrow X(t)$ in $\mathbb{R}$ by the continuous mapping theorem, and
  $$
    \sup_{n \in \mathbb{N}} \mathbb{E} |X_n(t)|^p \leq \sup_{n \in \mathbb{N}} \mathbb{E}^*\|X_n\|_T^p < \infty.
  $$
  By Theorem 25.12 in \cite{billingsley:prob_and_measure}, it therefore holds that $X(t)$ is integrable, and
  \begin{equation}
    \label{eq:uniform_convergence_pintwise}
    \mu_n(t) \xrightarrow[n \to \infty]{} \mu(t)
  \end{equation}
  for any fixed $t \in T$.

  For $C > 0$ define a new process $X_n^{(C)}$ by
  $$
    X_n^{(C)}(t) = X_n(t) \textbf{1}\{|X_n(t)| < C\} + C \textbf{1}\{|X_n(t)| \geq C\},
  $$
  and $X^{(C)}$ analogously. Denote the corresponding mean functions by $\mu_n^{(C)}$ and $\mu^{(C)}$, respectively. Then
  \begin{equation}
    \label{eq:uniform_convergence_pintwise_truncated}
    \mu_n^{(C)}(t) \xrightarrow[n \to \infty]{} \mu^{(C)}(t)
  \end{equation}
  for any fixed $t \in T$, which can be shown in exactly the same way as Eq.\@ \eqref{eq:uniform_convergence_pintwise}. Furthermore,
  \begin{equation}
    \label{eq:outer_uniform_integrability_1}
    \|\mu_n - \mu_n^{(C)}\|_T \leq \mathbb{E}^* \left\|X_n - X_n^{(C)}\right\|_T \leq \mathbb{E}^* \left[\|X_n\|_T \textbf{1}\{\|X_n\|_T \geq C\}\right].
  \end{equation}
  Recall the following elementary proof of Hölder's inequality: For random variables $U$ and $V$, and a pair of Hölder conjugates $1/p + 1/q = 1$, it holds by Young's inequality \citep[e.g.\@ Theorem 156 in][with $\phi(x) = x^{p-1}$]{hardy_littlewood_polya:inequalities} that
  $$
    \frac{|U V|}{\|U\|_{L_p} \|V\|_{L_p}} \leq \frac{(|U|/\|U\|_{L_p})^p}{p} + \frac{(|V|/\|V\|_{L_q})^q}{q}.
  $$
  Integrate to see that $\|UV\|_{L_1} \leq \|U\|_{L_p}\|V\|_{L_q} (1/p + 1/q) = \|U\|_{L_p}\|V\|_{L_q}$. This argument only uses the sublinearity (as opposed to full linearity) of the expected value, which is also satisfied by the outer expectation. Therefore,
  $$
    \mathbb{E}^*|UV| \leq (\mathbb{E}^*|U|^p)^{1/p} (\mathbb{E}^*|V|^q)^{1/q}
  $$
  for any two real-valued random elements $U,V$. Apply this to the right-hand side of Eq.\@ \eqref{eq:outer_uniform_integrability_1} to see that
  \begin{align}
    \begin{split}
      \label{eq:outer_uniform_integrability_2}
      \|\mu_n(t) - \mu_n^{(C)}(t)\|_T & \leq (\mathbb{E}^* \|X_n\|_T^p)^{1/p} (\mathbb{E}^* \textbf{1}\{\|X_n\|_T \geq C\})^{1/q}                         \\
                                      & \leq \sup_{n \in \mathbb{N}}(\mathbb{E}^* \|X_n\|_T^p)^{1/p} (\mathbb{E}^* \textbf{1}\{\|X_n\|_T \geq C\})^{1/q}.
    \end{split}
  \end{align}
  The function $U \mapsto \textbf{1}\{\|U\|_T \geq C\}$ is upper semicontinuous on $\ell^\infty(T)$. By the Portmanteau theorem \citep[Theorem 1.3.4 in][]{van_der_vaart_wellner:weak_convergence}, it therefore holds that
  \begin{equation}
    \label{eq:outer_uniform_integrability_3}
    \limsup_{n \to \infty} \mathbb{E}^* \textbf{1}\{\|X_n\|_T \geq C\} \leq \mathbb{P}(\|X\|_T \geq C),
  \end{equation}
  and the right-hand side monotonically decreases to $0$ for $C \to \infty$ because $\|X\|_T$ is tight in $\mathbb{R}$. Now fix some $\varepsilon > 0$. Eqs.\@ \eqref{eq:outer_uniform_integrability_2} and \eqref{eq:outer_uniform_integrability_3} together imply that we can find two numbers $N_1 = N_1(\varepsilon)$ and $C = C(\varepsilon)$ such that
  \begin{equation}
    \label{eq:cauchy_1}
    \|\mu_n - \mu_n^{(C)}\|_T \leq \varepsilon/3
  \end{equation}
  for any $n \geq N_1$. For the remainder of the proof, $C$ will be this special choice $C(\varepsilon)$. Furthermore, the assumption $X_n \rightsquigarrow X$ in $\ell^\infty(T)$ also implies $X_n^{(C)} \rightsquigarrow X^{(C)}$ by the continuous mapping theorem. Fix some arbitrary $\delta > 0$. By Theorems 1.5.4 and 1.5.6 in \cite{van_der_vaart_wellner:weak_convergence}, we get the existence of a partition $T = T_1 \cup \ldots \cup T_K$ and some integer $N_2 = N_2(\delta)$ such that
  $$
    \mathbb{P}^*\left(\max_{k=1, \ldots, K} \sup_{s,t \in T_k} \left|X_n^{(C)}(s) - X_n^{(C)}(t)\right| > \delta/2\right) \leq \frac{\delta}{4C}
  $$
  for any $n \geq N_2$. Consider any $k =1, \ldots, K$ and $s,t \in T_k$. Since $|X_n^{(C)}(s) - X_n^{(C)}(t)| \leq 2C$ is always satisfied, the above inequality means that
  \begin{align}
    \begin{split}
      \label{eq:truncated_mean_functions_relatively_compact}
      \left|\mu_n^{(C)}(s) - \mu_n^{(C)}(t)\right| & = \left|\mathbb{E}X_n^{(C)}(s) - \mathbb{E}X_n^{(C)}(t)\right|                                                                           \\
                                                   & \leq \mathbb{E}^*\left[\max_{k=1, \ldots, K} \sup_{s,t \in T_k} \left|X_n^{(C)}(s) - X_n^{(C)}(t)\right|\right]                          \\
                                                   & \leq \delta/2 + 2C \mathbb{P}^*\left(\max_{k=1, \ldots, K} \sup_{s,t \in T_k} \left|X_n^{(C)}(s) - X_n^{(C)}(t)\right| > \delta/2\right) \\
                                                   & \leq \delta
    \end{split}
  \end{align}
  for all $n \geq N_2$. The identity $\mathbb{E}^* \textbf{1}_B = \mathbb{P}^*(B)$ for any set $B$, which we have used in the second inequality, holds by Lemma 1.2.3 in \cite{van_der_vaart_wellner:weak_convergence}. On the other hand, it is trivial to construct a partition $\tilde{T}_1, \ldots, \tilde{T}_L$ of $T$ such that
  \begin{equation}
    \label{eq:truncated_mean_functions_relatively_compact_2}
    \max_{n = 1, \ldots, N_2}\left|\mu_n^{(C)}(s) - \mu_n^{(C)}(t)\right| \leq \delta
  \end{equation}
  for any $s,t \in \tilde{T}_l$, $l = 1, \ldots, L$ (just observe that this can certainly be done for any fixed $n$, since every $\mu_n^{(C)}$ is bounded; then take intersections). Eqs.\@ \eqref{eq:truncated_mean_functions_relatively_compact} and \eqref{eq:truncated_mean_functions_relatively_compact_2} can be combined to see that
  $$
    \sup_{n \in \mathbb{N}}\left|\mu_n^{(C)}(s) - \mu_n^{(C)}(t)\right| \leq \delta
  $$
  for any $s,t \in T_k \cap \tilde{T}_l$, $k = 1, \ldots, K$, $l = 1, \ldots, L$. But this is just the characterisation of relative compactness in $\ell^\infty(T)$ \citep[Theorem IV.5.6 in][]{dunford_schwartz:1958}, and so the set $\{\mu_n^{(C)} ~|~ n \in \mathbb{N}\}$ is relatively compact in $\ell^\infty(T)$. Since compactness and sequential compactness are equivalent in metric spaces, Eq.\@ \eqref{eq:uniform_convergence_pintwise_truncated} now implies
  $$
    \|\mu_n^{(C)} - \mu^{(C)}\|_T \xrightarrow[n \to \infty]{} 0.
  $$
  Any convergent sequence in a metric space is Cauchy, and so we may choose some $N_3 = N_3(\varepsilon)$ such that
  \begin{equation}
    \label{eq:cauchy_2}
    \|\mu_n^{(C)} - \mu_m^{(C)}\|_T \leq \varepsilon/3
  \end{equation}
  for any $n,m \geq N_3$. Eqs.\@ \eqref{eq:cauchy_1} and \eqref{eq:cauchy_2} now give us
  $$
    \|\mu_n - \mu_m\|_T \leq \|\mu_n^{(C)} - \mu_m^{(C)}\|_T + \|\mu_n - \mu_n^{(C)}\|_T + \|\mu_m - \mu_m^{(C)}\|_T \leq \varepsilon
  $$
  for any $n,m \geq N_1 \lor N_3$, i.e.\@ the sequence $\mu_n$, $n \in \mathbb{N}$, is Cauchy and hence convergent in the Banach space $\ell^\infty(T)$. By Eq.\@ \eqref{eq:uniform_convergence_pintwise}, its limit must be equal to $\mu$.
\end{proof}

The proof of Theorem \ref{thm:product_uniformly_integrable} holds no surprises. It is a simple application of Lemma \ref{lem:uniformly_integrable}.

\begin{proof}[Proof of Theorem \ref{thm:product_uniformly_integrable}]
  We begin by proving the claim in Eq.\@ \eqref{eq:moment_convergence_conclusion_1}. The case $m = 1$ is clearly fulfilled since it is just a restatement of Eq.\@ \eqref{eq:moment_convergence_assumption} (recall that $\mathcal{F}_1 = \mathcal{F}$). Let us therefore assume that we have proven our claim for some $m$, and we will show that it must also hold for $m+1$. Fix some $h \in \mathcal{F}_{m+1}$. Similar to Eq.\@ \eqref{eq:v-1}, write
  \begin{align*}
     & r_n[P_n^{m+1}(h) - \mu_n^{m+1}(h)]                                                                                            \\
     & = \iint h(x_1, \ldots, x_{m+1}) ~\mathrm{d}\left[r_n(P_n - \mu_n)\right](x_{m+1})~\mathrm{d}\mu_n^m(x_1, \ldots, x_m)         \\
     & \quad + \iint h(x_1, \ldots, x_{m+1}) ~\mathrm{d}\left[r_n(P_n^m - \mu_n^m)\right](x_1, \ldots, x_m) ~\mathrm{d}P_n(x_{m+1}).
  \end{align*}
  By Jensen's inequality, this implies
  \begin{align*}
     & \left|r_nP_n^{m+1}(h) - \mu_n^{m+1}(h)\right|                                                                                                 \\
     & \leq \int \left|\int h(x_1, \ldots, x_{m+1})  ~\mathrm{d}\left[r_n(P_n - \mu_n)\right](x_{m+1}) \right|~\mathrm{d}\mu_n^m(x_1, \ldots, x_m)   \\
     & \quad + \int \left|\int h(x_1, \ldots, x_{m+1})~\mathrm{d}\left[r_n(P_n^m - \mu_n^m)\right](x_1, \ldots, x_m) \right| ~\mathrm{d}P_n(x_{m+1}) \\
     & \leq \|r_n(P_n - \mu_n)\|_\mathcal{F} + \left\|r_n\left(P_n^m - \mu_n^m\right)\right\|_{\mathcal{F}_m},
  \end{align*}
  where the last inequality holds because $h \in \mathcal{F}_{m+1}$. The right-hand side is independent of $h$, and so taking the supremum over all $h \in \mathcal{F}_{m+1}$ does not change this bound. Therefore,
  \begin{align}
    \begin{split}
      \label{eq:bias_bound}
       & \sup_{n \in \mathbb{N}}\mathbb{E}^* \left\|r_n(P_n^{m+1} - \mu_n^{m+1})\right\|_{\mathcal{F}_m+1}                                                                                 \\
       & \leq \sup_{n \in \mathbb{N}}\mathbb{E}^* \left\|r_n(P_n - \mu_n)\right\|_{\mathcal{F}} + \sup_{n \in \mathbb{N}}\mathbb{E}^* \left\|r_n(P_n^m - \mu_n^m)\right\|_{\mathcal{F}_m},
    \end{split}
  \end{align}
  and the right-hand is bounded by assumption. This proves Eq.\@ \eqref{eq:moment_convergence_conclusion_1}. Under the additional assumption that the original limiting process $G$ is tight and mean-zero, Eq.\@ \eqref{eq:moment_convergence_conclusion_2} is an immediate consequence of Eq.\@ \eqref{eq:moment_convergence_conclusion_1} and Lemma \ref{lem:uniformly_integrable}. Finally, the statement about replacing $(\mathbb{E}P_n)^m$ with $\mathbb{E}[P_n^m]$ is a consequence of Eqs.\@ \eqref{eq:moment_convergence_identity} and \eqref{eq:moment_convergence_conclusion_2}.
\end{proof}

\section{Proofs for Processes Based on Order Statistics and Their Concomitants}
\label{sec:proofs_order_concomitants}
\subsection{Order Statistics}
\label{sec:proofs_order_statistics}
Our goal in this section is to prove a process convergence result of the form
\begin{equation}
  \label{eq:goal_order}
  \left[(f,g) \mapsto \frac{1}{\sqrt{n}} \sum_{i=1}^{n-1} \left\{f(X_{n,i}') g(X_{n,i+1}') - \mathbb{E}\left[f(X_{n,i}') g(X_{n,i+1}')\right]\right\}\right] \rightsquigarrow G
\end{equation}
in $\ell^\infty(\mathcal{F} \times \mathcal{F})$ for appropriate function classes $\mathcal{F}$, where $X_{n,1}' \leq \ldots \leq X_{n,n}'$ are the order statistics of an i.i.d.\@ sample $X_1, \ldots, X_n$ (with ties broken at random, as usual). This convergence result will be used in the following sections to investigate the limiting behaviour of the empirical process associated with the concomitants of the order statistics, and ultimately in the study of Chatterjee's rank correlation. Following the usual procedure, we break up the weak convergence \eqref{eq:goal_order} into two parts: Convergence of the finite-dimensional projections and asymptotic tightness. Unfortunately, both of these are somewhat difficult to establish.

We begin with two technical lemmas: A purely analytical one, and one giving a tail bound on uniform spacings.

\begin{lemma}
  \label{lem:alpha_exp_probs}
  For $\alpha > 0$ and $t \geq 1$, define $f_\alpha(t) = 1 - [1 - \exp(-t^\alpha)/t]^t$. Then $\exp(t^\alpha)f_\alpha(t) \to 1$ for $t \to \infty$.
\end{lemma}
\begin{proof}
  By Taylor's theorem, it holds for any $x \in (0,1)$ that
  $$
    \log(1-x) = -x + R(x),
  $$
  where $|R(x)| \leq |1 - \xi|^{-2} x^2/2$ for some $\xi \in (0,x)$. Set $x = x_\alpha = \exp(-t^\alpha)/t$ to find that
  $$
    \left[1 - \frac{\exp(-t^\alpha)}{t}\right]^t = \exp\{t[-x_\alpha + R(x_\alpha)]\} = \exp\left\{-\exp(-t^\alpha) + t R(x_\alpha)\right\}.
  $$
  Next use Taylor's theorem on the exponential function to find that $\exp(x) = 1 + x + S(x)$ for any $x \in (-1,1)$, where $|S(x)| \leq (e/2) x^2$. Using this Taylor expansion on the outer exponential function in the above display, we find that
  \begin{equation}
    \label{eq:exp_diffs}
    1 - \left[1 - \frac{\exp(-t^\alpha)}{t}\right]^t = \exp(-t^\alpha) - t R(x_\alpha) - S(y_\alpha),
  \end{equation}
  where $y_\alpha = - \exp(-t^\alpha) + t R(x_\alpha)$. Since $t \geq 1$, we have $x_\alpha \leq e^{-1}$, and the term $|1 - \xi|^{-2}$ in the bound for $R(x_\alpha)$ can be bounded by $|1 - e^{-1}|^{-2}$ for any $t$. We therefore get
  $$
    |t R(x_\alpha)| \leq \frac{1}{2|1 - e^{-1}|^2} \exp(-2t^\alpha)/t,
  $$
  which also implies
  $$
    |S(y_\alpha)| \lesssim2 \exp(-2t^\alpha) + 2 t^2 R^2(x_\alpha) \lesssim \exp(-2t^\alpha),
  $$
  where $\lesssim$ is hiding universal constants. Combining these two error bounds with Eq.\@ \eqref{eq:exp_diffs} proves our claim.
\end{proof}

\begin{lemma}
  \label{lem:spacings_lemma}
  Let $Z_1, \ldots, Z_n$ be i.i.d.\@ $\mathcal{U}[0,1]$ random variables. For $i = 1, \ldots, n+1$, define the $i$-th spacing as $\delta_{n,i} = Z_{n,i}' - Z_{n,i-1}$, where we set $Z_{n,0}' = 0$ and $Z_{n,n+1}' = 1$, and let $D_n = \max_{i=1, \ldots, n+1} \delta_{n,i}$ be the largest spacing. Then the following holds: For any $e^{-1} \leq \alpha < 1$, there exists a constant $C_\alpha > 0$ depending only on $\alpha$ such that
  $$
    \mathbb{P}\left(D_n > 5 (n+1)^{\alpha-1}\right) \leq C_\alpha \exp[-(n+1)^\alpha]
  $$
  for all $n \in \mathbb{N}$.
\end{lemma}
\begin{proof}
  By Eqs.\@ (d) and (e) in \cite{shorack_wellner:1986}, Chapter 21, Section 2, pp.\@ 726f.\@ it holds for any $t > 0$ that
  \begin{align*}
    \left(1 - \frac{\exp(-t)}{n+1}\right)^{n+1} & = \mathbb{P}\left(\left[(n+1)D_n - \log (n+1)\right]\Sigma_n + (\Sigma_n - 1) \log (n+1) \leq t\right)   \\
                                                & = \mathbb{P}\left(D_n \leq \frac{\log (n+1) + \Sigma_n^{-1} [t - (\Sigma_n - 1)\log (n+1)]}{n+1}\right),
  \end{align*}
  where $\Sigma_n = Y_1 + \ldots + Y_{n+1}$ for $Y_1, \ldots, Y_n$ being independent random variables following an exponential distribution with expected value $1/(n+1)$. Since $\Sigma_n \geq 0$, we have $-(\Sigma_n - 1)\log (n+1) \leq \log (n+1)$. The factor $\Sigma_n^{-1}$ can be controlled by considering the cases $\Sigma_n \geq 1/2$ and $\Sigma_n < 1/2$ separately. We therefore get
  $$
    \left(1 - \frac{\exp(-t)}{n+1}\right)^{n+1} \leq \mathbb{P}\left(D_n \leq \frac{3 \log (n+1) + 2t}{n+1}\right) + \mathbb{P}(\Sigma_n < 1/2).
  $$
  Choose $t = (n+1)^\alpha$ for some $e^{-1} < \alpha < 1$ to see that
  $$
    \left(1 - \frac{\exp[-(n+1)^\alpha]}{n+1}\right)^{n+1} - \mathbb{P}(\Sigma_n < 1/2) \leq \mathbb{P}\left(D_n \leq \frac{3 \log (n+1) + 2 (n+1)^\alpha}{n+1}\right).
  $$
  Now use the inequality $\log(n+1) \leq (n+1)^\alpha$ which holds because $\alpha \geq e^{-1}$ to further bound the probability on the right-hand side. Then by subtracting both sides from $1$ we find that
  \begin{align}
    \begin{split}
      \label{eq:spacings_prob_1}
      \mathbb{P}\left(D_n > 5 (n+1)^{\alpha-1}\right) & \leq 1 - \left(1 - \frac{\exp[-(n+1)^\alpha]}{n+1}\right)^{n+1} + \mathbb{P}(\Sigma_n < 1/2) \\
                                                      & \leq c_\alpha \exp[-(n+1)^\alpha] + \mathbb{P}(\Sigma_n < 1/2)
    \end{split}
  \end{align}
  for some constant $c_\alpha$ depending only on $\alpha$ by Lemma \ref{lem:alpha_exp_probs}. Let us now bound the probability that $\Sigma_n < 1/2$. Recall that an exponential distribution with expected value $1/\theta$ has moment generating function $t \mapsto \theta / (\theta - t)$ for $t < \theta$ \citep[Example 21.3 in][]{billingsley:prob_and_measure}. Chernoff's bound therefore gives us
  \begin{equation}
    \label{eq:spacings_prob_2}
    \mathbb{P}(\Sigma_n < 1/2) \leq \inf_{t < 0} \left[\exp(-t/2) \left(\frac{n+1}{n+1-t}\right)^{n+1}\right] = \exp[-(n+1)(\log 2 - 1/2)],
  \end{equation}
  where we have used that the infimum is attained at $t = -(n+1)$. The right-hand side decreases exponentially in $n+1$ because $\log 2 - 1/2 > 0$. At the cost of increasing the constant $c_\alpha$ to some larger constant $C_\alpha$, still only depending on $\alpha$, we can combine Eqs.\@ \eqref{eq:spacings_prob_1} and \eqref{eq:spacings_prob_2} to find that
  $$
    \mathbb{P}\left(D_n > 5 (n+1)^{\alpha-1}\right) \leq C_\alpha \exp[-(n+1)^\alpha],
  $$
  and this is what we wanted to prove.
\end{proof}

Our approach to prove convergence of the finite-dimensional projection of the process in \eqref{eq:goal_order} is to write each summand as
$$
  f(X_{n,i}') g(X_{n,i+1}') = f(X_{n,i}')g(X_{n,i}') + f(X_{n,i}')\left[g(X_{n,i+1}') - g(X_{n,i}')\right].
$$
Subtracting the expectations and summing the first term on right-hand side over $i = 1, \ldots, n-1$ yields
$$
  \frac{1}{\sqrt{n}} \sum_{i=1}^n f(X_i)g(X_i) - \mathbb{E}\left[f(X_i)g(X_i)\right] + \mathcal{O}\left(\frac{1}{n}\right),
$$
assuming that $|f|, |g| \leq 1$, since the order statistics $X_{n,1}', \ldots, X_{n,n}'$ are simply a permutation of the original sample $X_1, \ldots, X_n$. For any fixed $f$ and $g$, this sum then converges by the usual central limit theorem. The difficult part is to show that the sum over the remainder terms is negligible. We will do this using the Efron-Stein inequality: For independent random variables $X_1, \ldots, X_n$, a measurable function $f$ and $Z = f(X_1, \ldots, X_n)$, it holds that
$$
  \mathrm{Var}(Z) \leq \frac{1}{2}\sum_{i=1}^n \mathbb{E}\left[(Z - Z_i')^2\right],
$$
where $Z_i'$ is defined just like $Z$, but with $X_i$ replaced by an independent copy $X_i'$. Thus, the Efron-Stein inequality relates the variance of a function of independent random variables to the expected change when replacing just one of the random variables by an independent copy. In the study of order statistics, this is a useful concept. For consider the order statistics $X_{n,1}', \ldots, X_{n,n}'$ of some sample $X_1, \ldots, X_n$, and imagine how the order statistics change when replacing one of the $X_i$ by some other number (here, an independent copy). They will change only in at most two places: Where the old observation $X_i$ is removed, and where the new observation $X_i'$ is inserted. Since we are interested in blocks $(X_{n,i}', X_{n,i+1}')$ of consecutive order statistics, this means that at most four such blocks will be affected by the switch from $X_i$ to $X_i'$. Thus, while in our application the random variables $Z$ and $Z_i'$ are defined as sums with $n$ summands each, their difference $Z - Z_i'$ is a sum with at most four summands, and it will be seen that these four summands can be bounded by differences of the form  $f(X_{n,i+1}') - f(X_{n,i}')$ and $g(X_{n,i+1}') - g(X_{n,i}')$. To control these differences, we will make the assumption that $f$ and $g$ are Borel measurable. Borel measurable functions are almost continuous by Lusin's theorem, and so $f(X_{n,i+1}') - f(X_{n,i}')$ and $g(X_{n,i+1}') - g(X_{n,i}')$ will be small with high probability.

\begin{lemma}
  \label{lem:nn_remainder_0}
  Let $f,g : [0,1] \to [0,1]$ be Borel measurable functions and define
  $$
    \Delta_n = \frac{1}{\sqrt{n}} \sum_{i=1}^{n-1} f(X_{n,i}')\left[g(X_{n,i+1}') - g(X_{n,i}')\right].
  $$
  Then $\mathrm{Var}(\Delta_n) \to 0$ as $n \to \infty$.
\end{lemma}
\begin{proof}
  In this proof we assume without loss of generality that the data $X_1, \ldots, X_n$ follow a $\mathcal{U}[0,1]$-distribution. If this is not the case, we can replace the functions $f$ and $g$ by $f \circ F^{-1}$ and $g \circ F^{-1}$, respectively, where $F$ is the distribution function of $X_1$, and apply these functions to an i.i.d.\@ $\mathcal{U}[0,1]$ process.

  Fix some $0 < \varepsilon < 1/4$. By Lusin's theorem \citep[Exercise III.9.18 in][]{dunford_schwartz:1958}, there exists a compact set $K \subseteq [0,1]$ with $\mathbb{P}(X_1 \in K) > 1 - \varepsilon$ such that the restrictions of both $f$ and $g$ onto $K$ are continuous. Let $(Y_1, \ldots, Y_n)$ be an independent copy of $(X_1, \ldots, X_n)$ and define, for any $i = 1, \ldots, n$, $\Delta_n^{(i)}$ just like $\Delta_n$, but with $X_i$ replaced by $Y_i$ (and all other $X_j$, $j \neq i$, unchanged). To bound the difference between $\Delta_n$ and $\Delta_n^{(i)}$, let $Z_j = X_j$ for $j = 1, \ldots, n$ and $Z_{n+1} = Y_i$, and then consider the order statistics $Z_{n,1}' \leq \ldots \leq Z_{n,n+1}'$ (there are almost surely no ties because we assumed the $X_i$ to be distributed uniformly on the unit interval). Let $A$ and $B$ be the random indices such that $Z_{n,A}' = X_i$ and $Z_{n,B}' = Y_i$. If $1 < A, B < n+1$ and $|A-B| > 1$, then $\Delta_n - \Delta_n^{(i)}$ consists of only four summands. For instance, if we assume that $B > A$, we get
  \begin{align}
    \begin{split}
      \label{eq:Delta_n_Delta_n_i}
       & \sqrt{n}\left(\Delta_n - \Delta_n^{(i)}\right)                                                                    \\
       & = f(Z_{n,A-1}')\left[g(Z_{n,A}') - g(Z_{n,A-1}')\right] - f(Z_{n,A-1}')\left[g(Z_{n,A+1}') - g(Z_{n,A-1}')\right] \\
       & + f(Z_{n,A}')\left[g(Z_{n,A+1}') - g(Z_{n,A}')\right] - f(Z_{n,A+1}')\left[g(Z_{n,A+2}') - g(Z_{n,A+1}')\right]   \\
       & + f(Z_{n,B-2}')\left[g(Z_{n,B-1}') - g(Z_{n,B-2}')\right] - f(Z_{n,B-1}')\left[g(Z_{n,B}) - g(Z_{n,B-1}')\right]  \\
       & + f(Z_{n,B-1}')\left[g(Z_{n,B+1}') - g(Z_{n,B-1}')\right] - f(Z_{n,B}')\left[g(Z_{n,B+1}') - g(Z_{n,B})\right].
    \end{split}
  \end{align}
  This identity comes frome the fact that we can construct the order statistics of both $X_1, \ldots, X_n$ and $X_1, \ldots, X_{i-1}, Y_i, X_{i+1}, \ldots, X_n$ from $Z_{n,1}' \leq \ldots \leq Z_{n,n+1}'$ by deleting either $Z_{n,B}'$ or $Z_{n,A}'$. A similar identity holds if $1 < A,B < n+1$, $|A-B| > 1$ and $B < A$. If one of the indexes $A$ and $B$ are $1$ or $n+1$, or if they are directly adjacent, i.e.\@ $|A-B| = 1$, then the corresponding identity consists of fewer than $4$ summands. Given $Z_{n,1}', \ldots, Z_{n,n+1}'$, the random vector $(A, B)$ is uniformly distributed on all possible choices $(a,b)$ such that $1 \leq a \neq b \leq n+1$, because the random variables $Z_1, \ldots, Z_{n+1}$ are i.i.d. Hence, from Eq.\@ \eqref{eq:Delta_n_Delta_n_i}, we get
  \begin{align}
    \begin{split}
      \label{eq:Delta_n_Delta_n_i_conditional}
       & n \mathbb{E}\left[\left(\Delta_n - \Delta_n^{(i)}\right)^2 ~|~ Z_{n,1}', \ldots, Z_{n,n+1}'\right]                                    \\
       & \lesssim \frac{1}{n} \sum_{a = 2}^{n+1} |g(Z_{n,a}') - g(Z_{n,a-1}')|^2 + \frac{1}{n} \sum_{a = 1}^{n}|g(Z_{n,a+1}') - g(Z_{n,a}')|^2 \\
       & \qquad + \frac{1}{n} \sum_{a = 2}^{n+1}|g(Z_{n,a+1}') - g(Z_{n,a-1}')|^2 + \mathcal{O}\left(\frac{1}{n}\right)
    \end{split}
  \end{align}
  where $\lesssim$ is hiding universal constants and the term $\mathcal{O}(1/n)$ is the result of the indices $a \neq b$ such that $\{a,b\} \cap \{1,n+1\} \neq \emptyset$ or $|a-b| = 1$, of which there are $\mathcal{O}(n)$ many. The constant hidden in the Landau symbol $\mathcal{O}(1/n)$ is also universal. Now define the event
  $$
    \mathcal{E}_n = \left\{\sum_{j=1}^{n+1} \textbf{1}_K(Z_j) \geq n(1-2\varepsilon) ~\textrm{and}~ D_{n+1} \leq 5(n+2)^{-1/2}\right\},
  $$
  where $D_{n+1}$ denotes the largest spacing of the $Z_1, \ldots, Z_{n+1}$ as in Lemma \ref{lem:spacings_lemma}. $\mathcal{E}_n$ is measurable with respect to the $\sigma$-algebra generated by $Z_{n,1}', \ldots, Z_{n,n+1}'$. Furthermore, on $\mathcal{E}_n$, at least $n(1 - 4\varepsilon)$ many pairs $(Z_{n,j}', Z_{n,j+1}')$ are elements of $K \times K$. This is because the number of indices $j$ for which this is not true is maximised if in the order statistics $Z_{n,1}' \leq \ldots \leq Z_{n,n+1}'$, an observation belonging to $K$ is followed by one not belonging to $K$, and vice-versa (for as long as this is possible, i.e.\@ until we have run out of observations not lying in $K$). Sorted in this way, every $Z_{n,j}'$ not belonging to $K$ causes the two blocks $(Z_{n,j-1}', Z_{n,j}')$ and $(Z_{n,j}', Z_{n,j+1}')$ to not belong to $K \times K$. On $\mathcal{E}_n$, there are at most $2n\varepsilon$ many observations among the $Z_1, \ldots, Z_{n+1}$ not belonging to $K$, and so $(Z_{n,j}', Z_{n,j+1}') \in K \times K$ for at least $n(1-4\varepsilon)$ many $j$. Essentially the same argument can be made for blocks of the form $(Z_{n,j-1}', Z_{n,j+1}')$. Furthermore, on $\mathcal{E}_n$, it holds that
  $$
    |Z_{n,j+1}' - Z_{n,j-1}| \leq 2 D_{n+1} \leq 10 (n+2)^{-1/2}
  $$
  for any $j = 2, \ldots, n$, and the same inequality holds for $|Z_{n,j}' - Z_{n,j-1}|$ for $j = 1, \ldots, n+1$. Writing
  $$
    \omega_n = \sup\left\{|g(x) - g(x')| ~:~ x,x' \in K \land |x-x'| \leq 10 (n+2)^{-1/2}\right\},
  $$
  it follows from Eq.\@ \eqref{eq:Delta_n_Delta_n_i_conditional} that
  \begin{equation}
    \label{eq:Delta_n_Delta_n_i_conditional_2}
    \mathbb{E}\left[\textbf{1}_{\mathcal{E}_n}\left(\Delta_n - \Delta_n^{(i)}\right)^2 ~|~ Z_{n,1}', \ldots, Z_{n,n+1}'\right] \lesssim \frac{(1-4\varepsilon)\omega_n^2 + 4\varepsilon + n^{-1}}{n},
  \end{equation}
  where $\lesssim$ is again hiding a universal constant. The term $4\varepsilon$ comes from the at most $4n\varepsilon$ many blocks $(Z_{n,j}', Z_{n,j+1}') \notin K \times K$, in which case we use trivial bound $|g(x) - g(x')| \leq 1$. The same trivial bound gives us, with $\Omega \setminus \mathcal{E}_n$ denoting the complement of $\mathcal{E}_n$,
  \begin{equation}
    \label{eq:Delta_n_Delta_n_i_conditional_3}
    \mathbb{E}\left[\textbf{1}_{\Omega \setminus \mathcal{E}_n}\left(\Delta_n - \Delta_n^{(i)}\right)^2 ~|~ Z_{n,1}', \ldots, Z_{n,n+1}'\right] \lesssim n^{-1} \textbf{1}_{\Omega \setminus \mathcal{E}_n}.
  \end{equation}
  Eqs.\@ \eqref{eq:Delta_n_Delta_n_i_conditional_2} and \eqref{eq:Delta_n_Delta_n_i_conditional_3} combined with the law of total expectation give us
  $$
    \mathbb{E}\left[\left(\Delta_n - \Delta_n^{(i)}\right)^2\right] \lesssim \frac{(1 - 4\varepsilon)\omega_n^2 + 4\varepsilon + n^{-1} + \mathbb{P}\left(\Omega \setminus \mathcal{E}_n\right)}{n}.
  $$
  By the Efron-Stein inequality \citep[Theorem 3.1 in][]{boucheron_etal:2013} we now get
  \begin{equation}
    \label{eq:efron_stein_Delta_n}
    \mathrm{Var}(\Delta_n) \leq \frac{1}{2} \sum_{i=1}^n \mathbb{E}\left[\left(\Delta_n - \Delta_n^{(i)}\right)^2\right] \lesssim (1 - 4\varepsilon)\omega_n^2 + 4\varepsilon + n^{-1} + \mathbb{P}\left(\Omega \setminus \mathcal{E}_n\right).
  \end{equation}
  Furthermore,
  \begin{align*}
    \mathbb{P}\left(\Omega \setminus \mathcal{E}_n\right) & \leq \mathbb{P}\left(\sum_{j=1}^{n+1} \textbf{1}_K(Z_j) < n(1-2\varepsilon)\right) +  \mathbb{P}\left(D_{n+1} > 5(n+2)^{-1/2}\right) \\
                                                          & \leq \exp\left(-2n\varepsilon^2\right) + C_{1/2} \exp\left(-\sqrt{n+2}\right)
  \end{align*}
  by Hoeffding's inequality and Lemma \ref{lem:spacings_lemma}. $C_{1/2}$ is the universal constant from the statement of that lemma. Finally, $\omega_n \to 0$ as $n \to 0$ because the restriction of $g$ to $K$ is continuous, hence uniformly continuous. Using these observations in Eq.\@ \eqref{eq:efron_stein_Delta_n} reveals that
  $$
    \limsup_{n \to \infty} \mathrm{Var}(\Delta_n) \lesssim 4\varepsilon.
  $$
  Since $\varepsilon$ can be arbitrarily small and the constant hidden by the symbol $\lesssim$ is universal, this proves our claim.
\end{proof}

\begin{lemma}
  \label{lem:Sn_fidi}
  Let $m \in \mathbb{N}$ and suppose that $f_j, g_j : [0,1] \to [0,1]$ are Borel measurable functions for $j = 1, \ldots, m$. Define the random vector $S_n = (S_{n,1}, \ldots, S_{n,m})$ by
  $$
    S_{n,j} = \frac{1}{\sqrt{n}} \sum_{i=1}^n f_j(X_{n,i}')g_j(X_{n,i+1}').
  $$
  Then $S_n - \mathbb{E}S_n \rightsquigarrow \mathcal{N}(0, \Sigma)$ as $n \to \infty$, where $\Sigma_{kl} = \mathrm{Cov}[f_k(X_1)g_k(X_1), f_l(X_1)g_l(X_1)]$.
\end{lemma}
\begin{proof}
  Define the random vectors $\tilde{S}_n = (\tilde{S}_{n,1}, \ldots, \tilde{S}_{n,m})$ and $\Delta_n = (\Delta_{n,1}, \ldots, \Delta_{n,m})$ by
  $$
    \tilde{S}_{n,j} = \frac{1}{\sqrt{n}} \sum_{i=1}^n f_j(X_i)g_j(X_i)
  $$
  and
  $$
    \Delta_{n,j} = \frac{1}{\sqrt{n}} \sum_{i=1}^{n-1} f_j(X_{n,i}')\left[g_j(X_{n,i+1}') - g_j(X_{n,i}')\right].
  $$
  Then, for any $j = 1, \ldots, m$,
  $$
    S_{n,j} = \frac{1}{\sqrt{n}} \sum_{i=1}^{n-1} f_j(X_{n,i}')g_j(X_{n,i}') + \Delta_{n,j} = \tilde{S}_{n,j} + \Delta_{n,j} + \mathcal{O}\left(\frac{1}{\sqrt{n}}\right),
  $$
  with the constant hidden in the Landau symbol being universal. By Lemma \ref{lem:nn_remainder_0}, we get that $\|\Delta_n - \mathbb{E}\Delta_n\|_2^2 \to 0$ in probability as $n \to \infty$. Hence $S_n - \mathbb{E}S_n = \tilde{S}_n - \mathbb{E}\tilde{S}_n + R_n$, where $\|R_n\|_2 \to 0$ in probability. Our claim follows from the usual multivariate central limit theorem \citep[e.g.\@ Theorem 29.5 in][]{billingsley:prob_and_measure} and Slutsky's Lemma.
\end{proof}

We have now established convergence of the finite dimensional projections of the process in \eqref{eq:goal_order}, and it remains to establish the asymptotic tightness, or equivalently, the asymptotic uniform equicontinuity in probability \citep[cf.\@ Section 1.5 in][]{van_der_vaart_wellner:weak_convergence}. For this, we need to bound a single increment of the process, which will later be used in a maximal inequality. We will bound the increments in the Orlicz norm.

Recall the definition of the exponential Orlicz norms $\|\cdot\|_{\psi_r}$, $r \geq 1$,
$$
  \|X\|_{\psi_r} = \inf\left\{C > 0 ~|~ \mathbb{E}\psi_r(X/C) \leq 1\right\},
$$
where $\psi_r(x) = \exp(x^r) - 1$. It is known that if $X$ satisfies the tail bound $\mathbb{P}(|X| \geq t) \leq K \exp(-C t^r)$ for all $t \geq 0$, then $\|X\|_{\psi_r} \leq (1+K)/C$ \citep[this is Lemma 2.2.1 in][]{van_der_vaart_wellner:weak_convergence}. The following lemma is a (very slight) generalisation of this statement, allowing for mixed tail bounds.

\begin{lemma}
  \label{lem:mixed_tailbound}
  Let $K, C_1, C_2, T > 0$ and $1 \leq r \leq 2$ be constants, and $X$ a random variable satisfying the mixed tail bound
  $$
    \mathbb{P}(|X| > t) \leq \begin{cases}
      K \exp(-C_1 t^2) & \textrm{if } t < T,    \\
      K \exp(-C_2 t^r) & \textrm{if } t \geq T.
    \end{cases}
  $$
  for all $t > 0$. Then $\|X\|_{\psi_r} \leq K^* \max\{C_1^{-1/2}, C_2^{-1/r}\}$, where $K^*$ is a constant depending only on $K$.
\end{lemma}
\begin{proof}
  Observe that
  \begin{align*}
    \mathbb{P}\left(\textbf{1}\{|X| \leq T\}|X| > t\right) & \leq \begin{cases}
                                                                    0 & \quad \textrm{if } t \geq T, \\ K \exp(-C_1 t^2) &\quad \textrm{if } t < T
                                                                  \end{cases} \\
                                                           & \leq K \exp(-C_1 t^2)
  \end{align*}
  for all $t > 0$, which implies $\|\textbf{1}\{|X| \leq T\}X\|_{\psi_2} \leq \sqrt{(1+K)/C_1}$ by Lemma 2.2.1 in \cite{van_der_vaart_wellner:weak_convergence}. Similarly,
  \begin{align*}
    \mathbb{P}\left(\textbf{1}\{|X| > T\}|X| > t\right) & \leq \begin{cases}
                                                                 \mathbb{P}(|X| > T) & \quad \textrm{if } t \leq T, \\ K \exp(-C_2 t^r) &\quad \textrm{if } t > T
                                                               \end{cases} \\
                                                        & \leq K \exp\left(- C_2 \max\{t^r, T^r\}\right)                                                       \\
                                                        & \leq K \exp\left(- C_2 t^r\right).
  \end{align*}
  Again by Lemma 2.2.1 in \cite{van_der_vaart_wellner:weak_convergence}, this gives us $\|\textbf{1}\{|X| > T\}X\|_{\psi_r} \leq [(1+K) / C_2]^{1/r}$. Furthermore, $\|\cdot\|_{\psi_r} \leq (\log 2)^{1/2 - 1/r} \|\cdot\|_{\psi_2}$ \citep[Exercise 2.2.5 in][]{van_der_vaart_wellner:weak_convergence}, which, combined with the triangle inequality, yields
  $$
    \|X\|_{\psi_r} \leq (\log 2)^{1/2 - 1/r}\|\textbf{1}\{|X| \leq T\}X\|_{\psi_2} + \|\textbf{1}\{|X| > T\}X\|_{\psi_r}.
  $$
  Our claim now follows from the previously established bounds for the Orlicz norms on the right-hand side, and the inequality $(\log 2)^{1/2 - 1/r} \leq (\log 2)^{-1/2}$.
\end{proof}

We can now prove the main fact about the increments of the process in \eqref{eq:goal_order}. As we have already stated, we will bound each increment in the Orlicz norm (of order $1 < r < 2$), and the bound will be in terms of the product semimetric on $\mathcal{F} \times \mathcal{F}$ induced by the $L_1(X_1)$-seminorm on $\mathcal{F}$. This result should be thought of as bounding the increment in a strong norm (the Orlicz norm) by a weak metric, which is crucial for the maximal inequality which we will employ later.

The statement covers two cases: One where the product semimetric is large, and one where it is small. The latter case is easy; most of the proof is occupied with the former, and it is rather involved. Its general idea is as follows: Since we want to control the increment in the Orlicz norm of order $1 < r < 2$, we will need to establish tail bounds which are somewhere between subexponential and sub-Gaussian. For this, we will use a version of Herbst's argument combined with a modified logarithmic Sobolev inequality. Herbst's argument is the following observation: If $Z$ is an integrable random variable such that
\begin{equation}
  \label{eq:goal_sobolev}
  \mathrm{Ent}\left(e^{\lambda Z}\right) \leq \frac{\lambda^2 v}{2} \mathbb{E}e^{\lambda Z}
\end{equation}
for some $v > 0$ and $\lambda > 0$, then
\begin{equation}
  \label{eq:goal_herbst}
  \log \mathbb{E} e^{\lambda(Z - \mathbb{E}Z)} \leq \frac{\lambda^2 v}{2}
\end{equation}
for all $\lambda > 0$. $\mathrm{Ent}$ denotes the entropy of a random variable defined by
$$
  \mathrm{Ent}(Y) = \mathbb{E}[Y\log Y] - (\mathbb{E} Y) \log \mathbb{E}Y.
$$
The left-hand side in Eq.\@ \eqref{eq:goal_herbst} is the logarithmic moment generating function of $Z - \mathbb{E}Z$, and the bound in that display corresponds to sub-Gaussian tails. The modified logarithmic Sobolev inequality is used to establish Eq.\@ \eqref{eq:goal_sobolev}. It states that, if $Z = f(X_1, \ldots, X_n)$ for independent random variables $X_1, \ldots, X_n$, it holds for all $\lambda > 0$ that
$$
  \mathrm{Ent}(e^{\lambda Z}) \leq \sum_{i=1}^n \mathbb{E}\left[e^{\lambda Z} h(-\lambda [Z - Z_i])\right],
$$
where $h(x) = e^x - x - 1$ and $Z_i = f_i(X_1, \ldots, X_{i-1}, X_{i+1}, \ldots, X_n)$ and arbitrary measurable functions $f_1, \ldots, f_n$. To deal with this part, we will use similar observations as in the proof of Lemma \ref{lem:nn_remainder_0} (how does removing one of the observations $X_1, \ldots, X_n$ affect the order statistics?). To complicate things, we will also have to introduce a certain truncation argument to facilitate analysis of the expression $h(-\lambda [Z - Z_i])$. But this should be thought of as a technical nuisance. The real heart of the proof is the combination of Herbst's argument, the modified logarithmic Sobolev inequality, and the nature of the order statistics and how removing one observation $X_i$ from the sample changes them.

\begin{lemma}
  \label{lem:increments_orlicz}
  Let $f_1, g_1, f_2, g_2 : [0,1] \to [0,1]$ be Borel measurable functions and let us define $\rho = \rho[(f_1,g_1), (f_2,g_2)] = \|f_1(X_1) - f_2(X_1)\|_{L_1} + \|g_1(X_1) - g_2(X_1)\|_{L_1}$. For $j = 1,2$, write
  $$
    S_n(f_j,g_j) = \frac{1}{\sqrt{n}}\sum_{i=1}^{n-1} f_j(X_{n,i}')g_j(X_{n,i+1}').
  $$
  Let $1 < r < 2$ and $\alpha \geq 1$ be two numbers. Then there exists a constant $K$ depending only on $\alpha$ and $r$ such that
  \begin{equation}
    \label{eq:orlicz_statement_large}
    \left\|S_n(f_1,g_1) - \mathbb{E}S_n(f_1,g_1) - S_n(f_2,g_2) + \mathbb{E}S_n(f_2,g_2)\right\|_{\psi_r} \leq K \frac{1}{\log (1 / \rho)},
  \end{equation}
  if $\rho \geq n^{-\alpha}$. On the other hand, there also exists a constant $\tilde{K}$ depending only on $r$ such that
  \begin{equation}
    \label{eq:orlicz_statement_small}
    \left\|S_n(f_1,g_1) - \mathbb{E}S_n(f_1,g_1) - S_n(f_2,g_2) + \mathbb{E}S_n(f_2,g_2)\right\|_{\psi_r} \leq \tilde{K} n^{1/2 - 1/r}
  \end{equation}
  if $\rho \leq n^{-1}$.
\end{lemma}
\begin{proof}
  We begin by proving Eq.\@ \eqref{eq:orlicz_statement_large}. Fix some $0 < s < 1/(2\alpha) - r/(4\alpha)$. We may assume without loss of generality that $\rho \leq e^{-2/s}$. If this is not the case, we can simply replace the functions $f_j$ and $g_j$, $j = 1,2$, with $f_j/(2e^{2/s}) $ and $g_j/(2e^{2/s})$, respectively. This will only increase the bound derived in this proof by a constant factor of $2e^{2/s}$. To account for the fact that this rescaling might affect the assumption $\rho \geq n^{-\alpha}$, we will prove our claim using only the weaker assumption that
  \begin{equation}
    \label{eq:rho_lower_bound_weak}
    \rho \geq \frac{n^{-\alpha}}{2e^{2/s}},
  \end{equation}
  which is always fulfilled.

  Define the function $\phi : [0,\infty) \to [0,1]$ by
  \begin{equation}
    \label{eq:definition_phi_interpolation}
    \phi(x) = \begin{cases}
      1                                                          & \quad \textrm{if } x < \rho^s,                    \\
      \log\left(\rho^{s/2}/x\right)/\log\left(\rho^{-s/2}\right) & \quad \textrm{if } \rho^s \leq x \leq \rho^{s/2}, \\
      0                                                          & \quad \textrm{if } x > \rho^{s/2}.
    \end{cases}
  \end{equation}
  $\phi$ is continuous everywhere and differentiable on $(\rho^s, \rho^{s/2})$ with derivative
  $$
    \phi'(x) = - \frac{2}{x s \log(1/\rho)}.
  $$
  We note for later the following fact: If $\rho^s \leq x \leq y \leq \rho^{s/2}$ with $|x-y| \leq 2/n$, then there exists (by the mean value theorem) an intermediate point $\xi \in [x,y]$ such that
  \begin{equation}
    \label{eq:phi_bound}
    |\phi(x) - \phi(y)| x \sqrt{n} \leq 2 |\phi'(\xi)| x n^{-1/2} = \frac{4x}{\xi s \log(1/\rho)} n^{-1/2} \leq \frac{4}{s \log(1/\rho)} n^{-1/2}.
  \end{equation}
  Because $\phi$ is constant on either side of the interval $[\rho^s, \rho^{s/2}]$, the above inequality even holds for all $0 \leq x \leq y$.

  Next, define the functions $a, \mu_n : [0,1]^n \to [0,1]$ and $b, \mu_{n-1} : [0,1]^{n-1} \to [0,1]$ by
  \begin{align*}
    \mu_n(x_1, \ldots, x_n)         & = \frac{1}{n}\sum_{i=1}^n \{|f_1 - f_2|(x_i) + |g_1 - g_2|(x_i)\},    \\
    \mu_{n-1}(x_1, \ldots, x_{n-1}) & = \frac{1}{n}\sum_{i=1}^{n-1} \{|f_1 - f_2|(x_i) + |g_1 - g_2|(x_i)\}
  \end{align*}
  and
  \begin{align*}
    a(x_1, \ldots, x_n)     & = \phi \circ \mu_n (x_1, \ldots, x_n),         \\
    b(x_1, \ldots, x_{n-1}) & = \phi \circ \mu_{n-1} (x_1, \ldots, x_{n-1}).
  \end{align*}
  Define furthermore the random variables
  $$
    Z = a(X_1, \ldots, X_n) \left[S_n(f_1,g_1) - S_n(f_2,g_2)\right]
  $$
  and, for any $i = 1, \ldots, n$,
  $$
    Z_i = b(X_1, \ldots, X_{i-1}, X_{i+1}, \ldots, X_n)\left[S_n^{(i)}(f_1,g_1) - S_n^{(i)}(f_2, g_2)\right],
  $$
  where $S_n^{(i)}(f_j,g_j)$ is defined just as $S_n(f_j, g_j)$, but based on the sample with $X_i$ removed. More precisely,
  $$
    S_n^{(i)}(f_j, g_j) = \frac{1}{\sqrt{n}}\sum_{k=1}^{n-2} f_j\left(\hat{X}_{n,k}'\right)g_j\left(\hat{X}_{n,k+1}'\right),
  $$
  with $(\hat{X}_1, \ldots, \hat{X}_{n-1}) = (X_1, \ldots, X_{i-1}, X_{i+1}, \ldots, X_n)$. Now it holds that
  \begin{align}
    \begin{split}
      \label{eq:Z-Zni_1}
       & |Z - Z_i|                                                                                                                  \\
       & \leq |a(X_1, \ldots, X_n) - b(X_1, \ldots, X_{i-1}, X_{i+1}, \ldots, X_n)| \cdot |S_n^{(i)}(f_1,g_1) - S_n^{(i)}(f_2,g_2)| \\
       & \quad + |a(X_1, \ldots, X_n)| \cdot \left|S_n(f_1,g_1) - S_n(f_2,g_2) - S_n^{(i)}(f_1,g_1) + S_n^{(i)}(f_2,g_2)\right|.
    \end{split}
  \end{align}
  By using the the triangle inequality, the bound $|xy - uv| \leq |x-u| + |y-v|$, valid for all $-1 \leq x,y,u,v \leq 1$, and the boundedness of the functions $f_j$ and $g_j$, we find that the following two inequalities hold:
  \begin{align}
    \begin{split}
      \label{eq:Sn_myn}
      |S_n(f_1,g_1) - S_n(f_2,g_2)|             & \leq \sqrt{n} \mu_n(X_1, \ldots, X_n),                               \\
      |S_n^{(i)}(f_1,g_1) - S_n^{(i)}(f_2,g_2)| & \leq \sqrt{n} \mu_{n-1}(X_1, \ldots, X_{i-1}, X_{i+1}, \ldots, X_n).
    \end{split}
  \end{align}
  To bound the first term on the right-hand side of Eq.\@ \eqref{eq:Z-Zni_1}, observe that
  $$
    0 \leq \mu_n(X_1, \ldots, X_n) - \mu_{n-1}(X_1, \ldots, X_{i-1}, X_{i+1}, \ldots, X_n) \leq 2/n,
  $$
  and so by Eqs.\@ \eqref{eq:phi_bound} and \eqref{eq:Sn_myn},
  \begin{align}
    \begin{split}
      \label{eq:Z-Zni_2}
       & |a(X_1, \ldots, X_n) - b(X_1, \ldots, X_{i-1}, X_{i+1}, \ldots, X_n)| \cdot |S_n^{(i)}(f_1,g_1) - S_n^{(i)}(f_2,g_2)|   \\
       & \leq \left|\phi \circ \mu_n(X_1, \ldots, X_n) - \phi \circ \mu_{n-1}(X_1, \ldots, X_{i-1}, X_{i+1}, \ldots, X_n)\right| \\
       & \quad \cdot \sqrt{n} \mu_{n-1}(X_1, \ldots, X_{i-1}, X_{i+1}, \ldots, X_n)                                              \\
       & \leq \frac{4}{s \log(1/\rho)} n^{-1/2}.
    \end{split}
  \end{align}

  For the next observation, suppose that $X_i$ corresponds to $X_{n,k}'$ in the order statistics. Then most summands in $S_n(f_1,g_1) - S_n(f_2,g_2)$ will also occur in $S_n^{(i)}(f_1,g_1) - S_n^{(i)}(f_2,g_2)$, and vice-versa. If $1 < k < n$, then the only summands which are not contained in both sums are
  \begin{align*}
     & n^{-1/2}\left[f_1(X_{n,k-1}')g_1(X_{n,k}') - f_2(X_{n,k-1}')g_2(X_{n,k}')\right], \\
     & n^{-1/2}\left[f_1(X_{n,k}')g_1(X_{n,k+1}') - f_2(X_{n,k}')g_2(X_{n,k+1}')\right],
  \end{align*}
  from $S_n(f_1,g_1) - S_n(f_2,g_2)$, and
  $$
    n^{-1/2}\left[f_1(X_{n,k-1}')g_1(X_{n,k+1}') - f_2(X_{n,k-1}')g_2(X_{n,k+1}')\right] \\
  $$
  from $S_n^{(i)}(f_1,g_1) + S_n^{(i)}(f_2,g_2)$. If $k = 1$, i.e.\@ $X_i$ is the smallest observation, then the only summand not contained in both sums is
  $$
    n^{-1/2}\left[f_1(X_{n,k}')g_1(X_{n,k+1}') - f_2(X_{n,k}')g_2(X_{n,k+1}')\right]
  $$
  from $S_n(f_1,g_1) - S_n(f_2,g_2)$. Finally, if $k = n$, i.e.\@ $X_i$ is the largest observation, then the only summand not contained in both sums is
  $$
    n^{-1/2}\left[f_1(X_{n,k-1}')g_1(X_{n,k}') - f_2(X_{n,k-1}')g_2(X_{n,k}')\right],
  $$
  again from $S_n(f_1,g_1) - S_n(f_2,g_2)$. From this it follows that, if $1 < k < n$,
  \begin{align*}
     & \left|S_n(f_1,g_1) - S_n(f_2,g_2) - S_n^{(i)}(f_1,g_1) + S_n^{(i)}(f_2,g_2)\right| \\
     & \leq n^{-1/2} |f_1-f_2|(X_{n,k-1}') + n^{-1/2}|g_1 - g_2|(X_{n,k}')                \\
     & \quad + n^{-1/2}|f_1 - f_2|(X_{n,k}') + n^{-1/2}|g_1 - g_2|(X_{n,k+1}')            \\
     & \quad + n^{-1/2}|f_1 - f_2|(X_{n,k-1}') + n^{-1/2}|g_1 - g_2|(X_{n,k+1}').
  \end{align*}
  If $m \in \mathbb{N}$ is an integer, we can raise both sides to the power $m$ and use the generic inequality $|t_1 + \ldots + t_l|^m \leq l^{m-1}(|t_1|^m + \ldots + |t_l|^m)$ together with the fact that $|f_1 - f_2|^m \leq |f_1 - f_2|$ and $|g_1 - g_2|^m \leq |g_1 - g_2|$ to find that
  \begin{align}
    \begin{split}
      \label{eq:Z-Zni_3}
       & \left|S_n(f_1,g_1) - S_n(f_2,g_2) - S_n^{(i)}(f_1,g_1) + S_n^{(i)}(f_2,g_2)\right|^m         \\
       & \leq 6^{m-1} n^{-m/2} |f_1-f_2|(X_{n,k-1}') + 6^{m-1} n^{-m/2} |g_1 - g_2|(X_{n,k}')         \\
       & \quad + 6^{m-1} n^{-m/2} |f_1 - f_2|(X_{n,k}') + 6^{m-1} n^{-m/2} |g_1 - g_2|(X_{n,k+1}')    \\
       & \quad + 6^{m-1} n^{-m/2} |f_1 - f_2|(X_{n,k-1}') + 6^{m-1} n^{-m/2} |g_1 - g_2|(X_{n,k+1}'),
    \end{split}
  \end{align}
  if $1 < k < n$. Similarly, if $k = 1$, then
  \begin{align}
    \begin{split}
      \label{eq:Z-Zni_4}
       & \left|S_n(f_1,g_1) - S_n(f_2,g_2) - S_n^{(i)}(f_1,g_1) + S_n^{(i)}(f_2,g_2)\right|^m  \\
       & \leq 2^{m-1} n^{-m/2}|f_1 - f_2|(X_{n,k}') + 2^{m-1} n^{-m/2}|g_1 - g_2|(X_{n,k+1}').
    \end{split}
  \end{align}
  for any $m \in \mathbb{N}$. Finally, if $k = n$, then
  \begin{align}
    \begin{split}
      \label{eq:Z-Zni_5}
       & \left|S_n(f_1,g_1) - S_n(f_2,g_2) - S_n^{(i)}(f_1,g_1) + S_n^{(i)}(f_2,g_2)\right|^m   \\
       & \leq 2^{m-1} n^{-m/2}|f_1 - f_2|(X_{n,{k-1}}') + 2^{m-1} n^{-m/2}|g_1 - g_2|(X_{n,k}')
    \end{split}
  \end{align}
  for $m \in \mathbb{N}$.

  We now want to use Eqs.\@ \eqref{eq:Z-Zni_2} through \eqref{eq:Z-Zni_5} to bound the sum
  $$
    \sum_{i=1}^n \left\{a(X_1, \ldots, X_n)\left|S_n(f_1,g_1) - S_n(f_2,g_2) - S_n^{(i)}(f_1,g_1) + S_n^{(i)}(f_2,g_2)\right|\right\}^m
  $$
  for any fixed $m \in \mathbb{N}$. First, since the function $a$ is bounded by $1$ in absolute value, we get
  \begin{align}
    \begin{split}
      \label{eq:sum_investigation}
       & \sum_{i=1}^n \left\{a(X_1, \ldots, X_n)\left|S_n(f_1,g_1) - S_n(f_2,g_2) - S_n^{(i)}(f_1,g_1) + S_n^{(i)}(f_2,g_2)\right|\right\}^m \\
       & \leq \sum_{i=1}^n \left|S_n(f_1,g_1) - S_n(f_2,g_2) - S_n^{(i)}(f_1,g_1) + S_n^{(i)}(f_2,g_2)\right|^m.
    \end{split}
  \end{align}
  Now recall that the index $k$ on the right-hand sides of Eqs.\@ \eqref{eq:Z-Zni_2} through \eqref{eq:Z-Zni_5} actually depends on the index $i$, i.e.\@ $k = k(i)$. For any fixed $i = 1, \ldots, n$, the corresponding summand in the last line of Eq.\@ \eqref{eq:sum_investigation} can be bounded by one of the right-hand sides of Eqs.\@ \eqref{eq:Z-Zni_2} through \eqref{eq:Z-Zni_5}. We can only increase the sum by including all of the right-hand sides of Eqs.\@ \eqref{eq:Z-Zni_2} through \eqref{eq:Z-Zni_5}, i.e.\@
  \begin{align}
    \begin{split}
      \label{eq:sum_investigation_2}
       & \sum_{i=1}^n \left|S_n(f_1,g_1) - S_n(f_2,g_2) - S_n^{(i)}(f_1,g_1) + S_n^{(i)}(f_2,g_2)\right|^m                                                                                         \\
       & \leq \sum_{i=1}^n \{\textrm{RHS of Eqs.\@ \eqref{eq:Z-Zni_3}}\} + \sum_{i=1}^n \{\textrm{RHS of Eqs.\@ \eqref{eq:Z-Zni_4}}\} + \sum_{i=1}^n \{\textrm{RHS of Eqs.\@ \eqref{eq:Z-Zni_5}}\} \\
       & \leq \frac{6^{m-1}}{n^{m/2 - 1}} n^{-1}\sum_{i=1}^n |f_1 - f_2|(X_{n,k(i)-1}') + \frac{6^{m-1}}{n^{m/2 - 1}} n^{-1}\sum_{i=1}^n |g_1-g_2|(X_{n,k(i)}') + \ldots,
    \end{split}
  \end{align}
  where the sum in the last line includes all single terms occurring on the right-hand sides of Eqs.\@ \eqref{eq:Z-Zni_3} through \eqref{eq:Z-Zni_5}. The matching $i \mapsto k$ of an observation $X_i$ to its position $k$ in the order statistics $X_{n,1}' \leq \ldots \leq X_{n,k}' ( = X_i) \leq \ldots \leq X_{n,n}'$ is bijective, because the order statistics are just a (random) permutation of the original sample $X_1, \ldots, X_n$. Therefore,
  $$
    \frac{1}{n}\sum_{i=1}^n |g_1 - g_2|(X_{n,k(i)}') = \frac{1}{n}\sum_{i=1}^n |g_1 - g_2|(X_i),
  $$
  and the same identity holds if we replace $|g_1 - g_2|$ by $|f_1 - f_2|$. This exact equality is not true for summands containing $X_{n,k(i)-1}'$ or $X_{n,k(i)+1}'$, but we have a corresponding inequality. For instance,
  \begin{align*}
    \frac{1}{n}\sum_{i=1}^n |f_1 - f_2|(X_{n,k(i)-1}') & = \frac{1}{n}\sum_{i=1}^n |f_1 - f_2|(X_i) - |f_1-f_2|(X_{n,1}') \\
                                                       & \leq \frac{1}{n}\sum_{i=1}^n |f_1 - f_2|(X_i).
  \end{align*}
  Hence, each sum on the last line in Eq.\@ \eqref{eq:sum_investigation_2} can be bounded by
  $$
    \frac{6^{m-1}}{n^{m/2 - 1}} n^{-1}\sum_{i=1}^n |f_1 - f_2|(X_i)
  $$
  or
  $$
    \frac{6^{m-1}}{n^{m/2 - 1}} n^{-1}\sum_{i=1}^n |g_1 - g_2|(X_i).
  $$
  The last line Eq.\@ \eqref{eq:sum_investigation_2} contains $10$ sums in total, $5$ corresponding to $|f_1 - f_2|$ and $5$ corresponding $|g_1 - g_2|$ (simply count the terms on the right-hand sides of Eqs.\@ \ref{eq:Z-Zni_3} through \ref{eq:Z-Zni_5} to see this). Hence, Eq.\@ \eqref{eq:sum_investigation_2} can further be bounded by
  \begin{equation}
    \label{eq:sum_investigation_3}
    \frac{5 \cdot 6^{m-1}}{n^{m/2 - 1}} n^{-1}\sum_{i=1}^n \left\{|f_1-f_2|(X_i) + |g_1 - g_2|(X_i)\right\},
  \end{equation}
  and of course we can trivially bound the numerator $5 \cdot 6^{m-1}$ by $6^m$ (purely to reduce notation later). All in all, we see from Eqs.\@ \eqref{eq:sum_investigation} through \eqref{eq:sum_investigation_3} that, for any $m \in \mathbb{N}$,
  \begin{align}
    \begin{split}
      \label{eq:sum_investigation_final}
       & \sum_{i=1}^n \left\{a(X_1, \ldots, X_n)\left|S_n(f_1,g_1) - S_n(f_2,g_2) - S_n^{(i)}(f_1,g_1) + S_n^{(i)}(f_2,g_2)\right|\right\}^m \\
       & \leq \sum_{i=1}^n \left|S_n(f_1,g_1) - S_n(f_2,g_2) - S_n^{(i)}(f_1,g_1) + S_n^{(i)}(f_2,g_2)\right|^m                              \\
       & \leq \frac{6^m}{n^{m/2 - 1}} n^{-1}\sum_{i=1}^n \left\{|f_1-f_2|(X_i) + |g_1 - g_2|(X_i)\right\}.
    \end{split}
  \end{align}

  Now combine Eqs.\@ \eqref{eq:Z-Zni_1}, \eqref{eq:Z-Zni_2} and \eqref{eq:sum_investigation_final} to see that, for any $m \in \mathbb{N}$,
  \begin{align*}
     & \sum_{i=1}^n |Z-Z_i|^m                                                                                                                             \\
     & \leq \left(\frac{4}{s \log(1/\rho)}\right)^m n^{-m/2 + 1}                                                                                          \\
     & \quad + \textbf{1}\{a(X_1, \ldots, X_n) > 0\} \frac{6^m}{n^{m/2 - 1}} n^{-1}\sum_{i=1}^n \left\{|f_1-f_2|(X_i) + |g_1 - g_2|(X_i)\right\}          \\
     & \leq \frac{6^m}{n^{m/2 - 1}} \left[\left(\frac{4}{s \log(1/\rho)}\right)^m +  \textbf{1}\{a(X_1, \ldots, X_n) > 0\} \mu_n(X_1, \ldots, X_n)\right] \\
     & \leq \frac{6^m}{n^{m/2 - 1}}\left[\left(\frac{4}{s \log(1/\rho)}\right)^m +  \rho^{s/2}\right].
  \end{align*}
  In the final inequality, we have used that $a(X_1, \ldots, X_n) = \phi \circ \mu_n (X_1, \ldots, X_n) > 0$ also implies that $\mu_n(X_1, \ldots, X_n) \leq \rho^{s/2}$ by definition of $\phi$. For ease of notation, let us write $v = v(\rho) = 4/\{s\log(1/\rho)\}$, then the above inequality may be restated as
  \begin{equation}
    \label{eq:sum_Z-Zi_simple}
    \sum_{i=1}^n |Z-Z_i|^m \leq \frac{6^m}{n^{m/2 - 1}}\left(v^m +  \rho^{s/2}\right).
  \end{equation}

  Define the entropy $\mathrm{Ent}(Y)$ of a real-valued random variable $Y$ by
  $$
    \mathrm{Ent}(Y) = \mathbb{E}[Y \log Y] - \mathbb{E}Y \log \mathbb{E}Y,
  $$
  with the convention that $0 \log 0 = 0$. By Theorem 6.6 in \cite{boucheron_etal:2013}, it holds that
  \begin{equation}
    \label{eq:modified_sobolev_inequality}
    \mathrm{Ent}\left(e^{\lambda Z}\right) \leq \sum_{i=1}^n \mathbb{E}\left[e^{\lambda Z} h(-\lambda[Z - Z_i])\right]
  \end{equation}
  for all $\lambda \in \mathbb{R}$, where $h(x) = e^x - x - 1$. Now we can use the Taylor expansion of the exponential function in $h(x)$ and apply the bound from Eq.\@ \eqref{eq:sum_Z-Zi_simple} to see that, for any $\lambda \geq 0$,
  \begin{align}
    \begin{split}
      \label{eq:h_bound_refined}
      \left|\sum_{i=1}^n h(-\lambda[Z - Z_i])\right| & \leq \sum_{i=1}^n \sum_{m=2}^\infty \frac{\lambda^m |Z - Z_i|^m}{m!} = \sum_{m=2}^\infty \frac{\lambda^m}{m!} \sum_{i=1}^n |Z - Z_i|^m                                                   \\
                                                     & \leq \sum_{m=2}^\infty \frac{(6 \lambda v)^m (n^{-1/2})^{m-2}}{m!} + \rho^{s/2}\sum_{m=2}^\infty \frac{(6 \lambda)^m (n^{-1/2})^{m-2}}{m!}                                               \\
                                                     & = \frac{\lambda^2}{2} (6v)^2 \left[ \sum_{m=0}^\infty \frac{(6 \lambda v / \sqrt{n})^m}{(m+2)!} + \frac{\rho^{s/2}}{v^2}\sum_{m=0}^\infty \frac{(6 \lambda / \sqrt{n})^m}{(m+2)!}\right] \\
                                                     & \leq \frac{\lambda^2}{2} (6v)^2 \left[\exp\left(\frac{6\lambda v}{\sqrt{n}}\right) + \frac{\rho^{s/2}}{v^2} \exp\left(\frac{6\lambda}{\sqrt{n}}\right)\right].
    \end{split}
  \end{align}
  In particular, if $C > 0$ is some constant and $\lambda \leq C \rho^{s/2} \sqrt{n}/v$, then
  \begin{equation}
    \label{eq:exponential_bound}
    \exp\left(\frac{6\lambda v}{\sqrt{n}}\right) + \frac{\rho^{s/2}}{v^2} \exp\left(\frac{6\lambda}{\sqrt{n}}\right) \leq \exp\left(6C \rho^{s/2}\right) + \frac{\rho^{s/2}}{v^2}\exp\left(\frac{6C \rho^{s/2}}{v}\right).
  \end{equation}
  Both the functions $\rho \mapsto \rho^{s/2} v^{-2}$ and  $\rho \mapsto \rho^{s/2} v^{-1}$ are non-negative and continuous on $[0,e^{-2/s}]$, and we can bound their respective maxima by some number $m_s$ depending only on $s$. Combining this observation with Eqs.\@ \eqref{eq:h_bound_refined} and \eqref{eq:exponential_bound}, we find that there exist some constant $M_s(C) > 0$ depending only on $s$ and $C$, such that
  \begin{align}
    \begin{split}
      \label{eq:h_bound_refined_2}
      \left|\sum_{i=1}^n h(-\lambda[Z - Z_i])\right| & \leq \frac{\lambda^2}{2} v^2 M_s(C),
    \end{split}
  \end{align}
  whenever $\lambda \leq C \rho^{s/2} \sqrt{n}/v$. We may assume without loss of generality that $M_s(C) > 1$ for any $C \geq 0$; otherwise replace it with $\max\{1, M_s(C)\}$. Combine Eqs.\@ \eqref{eq:modified_sobolev_inequality} and \eqref{eq:h_bound_refined_2} to see that
  \begin{equation}
    \label{eq:herbst_1_new}
    \mathrm{Ent}\left(e^{\lambda Z}\right) \leq \mathbb{E}\left[e^{\lambda Z}\right] \frac{\lambda^2}{2} v^2 M_s(C)
  \end{equation}
  for any $\lambda \leq C \rho^{s/2} \sqrt{n}/v$ By Herbst's argument \citep[Proposition 6.1 in][]{boucheron_etal:2013}, this implies
  \begin{equation}
    \label{eq:herbst_2_new}
    \log \mathbb{E}\left[\exp\left(\lambda(Z - \mathbb{E}Z)\right)\right] \leq \frac{\lambda^2}{2} v^2 M_s(C)
  \end{equation}
  for any $\lambda \leq C \rho^{s/2} \sqrt{n}/v$. \citep[Note: the cited Proposition 6.1 in][actually starts with the assumption that Eq.\@ \ref{eq:herbst_1_new} holds for all $\lambda > 0$, which then implies that Eq.\@ \ref{eq:herbst_2_new} also holds for all $\lambda >0$. But an inspection of the proof of that proposition reveals that one can indeed relax the assumption like we have done here: if Eq.\@ \ref{eq:herbst_1_new} holds for all $0 < \lambda \leq L$ for some fixed $L > 0$, then Eq.\@ \ref{eq:herbst_2_new} also holds for all $0 < \lambda \leq L$.]{boucheron_etal:2013}

  We are now interested in bounding $Z$ in absolute value. First, by Eq.\@ \eqref{eq:Sn_myn},
  $$
    |Z| \leq \sqrt{n} a(X_1, \ldots, X_n) \mu_n(X_1, \ldots, X_n) = \sqrt{n} \left[\phi \circ \mu_n(X_1, \ldots, X_n)\right] \mu_n(X_1, \ldots, X_n).
  $$
  An upper bound for $x \mapsto \sqrt{n} x \phi(x)$ will therefore also serve as an upper bound for $|Z|$. But it is clear from the definition of $\phi$ in Eq.\@ \eqref{eq:definition_phi_interpolation} that any maximal point $x_0$ of $\phi$ must lie in the interval $[\rho^s, \rho^{s/2}]$. It therefore suffices to find a maximum of the interpolating part of $x \mapsto x \phi(x)$ given by the function
  \begin{align*}
    u : [\rho^s, \rho^{s/2}] & \to [0,\infty),                                         \\
    x                        & \mapsto x \frac{\log(\rho^{s/2}/x)}{\log(\rho^{-s/2})}.
  \end{align*}
  $u$ has a maximum at $x = \rho^{s/2}/e$ with value $u(\rho^{s/2}/e) = 2\rho^{s/2}/\{e s \log (1/\rho)\}$ (the fact that $\rho^{s/2}/e \in [\rho^s, \rho^{s/2}]$ follows from our assumption $\rho \leq e^{-2/s}$). Therefore,
  \begin{equation}
    \label{eq:Z_absolute_bound}
    |Z| \leq\sqrt{n} \max_{\rho^s \leq x \leq \rho^{s/2}} u(x) = e^{-1}\rho^{s/2} \sqrt{n} \frac{2}{s \log(1/\rho)}  = \frac{e^{-1}}{2}  \rho^{s/2} v \sqrt{n}.
  \end{equation}
  Consider some $0 \leq t \leq e^{-1} \rho^{s/2} v \sqrt{n}$. The left-hand side of Eq.\@ \eqref{eq:herbst_2_new} is the log-moment generating function of $Z - \mathbb{E}Z$. If we choose
  $$
    \lambda = \frac{t}{v^2 M_s(1/e)} \leq \frac{\rho^{s/2} v \sqrt{n}}{e v^2 M_s(1/e)} \leq e^{-1} \frac{\rho^{s/2}\sqrt{n}}{v},
  $$
  we see by Chernoff's inequality and Eq.\@ \eqref{eq:herbst_2_new} that
  $$
    \mathbb{P}\left(Z - \mathbb{E}Z \geq t\right) \leq \exp\left(\frac{t^2}{2v^2 M_s(1/e)} - \frac{t^2}{v^2 M_s(1/e)}\right) = \exp\left(- \frac{t^2}{2v^2 M_s(1/e)}\right).
  $$
  We assumed that $0 \leq t \leq e^{-1} \rho^{s/2} v \sqrt{n}$, but in fact this inequality is valid for all $t \geq 0$. To see this, just note that $|Z - \mathbb{E}Z| \leq e^{-1} \rho^{s/2} v \sqrt{n}$ as a consequence of Eq.\@ \eqref{eq:Z_absolute_bound}, and so $\mathbb{P}\left(Z - \mathbb{E}Z \geq t\right) = 0$ for all $t > e^{-1} \rho^{s/2} v \sqrt{n}$. Furthermore, all arguments up to this point remain valid if we replace $Z$ with $-Z$ and $Z_i$ with $-Z_i$. In exactly the same way, we therefore get a bound for the lower tail of $Z-\mathbb{E}Z$,
  $$
    \mathbb{P}\left(Z - \mathbb{E}Z \leq -t\right) = \mathbb{P}\left(- Z - \mathbb{E}[-Z] \geq t\right)\leq \exp\left(- \frac{t^2}{2v^2 M_s(1/e)}\right),
  $$
  for all $t \geq 0$. This implies the following two-sided tail bound for all $t \geq 0$:
  $$
    \mathbb{P}\left(|Z - \mathbb{E}Z| \geq t\right) \leq 2\exp\left(- \frac{t^2}{2v^2 M_s(1/e)}\right).
  $$
  Hence, $Z - \mathbb{E}Z_i$ has sub-Gaussian tails, and by Lemma 2.2.1 in \cite{van_der_vaart_wellner:weak_convergence} its sub-Gaussian Orlicz norm can be bounded by
  \begin{equation}
    \label{eq:orlicz_part_a}
    \|Z - \mathbb{E}Z\|_{\psi_2} \leq v \sqrt{6M_s(1/e)}.
  \end{equation}

  Having bounded the truncated random variable $Z$, we now need to investigate the tail behaviour of the remainder
  $$
    D = S_n(f_1, g_1) - S_n(f_2,g_2) - Z = [1 - a(X_1, \ldots, X_n)]\left\{S_n(f_1, g_1) - S_n(f_2,g_2)\right\}.
  $$
  We observe two things: First, $|D| > 0$ implies that $\mu_n(X_1, \ldots, X_n) \geq \rho^s$ by definition of the function $a(X_1, \ldots, X_n)$. $\mu_n(X_1, \ldots, X_n)$ is simply a mean of i.i.d.\@ random variables and $\rho = \mathbb{E} \mu_n(X_1, \ldots, X_n)$. We can therefore use Hoeffding's inequality \citep[Theorem 2.8 in][]{boucheron_etal:2013} to see that, for any $0 \leq t \leq 2 \sqrt{n}$,
  \begin{align*}
    \mathbb{P}\left(|D| \geq t\right) & \leq \mathbb{P}\left(\mu_n(X_1, \ldots, X_n) \geq \rho^s\right)                                         \\
                                      & = \mathbb{P}\left(\mu_n(X_1, \ldots, X_n) - \mathbb{E}\mu_n(X_1, \ldots, X_n) \geq \rho^s - \rho\right) \\
                                      & \leq \exp\left(-\frac{n (\rho^s - \rho)^2}{2}\right).
  \end{align*}
  Our assumption $\rho \leq e^{-2/s}$ implies $\rho \leq 2^{-1/(1-s)}$, which in turn implies $\rho^s - \rho \geq \rho^s/2$. Inserting this into the last line of the previous display, we obtain
  \begin{align}
    \begin{split}
      \label{eq:D_remainder_1}
      \mathbb{P}\left(|D| \geq t\right) & \leq \exp\left(-\frac{n \rho^{2s}}{8}\right)                            \\
                                        & = \exp\left(- (2 \sqrt{n})^r \frac{n^{1-r/2} \rho^{2s}}{2^{3+r}}\right) \\
                                        & \leq \exp\left(- t^r \frac{n^{1-r/2} \rho^{2s}}{2^{3+r}}\right)
    \end{split}
  \end{align}
  for all $0 \leq t \leq 2\sqrt{n}$. The second observation is that, since
  $$
    |D| \leq \left|S_n(f_1,g_1) - S_n(f_2,g_2)\right| \leq \max\left\{S_n(f_1,g_1), S_n(f_2,g_2)\right\} \leq 2\sqrt{n},
  $$
  we trivially have
  $$
    \mathbb{P}\left(|D| \geq t\right) = 0
  $$
  for any $t > 2\sqrt{n}$. This means that the bound in Eq.\@ \eqref{eq:D_remainder_1} is actually valid for any $t \geq 0$, and so, by Lemma 2.2.1 in \cite{van_der_vaart_wellner:weak_convergence},
  \begin{align}
    \begin{split}
      \label{eq:D_remainder_2}
      \|D\|_{\psi_r} \leq \left[2^{4+r} n^{r/2 - 1} \rho^{-2s}\right]^{1/r} = & 2^{1 + 4/r} n^{1/2 - 1/r} \rho^{-2s/r}                       \\
                                                                              & \leq 2^{1 + (4 + 2s)/r} e^{4/r} n^{1/2 - (1 - 2\alpha s)/r}.
    \end{split}
  \end{align}
  In the last inequality we have used Eq.\@ \eqref{eq:rho_lower_bound_weak}. Furthermore,
  $$
    |\mathbb{E}D| \leq \|D\|_{L_1} \leq \|D\|_{\psi_1} \leq \|D\|_{\psi_r} (\log 2)^{1/r - 1}
  $$
  by the inequalities given on p.\@ 95 of \cite{van_der_vaart_wellner:weak_convergence}. Hence,
  \begin{align}
    \begin{split}
      \label{eq:D_remainder_3}
      \|D - \mathbb{E}D\|_{\psi_r} & \leq \|D\|_{\psi_r} + |\mathbb{E}D| (\log 2)^{-1/r}                                         \\
                                   & \leq \|D\|_{\psi_r} \left[1 + (\log 2)^{-1}\right]                                          \\
                                   & \leq  n^{1/2 - (1 - 2\alpha s)/r} 2^{1 + (4 + 2s)/r} e^{4/r} \left[1 + (\log 2)^{-1}\right] \\
                                   & =:  n^{1/2 - (1 - 2\alpha s)/r} Q(s,r)                                                      \\
    \end{split}
  \end{align}
  by Eq.\@ \eqref{eq:D_remainder_2}. In the first inequality we have used that the constant $1$-function has Orlicz norm $\|1\|_{\psi_r} = (\log 2)^{-1/r}$ \citep[Problem 2.2.2 in][]{van_der_vaart_wellner:weak_convergence}.

  Now by construction of $D$ we have
  $$
    S_n(f_1, g_1) - S_n(f_2,g_2) = D + Z.
  $$
  Therefore, using the inequality $\|\cdot\|_{\psi_r} \leq \|\cdot\|_{\psi_2} (\log 2)^{1/2 - 1/r}$ \citep[p.\@ 95 in][]{van_der_vaart_wellner:weak_convergence}, Eqs.\@ \eqref{eq:orlicz_part_a} and \eqref{eq:D_remainder_3} give us
  \begin{align}
    \begin{split}
      \label{eq:final_orlicz_bound}
       & \|S_n(f_1, g_1)- \mathbb{E}S_n(f_1, g_1) - S_n(f_2,g_2) + \mathbb{E}S_n(f_2,g_2)\|_{\psi_r}              \\
       & = \|D + Z - \mathbb{E}[D + Z]\|_{\psi_r}                                                                 \\
       & \leq \|Z - \mathbb{E}Z\|_{\psi_2} (\log 2)^{1/2 - 1/r} + \|D - \mathbb{E}D\|_{\psi_r}                    \\
       & \leq v \sqrt{6 M_s(1/e)}(\log 2)^{1/2 - 1/r} + n^{1/2 - (1 - 2s)/r} Q(s,r)                               \\
       & \leq v \left[\sqrt{6 M_s(1/e)}(\log 2)^{1/2 - 1/r} + Q(s,r)\frac{n^{1/2 - (1 - 2\alpha s)/r}}{v}\right].
    \end{split}
  \end{align}
  Finally, use the definition of $v$ and the lower bound on $\rho$ from Eq.\@ \eqref{eq:rho_lower_bound_weak} to find that
  $$
    \frac{n^{1/2 - (1 - 2\alpha s)/r}}{v} = \frac{s n^{1/2 - (1 - 2\alpha s)/r}}{4}\log(1/\rho) \leq \frac{sn^{1/2 - (1 - 2\alpha s)/r}}{4} \log\left(2e^{2/s} n^\alpha\right).
  $$
  Since $0 < s < 1/(2\alpha) - r/(4\alpha)$, we have $1/2 - (1-2\alpha s)/r < 0$, and so the right-hand side converges to $0$ for $n \to \infty$. In particular, the right-hand side (understood as a sequence in $n$) attains a finite maximum which depends only on $\alpha$, $s$ and $r$, and therefore
  \begin{align*}
     & \|S_n(f_1, g_1)- \mathbb{E}S_n(f_1, g_1) - S_n(f_2,g_2) + \mathbb{E}S_n(f_2,g_2)\|_{\psi_r}                                                                                          \\
     & \leq v \left[\sqrt{6 M_s(1/e)}(\log 2)^{1/2 - 1/r} + Q(s,r) \max_{n \in \mathbb{N}} \left\{\frac{sn^{1/2 - (1 - 2\alpha s)/r}}{4} \log\left(2e^{2/s} n^\alpha \right)\right\}\right]
  \end{align*}
  by Eq.\@ \eqref{eq:final_orlicz_bound}. A valid choice for the constant $K$ from Eq.\@ \eqref{eq:orlicz_statement_large} is given by $(4/s) \cdot 2e^{2/s} = 8e^{2/s}/s$ times the bracketed term in the preceding display. The factor $4/s$ is the one included in the definition of $v$, and the factor $2e^{2/s}$ accounts for the fact that we assumed in the beginning that $\rho \leq e^{-2/s}$ (see the paragraph before Eq.\@ \ref{eq:rho_lower_bound_weak}). To make the constant $K$ depend only on $\alpha$ and $r$ (and not on $s$), all we need to do is specify some arbitrary but unique way to choose the constant $s$. For instance, we could simply pick $s = 1/(2\alpha) - r/(8\alpha)$. This concludes the proof of Eq.\@ \eqref{eq:orlicz_statement_large}.

  The proof of Eq.\@ \eqref{eq:orlicz_statement_small} is much simpler. If $\rho \leq n^{-1}$, then we know from Eq.\@ \eqref{eq:Sn_myn} that
  \begin{align}
    \begin{split}
      \label{eq:small_rho_1}
       & \left\|S_n(f_1,g_1) - \mathbb{E}S_n(f_1,g_1) - S_n(f_2,g_2) + \mathbb{E}S_n(f_2,g_2)\right\|_{\psi_r}                                                                 \\
       & \leq \left\|S_n(f_1,g_1) - S_n(f_2,g_2)\right\|_{\psi_r} + \left\|\mathbb{E}S_n(f_1,g_1) - \mathbb{E}S_n(f_2,g_2)\right\|_{\psi_r}                                    \\
       & \leq \sqrt{n}\left\|\mu_n(X_1, \ldots, X_n)\right\|_{\psi_r} + \sqrt{n} \left\|\mathbb{E}\mu_n(X_1, \ldots,X_n)\right\|_{\psi_r}                                      \\
       & \leq \sqrt{n}\left\|\mu_n(X_1, \ldots, X_n) - \mathbb{E}\mu_n(X_1, \ldots,X_n)\right\|_{\psi_r} + 2 \sqrt{n} \left\|\mathbb{E}\mu_n(X_1, \ldots,X_n)\right\|_{\psi_r} \\
       & = \sqrt{n}\left\|\mu_n(X_1, \ldots, X_n) - \mathbb{E}\mu_n(X_1, \ldots,X_n)\right\|_{\psi_r} + 2 \sqrt{n} \rho (\log 2)^{-1/r}.
    \end{split}
  \end{align}
  Here we have used that $\mathbb{E}\mu_n(X_1, \ldots, X_n) = \rho$ and the fact that the Orlicz norms are monotone, i.e.\@ if $|X| \leq |Y|$ almost surely, then $\|X\|_{\psi_r} \leq \|Y\|_{\psi_r}$. Since $\rho \leq n^{-1}$ by assumption, we immediately get
  \begin{equation}
    \label{eq:small_rho_2}
    2\sqrt{n} \rho (\log 2)^{-1/r} < 2 (\log 2)^{-1/r} n^{-1/2}.
  \end{equation}
  On the other hand, observe that
  $$
    n \mu_n(X_1, \ldots, X_n) = \sum_{i=1}^n \{|f_1 - f_2|(X_i) + |g_1 - g_2|(X_i)\}
  $$
  is just the sum of $n$ i.i.d.\@ observations, each of which takes values in the interval $[0,2]$, and
  $$
    \sum_{i=1}^n \mathbb{E}\left[\{|f_1 - f_2|(X_i) + |g_1 - g_2|(X_i)\}^2\right] \leq 2n \rho.
  $$
  Hence, by Bernstein's inequality \citep[Corollary 2.11 in][]{boucheron_etal:2013}, it holds for any $0 < t \leq 4$ that
  \begin{align}
    \begin{split}
      \label{eq:small_rho_3}
       & \mathbb{P}\left(|\mu_n(X_1, \ldots, X_n) - \mathbb{E}\mu_n(X_1, \ldots, X_n)| \geq t\right)      \\
       & = \mathbb{P}\left(n |\mu_n(X_1, \ldots, X_n) - \mathbb{E}\mu_n(X_1, \ldots, X_n)| \geq nt\right) \\
       & \leq 2\exp\left(- \frac{n^2 t^2}{2(2n\rho + 2nt/3)}\right)                                       \\
       & \leq 2\exp\left(- \frac{n t^2}{4(\rho + t/3)}\right).
    \end{split}
  \end{align}
  If $t < \rho$, then $\rho + t/3 < 4\rho/3 < 4n^{-1}/3$, and so
  \begin{equation}
    \label{eq:small_rho_3a}
    \exp\left(- \frac{n t^2}{4(\rho + t/3)}\right) \leq \exp\left(- \frac{n^{2} t^2}{16/3}\right).
  \end{equation}
  On the other hand, if $\rho \leq t \leq 4$, then $\rho + t/3 \leq 4t/3$, and so
  \begin{equation}
    \label{eq:small_rho_3b}
    \exp\left(- \frac{n t^2}{4(\rho + t/3)}\right) \leq \exp\left(- \frac{nt}{16/3}\right) = \exp\left(- \frac{nt^r}{16/3} t^{1-r}\right) \leq \exp\left(- \frac{nt^r}{16/3} 4^{1-r}\right).
  \end{equation}
  The last inequality holds because $0 < t \leq 4$ and $r \geq 1$, and so $t^{1-r} \geq 4^{1-r}$. On the other hand, if $t \geq 4$, then trivially
  \begin{equation}
    \label{eq:small_rho_3c}
    \mathbb{P}\left(|\mu_n(X_1, \ldots, X_n) - \mathbb{E}\mu_n(X_1, \ldots, X_n)| \geq t\right) = 0,
  \end{equation}
  since $|\mu_n(X_1, \ldots, X_n)| \leq 2$. From Eqs.\@ \eqref{eq:small_rho_3} through \eqref{eq:small_rho_3c} and Lemma \ref{lem:mixed_tailbound}, we get
  \begin{align}
    \begin{split}
      \label{eq:small_rho_4}
       & \left\|\mu_n(X_1, \ldots, X_n) - \mathbb{E}\mu_n(X_1, \ldots, X_n)\right\|_{\psi_r}                          \\
       & \leq K^* \max\left\{n^{-1}\frac{\sqrt{3}}{4}, n^{-1/r} \left(\frac{3 \cdot 4^{r-1}}{16}\right)^{1/r}\right\} \\
       & \leq K_r^{**} n^{-1/r},
    \end{split}
  \end{align}
  where $K^*$ is a universal constant and $K_r^{**}$ depends on $r$. Combine Eqs.\@ \eqref{eq:small_rho_1}, \eqref{eq:small_rho_2} and \eqref{eq:small_rho_4} to see that
  \begin{align*}
     & \left\|S_n(f_1,g_1) - \mathbb{E}S_n(f_1,g_1) - S_n(f_2,g_2) + \mathbb{E}S_n(f_2,g_2)\right\|_{\psi_r} \\
     & \leq n^{1/2 - 1/r} \left(K_r^{**} + 2(\log 2)^{-1/r}\right).
  \end{align*}
  This proves Eq.\@ \eqref{eq:orlicz_statement_small}.
\end{proof}

Finally, we can prove the main theorem of this section, which is the convergence of the process in \eqref{eq:goal_order}. Compared with the proof of Lemma \ref{lem:increments_orlicz}, this one is easy. It consists of an application of a known chaining-type result by \cite{kley_etal:2016} to get from a single increment as in Lemma \ref{lem:increments_orlicz} to a bound for the supremal increment. The remaining steps are technical arguments to show that the resulting bound is small.

\begin{theorem}
  \label{thm:Sn_convergence}
  Let $(X_k)_{k \in \mathbb{N}}$ be a $[0,1]$-valued i.i.d.\@ process. Let $\mathcal{F}$ be a class of Borel measurable functions from $[0,1]$ to $[0,1]$ with $N_{[\,]}(\varepsilon, \mathcal{F}, L_1(X_1)) \leq K_{[\,]}\varepsilon^{-s}$ for some fixed constants $s, K_{[\,]} \geq 1$. Assume further that there is a countable subset $\mathcal{F}_0 \subseteq \mathcal{F}$ which is dense in $\mathcal{F}$ with respect to pointwise convergence. For $n \in \mathbb{N}$, define the process $S_n \in \ell^\infty(\mathcal{F} \times \mathcal{F})$ by
  $$
    S_n(f,g) = \frac{1}{\sqrt{n}} \sum_{i=1}^n f(X_{n,i}')g(X_{n,i+1}').
  $$
  Then $S_n - \mathbb{E}S_n \rightsquigarrow G$ as $n \to \infty$, where $G$ is a tight mean-zero Gaussian process with covariance function
  $$
    \gamma[(f,g), (f',g')] = \mathrm{Cov}[f(X_1)g(X_1), f'(X_1)g'(X_1)].
  $$
\end{theorem}
\begin{proof}
  For any fixed collection $(f_1, g_1), \ldots, (f_m, g_m) \in \mathcal{F} \times \mathcal{F}$, we have
  $$
    \left[S_n(f_j,g_j)\right]_{j=1, \ldots, m} \rightsquigarrow \left[G(f_j,g_j)\right]_{j=1, \ldots, m}
  $$
  in $\mathbb{R}^m$ by Lemma \ref{lem:Sn_fidi}. It remains to establish asymptotic tightness in $\ell^\infty(\mathcal{F} \times \mathcal{F})$ of the sequence $S_n - \mathbb{E}S_n$, $n \in \mathbb{N}$.

  Assume without loss of generality that the $L_1(X_1)$-diameter of $\mathcal{F}$ does not exceed $e^{-2}/2$. If this is not the case, we can instead prove asymptotic tightness of the rescaled sequence $(S_n - \mathbb{E}S_n)/[e^{-2}/2]$, which then also implies asymptotic tightness of the original sequence $S_n - \mathbb{E}S_n$.

  We define two semimetrics $\rho$ and $d$ on $\mathcal{F} \times \mathcal{F}$, given by
  $$
    \rho\left[(f,g), (f',g')\right] = \mathbb{E}|f(X_1) - f'(X_1)| + \mathbb{E}|g(X_1) - g'(X_1)|
  $$
  and
  $$
    d = \frac{1}{\log (1/\rho)}.
  $$
  Since the $L_1(X_1)$-diameter of $\mathcal{F}$ is bounded by $e^{-2}/2$, the $\rho$-diameter of $\mathcal{F} \times \mathcal{F}$ is bounded by $e^{-2}$. The function $x \mapsto 1/\log(1/x)$ is concave on $[0,e^{-2}]$, and so $d$ really is a well-defined semimetric on $\mathcal{F} \times \mathcal{F}$ (every non-negative concave function on the real line is subadditive). We also point out that $S_n$ has measurable suprema. More precisely, it holds for any $\delta > 0$ that
  \begin{align*}
     & \sup_{f,g,f',g' \in \mathcal{F} : d[(f,g), (f',g')] < \delta} \left|S_n(f,g) - \mathbb{E}S_n(f,g) - S_n(f',g') + \mathbb{E}S_n(f',g')\right|      \\
     & = \sup_{f,g,f',g' \in \mathcal{F}_0 : d[(f,g), (f',g')] < \delta} \left|S_n(f,g) - \mathbb{E}S_n(f,g) - S_n(f',g') + \mathbb{E}S_n(f',g')\right|.
  \end{align*}
  The right-hand side (and therefore also the left-hand side) is measurable because the supremum runs over the countable set $\mathcal{F}_0$. Finally, since any $\varepsilon$-ball with respect to $d$ defines an $\exp(-1/\varepsilon)$-ball with respect to $\rho$, we have
  \begin{equation}
    \label{eq:packing_number_d}
    D(\varepsilon, \mathcal{F} \times \mathcal{F}, d) \leq N(\varepsilon/2, \mathcal{F} \times \mathcal{F}, d) \leq N_{[\,]}(\exp(-2/\varepsilon), \mathcal{F} \times \mathcal{F}, \rho) \leq K_{[\,]} \exp(2s/\varepsilon).
  \end{equation}
  $D(\varepsilon, \mathcal{F} \times \mathcal{F}, d)$ denotes the packing number of $\mathcal{F} \times \mathcal{F}$ with respect to $d$, i.e.\@ the maximal number of $\varepsilon$-separated $d$-balls in $\mathcal{F} \times \mathcal{F}$. The first inequality can be found on p.\@ 98 in \cite{van_der_vaart_wellner:weak_convergence}.

  We will prove asymptotic tightness by showing asymptotic uniform $d$-equicontinuity in probability, i.e.\@ we will show that for any $\mu, \nu > 0$ there exists a $\delta > 0$ such that
  \begin{equation}
    \label{eq:zz_asymptotic_d_equicont}
    \limsup_{n \to \infty} \mathbb{P}\left(\sup_{d[(f,g), (f',g')] < \delta} \left|S_n(f,g) - \mathbb{E}S_n(f,g) - S_n(f',g') + \mathbb{E}S_n(f',g')\right| > \mu \right) < \nu.
  \end{equation}
  To this end, fix some $\mu, \nu > 0$. Pick some $1 < r < 2$ and $\alpha > 2$, and define $\bar{\eta} = 2/[\alpha \log n]$. Then $d[(f,g), (f',g')] \geq \bar{\eta}/2$ is equivalent to $\rho[(f,g), (f',g')] \geq n^{-\alpha}$. Therefore, by Lemma \ref{lem:increments_orlicz}, there exists a constant $K$ depending only on $r$ and $\alpha$ such that
  \begin{equation}
    \label{eq:lemma_increment_application}
    \left\|S_n(f_1,g_1) - \mathbb{E}S_n(f_1,g_1) - S_n(f_2,g_2) + \mathbb{E}S_n(f_2,g_2)\right\|_{\psi_r} \leq K d[(f,g), (f',g')]
  \end{equation}
  for all $(f,g), (f',g') \in \mathcal{F} \times \mathcal{F}$ with $d[(f,g),(f',g')] \geq \bar{\eta}/2$. By Lemma A.1 in \cite{kley_etal:2016}, there exists for any $\eta \geq \bar{\eta}$ a random variable $Z$ and a subset $\tilde{T} \subseteq \mathcal{F} \times \mathcal{F}$ such that
  \begin{align}
    \begin{split}
      \label{eq:kley_statement}
       & \sup_{d[(f,g), (f',g')] \leq \delta} \left|S_n(f,g) - \mathbb{E}S_n(f,g) - S_n(f',g') + \mathbb{E}S_n(f',g')\right|                                         \\
       & \leq Z + 2 \sup_{d[(f,g), (f',g')] \leq \bar{\eta} : (f',g') \in \tilde{T}} \left|S_n(f,g) - \mathbb{E}S_n(f,g) - S_n(f',g') + \mathbb{E}S_n(f',g')\right|,
    \end{split}
  \end{align}
  where
  \begin{equation}
    \label{eq:kley_rv_bound}
    \|Z\|_{\psi_r} \lesssim \int_{\bar{\eta}/2}^\eta \psi_r^{-1}[D(\varepsilon, \mathcal{F} \times \mathcal{F}, d)] ~\mathrm{d}\varepsilon + (\delta + 2\bar{\eta}) \psi_r^{-1}[D^2(\eta, \mathcal{F} \times \mathcal{F}, d)],
  \end{equation}
  and the set $\tilde{T}$ has at most $D(\bar{\eta}, \mathcal{F} \times \mathcal{F}, d)$ elements. The symbol $\lesssim$ in Eq.\@ \eqref{eq:kley_rv_bound} is hiding a constant which depends only on $r$ through the norm $\psi_r$ and the constant $K$ from Eq.\@ \eqref{eq:lemma_increment_application}. This last claim is not included in the statement of Lemma A.1 in \cite{kley_etal:2016}, but it is stated in the proof of that lemma.

  We now choose $\eta = \delta$. We also point out (only for notational simplicity) that if $\delta > 0$ is sufficiently small that $D(\delta, \mathcal{F} \times \mathcal{F}, d) \geq 2$, it also holds that
  \begin{equation}
    \label{eq:psi_r_inverse}
    \psi_r^{-1}[D(\varepsilon, \mathcal{F} \times \mathcal{F}, d)] = \left[\log\left(1 + D(\varepsilon, \mathcal{F} \times \mathcal{F}, d)\right)\right]^{1/r} \leq \left[2 \log D(\varepsilon, \mathcal{F} \times \mathcal{F}, d)\right]^{1/r}
  \end{equation}
  for any $\varepsilon \leq \delta$. From Eqs.\@ \eqref{eq:packing_number_d}, \eqref{eq:kley_rv_bound} and \eqref{eq:psi_r_inverse}, we now immediately get
  \begin{align}
    \begin{split}
      \label{eq:kley_rv_bound_2}
      \|Z\|_{\psi_r} & \lesssim \int_0^\delta \left(\frac{4s}{\varepsilon}\right)^{1/r} ~\mathrm{d}\varepsilon + (\delta + 2\bar{\eta}) \left(\frac{8s}{\delta}\right)^{1/r} \\
                     & \xrightarrow[n \to \infty]{} \int_0^\delta \left(\frac{4s}{\varepsilon}\right)^{1/r} ~\mathrm{d}\varepsilon + (8s)^{1/r} \delta^{1 - 1/r}
    \end{split}
  \end{align}
  for sufficiently small $\delta$.

  It remains to consider the supremum on the right-hand side of Eq.\@ \eqref{eq:kley_statement}. For this, let $[l_i, u_i]$, $i = 1, \ldots, M_1$ be a collection of $n^{-\alpha/2}/2$-brackets of $\mathcal{F}$ with respect to $L_1(X_1)$. By assumption, we can choose $M_1 \lesssim n^{s\alpha/2}$. We may also assume without loss of generality that $0 \leq l_i, u_i \leq 1$ for all $i = 1, \ldots, M_1$, since otherwise we can replace $l_i$ by
  $$
    \tilde{l}_i = \begin{cases}
      0 & \quad \textrm{if } l_i < 0, \\ l_i &\quad \textrm{if } 0 \leq l_i \leq 1, \\ 1 &\quad \textrm{if } l_i > 1,
    \end{cases}
  $$ and $u_i$ by an analogously defined $\tilde{u}_i$. Furthermore, let us denote the elements of $\tilde{T}$ by $(\tilde{f}_k, \tilde{g}_k)$, $k = 1, \ldots, M_2$ for some $M_2 \leq D(\bar{\eta}, \mathcal{F} \times \mathcal{F}, d)$. Now for $i,j = 1, \ldots, M_1$ and $k = 1, \ldots, M_2$, define the sets
  \begin{equation}
    \label{eq:Uijk_definition}
    U_{ijk} = \left\{(f,g) \in \mathcal{F} \times \mathcal{F} ~|~ (f,g) \in [l_i, u_i] \times [l_j, u_j] \land d[(f,g), (\tilde{f}_k, \tilde{g}_k)] \leq \bar{\eta}\right\}.
  \end{equation}
  Suppose that $(f,g) \in U_{ijk}$. Then
  \begin{align*}
     & S_n(f,g) - \mathbb{E}S_n(f,g) - S_n(\tilde{f}_k, \tilde{g}_k) + \mathbb{E}S_n(\tilde{f}_k, \tilde{g}_k)                                                                                        \\
     & = \frac{1}{\sqrt{n}} \sum_{t = 1}^{n-1} \left\{f(X_{n,t}')g(X_{n,t+1}') - \mathbb{E}f(X_{n,t}')g(X_{n,t+1}')\right\} - S_n(\tilde{f}_k, \tilde{g}_k) + \mathbb{E}S_n(\tilde{f}_k, \tilde{g}_k) \\
     & \leq \frac{1}{\sqrt{n}} \sum_{t = 1}^{n-1} \left\{u_i(X_{n,t}')u_j(X_{n,t+1}') - \mathbb{E}u_i(X_{n,t}')u_j(X_{n,t+1}')\right\}                                                                \\
     & \quad + \frac{1}{\sqrt{n}} \sum_{t = 1}^{n-1} \mathbb{E}\left[u_i(X_{n,t}')u_j(X_{n,t+1}') - f(X_{n,t}')g(X_{n,t+1}')\right]                                                                   \\
     & \quad - S_n(\tilde{f}_k, \tilde{g}_k) + \mathbb{E}S_n(\tilde{f}_k, \tilde{g}_k)                                                                                                                \\
     & = S_n(u_i,u_j) - \mathbb{E}S_n(u_i,u_j) - S_n(\tilde{f}_k, \tilde{g}_k) + \mathbb{E}S_n(\tilde{f}_k, \tilde{g}_k)                                                                              \\
     & \quad + \frac{1}{\sqrt{n}} \sum_{t = 1}^{n-1} \mathbb{E}\left[u_i(X_{n,t}')u_j(X_{n,t+1}') - f(X_{n,t}')g(X_{n,t+1}')\right].
  \end{align*}
  We can bound the sum in the last line by observing that, for any $t = 1, \ldots, n-1$,
  \begin{align*}
     & \mathbb{E}\left[u_i(X_{n,t}')u_j(X_{n,t+1}') - f(X_{n,t}')g(X_{n,t+1}')\right]                 \\
     & \leq \mathbb{E}[u_i(X_{n,t}') - f(X_{n,t}')] + \mathbb{E}[u_j(X_{n,t+1}') - g(X_{n,t+1}')]     \\
     & \leq \mathbb{E}[u_i(X_{n,t}') - l_i(X_{n,t}')] + \mathbb{E}[u_j(X_{n,t+1}') - l_j(X_{n,t+1}')] \\
     & = \mathbb{E}[u_i(X_1) - l_i(X_1)] + \mathbb{E}[u_j(X_1) - l_j(X_1)]                            \\
     & \leq  n^{-\alpha/2}.
  \end{align*}
  Therefore,
  \begin{align}
    \begin{split}
      \label{eq:T_tilde_1}
       & S_n(f,g) - \mathbb{E}S_n(f,g) - S_n(\tilde{f}_k, \tilde{g}_k) + \mathbb{E}S_n(\tilde{f}_k, \tilde{g}_k)                                  \\
       & \leq S_n(u_i,u_j) - \mathbb{E}S_n(u_i,u_j) - S_n(\tilde{f}_k, \tilde{g}_k) + \mathbb{E}S_n(\tilde{f}_k, \tilde{g}_k) + n^{(1-\alpha)/2}.
    \end{split}
  \end{align}
  In a similar way, one shows that
  \begin{align}
    \begin{split}
      \label{eq:T_tilde_2}
       & S_n(f,g) - \mathbb{E}S_n(f,g) - S_n(\tilde{f}_k, \tilde{g}_k) + \mathbb{E}S_n(\tilde{f}_k, \tilde{g}_k)                                  \\
       & \geq S_n(l_i,l_j) - \mathbb{E}S_n(l_i,l_j) - S_n(\tilde{f}_k, \tilde{g}_k) + \mathbb{E}S_n(\tilde{f}_k, \tilde{g}_k) - n^{(1-\alpha)/2}.
    \end{split}
  \end{align}
  Eqs.\@ \eqref{eq:T_tilde_1} and \eqref{eq:T_tilde_2} together imply that
  \begin{align}
    \begin{split}
      \label{eq:T_tilde_3}
       & \left|S_n(f,g) - \mathbb{E}S_n(f,g) - S_n(\tilde{f}_k, \tilde{g}_k) + \mathbb{E}S_n(\tilde{f}_k, \tilde{g}_k)\right|                 \\
       & \leq \left|S_n(u_i,u_j) - \mathbb{E}S_n(u_i,u_j) - S_n(\tilde{f}_k, \tilde{g}_k) + \mathbb{E}S_n(\tilde{f}_k, \tilde{g}_k)\right|    \\
       & \quad + \left|S_n(l_i,l_j) - \mathbb{E}S_n(l_i,l_j) - S_n(\tilde{f}_k, \tilde{g}_k) + \mathbb{E}S_n(\tilde{f}_k, \tilde{g}_k)\right| \\
       & \quad + n^{(1-\alpha)/2}.
    \end{split}
  \end{align}
  Furthermore, because the cartesian products $[l_i, u_i] \times [l_j, u_j]$ cover $\mathcal{F} \times \mathcal{F}$, any $(f,g) \in \mathcal{F} \times \mathcal{F}$ with $d[(f,g), (\tilde{f}_k, \tilde{g}_k)] \leq \bar{\eta}$ must lie in some $U_{ijk}$. From this observation and Eq.\@ \eqref{eq:T_tilde_3}, we see that
  \begin{align}
    \begin{split}
      \label{eq:T_tilde_orlicz_1}
       & \left\|\sup_{d[(f,g), (f',g')] \leq \bar{\eta} : (f',g') \in \tilde{T}} \left|S_n(f,g) - \mathbb{E}S_n(f,g) - S_n(f',g') + \mathbb{E}S_n(f',g')\right|\right\|_{\psi_r}                                                        \\
       & \leq \left\|\max_{i,j = 1, \ldots, M_1} \max_{k=1, \ldots, M_2} \sup_{(f,g) \in U_{ijk}} \left|S_n(f,g) - \mathbb{E}S_n(f,g) - S_n(\tilde{f}_k, \tilde{g}_k) + \mathbb{E}S_n(\tilde{f}_k, \tilde{g}_k)\right|\right\|_{\psi_r} \\
       & \leq \left\|\max_{i,j = 1, \ldots, M_1} \max_{k=1, \ldots, M_2} \left|S_n(u_i,u_j) - \mathbb{E}S_n(u_i,u_j) - S_n(\tilde{f}_k, \tilde{g}_k) + \mathbb{E}S_n(\tilde{f}_k, \tilde{g}_k)\right| \right\|_{\psi_r}                 \\
       & \quad + \left\|\max_{i,j = 1, \ldots, M_1} \max_{k=1, \ldots, M_2} \left|S_n(l_i,l_j) - \mathbb{E}S_n(l_i,l_j) - S_n(\tilde{f}_k, \tilde{g}_k) + \mathbb{E}S_n(\tilde{f}_k, \tilde{g}_k)\right| \right\|_{\psi_r}              \\
       & \quad + n^{(1-\alpha)/2}.
    \end{split}
  \end{align}
  Because $[l_i, u_i]$ and $[l_j, u_j]$ are $n^{-\alpha/2}/2$-brackets with respect to $L_1(X_1)$, it follows that
  \begin{align*}
    \rho[(u_i, u_j), (\tilde{f}_k, \tilde{g}_k)] & \leq \rho[(u_i, u_j), (f,g)] + \rho[(f,g), (\tilde{f}_k, \tilde{g}_k)] \\
                                                 & \leq n^{-\alpha/2} + n^{-\alpha/2}                                     \\
                                                 & = 2 n^{-\alpha/2}.
  \end{align*}
  Here, $(f,g)$ is an arbitrary element of $U_{ijk}$. In the second inequality, we have used the definition of $\rho$ and the fact that $d[(f,g), (\tilde{f}_k, \tilde{g}_k)] \leq \bar{\eta}$ since $(f,g) \in U_{ijk}$, which is equivalent to $\rho[(f,g), (\tilde{f}_k, \tilde{g}_k)] \leq n^{-\alpha/2}$. Now apply the second part of Lemma \ref{lem:increments_orlicz} to find that
  \begin{equation}
    \label{eq:T_tilde_orlicz_2}
    \left\|S_n(u_i,u_j) - \mathbb{E}S_n(u_i,u_j) - S_n(\tilde{f}_k, \tilde{g}_k) + \mathbb{E}S_n(\tilde{f}_k, \tilde{g}_k)\right\|_{\psi_r} \lesssim n^{1/2 - 1/r}
  \end{equation}
  for sufficiently large $n$, where $\lesssim$ is hiding a constant depending only on $r$. The notion of `sufficiently large $n$' only depends on $\alpha$; the threshold is the smallest $n$ for which $2n^{-\alpha/2} < n^{-1}$, so that Eq.\@ \eqref{eq:orlicz_statement_small} becomes applicable. Eq.\@ \eqref{eq:T_tilde_orlicz_2} and Lemma 2.2.2 in \cite{van_der_vaart_wellner:weak_convergence} now give us
  \begin{align*}
     & \left\|\max_{i,j = 1, \ldots, M_1} \max_{k=1, \ldots, M_2} \left|S_n(u_i,u_j) - \mathbb{E}S_n(u_i,u_j) - S_n(\tilde{f}_k, \tilde{g}_k) + \mathbb{E}S_n(\tilde{f}_k, \tilde{g}_k)\right| \right\|_{\psi_r} \\
     & \lesssim \psi_r^{-1}(M_1^2 M_2) n^{1/2 - 1/r}
  \end{align*}
  for sufficiently large $n$. The symbol $\lesssim$ is hiding constants depending only on $r$ and $\alpha$. Now use the bounds on $M_1$ and $M_2$ and Eqs. \@ \eqref{eq:packing_number_d} and \eqref{eq:psi_r_inverse} to see that
  $$
    \psi_r^{-1}(M_1^2 M_2) \lesssim \psi_r^{-1}\left(n^{3s\alpha}\right) \lesssim \log n
  $$
  for sufficiently large $n$ and small $\delta$, where $\lesssim$ is hiding constants depending on $s$, $\alpha$ and $K_{[\,]}$. Therefore,
  \begin{align}
    \begin{split}
      \label{eq:T_tilde_orlicz_3}
       & \left\|\max_{i,j = 1, \ldots, M_1} \max_{k=1, \ldots, M_2} \left|S_n(u_i,u_j) - \mathbb{E}S_n(u_i,u_j) - S_n(\tilde{f}_k, \tilde{g}_k) + \mathbb{E}S_n(\tilde{f}_k, \tilde{g}_k)\right| \right\|_{\psi_r} \\
       & \lesssim n^{1/2 - 1/r} \log n
    \end{split}
  \end{align}
  for sufficiently large $n$ and small $\delta$. In a similar way, we can show that
  \begin{align}
    \begin{split}
      \label{eq:T_tilde_orlicz_4}
       & \left\|\max_{i,j = 1, \ldots, M_1} \max_{k=1, \ldots, M_2} \left|S_n(l_i,l_j) - \mathbb{E}S_n(l_i,l_j) - S_n(\tilde{f}_k, \tilde{g}_k) + \mathbb{E}S_n(\tilde{f}_k, \tilde{g}_k)\right| \right\|_{\psi_r} \\
       & \lesssim n^{1/2 - 1/r} \log n
    \end{split}
  \end{align}
  for sufficiently large $n$ and small $\delta$. Eqs.\@ \eqref{eq:T_tilde_orlicz_1}, \eqref{eq:T_tilde_orlicz_3} and \eqref{eq:T_tilde_orlicz_4} now give us
  \begin{align}
    \begin{split}
      \label{eq:T_tilde_orlicz_final}
       & \left\|\sup_{d[(f,g), (f',g')] \leq \bar{\eta} : (f',g') \in \tilde{T}} \left|S_n(f,g) - \mathbb{E}S_n(f,g) - S_n(f',g') + \mathbb{E}S_n(f',g')\right|\right\|_{\psi_r} \\
       & \lesssim n^{1/2 - 1/r} \log n + n^{(1-\alpha)/2}                                                                                                                        \\
       & \lesssim n^{1/2 - 1/r} \log n
    \end{split}
  \end{align}
  for sufficiently large $n$ and small $\delta$, where $\lesssim$ is hiding constants depending only on $r$, $s$, $\alpha$ and $K_{[\,]}$, and the notion of `sufficiently large $n$' depends only on $\alpha$.

  Now combine Eqs.\@ \eqref{eq:kley_statement}, \eqref{eq:kley_rv_bound_2} and \eqref{eq:T_tilde_orlicz_final} to see that
  \begin{align}
    \begin{split}
      \label{eq:limsup_psi_r_norm}
       & \limsup_{n \to \infty} \left\|\sup_{d[(f,g), (f',g')] \leq \delta} \left|S_n(f,g) - \mathbb{E}S_n(f,g) - S_n(f',g') + \mathbb{E}S_n(f',g')\right|\right\|_{\psi_r} \\
       & \lesssim \int_0^\delta \left(\frac{4s}{\varepsilon}\right)^{1/r} ~\mathrm{d}\varepsilon + (8s)^{1/r} \delta^{1 - 1/r}
    \end{split}
  \end{align}
  for all $\delta$ sufficiently small that $D(\delta, \mathcal{F} \times \mathcal{F}, d) \geq 2$. The right-hand side tends to $0$ for $\delta \downarrow 0$ because $r > 1$, and so Eq.\@ \eqref{eq:zz_asymptotic_d_equicont} follows from Markov's inequality. Furthermore, the sequence $S_n(f,g) - \mathbb{E}S_n(f,g)$, $n \in \mathbb{N}$, is asymptotically tight in $\mathbb{R}$ for any fixed $(f,g) \in \mathcal{F} \times \mathcal{F}$ as a consequence of Lemma \ref{lem:Sn_fidi}. By Theorem 1.5.7 in \cite{van_der_vaart_wellner:weak_convergence}, this means that the sequence of processes $S_n - \mathbb{E}S_n$, $n \in \mathbb{N}$, is asymptotically tight in $\ell^\infty(\mathcal{F} \times \mathcal{F})$. Our claim now follows from Theorem 1.5.4 in \cite{van_der_vaart_wellner:weak_convergence}.
\end{proof}

As the final result in this section, we state a short corollary to the previous theorem, which we formally state for easy reference. We have essentially already proven it in the previous theorem's proof.

\begin{corollary}
  \label{cor:order_psi_sup}
  Let $1 < r < 2$. Under the assumptions of Theorem \ref{thm:Sn_convergence}, it holds that
  $$
    \sup_{n \in \mathbb{N}} \left\|\sup_{(f,g) \in \mathcal{F} \times \mathcal{F}} \left|S_n(f,g) - \mathbb{E}S_n(f,g) \right|\right\|_{\psi_r} < \infty.
  $$
\end{corollary}
\begin{proof}
  Assume without loss of generality that $\mathcal{F}$ contains the constant $0$-function. If this is not the case, we can add it to both $\mathcal{F}$ and the countably dense subset $\mathcal{F}_0$ without changing any of the assumptions. Assume furthermore without loss of generality, as we did in the proof of Theorem \ref{thm:Sn_convergence}, that the $L_1(X_1)$-diameter of $\mathcal{F}$ is bounded by $e^{-2}/2$. Pick some $\delta > 0$ such that $D(\delta, \mathcal{F} \times \mathcal{F}, d) = 2$. Eq.\@ \eqref{eq:limsup_psi_r_norm} implies that
  $$
    M = \sup_{n \in \mathbb{N}} \left\|\sup_{d[(f,g), (f',g')] \leq \delta} \left|S_n(f,g) - \mathbb{E}S_n(f,g) - S_n(f',g') + \mathbb{E}S_n(f',g')\right|\right\|_{\psi_r} < \infty.
  $$
  Since we can cover $\mathcal{F} \times \mathcal{F}$ with $N(\delta, \mathcal{F} \times \mathcal{F}, d) \leq D(\delta, \mathcal{F} \times \mathcal{F}, d) = 2$ many $\delta$-balls, and because $\mathcal{F}$ contains the constant $0$ function by assumption, we get from this
  $$
    \sup_{n \in \mathbb{N}} \left\|\sup_{(f,g) \in \mathcal{F} \times \mathcal{F}} \left|S_n(f,g) - \mathbb{E}S_n(f,g) \right|\right\|_{\psi_r} \leq 2 M < \infty.
  $$
\end{proof}

\subsection{Concomitants of Order Statistics}
The goal of this section is to establish convergence of the process
\begin{equation}
  \label{eq:goal_concomitants}
  [0,1]^2 \ni a \mapsto \frac{1}{\sqrt{n}} \sum_{i=1}^{n-1} \textbf{1}_{[0,a]}(Y_{n,i}', Y_{n,i+1}') - \mathbb{E} \textbf{1}_{[0,a]}(Y_{n,i}', Y_{n,i+1}'),
\end{equation}
from which the process convergence results of Section \ref{sec:mains_chatterjee} will follow by our general methods from Section \ref{sec:mains}. The results specific to Chatterjee's rank correlation will be proven later in Section \ref{sec:proofs_chatterjee}.

We begin with a certain representation for the concomitants $Y_{n,1}', \ldots, Y_{n,n}'$. The idea is to express each $Y_k = \tau(X_k, U_k)$ as a measurable function of $X_k$ and some independent noise. This representation is well-known \citep[see, for instance, Lemma 7.11 in][]{vandervaart:asymptotic_statistics}, and we only state this lemma here because the specific form of the function $\tau$ is important for us.

\begin{lemma}
  \label{lem:representation_lemma}
  Let $(X_k,Y_k)$, $k \in \mathbb{N}$, be independent copies of some random vector $(X,Y) \in \mathbb{R}^2$. Write $F_x(y) = \mathbb{P}(Y \leq y ~|~ X = x)$ and define $\tau(x,u) = F_x^{-1}(u)$, the generalised inverse of $F_x$ evaluated at $u$. Then there exists a copy $\left(\bar{X}_k, \bar{Y}_k\right)_{k \in \mathbb{N}}$ of $(X_k, Y_k)_{k \in \mathbb{N}}$ and an i.i.d.\@ sequence $U_k \sim \mathcal{U}[0,1]$, $k \in \mathbb{N}$, with the following properties:
  \begin{enumerate}
    \item\label{it:darstellung} $\left(\bar{X}_k, \bar{Y}_k\right)_{k \in \mathbb{N}} = \left(\bar{X}_k, \tau\left(\bar{X}_k, U_k\right)\right)_{k \in \mathbb{N}}$ almost surely,
    \item\label{it:independence} $(U_k)_{k \in \mathbb{N}}$ is independent of $\left(\bar{X}_k\right)_{k \in \mathbb{N}}$,
    \item\label{it:reordering} for all $n \in \mathbb{N}$, $U_{n,1}', \ldots, U_{n,n}'$ are i.i.d.\@ and independent of $\left(\bar{X}_k\right)_{k \in \mathbb{N}}$, where $U_{n,1}', \ldots, U_{n,n}'$ are a reordering of $U_1, \ldots, U_n$ according to the same permutation that maps $\bar{X}_1, \ldots, \bar{X}_n$ onto $\bar{X}_{n,1}, \ldots, \bar{X}_{n,n}$.
  \end{enumerate}
\end{lemma}
\begin{proof}
  By Proposition 4.1 in \cite{peters_et_al:elements_of_causal_inference}, it holds that $(X_1, Y_1) = (X_1, \tau(X_1, U_1))$ in distribution for a random variable $U_1 \sim \mathcal{U}[0,1]$ independent of $X_1$. Let $\left(\bar{X}_k, U_k\right)$ be i.i.d. copies of $(X_1, U_1)$, $k \geq 2$, set $\bar{X}_1 = X_1$ and $\bar{Y}_k = \tau\left(\bar{X}_k, U_k\right)$, $k \geq 1$. This proves (\ref{it:darstellung}) and (\ref{it:independence}). To prove (\ref{it:reordering}), notice that we can write $U_1', \ldots, U_n'$ as $\sigma_{\bar{X}_1, \ldots, \bar{X}_n}(U_1, \ldots, U_n)$, where $\sigma_{\bar{X}_1, \ldots, \bar{X}_n}$ is a random permutation fully determined by $\bar{X}_1, \ldots, \bar{X}_n$. Furthermore, with $\eta$ denoting the distribution of $U_1$, it holds that $(U_1, \ldots, U_n)$ has distribution $\eta^n$ by the i.i.d. property of the $U_k$. Let $M$ be a measurable set, and write $\mu_n$ for the joint distribution of the $\bar{X}_1, \ldots, \bar{X}_n$, then
  \begin{align*}
    \mathbb{P}((U_1', \ldots, U_n') \in M) & = \iint \textbf{1}_M(\sigma_{x_1, \ldots, x_n}(u_1, \ldots, u_n)) ~\mathrm{d}\eta^n(u_1, \ldots, u_n) ~\mathrm{d}\mu_n(x_1, \ldots, x_n) \\
                                           & = \iint \textbf{1}_M(u_1, \ldots, u_n) ~\mathrm{d}\eta^n(u_1, \ldots, u_n) ~\mathrm{d}\mu_n(x_1, \ldots, x_n)                            \\
                                           & = \int \mathbb{P}((U_1, \ldots, U_n) \in M) ~\mathrm{d}\mu_n = \mathbb{P}((U_1, \ldots, U_n) \in M),
  \end{align*}
  where the first equality holds because the $U_1, \ldots, U_n$ are independent of the $\bar{X}_1, \ldots, \bar{X}_n$ and the second one holds because product measures are invariant under permutations. Therefore the $U_1', \ldots, U_n'$ are i.i.d. since the same holds for the $U_1, \ldots, U_n$. By the same argument, it holds that
  \begin{align*}
    \mathbb{E}\left[\textbf{1}_M(U_1', \ldots, U_n') ~|~ \left(\bar{X}_k\right)_{k \in \mathbb{N}}\right] & = \int \textbf{1}_M(\sigma_{\bar{X}_1, \ldots, \bar{X}_n}(u_1, \ldots, u_n)) ~\mathrm{d}\eta^n(u_1, \ldots, u_n) \\
                                                                                                          & = \int \textbf{1}_M(u_1, \ldots, u_n) ~\mathrm{d}\eta^n(u_1, \ldots, u_n)                                        \\
                                                                                                          & = \mathbb{P}((U_1', \ldots, U_n') \in M)
  \end{align*}
  almost surely, and so $(U_1', \ldots, U_n')$ is independent of $\left(\bar{X}_k\right)_{k \in \mathbb{N}}$.
\end{proof}

Before we discuss the value of this lemma for the concomitant process \eqref{eq:goal_concomitants}, we need the following Glivenko-Cantelli type result for the functions of the order statistics $X_{n,1}, \ldots, X_{n,n}$. Its proof is very similar to the usual Glivenko-Cantelli theorem in function classes with bracketing. The only unusual part is that we have to combine these classical arguments with Lusin's theorem -- a technique which we have already used in Section \ref{sec:proofs_order_statistics}.

\begin{lemma}
  \label{lem:dreifach_glivenko_cantelli}
  Let $\mathcal{F}$ be a class of Borel measurable functions from the unit interval onto itself with $N_{[\,]}(\varepsilon, \mathcal{F}, L_1(X)) < \infty$ for all $\varepsilon > 0$. Let $(X_k)_{k \in \mathbb{N}}$ be an i.i.d.\@ process taking values in the unit interval. Then
  \begin{equation}
    \label{eq:ordnungsstatistiken_1}
    \sup_{f,g,h \in \mathcal{F}} \left| \frac{1}{n} \sum_{j=1}^{n-2} f(X_{n,j}') g(X_{n, j+1}') h(X_{n, j+2}') - \mathbb{E} [f(X_1) g(X_1) h(X_1)] \right| \xrightarrow[n \to \infty]{as*} 0,
  \end{equation}
  and, for any $d \in \mathbb{N}$,
  \begin{equation}
    \label{eq:ordnungsstatistiken_2}
    \sup_{f \in \mathcal{F}} \frac{1}{n} \sum_{j=1}^{n-d} \left|f(X_{n,j+d}')- f(X_{n,j}')\right| \xrightarrow[n \to \infty]{as*} 0.
  \end{equation}
  Furthermore, if $(U_k)_{k \in \mathbb{N}}$ is another i.i.d.\@ process independent of $(X_k)_{k \in \mathbb{N}}$ with $U_k \sim \mathcal{U}[0,1]$ for each $k \in \mathbb{N}$, then, for any $d \in \mathbb{N}$,
  \begin{equation}
    \label{eq:ordnungsstatistiken_3}
    \sup_{f \in \mathcal{F}} \frac{1}{n} \sum_{j=1}^{n-d} \left|\textbf{1}\{U_j \leq f(X_{n,j+d}')\}- \textbf{1}\{U_j \leq f(X_{n,j}')\}\right| \xrightarrow[n \to \infty]{as*} 0.
  \end{equation}
\end{lemma}
\begin{proof}
  Consider any single measurable function $f : [0,1] \to [0,1]$. We claim that, for any $d \in \mathbb{N}$,
  \begin{equation}
    \label{eq:slln_order_shift}
    \frac{1}{n} \sum_{j=1}^{n-d} |f(X_{n,j+d}') - f(X_{n,j}')| \xrightarrow[n \to \infty]{a.s.} 0,
  \end{equation}
  as well as
  \begin{equation}
    \label{eq:slln_order_shift_indikator}
    \frac{1}{n} \sum_{j=1}^{n-d} |\textbf{1}\{U_j \leq f(X_{n,j+d}')\} - \textbf{1}\{U_j \leq f(X_{n,j}')\}| \xrightarrow[n \to \infty]{a.s.} 0.
  \end{equation}

  Fix some $\varepsilon > 0$. By Lusin's theorem, there is a compact subset $K_\varepsilon \subseteq [0,1]$ such that the restriction of $f$ to $K_\varepsilon$ is continuous and $\mathbb{P}(X \notin K_\varepsilon) < \varepsilon$. As a continuous function on a compact set, this restriction is uniformly continuous, and we can choose $\delta > 0$ such that
  $$
    \sup_{x,y \in K_\varepsilon : |x-y|<\delta} |f(x) - f(y)| < \varepsilon.
  $$
  Let $T_1, \ldots, T_K$ be a partition of the unit interval into subintervals of length less than $\delta$ and write $C_n = \{j = 1, \ldots, n-d ~|~ X_{n,j}', X_{n,j+d}' \in K_\varepsilon\}$. For any $k = 1, \ldots, K$, let $t_{k,n} = \{j \in C_n ~|~ X_{n,j}' \in T_k\}$. We have
  $$
    \sum_{j \in t_{k,n}} |f(X_{n,j+d}') - f(X_{n,j}')| \leq d + \# t_{k,n} \varepsilon,
  $$
  where the $d$ on the right-hand side occurs to account for the largest $d$ indices $j$ in $t_{k,n}$, in which case it is possible that $X_{n,j+d}' \notin T_k$. Therefore,
  \begin{align*}
    \frac{1}{n} \sum_{j \in C_n} |f(X_{n,j+d}') - f(X_{n,j}')| & = \frac{1}{n}\sum_{k=1}^K \sum_{j \in t_{k,n}} |f(X_{n,j+d}') - f(X_{n,j}')|                       \\
                                                               & \leq \frac{Kd}{n} + \varepsilon \sum_{k=1}^K \frac{\# t_{k,n}}{n} \leq \frac{Kd}{n} + \varepsilon.
  \end{align*}
  The right-hand side is bounded by $2\varepsilon$ for large $n$. On the other hand,
  \begin{align}
    \begin{split}
      \label{eq:not_Cn}
       & \frac{1}{n} \sum_{j \notin C_n} |f(X_{n,j+d}') - f(X_{n,j}')|                                                                            \\
       & \leq \frac{\# (\{1, \ldots, n\} \setminus C_n)}{n}                                                                                       \\
       & \leq \frac{\# \{j = 1, \ldots, n ~|~ X_{n,j}' \notin K_\varepsilon\} + \# \{j = 1, \ldots, n-d ~|~ X_{n,j+d}' \notin K_\varepsilon\}}{n} \\
       & \xrightarrow[n \to \infty]{a.s.} 2 \mathbb{P}(X \notin K_\varepsilon) < 2 \varepsilon.
    \end{split}
  \end{align}
  From these two observations we get
  $$
    \limsup_{n \to \infty} \frac{1}{n} \sum_{j = 1}^n |f(X_{n,j+d}') - f(X_{n,j}')| < 4 \varepsilon
  $$
  almost surely. Since $\varepsilon > 0$ was arbitrary, we obtain Eq.\@ \eqref{eq:slln_order_shift}. To prove Eq.\@ \eqref{eq:slln_order_shift_indikator}, observe that if $j, j+d \in t_{k,n}$,
  $$
    |\textbf{1}\{U_j \leq f(X_{n,j+d}')\} - \textbf{1}\{U_j \leq f(X_{n,j}')\}| \leq \textbf{1}\{f(X_{n,j}') - \varepsilon \leq U_j \leq f(X_{n,j}') + \varepsilon\}.
  $$
  Thus,
  \begin{align*}
     & \sum_{j \in t_{k,n}} |\textbf{1}\{U_j \leq f(X_{n,j+d}')\} - \textbf{1}\{U_j \leq f(X_{n,j}')\}|                      \\
     & \quad \leq 2d + \sum_{j \in t_{k,n}} \textbf{1}\{f(X_{n,j}') - \varepsilon \leq U_j \leq f(X_{n,j}') + \varepsilon\},
  \end{align*}
  where the term $2d$ again accounts for the $d$ largest indices in $t_{k,n}$, in which case we use the trivial bound
  $$
    |\textbf{1}\{U_j \leq f(X_{n,j+d}')\} - \textbf{1}\{U_j \leq f(X_{n,j}')\}| \leq 2 + \textbf{1}\{f(X_{n,j}') - \varepsilon \leq U_j \leq f(X_{n,j}') + \varepsilon\}.
  $$
  Similarly to before, this gives us
  \begin{align}
    \begin{split}
      \label{eq:indikator_Cn}
       & \frac{1}{n}\sum_{j \in C_n} |\textbf{1}\{U_j \leq f(X_{n,j+d}')\} - \textbf{1}\{U_j \leq f(X_{n,j}')\}|                                 \\
       & \quad \leq \frac{2Kd}{n} + \frac{1}{n} \sum_{j \in C_n} \textbf{1}\{f(X_{n,j}') - \varepsilon \leq U_j \leq f(X_{n,j}') + \varepsilon\} \\
       & \quad  \leq \frac{2Kd}{n} + \frac{1}{n} \sum_{j=1}^n \textbf{1}\{f(X_{n,j}') - \varepsilon \leq U_j \leq f(X_{n,j}') + \varepsilon\}.
    \end{split}
  \end{align}
  To keep things short, let us write $\chi_{n,j} = \textbf{1}\{f(X_{n,j}') - \varepsilon \leq U_j \leq f(X_{n,j}') + \varepsilon\}$. Fix any $\kappa > 0$. Then by Hoeffding's inequality, and since $0 \leq \mathbb{E}[\chi_{n,j} ~|~ (X_k)_{k \in \mathbb{N}}] \leq 2\varepsilon$,
  \begin{align*}
     & \mathbb{P}\left(\frac{1}{n} \sum_{j=1}^n \chi_{n,j} - 2\varepsilon \geq \kappa ~\Big|~ (X_k)_{k \in \mathbb{N}}\right)                                                                                                                             \\
     & \quad \leq \mathbb{P}\left(\frac{1}{n} \sum_{j=1}^n \left\{\chi_{n,j} - \mathbb{E}[\chi_{n,j} ~|~ (X_k)_{k \in \mathbb{N}}]\right\} \geq \kappa ~\Big|~ (X_k)_{k \in \mathbb{N}}\right) \leq \exp\left\{-2n \kappa^2 (1 + 2\varepsilon)^2\right\}.
  \end{align*}
  Since the last bound does not depend on the process $(X_k)_{k \in \mathbb{N}}$, we also have
  $$
    \mathbb{P}\left(\frac{1}{n} \sum_{j=1}^n \chi_{n,j} - 2\varepsilon \geq \kappa\right) \leq  \exp\left\{-2n \kappa^2 (1 + 2\varepsilon)\right\}
  $$
  by the law of total probability. Summing the right-hand side over all $n \in \mathbb{N}$ results in a finite limit for any choice of $\kappa > 0$. The Borel-Cantelli lemma now implies
  $$
    \limsup_{n \to \infty} n^{-1} \sum_{j=1}^n \chi_{n,j} \leq 2 \varepsilon
  $$
  almost surely. Hence, by Eq.\@ \eqref{eq:indikator_Cn},
  $$
    \limsup_{n \to \infty} \frac{1}{n}\sum_{j \in C_n} |\textbf{1}\{U_j \leq f(X_{n,j+d}')\} - \textbf{1}\{U_j \leq f(X_{n,j}')\}| \leq 2\varepsilon
  $$
  almost surely. On the other hand, by essentially the same argument as in Eq.\@ \eqref{eq:not_Cn}, we have
  $$
    \limsup_{n \to \infty} \frac{1}{n}\sum_{j \notin C_n} |\textbf{1}\{U_j \leq f(X_{n,j+d}')\} - \textbf{1}\{U_j \leq f(X_{n,j}')\}| \leq 4\varepsilon
  $$
  almost surely. This last two results combine to imply \eqref{eq:slln_order_shift_indikator} since $\varepsilon > 0$ was arbitrary.

  Let us now prove the claims in the statement of the lemma. Using the identity $xyz - uvw = xy(z-w) + xw(y-v) + vw(x-u)$, Eq.\@ \eqref{eq:slln_order_shift} leads to
  \begin{equation}
    \label{eq:slln_three}
    \frac{1}{n} \sum_{j=1}^{n-2} f(X_{n,j}') g(X_{n, j+1}') h(X_{n, j+2}') - \mathbb{E} [f(X_1) g(X_1) h(X_1)] \xrightarrow[n \to \infty]{a.s.} 0
  \end{equation}
  for any fixed measurable $f,g,h : [0,1] \to [0,1]$. Finally, if $[l_i, u_i], i = 1, \ldots, L,$ is an $\varepsilon$-bracketing of $\mathcal{F}$, we can assume without loss of generality that $|u_i|, |l_i| \leq 1$ for all $i$, otherwise we replace each $u_i$ by $u_i \land 1$ and each $l_i$ by $l_i \lor 0$. Then $[l_i l_j l_k, u_i u_j u_k], i,j,k = 1, \ldots, L,$ defines a $3 \varepsilon$-bracketing of $\mathcal{F}_3 = \{fgh ~|~ f,g,h \in \mathcal{F}\}$. Hence, $N_{[\,]}(\varepsilon, \mathcal{F}_3, L_1(X)) < \infty$ for all $\varepsilon > 0$, and the claim in Eq. \eqref{eq:ordnungsstatistiken_1} follows by standard arguments; see the proof of Theorem 2.4.1 in \cite{van_der_vaart_wellner:weak_convergence}. The appeal to the strong law of large numbers in that proof can be replaced by Eq.\@ \eqref{eq:slln_three} to make it applicable to our situation.

  To strengthen Eq.\@ \eqref{eq:slln_order_shift} to the claim in Eq.\@ \eqref{eq:ordnungsstatistiken_2}, which is uniform in $f \in \mathcal{F}$, fix some $\varepsilon > 0$ and suppose that $f$  is in the $\varepsilon$-bracket $[l_i, u_i]$. We have the inequalities
  \begin{align*}
    f(X_{n,j+d}') - f(X_{n,j}') & = f(X_{n,j+d}') - u_i(X_{n,j}') + u_i(X_{n,j}') - f(X_{n,j}')          \\
                                & \leq u_i(X_{n,j+d}') - u_i(X_{n,j}') + u_i(X_{n,j}') - l_i(X_{n,j}')   \\
                                & \leq |u_i(X_{n,j+d}') - u_i(X_{n,j}')| + u_i(X_{n,j}') - l_i(X_{n,j}')
  \end{align*}
  and
  \begin{align*}
    f(X_{n,j}') - f(X_{n,j+d}') & = - f(X_{n,j+d}') + l_i(X_{n,j}') - l_i(X_{n,j}') + f(X_{n,j}')        \\
                                & \leq - l_i(X_{n,j+d}') + l_i(X_{n,j}') - l_i(X_{n,j}') + u_i(X_{n,j}') \\
                                & \leq |l_i(X_{n,j+d}') - l_i(X_{n,j}')| + u_i(X_{n,j}') - l_i(X_{n,j}')
  \end{align*}
  This implies that $|f(X_{n,j}') - f(X_{n,j+d}')|$ is bounded by
  $$
    |u_i(X_{n,j+d}') - u_i(X_{n,j}')| + |l_i(X_{n,j+d}') - l_i(X_{n,j}')| + u_i(X_{n,j}') - l_i(X_{n,j}'),
  $$
  and so
  \begin{align*}
    \frac{1}{n} \sum_{j=1}^{n-d} |f(X_{n,j+d}') - f(X_{n,j}')| & \leq \frac{1}{n} \sum_{j=1}^{n-d} |u_i(X_{n,j+d}') - u_i(X_{n,j}')|     \\
                                                               & \qquad + \frac{1}{n} \sum_{j=1}^{n-d} |l_i(X_{n,j+d}') - l_i(X_{n,j}')| \\
                                                               & \qquad  + \frac{1}{n} \sum_{j=1}^n u_i(X_{n,j}') - l_i(X_{n,j}').
  \end{align*}
  The first two sums converge to $0$ almost surely by Eq.\@ \eqref{eq:slln_order_shift}, whereas the third sum converges to $\mathbb{E}[u_i(X) - l_i(X)] \leq \varepsilon$ almost surely by the usual strong law of large numbers. It is therefore bounded by $2\varepsilon$ for large $n$. Taking the supremum over $f \in \mathcal{F}$ on the left-hand side translates to taking a finite maximum over $i = 1, \ldots, N_{[\,]}(\varepsilon, \mathcal{F}, L_1(X))$ on the right-hand side. So,
  \begin{align*}
     & \sup_{f \in \mathcal{F}} \frac{1}{n} \sum_{j=1}^{n-d} |f(X_{n,j+d}') - f(X_{n,j}')|                                                                      \\
     & \quad \leq \max_i \frac{1}{n} \sum_{j=1}^{n-d} |u_i(X_{n,j+d}') - u_i(X_{n,j}')| + \max_i \frac{1}{n} \sum_{j=1}^{n-d} |l_i(X_{n,j+d}') - l_i(X_{n,j}')| \\
     & \qquad + \max_i \frac{1}{n} \sum_{j=1}^n u_i(X_{n,j}') - l_i(X_{n,j}')                                                                                   \\
     & \quad \leq 2\varepsilon + 2 \varepsilon = 4 \varepsilon
  \end{align*}
  almost surely for $n$ sufficiently large. Since $\varepsilon > 0$ was arbitrary, this proves Eq.\@ \eqref{eq:ordnungsstatistiken_2}.

  Finally, let us prove Eq.\@ \eqref{eq:ordnungsstatistiken_3}. By the same arguments as before, it holds that
  \begin{align}
    \begin{split}
      \label{eq:partialsumme_indikator_sup}
       & \sup_{f \in \mathcal{F}}\frac{1}{n} \sum_{j=1}^{n-d} |\textbf{1}\{U_j \leq f(X_{n,j+d}')\} - \textbf{1}\{U_j \leq f(X_{n,j}')\}| \\
       & \leq \max_i \frac{1}{n} \sum_{j=1}^{n-d} |\textbf{1}\{U_j \leq u_i(X_{n,j+d}')\} - \textbf{1}\{U_j \leq u_i(X_{n,j}')\}|         \\
       & \quad + \max_i \frac{1}{n} \sum_{j=1}^{n-d} |\textbf{1}\{U_j \leq l_i(X_{n,j+d}')\} - \textbf{1}\{U_j \leq l_i(X_{n,j}')\}|      \\
       & \quad + \max_i \frac{1}{n} \sum_{j=1}^n \textbf{1}\{U_j \leq u_i(X_{n,j}')\} - \textbf{1}\{U_j \leq l_i(X_{n,j}')\}.
    \end{split}
  \end{align}
  As before, the first two terms tend to $0$ almost surely by Eq.\@ \eqref{eq:slln_order_shift_indikator}. For the third term, consider any fixed bracket $[l_i, u_i]$. Again, we assume without loss of generality that $0 \leq l_i, u_i \leq 1$. Using Hoeffding's inequality, we see that, for any $\kappa > 0$,
  $$
    \mathbb{P}\left(\frac{1}{n} \sum_{j=1}^n \textbf{1}\{U_j \leq u_i(X_{n,j}')\} - \textbf{1}\{U_j \leq l_i(X_{n,j}')\} - u_i(X_{n,j}') + l_i(X_{n,j}')\geq \kappa ~\Big|~ (X_k)_{k \in \mathbb{N}}\right)
  $$
  is bounded by $\exp\{-8n \kappa^2\}$. By the law of total probability,
  $$
    \mathbb{P}\left(\frac{1}{n} \sum_{j=1}^n \left\{\textbf{1}\{U_j \leq u_i(X_{n,j}')\} - \textbf{1}\{U_j \leq l_i(X_{n,j}')\} - u_i(X_{n,j}') + l_i(X_{n,j}')\right\} \geq \kappa\right)
  $$
  is also bounded by $\exp\{-8n \kappa^2\}$, and this bound is summable over $n \in \mathbb{N}$. The Borel-Cantelli lemma now gives us
  \begin{align*}
    \limsup_{n \to \infty} \frac{1}{n} \sum_{j=1}^n \textbf{1}\{U_j \leq u_i(X_{n,j}')\} - \textbf{1}\{U_j \leq l_i(X_{n,j}')\} & \leq \limsup_{n \to \infty} \frac{1}{n} \sum_{j=1}^n u_i(X_{n,j}') - l_i(X_{n,j}') \\
                                                                                                                                & = \mathbb{E}\left[u_i(X_1) - l_i(X_1)\right] \leq \varepsilon
  \end{align*}
  almost surely. The left-hand side is non-negative since $u_i \geq l_i$. Hence, for any $\varepsilon > 0$, the right-hand side in Eq.\@ \eqref{eq:partialsumme_indikator_sup} is bounded by $4\varepsilon$ for large enough $n$, which proves Eq.\@ \eqref{eq:ordnungsstatistiken_3}.
\end{proof}

We can now come to one of the main theorems in this section. Recall Lemma \ref{lem:representation_lemma}. The value of that result is that it also gives us a distributional representation for the concomitants $Y_{n,1}', \ldots, Y_{n,n}'$. By statement \ref{it:reordering}, we get that the joint distribution of $Y_{n,1}', \ldots, Y_{n,n}'$ is the same as that of $\tau(U_1, X_{n,1}'), \ldots, \tau(U_n, X_{n,n}')$. The value of the latter objects is that they contain the i.i.d.\@ random variables $U_1, \ldots, U_n$ directly. Later, we will combine this fact with the specific nature of the function $\tau$ as a generalised inverse in the following way: Consider a fixed $a = (a_1, a_2) \in [0,1]^2$. Since $F^{-1}(u) \leq y$ if and only if $u \leq F(y)$ for any distribution function $F$, it follows that
$$
  \textbf{1}\left\{\tau(U_i, X_{n,i}') \leq a_1 \land \tau(U_{i+1}, X_{n,i+1}') \leq a_2\right\} = \textbf{1}\left\{U_i \leq f_{a_1}(X_{n,i}') \land U_{i+1} \leq f_{a_2}(X_{n,i+1}')\right\},
$$
where $f_{a_1}(x) = \mathbb{P}(Y \leq a_1 ~|~ X = x)$, and $f_{a_2}$ analogously. In distribution, the left-hand side is equal to the $i$-th summand in \eqref{eq:goal_concomitants}, except for the centring with the expected value. The following theorem will establish process convergence for the centred partial sums of the right-hand side in the preceding display.

We will do this by conditioning on $X_1, \ldots, X_n$, which results in a process depending only on $(U_i, U_{i+1})$, which is a $1$-dependent sequence. After employing a very simple version of a common blocking technique (to get rid of the $1$-dependence), we can use standard methods from \cite{van_der_vaart_wellner:weak_convergence} for triangular arrays of partial sums of independent processes to get the desired convergence result conditional on $X_1, \ldots, X_n$. The extension to the unconditional convergence will be accomplished using Lemma \ref{lem:dreifach_glivenko_cantelli}. There are some technicalities here, because because `conditioning' is not well-defined when talking about potentially non-measurable random elements, but in essence this is the main idea.

\begin{theorem}
  \label{thm:weak_convergence_1}
  Let $(X_k)_{k \in \mathbb{N}}$ and $(U_k)_{k \in \mathbb{N}}$ be two independent i.i.d.\@ processes in the unit interval such that $U_k \sim \mathcal{U}[0,1]$. Let $\mathcal{F}$ be a class of Borel measurable functions from $[0,1]$ to $[0,1]$ with $\sup_Q N_{[\,]}(\varepsilon, \mathcal{F}, L_1(Q)) \lesssim \varepsilon^{-r}$ for some fixed $r \geq 1$, where the supremum is taken over all probability measures $Q$, and $\lesssim$ is hiding a universal constant. Assume further that there is a countable subset $\mathcal{F}_0 \subseteq \mathcal{F}$ which is dense in $\mathcal{F}$ with respect to pointwise convergence. Define the processes $G_n, \hat{G}_n \in \ell^\infty(\mathcal{F} \times \mathcal{F})$ by
  $$
    G_n(f,g) = \frac{1}{\sqrt{n}} \sum_{i=1}^{n-1} \left\{\textbf{1}(U_i \leq f(X'_{n,i}))\textbf{1}(U_{i+1} \leq g(X'_{n,i+1})) - f(X_{n,i}')g(X_{n,i+1}')\right\},
  $$
  and
  $$
    \hat{G}_n(f,g) = \frac{1}{\sqrt{n}} \sum_{i=1}^{n-1} \left\{\textbf{1}(U_i \leq f(X'_{n,i}))\textbf{1}(U_{i+1} \leq g(X'_{n,i+1})) - \mathbb{E}\left[f(X_{n,i}')g(X_{n,i+1}')\right]\right\}.
  $$
  Then the following statements hold:
  \begin{enumerate}
    \item \label{it:theorem_weak}
          $G_n \rightsquigarrow G$ in $\ell^\infty(\mathcal{F} \times \mathcal{F})$ for some tight mean-zero Gaussian process $G$ with covariance function
          \begin{align*}
            \Gamma((f,g), (f',g')) & = \mathbb{E}[(f \land f')(X_1) (g \land g')(X_1) - f(X_1) g(X_1) f'(X_1) g'(X_1)]     \\
                                   & \quad + \mathbb{E}[f'(X_1) g(X_1) (f \land g')(X_1) - f(X_1) g(X_1) f'(X_1) g'(X_1)]  \\
                                   & \quad + \mathbb{E}[f(X_1) g'(X_1) (f' \land g)(X_1) - f(X_1) g(X_1) f'(X_1) g'(X_1)].
          \end{align*}
    \item \label{it:theorem_EW_zentrierung} $\hat{G}_n \rightsquigarrow G + B$ in $\ell^\infty(\mathcal{F} \times \mathcal{F})$, where $G$ is the limiting process from Statement (\ref{it:theorem_weak}) and $B$ is another tight mean-zero Gaussian process, independent of $G$, with covariance function
          $$
            \beta((f,g), (f',g')) = \mathbb{E}[f(X_1)g(X_1)f'(X_1)g'(X_1)] - \mathbb{E}[f(X_1)g(X_1)] \mathbb{E}[f'(X_1) g'(X_1)].
          $$
  \end{enumerate}
\end{theorem}
\begin{proof}
  Since $(X_k)_{k \in \mathbb{N}}$ and $(U_k)_{k \in \mathbb{N}}$ are independent, we can assume without loss of generality that they are defined on two separate probability spaces $\Omega_1$ and $\Omega_2$, respectively. $\mathcal{F}$ satisfies the assumptions of Lemma \ref{lem:dreifach_glivenko_cantelli}, and so do the classes $|\mathcal{F} - \mathcal{F}| = \{|f-g| ~|~ f,g \in \mathcal{F}\}$ and $\mathcal{F} \land \mathcal{F} = \{f \land g ~|~ f,g \in \mathcal{F}\}$, since we can construct $\varepsilon$-brackets for either of these two classes from $\varepsilon$-brackets for $\mathcal{F}$. Therefore, the union $\mathcal{F}_\cup = \mathcal{F} \cup |\mathcal{F} - \mathcal{F}| \cup \mathcal{F} \land \mathcal{F}$ also satisfies the conditions of Lemma \ref{lem:dreifach_glivenko_cantelli}. By Lemma 1.9.2 in \cite{van_der_vaart_wellner:weak_convergence}, outer almost sure convergence and almost uniform convergence are equivalent for sequences. By definition of almost uniform convergence, this implies the following: For arbitrary but fixed $\gamma > 0$, we can find a measurable $A_\gamma \subseteq \Omega_1$ with $\mathbb{P}(A_\gamma) \geq 1 - \gamma$ such that, uniformly in $\omega_1 \in A_\gamma$,
  \begin{equation}
    \label{eq:au_convergence}
    \sup_{f,g,h \in \mathcal{F}_\cup} \left| \frac{1}{n} \sum_{j=1}^{n-2} f(X_{n,j}') g(X_{n, j+1}') h(X_{n, j+2}') - \mathbb{E} [f(X_1) g(X_1) h(X_1)] \right|(\omega_1) \xrightarrow[n \to \infty]{} 0
  \end{equation}
  as well as
  \begin{equation}
    \label{eq:au_convergence_difference}
    \sup_{f \in \mathcal{F}_\cup} \frac{1}{n} \sum_{j=1}^{n-d} \left| f(X_{n,j+d}') - f(X_{n,j}') \right|(\omega_1) \xrightarrow[n \to \infty]{} 0
  \end{equation}
  for all $d = 1, \ldots, D$ up to some pre-specified $D \in \mathbb{N}$. For the remainder of this proof, until pointed out otherwise, we will consider the processes $G_n = G_n(\omega_1, \omega_2)$ and $G = G(\omega_1, \omega_2)$ and all related random variables as functions in $\omega_2$ only for fixed $\omega_1 \in A_\gamma$. Essentially, we first show weak convergence conditionally on $(X_k)_{k \in \mathbb{N}}$, to then later extend it to the unconditional case. We also adopt the following notation throughout this proof: For a measurable function $h : \Omega_1 \times \Omega_2 \to \mathbb{R}$, the expression $\mathbb{E}_1h$ describes the expected value of $h$ over $\omega_1$ only while keeping $\omega_2$ fixed. More formally, $\mathbb{E}_1h$ is a function $\Omega_2 \to \mathbb{R}$ defined by $(\mathbb{E}_1 h)(\omega_2) = \int h(\omega_1, \omega_2) ~\mathbb{P}_1(\omega_1)$. $\mathbb{E}_2h$ is defined analogously, and $\mathbb{E}h$ denotes the expected value taken over both $\omega_1$ and $\omega_2$. In a few instances, we similarly write $\mathbb{E}_2^*$ for the outer expectation taken over $\omega_2$. A formal definition of this can be found on p.\@ 11 in \cite{van_der_vaart_wellner:weak_convergence}; but as we will show in this proof, all relevant expressions are measurable anyway, so that our definition for the measurable case suffices.

  We equip $\mathcal{F} \times \mathcal{F}$ with the semimetric $\rho((f,g), (f',g')) = \mathbb{E}|f(X_1) - f'(X_1)| + \mathbb{E}|g(X_1) - g'(X_1)|$. Then every $\varepsilon$-covering of $\mathcal{F}$ with respect to $L_1(X)$, centred at functions $f_1, \ldots, f_L$ induces a $2\varepsilon$-covering $(f_i, f_j)$, $i,j = 1, \ldots, L$, of $\mathcal{F} \times \mathcal{F}$ with respect to $\rho$. In particular, $(\mathcal{F} \times \mathcal{F}, \rho)$ is totally bounded.

  \paragraph*{Proof of Statement (\ref{it:theorem_weak})} Let us first prove that $G_n \rightsquigarrow G$. Fix some $K \in \mathbb{N}$. Define the processes $Z_{n,j}$ and $S_{n,i}$ by
  \begin{align*}
    Z_{n,j}(f,g) & = n^{-1/2} \textbf{1}\{U_j \leq f(X'_{n,j})\}\textbf{1}\{U_{j+1} \leq g(X'_{n,j+1})\}, \\
    S_{n,i}(f,g) & = \sum_{j= (i-1)(K+1) + 1}^{i(K+1) - 1} Z_{n,j}(f,g) - \mathbb{E}_2 Z_{n,j}(f,g),
  \end{align*}
  where $j = 1, \ldots, n$ and $i = 1, \ldots, \lfloor n/(K+1) \rfloor$. Write $m_n = \lfloor n/(K+1) \rfloor$. The sums in the definition of $S_{n,i}$ contain $K$ summands each, with $S_{n,1}$ being the sum over $Z_{n,1}, \ldots, Z_{n,K}$, $S_{n,2}$ the sum over $Z_{n,K+2}, \ldots, Z_{n,2K+1}$, and so on. Because the summands in $S_{n,i}$ and $S_{n,i+1}$ are separated by one observation, the processes $S_{n,1}, \ldots, S_{n,m_n}$ are independent for any fixed $\omega_1 \in \Omega_1$. For any finite collection $f_r, g_r \in \mathcal{F}$, $r = 1, \ldots, R$, Eq.\@ \eqref{eq:au_convergence} implies that the covariance matrix of
  $$
    \left\{\sum_{i=1}^{m_n}S_{n,i}(f_r,g_r)\right\}_{r = 1, \ldots, R}
  $$
  converges to $(\Gamma(f_r, g_r))_{r = 1, \ldots, R}$ uniformly in $\omega_1 \in A_\gamma$. By the usual multivariate Lindeberg theorem,
  \begin{equation}
    \label{eq:fidi_convergence_sni}
    \left\{\sum_{i=1}^{m_n} S_{n,i}(f_r,g_r)\right\}_{r = 1, \ldots, R} \rightsquigarrow (G(f_r, g_r))_{r=1, \ldots, R}
  \end{equation}
  uniformly in $\omega_1 \in A_\gamma$.

  For any $\eta > 0$, it holds that
  \begin{equation}
    \label{eq:bracketing_clt_1}
    \sum_{i=1}^{m_n} \mathbb{E}_2^*\left[\sup_{f,g \in \mathcal{F}} | S_{n,i}(f,g)| \textbf{1}\left(\sup_{f,g \in \mathcal{F}} |S_{n,i}(f,g)| > \eta\right)\right] \leq \frac{2K}{\sqrt{n}} \sum_{i=1}^{m_n} \textbf{1}\left(2Kn^{-1/2} > \eta\right) = 0
  \end{equation}
  for all $n \geq (2K/\eta)^2$ uniformly in $\omega_1 \in \Omega_1$.

  Fix some $f,g,f',g' \in \mathcal{F}$ and write $D_{n,j} = Z_{n,j}(f,g) - \mathbb{E}_2Z_{n,j}(f,g) - Z_{n,j}(f',g') + \mathbb{E}_2Z_{n,j}(f',g')$. Since $|D_{n,j}| \leq 4n^{-1/2}$, we have $(n/4)|\mathbb{E}_2[D_{n,i} D_{n,j}]| \leq \sqrt{n} \, \mathbb{E}_2|D_{n,j}|$ for any $i,j$, and the right-hand side is in turn bounded by
  \begin{align}
    \begin{split}
      \label{eq:Dnj_difference}
       & \mathbb{E}_2\left|\textbf{1}\{U_j \leq f(X_{n,j}')\} \textbf{1}\{U_{j+1} \leq g(X_{n,j+1}')\} - \textbf{1}\{U_j \leq f'(X_{n,j}')\} \textbf{1}\{U_{j+1} \leq g'(X_{n,j+1}')\}\right| \\
       & \quad + \left|f(X_{n,j}')g(X_{n,j+1}') - f'(X_{n,j}')g'(X_{n,j+1}')\right|                                                                                                           \\
       & \leq \mathbb{E}_2\left|\textbf{1}\{U_j \leq f(X_{n,j}')\} - \textbf{1}\{U_j \leq f'(X_{n,j}')\} \right|                                                                              \\
       & \quad + \mathbb{E}_2\left|\textbf{1}\{U_{j+1} \leq g(X_{n,j+1}')\} - \textbf{1}\{U_{j+1} \leq g'(X_{n,j+1}')\} \right|                                                               \\
       & \quad + \left|f(X_{n,j}')- f'(X_{n,j}')\right| + \left|g(X_{n,j+1}')- g'(X_{n,j+1}')\right|                                                                                          \\
       & = 2\left|f(X_{n,j}')- f'(X_{n,j}')\right| + 2\left|g(X_{n,j+1}')- g'(X_{n,j+1}')\right| ,
    \end{split}
  \end{align}
  where we have used the fact that the functions in $\mathcal{F}$ are bounded by $1$. On the other hand, since the $D_{n,j}$ are $1$-dependent for any fixed $\omega_1 \in \Omega_1$, $\mathbb{E}_2[D_{n,i} D_{n,j}] = 0$ whenever $|i-j| > 1$. From this we see that
  \begin{align}
    \begin{split}
      \label{eq:bracketing_beweis_1}
       & \sum_{i=1}^{m_n} \mathbb{E}_2\left[\left\{S_{n,i}(f,g) - S_{n,i}(f',g')\right\}^2\right]                                                                    \\
       & = \sum_{i=1}^{m_n} \sum_{j = (i-1)(K+1)+1}^{i(K+1)-1} \mathbb{E}_2[D_{n,j-1}D_{n,j}] + \mathbb{E}_2\left[D_{n,j}^2\right] + \mathbb{E}_2[D_{n,j} D_{n,j+1}] \\
       & \leq \frac{24}{n}  \sum_{i=1}^{m_n} \sum_{j = (i-1)(K+1)+1}^{i(K+1)-1} \left|f - f'\right|(X_{n,j}') + \left|g - g'\right|(X_{n,j+1}')                      \\
       & \leq \frac{24}{n} \sum_{j=1}^n \left|f - f'\right|(X_{n,j}') + \left|g - g'\right|(X_{n,j+1}')
    \end{split}
  \end{align}
  for all $\omega_1 \in \Omega_1$. By Eq.\@ \eqref{eq:au_convergence},
  \begin{align*}
     & \sup_{f,g,f',g' \in \mathcal{F}}\left|\frac{1}{n} \sum_{j=1}^n |f-f'|(X_{n,j}') + |g-g'|(X_{n,j+1}') - \rho((f,g), (f',g'))\right|                                 \\
     & = \sup_{f,g,f',g' \in \mathcal{F}}\left|\frac{1}{n} \sum_{j=1}^n |f-f'|(X_{n,j}') + |g-g'|(X_{n,j+1}') - \mathbb{E}\left[|f-f'|(X_1) + |g -g'|(X_1)\right]\right|,
  \end{align*}
  and the right-hand side tends to $0$ uniformly in $\omega_1 \in A_\gamma$. This also implies
  \begin{align*}
     & \sup_{\rho((f,g),(f',g')) < \delta} \left|\sum_{i=1}^{m_n} \mathbb{E}_2\left[(S_{n,i}(f,g) - S_{n,i}(f',g'))^2 \right]\right|                                        \\
     & \leq 24 \sup_{\rho((f,g),(f',g')) < \delta} \left| \frac{1}{n} \sum_{j=1}^n \left|f - f'\right|(X_{n,j}') + \left|g - g'\right|(X_{n,j+1}')\right|                   \\
     & \leq \sup_{f,g,f',g' \in \mathcal{F}}\left|\frac{1}{n} \sum_{j=1}^n |f-f'|(X_{n,j}') + |g-g'|(X_{n,j+1}') - \mathbb{E}\left[|f-f'|(X_1) + |g -g'|(X_1)\right]\right| \\
     & \quad +  24 \delta,
  \end{align*}
  and the right-hand side tends to $24\delta$ almost surely for any $\delta > 0$ uniformly in $\omega_1 \in A_\gamma$. In particular
  \begin{equation}
    \label{eq:bracketing_clt_2}
    \sup_{\rho((f,g),(f',g')) < \delta_n} \left|\sum_{i=1}^{m_n} \mathbb{E}_2\left[(S_{n,i}(f,g) - S_{n,i}(f',g'))^2 \right]\right| \xrightarrow[n \to \infty]{} 0
  \end{equation}
  uniformly in $\omega_1 \in A_\gamma$ for any sequence $\delta_n \downarrow 0$.

  Now for any $n \in \mathbb{N}$, let $\rho_n$ be the semimetric on $\mathcal{F}$ defined by
  $$
    \rho_n(f,f') = \frac{1}{n} \sum_{j=1}^n |f-f'|(X_{n,j}').
  $$
  $\rho_n$ is the $L_1$-seminorm with respect to a discrete probability measure and thus its bracketing number is bounded by $\varepsilon^{-r}$ by assumption. For any two $\delta$-brackets $[v_1, w_1]$ and $[v_2, w_2]$ and functions $f,f' \in [v_1, w_1]$ and $g,g' \in [v_2, w_2]$, we see by a similar argument as in Eq.\@ \eqref{eq:bracketing_beweis_1} that
  \begin{align}
    \begin{split}
      \label{eq:Sni_quadratische_differenz}
       & (S_{n,i}(f,g) - S_{n,i}(f',g'))^2                                                                                                                   \\
       & \leq \frac{K}{n} \sum_{j= (i-1)(K+1) + 1}^{i(K+1) - 1} \textbf{1}\left\{(f \land f')(X_{n,j}')\leq U_j \leq (f \lor f')(X'_{n,j})\right\}           \\
       & \quad + \frac{K}{n} \sum_{j= (i-1)(K+1) + 1}^{i(K+1) - 1}\textbf{1}\left\{(g \land g')(X_{n,j+1}')\leq U_{j+1} \leq (g \lor g')(X'_{n,j+1})\right\} \\
       & \quad + \frac{K}{n}  \sum_{j= (i-1)(K+1) + 1}^{i(K+1) - 1} |f-f'|(X_{n,j}') + |g-g'|(X_{n,j+1}')                                                    \\
       & \leq \frac{K}{n} \sum_{j= (i-1)(K+1) + 1}^{i(K+1) - 1} \textbf{1}\left\{v_1(X_{n,j}')\leq U_j \leq w_1(X'_{n,j})\right\}                            \\
       & \quad + \frac{K}{n} \sum_{j= (i-1)(K+1) + 1}^{i(K+1) - 1}\textbf{1}\left\{v_2(X_{n,j+1}')\leq U_{j+1} \leq w_2(X'_{n,j+1})\right\}                  \\
       & \quad + \frac{K}{n}  \sum_{j= (i-1)(K+1) + 1}^{i(K+1) - 1} |w_1-v_1|(X_{n,j}') + |w_2-v_2|(X_{n,j+1}')  ,
    \end{split}
  \end{align}
  and the same bound holds if we take the supremum over all $f,f' \in [v_1,w_1]$ and $g,g' \in [v_2, w_2]$ on the left-hand side. Hence,
  \begin{align*}
     & \sum_{i=1}^{m_n} \mathbb{E}\left[\sup_{f,f',g,g'}(S_{n,i}(f,g) - S_{n,i}(f',g'))^2 \right]                                             \\
     & \leq \frac{2K}{n} \sum_{i=1}^{m_n}\sum_{j= (i-1)(K+1) + 1}^{i(K+1) - 1} \left\{|w_1 - v_1|(X_{n,j}') + |w_2 - v_2|(X_{n,j+1}')\right\} \\
     & \leq \frac{2K}{n} \sum_{j=1}^n \left\{|w_1 - v_1|(X_{n,j}') + |w_2 - v_2|(X_{n,j+1}')\right\}                                          \\
     & = 2K\left\{\rho_n(v_1, w_1) + \rho_n(v_2, w_2)\right\} < 4K\delta,
  \end{align*}
  where the supremum in the first line is taken over all $f,f' \in [v_1,w_1]$ and $g,g' \in [v_2, w_2]$. The entire collection $[v_i, w_i] \times [v_j, w_j]$, $i,j = 1, \ldots, N_{[\,]}(\delta, \mathcal{F}, \rho_n)$ forms a partition of $\mathcal{F} \times \mathcal{F}$, and the size of this partition at most $\delta^{-2r}$. For the choice $\delta = \varepsilon^2/(4K)$, the last bound in the above set of inequalities becomes $\varepsilon^2$, and $\sqrt{\log \{(\varepsilon^2/4K)^{-2r}\}}$ is integrable on the unit interval. In particular, for any $\delta_n \downarrow 0$,
  \begin{equation}
    \label{eq:bracketing_clt_3}
    \int_0^{\delta_n} \sqrt{\log \left\{(\varepsilon^2/4K)^{-2r}\right\}} ~\mathrm{d}\varepsilon \xrightarrow[n \to \infty]{} 0.
  \end{equation}
  Eqs.\@ \eqref{eq:bracketing_clt_1}, \eqref{eq:bracketing_clt_2} and \eqref{eq:bracketing_clt_3} and the finite dimensional convergence in Eq.\@ \eqref{eq:fidi_convergence_sni} are the conditions of Theorem 2.11.9 in \cite{van_der_vaart_wellner:weak_convergence}. Therefore, by that theorem,
  \begin{equation}
    \label{eq:Gn_tilde_konvergenz}
    \tilde{G}_n = \sum_{i=1}^{m_n} S_{n,i} \rightsquigarrow G
  \end{equation}
  in $\ell^\infty(\mathcal{F} \times \mathcal{F})$ uniformly in $\omega_1 \in A_\gamma$. The difference $G_n - \tilde{G}_n$ is the sum over all $Z_{n,j}$ for which $j$ is not included in any of the $S_{n,i}$. In particular, the summands of this difference are independent for every $\omega_1$. Thus, even without resorting to a blocking scheme, one can check by the same methods as before that
  $$
    G_n - \tilde{G}_n \rightsquigarrow \frac{1}{K} \, G.
  $$
  in $\ell^\infty(\mathcal{F} \times \mathcal{F})$ uniformly in $\omega_1 \in A_\gamma$. The factor $1/K$ accounts for the fact that we only have $m_n = n/K$ summands in $G_n - \tilde{G}_n$, but still normalise with $n^{-1/2}$. This also implies that
  \begin{equation}
    \label{eq:billingsley_dehling_trick}
    \lim_{K \to \infty} \limsup_{n \to \infty} \mathbb{P}_2^*\left(\sup_{f,g \in \mathcal{F}} \left| G_n(f,g) - \tilde{G}_n(f,g)\right| > \varepsilon\right) = 0
  \end{equation}
  for any $\varepsilon > 0$. By Theorem 5.1 in \cite{dehling_durieu_tusche:2014}, it follows that $G_n \rightsquigarrow G$ in $\ell^\infty(\mathcal{F} \times \mathcal{F})$ as processes in $\omega_2$, uniformly in $\omega_1 \in A_\gamma$.

  By Theorem 1.7.2 in \cite{van_der_vaart_wellner:weak_convergence}, this is the case if and only if
  \begin{equation}
    \label{eq:fubini_1}
    \mathbb{E}_2^* f(G_n) \xrightarrow[n \to \infty]{} \mathbb{E}f(G)
  \end{equation}
  uniformly in $\omega_1 \in A_\gamma$ for all bounded, continuous and ball-measurable $f : \ell^\infty(\mathcal{F} \times \mathcal{F}) \to \mathbb{R}$. Ball-measurability means measurability with respect to the ball-$\sigma$-algebra on $\ell^\infty(\mathcal{F} \times \mathcal{F})$, i.e.\@ the $\sigma$-algebra generated by all open balls in $\ell^\infty(\mathcal{F} \times \mathcal{F})$; see Section 1.7 in \cite{van_der_vaart_wellner:weak_convergence} for details. The class $\mathcal{F}$ is uniformly bounded and allows for a pointwise dense countable subset by assumption. By Example 1.7.4 in \cite{van_der_vaart_wellner:weak_convergence}, the process $G_n$ is therefore ball-measurable for any $n \in \mathbb{N}$, and so the outer expectation in Eq.\@ \eqref{eq:fubini_1} can be replaced by the usual expectation. This allows us to apply Fubini's theorem to obtain
  \begin{align*}
     & \left|\mathbb{E}^* f(G_n) - \mathbb{E}f(G)\right|  = \left|\mathbb{E} \left[f(G_n) - f(G)\right]\right|                                                                                                   \\
     & = \left|\mathbb{E}_1 \left[\textbf{1}_{A_\gamma}\mathbb{E}_2 \left\{f(G_n) - f(G)\right\} + \textbf{1}_{A_\gamma^C}\mathbb{E}_2 \left\{f(G_n) - f(G)\right\}\right]\right|                                \\
     & \leq \mathbb{P}(A_\gamma) \sup_{\omega_1 \in A_\gamma}|\mathbb{E}_2[f(G_n) - f(G)]|(\omega_1) + \mathbb{P}\left(A_\gamma^C\right) \sup_{\omega_1 \notin A_\gamma}|\mathbb{E}_2[f(G_n) - f(G)]|(\omega_1).
  \end{align*}
  The first supremum tends to $0$ by Eq.\@ \eqref{eq:fubini_1}, whereas the second supremum is bounded by some constant uniform in $n$ since $f$ is assumed to be bounded. Since $\mathbb{P}(A_\gamma) \geq \gamma$, we have
  \begin{align}
    \begin{split}
      \label{eq:fubini_restterm}
      \limsup_{n \to \infty } \left|\mathbb{E}^* f(G_n) - \mathbb{E}f(G)\right| & \leq (1-\gamma) \limsup_{n \to \infty}\sup_{\omega_1 \notin A_\gamma}|\mathbb{E}_2[f(G_n) - f(G)]|(\omega_1) \\
                                                                                & \leq (1-\gamma) 2 \|f\|_\infty.
    \end{split}
  \end{align}
  $\gamma > 0$ is arbitrary, which proves that the left-hand side is $0$ for any continuous, bounded and ball-measurable $f$. Hence, by another appeal to Theorem 1.7.2 in \cite{van_der_vaart_wellner:weak_convergence}, we have $G_n \rightsquigarrow G$ in $\ell^\infty(\mathcal{F} \times \mathcal{F})$.

  \paragraph*{Proof of Statement (\ref{it:theorem_EW_zentrierung})}
  By Statement (\ref{it:theorem_weak}), it holds that $G_n \rightsquigarrow G$ in $\ell^{\infty}(\mathcal{F} \times \mathcal{F})$. Furthermore,
  $$
    (\hat{G}_n - G_n)(f,g) = \frac{1}{\sqrt{n}}\sum_{i=1}^{n-1} \left\{f(X_{n,i}')g(X_{n,i}') - \mathbb{E}\left[f(X_{n,i}')g(X_{n,i+1}')\right]\right\}.
  $$
  Thus, by Theorem \ref{thm:Sn_convergence}, it holds that $D_n = \hat{G}_n - G_n \rightsquigarrow B$.

  Recall the construction of the set $A_\gamma$ from the beginning of the proof. We claim that, for any measurable function $f$, $\mathbb{E}_2[f(G)](\omega_1)$ is constant as $\omega_1$ ranges through $A_\gamma$. Consider Eq.\@ \eqref{eq:fidi_convergence_sni}. Since the limiting distribution on the right-hand side is fully specified by the deterministic covariance matrix $(\Gamma(f_r, g_r))_{r = 1, \ldots, R}$, it does not depend on $\omega_1 \in A_\gamma$ (this is not a statement about stochastic independence of random vectors, only on the lack of influence that $\omega_1$ has on the limiting distribution in a marginal sense). Hence, by combining this observation with Eq.\@ \eqref{eq:billingsley_dehling_trick}, we know that the limiting distribution of $(G_n(f_r, g_r))_{r = 1, \ldots, R}$ does not depend on $\omega_1 \in A_\gamma$. Since the finite-dimensional projections of a tight process determine the entire distribution of the process \citep[Lemma 1.5.3 in][]{van_der_vaart_wellner:weak_convergence}, the limiting distribution of $G_n$ as a process in $\omega_2$ with $\omega_1 \in A_\gamma$ fixed does not depend on $\omega_1$. But this limiting distribution is exactly that of $G$, considered as a process in $\omega_2$ only, with $\omega_1 \in A_\gamma$ fixed. This implies constancy of $\mathbb{E}_2[f(G)](\omega_1)$ in $\omega_1 \in A_\gamma$. On the other hand, the process $D_n$ is fully determined by $\omega_1$, since it only depends on $(X_k)_{k \in \mathbb{N}}$, and so its limiting process $B$ must only depend on $\omega_1$, too. Therefore, if $f$ and $g$ are measurable functions, then
  \begin{align*}
    \mathbb{E}[f(G) g(B)] & = \mathbb{E}_1[ \textbf{1}_{A_\gamma} g(B) \mathbb{E}_2f(G)] + \mathbb{E}_1[\textbf{1}_{A_\gamma^C} g(B) \mathbb{E}_2f(G)]              \\
                          & = \mathbb{E}_1[ \textbf{1}_{A_\gamma} g(B)] \mathbb{E}_2f(G)(\omega_1^*) + \mathbb{E}_1[\textbf{1}_{A_\gamma^C} g(B) \mathbb{E}_2f(G)],
  \end{align*}
  where $\omega_1^* \in A_\gamma$ is arbitrary but fixed. Recall that the set $A_\gamma$ is measurable, so that we can really work with the usual expectations instead of outer expectations. Similarly,
  $$
    \mathbb{E}f(G) \mathbb{E}[g(B)] = \mathbb{E}_1[ \textbf{1}_{A_\gamma} g(B)]\mathbb{E}_2f(G)(\omega_1^*) + \mathbb{E}_1[\textbf{1}_{A_\gamma^C} g(B) \mathbb{E}_2f(\tilde{G})],
  $$
  where $\tilde{G}$ is a copy of $G$ independent of everything else. Thus,
  \begin{align*}
    |\mathbb{E}[f(G) g(B)] - \mathbb{E}f(G) \mathbb{E}g(B)| & \leq |\mathbb{E}_1[\textbf{1}_{A_\gamma^C} g(B) \mathbb{E}_2f(G)]| + |\mathbb{E}_1[\textbf{1}_{A_\gamma^C} g(B) \mathbb{E}_2f(\tilde{G})]| \\
                                                            & \leq 2 \|g\|_\infty \|f\|_\infty (1-\gamma).
  \end{align*}
  For $\gamma \uparrow 1$, the bound tends to $0$, which proves independence of $G$ and $B$.
\end{proof}

Next, we establish uniform $p$-integrability of the process $\hat{G}_n$ from the previous theorem, which is required for some of our general results from Section \ref{sec:mains}. For this, we will again use the incomplete chaining result from \cite{kley_etal:2016} which we have already encountered in Section \ref{sec:proofs_order_statistics}. Luckily, in this instance, we do not need an analogue of the unwieldy Lemma \ref{lem:increments_orlicz}. Instead, we will write $\hat{G}_n = G_n + (\hat{G}_n - G_n)$. Conditionally on $X_1, \ldots, X_n$, the process $G_n$ is $1$-dependent and centred, and we will be able to control its increments by standard arguments. On the other hand, the difference process $\hat{G}_n - G_n$ is precisely of the form as we have considered in Section \ref{sec:proofs_order_statistics}, and we can use Corollary \ref{cor:order_psi_sup} to deal with it. The proof of the following lemma has some overlap in terms of technique with that of Theorem \ref{thm:Sn_convergence}, because they both use the same chaining argument from \cite{kley_etal:2016}.

\begin{lemma}
  \label{lem:concomitants_uniformly_integrable}
  Assume the conditions of Theorem \ref{thm:weak_convergence_1}, and recall in particular the number $r$, the countable subset $\mathcal{F}_0$ and the process $\hat{G}_n$ from the statement of that theorem. Suppose that the countable class $\mathcal{F}_0 \subseteq \mathcal{F}$ is not only dense with respect to pointwise converge, but with respect to monotonically decreasing convergence, i.e.\@ for any $f \in \mathcal{F}$ there exists a sequence $f_k \in \mathcal{F}_0$, $k \in \mathbb{N}$, such that $f_k \downarrow f$ as $k \to \infty$. Then $\|\hat{G}_n\|_{\mathcal{F} \times \mathcal{F}}$ is Borel measurable, and it holds for any $p > 12r$ that $\sup_{n \in \mathbb{N}} \mathbb{E} \|\hat{G}_n\|_{\mathcal{F} \times \mathcal{F}}^p < \infty$.
\end{lemma}
\begin{proof}
  We may assume without loss of generality that $\mathcal{F}$ contains the constant $0$-function. If this is not the case, we can simply add it to both $\mathcal{F}$ and $\mathcal{F}_0$, which does not impact any of the assumptions.

  $\hat{G}_n$ has continuous sample paths with respect to monotonically decreasing convergence: the monotonicity ensures that the indicator functions on the definition of $\hat{G}_n$ converge pointwise, and so we can apply the dominated convergence theorem to get continuity of the sample paths. Therefore,
  $$
    \|\hat{G}_n\|_{\mathcal{F} \times \mathcal{F}} = \|\hat{G}_n\|_{\mathcal{F}_0 \times \mathcal{F}_0},
  $$
  and the right-hand side is Borel measurable, because it is the countable supremum of Borel measurable random variables. A similar argument can be made for the processes $G_n$ from Theorem \ref{thm:weak_convergence_1}.

  Observe that $\hat{G}_n = G_n + (\hat{G}_n - G_n)$, and that for any $1 < r < 2$,
  \begin{equation}
    \label{eq:Gnhat_Gn_sup}
    \sup_{n \in \mathbb{N}}\mathbb{E}\left\|\hat{G}_n - G_n\right\|_{\mathcal{F} \times \mathcal{F}}^p < \sup_{n \in \mathbb{N}}\left\|\left\|\hat{G}_n - G_n\right\|_{\mathcal{F} \times \mathcal{F}}\right\|_{\psi_r} < \infty
  \end{equation}
  by Corollary \ref{cor:order_psi_sup}, and because the $L_1$-norm can be controlled by any $\psi_r$-norm. It remains to control the supremal $L_1$-norm of $\|G_n\|_{\mathcal{F} \times \mathcal{F}}$.

  Now fix some $p > 12r$. Let us denote (as in the proof of Theorem \ref{thm:weak_convergence_1}) by $\mathbb{E}_1$ the expectation only over $X_1, \ldots, X_n$, and by $\mathbb{E}_2$ the expectation only over $U_1, \ldots, U_n$. We will show
  \begin{equation}
    \label{eq:uniform_integrability_goal}
    \sup_{n \in \mathbb{N}}\mathbb{E}_2\|G_n\|_{\mathcal{F} \times \mathcal{F}}^p < M
  \end{equation}
  for some deterministic constant $M$. Let us write $G_n = G_n^{(1)} + G_n^{(2)}$, where $G_n^{(1)}$ includes all odd indices in the sum from the definition of $G_n$, and $G_n^{(2)}$ all even indices. Then, since
  \begin{equation}
    \label{eq:uniform_integrability_goal_2}
    \sup_{n \in \mathbb{N}}\mathbb{E}_2\|G_n\|_{\mathcal{F} \times \mathcal{F}}^p \leq \sup_{n \in \mathbb{N}}\mathbb{E}_2\left\|G_n^{(1)}\right\|_{\mathcal{F} \times \mathcal{F}}^p + \sup_{n \in \mathbb{N}}\mathbb{E}_2\left\|G_n^{(2)}\right\|_{\mathcal{F} \times \mathcal{F}}^p,
  \end{equation}
  it suffices to establish Eq.\@ \eqref{eq:uniform_integrability_goal} for $G_n^{(1)}$ and $G_n^{(2)}$ instead of $G_n$. We will consider $G_n^{(1)}$, as the proof for $G_n^{(2)}$ is completely analogous.

  We will use the chaining-type argument from \cite{kley_etal:2016}, for which we need to first consider a fixed but arbitrary increment of our process. To this end, let us fix some $(f,g), (f',g') \in \mathcal{F} \times \mathcal{F}$, and write
  \begin{align*}
    D_i & = \textbf{1}\{U_{2i-1} \leq f(X_{n,2i-1}')\} \textbf{1}\{U_{2i} \leq g(X_{n,2i}')\} - f(X_{n,2i-1}')g(X_{n,2i}')            \\
        & \quad - \textbf{1}\{U_{2i-1} \leq f'(X_{n,2i-1}')\} \textbf{1}\{U_{2i} \leq g'(X_{n,2i}')\} + f'(X_{n,2i-1}')g'(X_{n,2i}'),
  \end{align*}
  for $i = 1, \ldots, \lceil (n-1)/2\rceil =: N$, so that $G_n^{(1)}(f,g) - G_n^{(1)}(f',g') = n^{-1/2}(D_1 + \ldots + D_N)$. Conditionally on $X_1, \ldots, X_n$, $D_1, \ldots, D_N$ are independent and centred random variables, and so Rosenthal's inequality \citep[Theorem 3 in][]{rosenthal:1970} gives
  \begin{equation}
    \label{eq:rosenthal_1}
    \mathbb{E}_2\left[\left|\sum_{i=1}^N D_i\right|^p\right] \leq K_p \max\left\{\sum_{i=1}^N \mathbb{E}_2\left[|D_i|^p\right], \left(\sum_{i=1}^N \mathbb{E}_2\left[D_i^2\right]\right)^{p/2}\right\}
  \end{equation}
  for some constant $K_p$ depending only on $p$. Consider two cases: First, if $\sum_i \mathbb{E}_2 \left[|D_i|^p\right] \leq 1$, Eq.\@ \eqref{eq:rosenthal_1} immediately yields
  \begin{align}
    \begin{split}
      \label{eq:rosenthal_2}
      \mathbb{E}_2\left[\left|G_n^{(1)}(f,g) - G_n^{(1)}(f',g')\right|^p\right] & = n^{-p/2} \mathbb{E}_2\left[\left|\sum_{i=1}^N D_i\right|^p\right]                                              \\
                                                                                & \leq K_p \max\left\{n^{-p/2}, \left(\frac{1}{n}\sum_{i=1}^N \mathbb{E}_2\left[D_i^2\right]\right)^{p/2}\right\}.
    \end{split}
  \end{align}
  Now assume $\sum_i \mathbb{E}_2 \left[|D_i|^p\right] > 1$. Then
  $$
    \sum_{i=1}^N \mathbb{E}_2 \left[|D_i|^p\right] < \left(\sum_{i=1}^N \mathbb{E}_2 \left[|D_i|^p\right]\right)^{p/2} \leq \left(\sum_{i=1}^N \mathbb{E}_2 \left[|D_i|^2\right]\right)^{p/2}.
  $$
  The last inequality holds because $|D_i| \leq 1$. Using this bound in Eq.\@ \eqref{eq:rosenthal_1} yields
  \begin{equation}
    \label{eq:rosenthal_3}
    \mathbb{E}_2\left[\left|G_n^{(1)}(f,g) - G_n^{(1)}(f',g')\right|^p\right]  = n^{-p/2} \mathbb{E}_2\left[\left|\sum_{i=1}^N D_i\right|^p\right] \leq K_p \left(\frac{1}{n}\sum_{i=1}^N \mathbb{E}_2\left[D_i^2\right]\right)^{p/2}.
  \end{equation}
  Combine Eqs.\@ \eqref{eq:rosenthal_2} and \eqref{eq:rosenthal_3} to see that
  \begin{equation}
    \label{eq:Gn_hat_1_increment_1}
    \mathbb{E}_2\left[\left|G_n^{(1)}(f,g) - G_n^{(1)}(f',g')\right|^p\right] \leq K_p \left(n^{-1} \lor \frac{1}{n}\sum_{i=1}^N \mathbb{E}_2\left[D_i^2\right]\right)^{p/2}
  \end{equation}
  always holds. Furthermore, we have
  $$
    \mathbb{E}_2\left[D_i^2\right] \leq 2 |f - f'|(X_{n,2i-1}') + 2|g-g'|(X_{n,2i}'),
  $$
  which can be seen in exactly the same way as in Eq.\@ \eqref{eq:Dnj_difference}. Write
  $$
    d_n\left[(f,g), (f',g')\right] = \frac{1}{n} \sum_{i=1}^n |f-f'|(X_i) + |g-g'|(X_i),
  $$
  then we get as a consequence of Eq.\@ \eqref{eq:Gn_hat_1_increment_1}
  \begin{equation}
    \label{eq:Gn_hat_1_increment_2}
    \left(\mathbb{E}_2\left[\left|G_n^{(1)}(f,g) - G_n^{(1)}(f',g')\right|^p\right]\right)^{1/p} \leq \sqrt{2}K_p^{1/p} \left\{n^{-1} \lor d_n\left[(f,g), (f',g')\right]\right\}^{1/2}.
  \end{equation}
  Since we assumed that the constant $0$-function lies in $\mathcal{F}$, and $G_n^{(1)}(0,0) = 0$ almost surely, we get that
  $$
    \left\|G_n^{(1)}\right\|_{\mathcal{F} \times \mathcal{F}} \leq \sup_{(f,g), (f',g') \in \mathcal{F} \times \mathcal{F}} \left|G_n^{(1)}(f,g) - G_n^{(1)}(f',g')\right|.
  $$
  To bound the supremal increment on the right-hand side, we use Lemma A.1 in \cite{kley_etal:2016} with (in their notation) $\delta = \eta = \mathrm{diam}(\mathcal{F} \times \mathcal{F}, d_n) = 2$ and $\bar{\eta} = 2/n$. This gives us
  \begin{equation}
    \label{eq:kley_Gn_1}
    \left\|G_n^{(1)}\right\|_{\mathcal{F} \times \mathcal{F}} \lesssim S_1 + 2 \sup_{d_n[(f,g), (f',g')] \leq 2/n : (f',g') \in \tilde{T}} \left|G_n^{(1)}(f,g) - G_n^{(1)}(f',g')\right|,
  \end{equation}
  where $\lesssim$ is hiding a constant depending only on $p$, $\tilde{T}$ is a set with at most $D(2/n, \mathcal{F} \times \mathcal{F}, d_n)$ many points, and $S_1$ is a random variable satisfying
  \begin{equation}
    \label{eq:kley_Gn_2}
    \left(\mathbb{E}_2\left[|S_1|^p\right]\right)^{1/p} \leq \tilde{K} \left[\int_0^2 D^{1/p}(\varepsilon, \mathcal{F} \times \mathcal{F}, d_n) ~\mathrm{d}\varepsilon + (2 + 4/n)\right]
  \end{equation}
  for some constant $\tilde{K}$ depending only on $p$ and the constant $K_p$ \citep[this dependence is not stated in Lemma A.1 of][directly, but follows from their proof]{kley_etal:2016}. Since $K_p$ in turn also only depends on $p$, $\tilde{K}$ depends only on $p$. $D(\varepsilon, \mathcal{F} \times \mathcal{F}, d_n)$ denotes the packing number of $\mathcal{F} \times \mathcal{F}$. Since packing numbers and covering numbers are comparable by the inequality $D(\varepsilon, \mathcal{F} \times \mathcal{F}, d_n) \leq N(\varepsilon/2, \mathcal{F} \times \mathcal{F}, d_n)$ \citep[cf.\@ p.\@ 96 in][]{van_der_vaart_wellner:weak_convergence}, we get
  $$
    D(\varepsilon, \mathcal{F} \times \mathcal{F}, d_n) \leq \sup_Q D^2(\varepsilon, \mathcal{F}, L_1(Q)) \leq \sup_Q N_{[\,]}^2(\varepsilon/2, \mathcal{F}, L_1(Q)) \lesssim (\varepsilon/2)^{-2r},
  $$
  where $\lesssim$ is hiding a universal constant (this is part of the assumptions of Theorem \ref{thm:weak_convergence_1}). Since $p > 12r > 2r$, Eq.\@ \eqref{eq:kley_Gn_2} gives us
  \begin{equation}
    \label{eq:kley_Gn_3}
    \left(\mathbb{E}_2\left[|S_1|^p\right]\right)^{1/p} \leq \tilde{K} \left[\int_0^2 (\varepsilon/2)^{-2r/p} ~\mathrm{d}\varepsilon + 6\right] < \infty.
  \end{equation}
  It remains to bound the (conditional) $L_p$-norm of the supremum in Eq.\@ \eqref{eq:kley_Gn_1}. This can be done in a similar way as in the proof of Theorem \ref{thm:Sn_convergence}. Let $[l'_i, u'_i]$ and $[l''_i, u_i'']$, $i = 1, \ldots, M_1$, be collections of $n^{-1}/2$-brackets with respect to $L_1(X)$ and $d_n$, respectively. We may assume without loss of generality that $l_i', l_i''$ and $u_i', u_i''$ take their values in $[0,1]$. $M_1$ is some number which can be assumed to be less than $(2n)^r$ up to some universal multiplicative constant. Set $l_{ij} = l_i' \lor l_j''$ and $u_{ij} = u_i' \land u_j''$. Then
  $$
    [l_{ij}, u_{ij}] \subseteq [l_i', u_i'] \cap [l_j'', u_j''],
  $$
  and since the original brackets cover all of $\mathcal{F}$, so do the brackets $[l_{ij}, u_{ij}]$. Furthermore, they are $n^{-1}/2$-brackets with respect to both $L_1(X)$ and $d_n$ simultaneously. Let us denote the finitely many elements of the set $\tilde{T}$ by $(\tilde{f}_m, \tilde{g}_m)$, $m = 1, \ldots, M_2 \leq D(2n^{-1}, \mathcal{F} \times \mathcal{F}, d_n)$. Define the sets
  $$
    U_{ijklm} = \left\{(f,g) \in \mathcal{F} \times \mathcal{F} ~|~ (f,g) \in [l_{ij}, u_{ij}] \times [l_{kl}, u_{kl}] \land d_n[(f,g), (\tilde{f}_m, \tilde{g}_m)] \leq 2/n\right\},
  $$
  $i,j,k,l = 1, \ldots, M_1$, $m = 1, \ldots, M_2$. The construction of the sets $U_{ijklm}$ as well as the following arguments are essentially the same as in the proof of Theorem \ref{thm:Sn_convergence}, starting at Eq.\@ \eqref{eq:Uijk_definition}. As in that proof, we can show that
  \begin{align*}
    \left|G_n^{(1)}(f,g) - G_n^{(1)}(\tilde{f}_m,\tilde{f}_m)\right| & \leq \left|G_n^{(1)}(u_{ij},u_{kl}) - G_n^{(1)}(\tilde{f}_m,\tilde{g}_m)\right|               \\
                                                                     & \quad + \left|G_n^{(1)}(l_{ij},l_{kl}) - G_n^{(1)}(\tilde{f}_m,\tilde{g}_m)\right| + n^{-1/2}
  \end{align*}
  for any $(f,g) \in U_{ijklm}$ and $m = 1, \ldots, M_2$. Any $(f,g) \in \mathcal{F} \times \mathcal{F}$ with $d_n[(f,g), (\tilde{f}_m, \tilde{g}_m)] \leq 2/n$ must lie in some $U_{ijklm}$. Therefore,
  \begin{align*}
     & \sup_{d_n[(f,g), (f',g')] \leq 2/n : (f',g') \in \tilde{T}} \left|G_n^{(1)}(f,g) - G_n^{(1)}(f',g')\right|                                              \\
     & \leq \max_{i,j,k,l = 1, \ldots, M_1} \max_{m = 1, \ldots, M_2} \left|G_n^{(1)}(u_{ij},u_{kl}) - G_n^{(1)}(\tilde{f}_m,\tilde{g}_m)\right|               \\
     & \quad + \max_{i,j,k,l = 1, \ldots, M_1} \max_{m = 1, \ldots, M_2} \left|G_n^{(1)}(l_{ij},l_{kl}) - G_n^{(1)}(\tilde{f}_m,\tilde{g}_m)\right| + n^{-1/2}
  \end{align*}
  By Lemma 2.2.2 in \cite{van_der_vaart_wellner:weak_convergence},
  \begin{align*}
     & \left(\mathbb{E}_2\left[\sup_{d_n[(f,g), (f',g')] \leq 2/n : (f',g') \in \tilde{T}} \left|G_n^{(1)}(f,g) - G_n^{(1)}(f',g')\right|^p\right]\right)^{1/p}                                                      \\
     & \lesssim (M_1^4 M_2)^{1/p }\max_{i,j,k,l = 1, \ldots, M_1} \max_{m = 1, \ldots, M_2} \left(\mathbb{E}_2\left[\left|G_n^{(1)}(u_{ij},u_{kl}) - G_n^{(1)}(\tilde{f}_m,\tilde{g}_m)\right|^p\right]\right)^{1/p} \\
     & \quad + (M_1^4 M_2)^{1/p}\max_{i,j,k,l = 1, \ldots, M_1} \max_{m = 1, \ldots, M_2}  \left(\mathbb{E}_2\left[\left|G_n^{(1)}(l_{ij},l_{kl}) - G_n^{(1)}(\tilde{f}_m,\tilde{g}_m)\right|^p\right]\right)^{1/p}  \\
     & \quad + n^{-1/2},
  \end{align*}
  where $\lesssim$ is hiding a constant depending only on $p$. Now use Eq.\@ \eqref{eq:Gn_hat_1_increment_2} to bound each of the expectations occurring on the right-hand side of this bound by some multiple of $n^{-1/2}$. This gives us
  \begin{align}
    \begin{split}
      \label{eq:kley_Gn_4}
       & \left(\mathbb{E}_2\left[\sup_{d_n[(f,g), (f',g')] \leq 2/n : (f',g') \in \tilde{T}} \left|G_n^{(1)}(f,g) - G_n^{(1)}(f',g')\right|^p\right]\right)^{1/p} \\
       & \lesssim (M_1^4 M_2)^{1/p } n^{-1/2}                                                                                                                     \\
       & \lesssim (n^{4r}n^{2r})^{1/p} n^{-1/2}                                                                                                                   \\
       & = n^{6r/p - 1/2},
    \end{split}
  \end{align}
  where $\lesssim$ is hiding constants depending only on $p$ and $r$. Now combine Eqs.\@ \eqref{eq:kley_Gn_1}, \eqref{eq:kley_Gn_3} and \eqref{eq:kley_Gn_4} to find that
  \begin{equation}
    \label{eq:kley_Gn_5}
    \left(\mathbb{E}_2 \left[\left\|G_n^{(1)}\right\|_{\mathcal{F} \times \mathcal{F}}^p\right]\right)^{1/p} \lesssim \tilde{K} \left[\int_0^2 (\varepsilon/2)^{-2r/p} ~\mathrm{d}\varepsilon + 6\right] + n^{6r/p - 1/2},
  \end{equation}
  for all $n \in \mathbb{N}$, where $\lesssim$ is hiding a constant depending only on $r$ and $p$. Furthermore, since $p > 12r$ by assumption, it holds that $6r/p - 1/2 < 0$. Eq.\@ \eqref{eq:kley_Gn_5} therefore implies that
  $$
    \sup_{n \in \mathbb{N}} \mathbb{E}_2 \left[\left\|G_n^{(1)}\right\|_{\mathcal{F} \times \mathcal{F}}^p\right] < M < \infty
  $$
  for some constant $M$ depending only on $r$ and $p$. The same arguments lead to an analogous bound for $G_n^{(2)}$, and so we have established Eq.\@ \eqref{eq:uniform_integrability_goal} by virtue of Eq.\@ \eqref{eq:uniform_integrability_goal_2}. Hence,
  $$
    \sup_{n \in \mathbb{N}} \mathbb{E} \left[\left\|G_n^{(1)}\right\|_{\mathcal{F} \times \mathcal{F}}^p\right] \leq \mathbb{E}_1 \sup_{n \in \mathbb{N}} \mathbb{E}_2 \left[\left\|G_n^{(1)}\right\|_{\mathcal{F} \times \mathcal{F}}^p\right] < M < \infty.
  $$
  Together with Eq.\@ \eqref{eq:Gnhat_Gn_sup}, this proves our claim.
\end{proof}

Next, as a corollary to Theorem \ref{thm:weak_convergence_1}, we finally establish convergence of the process in \eqref{eq:goal_concomitants}. The technique of how to do this has already been described in the discussion preceding Theorem \ref{thm:weak_convergence_1}. Following this corollary are a few more results, all of which are simple applications of our general methods, or technical lemmas in preparation thereof.

\begin{corollary}
  \label{cor:weak_convergence_Wtilde}
  Let $P_n$ be the empirical measure of $(Y_{n,1}', Y_{n,2}'), \ldots, (Y_{n,n-1}', Y_{n,n}')$. Then it holds that $\sqrt{n}(P_n - \mathbb{E}P_n) \rightsquigarrow G + B$ in in $\ell^{\infty}(\mathcal{G})$, where $\mathcal{G}$ is the class of all indicator functions $\textbf{1}_{[0,a]}$, $a \in [0,1]^2$, and $G$ and $B$ are two independent tight mean-zero Gaussian process with covariance functions $\Gamma$ and $\beta$, respectively, which are given by

  \begin{align*}
    \Gamma(\textbf{1}_{[0,s]}, \textbf{1}_{[0,t]}) & = \int \mathbb{E}[\textbf{1}_{[0,s]}(Y_1, Y_2) \textbf{1}_{[0,t]}(Y_1, Y_2) ~|~ (X_1, X_2) = (x,x)] ~\mathrm{d}\mathbb{P}^X(x)                     \\
                                                   & \quad + \int \mathbb{E}[\textbf{1}_{[0,s]}(Y_1, Y_2) \textbf{1}_{[0,t]}(Y_3, Y_1) ~|~ (X_1, X_2, X_3) = (x,x,x)] ~\mathrm{d}\mathbb{P}^X(x)        \\
                                                   & \quad + \int \mathbb{E}[\textbf{1}_{[0,s]}(Y_1, Y_2) \textbf{1}_{[0,t]}(Y_2, Y_3) ~|~ (X_1, X_2, X_3) = (x,x,x)] ~\mathrm{d}\mathbb{P}^X(x)        \\
                                                   & \quad - 3 \int \mathbb{E}[\textbf{1}_{[0,s]}(Y_1, Y_2) \textbf{1}_{[0,t]}(Y_3, Y_4) ~|~ (X_1, \ldots, X_4) = (x,x,x,x)] ~\mathrm{d}\mathbb{P}^X(x)
  \end{align*}
  and
  \begin{align*}
    \beta(\textbf{1}_{[0,s]}, \textbf{1}_{[0,t]})
     & = \int \mathbb{E}[\textbf{1}_{[0,s]}(Y_1, Y_2) \textbf{1}_{[0,t]}(Y_3, Y_4) ~|~ (X_1, \ldots, X_4) = (x,x,x,x)]~\mathrm{d}\mathbb{P}^X(x) \\
     & \quad - \left\{\int \mathbb{E}[\textbf{1}_{[0,s]}(Y_1, Y_2) ~|~ (X_1, X_2) = (x,x)] ~\mathrm{d}\mathbb{P}^X(x)   \right.                  \\
     & \qquad\qquad \left.\cdot \int \mathbb{E}[\textbf{1}_{[0,t]}(Y_1, Y_2) ~|~ (X_1, X_2) = (x,x)] ~\mathrm{d}\mathbb{P}^X(x)\right\}.
  \end{align*}
\end{corollary}
\begin{proof}
  By Lemma \ref{lem:representation_lemma} we may assume without loss of generality that $Y_{n,i}' = \tau(X_{n,i}', U_i')$, where $\tau$ and $U_1, \ldots, U_n$ are explained in that lemma. Consider any $a = (a_1, a_2) \in [0,1]^2$. Then it holds that $(Y_{n,i}', Y_{n,i+1}') \in [0,a]$ if and only if $\tau(X_{n,i}', U_i) \leq a_1$ and $\tau(X_{n,i+1}', U_{i+1}) \leq a_2$. Recall that $\tau$ is defined by $\tau(x,u) = F_x^{-1}(u)$, where $F_x$ is the cumulative distribution function of $Y$ conditional on $X = x$. Since $F^{-1}(u) \leq a_i$ if and only if $u \leq F(a_i)$ for any distribution function $F$ \citep[Lemma 21.1 in][]{vandervaart:asymptotic_statistics}, it follows that
  $$
    \textbf{1}_{[0,a]}(Y_{n,i}', Y_{n,i+1}') = \textbf{1}\{U_i \leq f_{a_1}(X_{n,i}')\}\textbf{1}\{U_{i+1} \leq f_{a_2}(X_{n,i+1}')\},
  $$
  where $f_t(x) = \mathbb{P}(Y \leq t ~|~ X = x)$. This also implies
  \begin{align*}
    \mathbb{E}P_n(\textbf{1}_{[0,a]}) & = \mathbb{E}\left[\frac{1}{n-1} \sum_{i=1}^{n-1} \textbf{1}\{U_i \leq f_{a_1}(X_{n,i}')\}\textbf{1}\{U_{i+1} \leq f_{a_2}(X_{n,i+1}')\}\right] \\
                                      & = \frac{1}{n-1} \sum_{i=1}^{n-1} f_{a_1}(X_{n,i}')f_{a_2}(X_{n,i+1}')
  \end{align*}
  by independence of $U_1, \ldots, U_n$ from everything else. Therefore,
  \begin{align}
    \begin{split}
      \label{eq:indikator_funktionenklasse_F}
       & \sqrt{n}(P_n - \mathbb{E}P_n)(\textbf{1}_{[0,a]})                                                                                                                                              \\
       & = \frac{n}{n-1} \frac{1}{\sqrt{n}} \sum_{i=1}^{n-1} \left[\textbf{1}\{U_i \leq f_{a_1}(X_{n,i}')\}\textbf{1}\{U_{i+1} \leq f_{a_2}(X_{n,i+1}')\} - f_{a_1}(X_{n,i}')f_{a_2}(X_{n,i+1}')\right] \\
       & = \frac{n}{n-1}\hat{G}_n(f_{a_1}, f_{a_2}),
    \end{split}
  \end{align}
  where $\hat{G}_n$ is the process from Theorem \ref{thm:weak_convergence_1}. It remains to show that the function class $\mathcal{F} = \{f_t ~|~ t \in [0,1]\}$ satisfies the assumptions of Theorem \ref{thm:weak_convergence_1}.

  Let $Q$ be any probability measure. Then $F_Q(t) = \int \mathbb{P}(Y \leq t ~|~ X = x) ~\mathrm{d}Q(x)$ defines a distribution function on $[0,1]$, regardless of whether $Q$ is the actual distribution of $X$ or another distribution. For any $\varepsilon > 0$, we can find a collection of at most $N = 2/\varepsilon$ many points $0 = s_1 < \ldots < s_{N} = 1$ such that $F_Q(s_{i+1} - ) - F_Q(s_{i}) < \varepsilon$ for any $i = 1, \ldots, N-1$, where $F_Q(s - )$ denotes the left-sided limit of $F_Q$ in $s$. The fact that this is possible can be seen by first including all points of discontinuity of $F_Q$ with $F_Q$-probability of at least $\varepsilon$ into the collection $s_1 < \ldots < s_N$ and then partitioning the spaces between these points into intervals with $F_Q$-probability of less than $\varepsilon$. The fact that $2/\varepsilon$ many points suffice comes from the observation that all these disjoint intervals have a combined $F_Q$-probability of at most $1$. Finally, for any $i = 1, \ldots, N-1$, define $u_i(x) = \mathbb{P}(Y < s_{i+1} ~|~ X = x)$ and $l_i(x) = \mathbb{P}(Y \leq s_i ~|~ X = x)$. The collection $[l_i, u_i]$, $i = 1, \ldots, N$, covers $\mathcal{F}$, and
  $$
    \|u_i - l_i\|_{L_1(Q)} = \int u_i - l_i ~\mathrm{d}Q = F_Q(s_{i+1}-) - F_Q(s_i) < \varepsilon,
  $$
  so each $[l_i, u_i]$ is an $\varepsilon$-bracket for $\mathcal{F}$ with respect to the $L_1(Q)$-norm. From this, it follows that $N_{[\,]}(\varepsilon, \mathcal{F}, L_1(Q)) \leq 2/\varepsilon$. Since the right-hand side does not depend on $Q$, taking the supremum over all $Q$ does not change the bound $2/\varepsilon$.

  As a countable subset $\mathcal{F}_0 \subseteq \mathcal{F}$ which is dense with respect to pointwise convergence, choose $\mathcal{F}_0 = \{f_q ~|~ q \in \mathbb{Q} \cap [0,1]\}$. We claim that this subset has the desired property. To see this, consider any $f_t \in \mathcal{F}$. If $t \in \mathbb{Q}$, then $f_t \in \mathcal{F}_0$ and we are done. Assume therefore that $t \notin \mathbb{Q}$. Now take a sequence $q_n \in \mathbb{Q} \cap [0,1]$ such that $q_n \downarrow t$. The sequence of functions $f_{q_n}$ approximates $f_t$ pointwise: Fix some $x \in [0,1]$ and let $D = \{d_1, d_2, \ldots\}$ denote the at most countable collection of mass points of $Y|X=x$, i.e.\@ the collection of all points $d_j$ such that $p_j = \mathbb{P}(Y = d_j ~|~ X = x) > 0$. For a given $\varepsilon > 0$, determine an index $J \in \mathbb{N}$ such that $\sum_{j > J} p_j < \varepsilon/2$. By choosing $n$ large enough we can ensure that $d_1, \ldots, d_J \notin (t, q_n]$. On the other hand, $(t,q_n] \setminus D$ does not contain any points of mass by definition, so we can also choose $n$ large enough that $\mathbb{P}((t, q_n]\setminus D ~|~ X = x) < \varepsilon/2$. Thus, for sufficiently large $n$, it holds that
  \begin{align*}
    0 \leq f_{q_n}(x) - f_t(x) & = \mathbb{P}(Y \in (t, q_n] ~|~ X = x)                                                   \\
                               & = \mathbb{P}(Y \in D \cap (t,q_n] ~|~ X = x) + \mathbb{P}((t, q_n]\setminus D ~|~ X = x) \\
                               & \leq \sum_{j > J} p_j + \mathbb{P}((t, q_n]\setminus D ~|~ X = x)                        \\
                               & < \varepsilon/2 + \varepsilon/2 = \varepsilon.
  \end{align*}
  Since $\varepsilon > 0$ was arbitrary, this proves $f_{q_n}(x) \to f_t(x)$, and hence $\mathcal{F}_0$ is dense in $\mathcal{F}$ with respect to pointwise convergence. The weak convergence now follows by Eq.\@ \eqref{eq:indikator_funktionenklasse_F} and Theorem \ref{thm:weak_convergence_1}. To work out the specific form of the covariance functions, we use the fact that $f(x) = \mathbb{E}[\textbf{1}(U \leq f(x))]$ for $U \sim \mathcal{U}[0,1]$ in combination with the equivalence $U \leq f_t(x)$ if and only if $\tau(x,U) \leq t$. To illustrate this, consider the first term in the generic formula for $\Gamma$ in Theorem \ref{thm:weak_convergence_1},

  \begin{align*}
     & \mathbb{E}[(f_{s_1} \land f_{t_1})(X_1) (f_{s_2} \land f_{t_2})(X_1)]                                                                                                            \\
     & = \mathbb{E}[\textbf{1}\{U_1 \leq f_{s_1}(X_1)\} \textbf{1}\{U_1 \leq f_{t_1}(X_1)\} \textbf{1}\{U_2 \leq f_{s_2}(X_1)\} \textbf{1}\{U_2 \leq f_{t_2}(X_1)\}]                    \\
     & = \mathbb{E}[\textbf{1}\{\tau(X_1, U_1) \leq s_1\} \textbf{1}\{\tau(X_1, U_1) \leq t_1\} \textbf{1}\{\tau(X_1, U_2) \leq s_2\} \textbf{1}\{\tau(X_1, U_2) \leq t_2\}]            \\
     & = \int \mathbb{E}[\textbf{1}\{Y_1 \leq s_1\} \textbf{1}\{Y_1 \leq t_1\} \textbf{1}\{Y_2 \leq s_2\} \textbf{1}\{Y_2 \leq t_2\} ~|~ (X_1, X_2) = (x,x)] ~\mathrm{d}\mathbb{P}^X(x) \\
     & = \int \mathbb{E}[\textbf{1}_{[0,s]}(Y_1, Y_2) \textbf{1}_{[0,t]}(Y_1, Y_2) ~|~ (X_1, X_2) = (x,x)] ~\mathrm{d}\mathbb{P}^X(x).
  \end{align*}

  The other terms can be dealt with analogously, resulting in the covariance functions from the statement of this corollary.
\end{proof}

\begin{corollary}
  \label{cor:convergence_sum_classes}
  Fix some $1 \leq c < \infty$. Using the notation of Corollary \ref{cor:weak_convergence_Wtilde}, we have the convergence $\sqrt{n}(P_n - \mathbb{E}P_n) \rightsquigarrow G$ in $\ell^\infty[\mathcal{HK}(c)]$ for some tight mean-zero Gaussian process $G$. $G$ has uniformly continuous sample paths with respect to the supremum norm on $\mathcal{HK}(c)$.
\end{corollary}
\begin{proof}
  Everything except for the characterisation of the degeneracy is a direct consequence of Corollary \ref{cor:weak_convergence_Wtilde} and Theorem \ref{thm:extension_hk_simple}, noting that tightness and separability of Borel measures are equivalent in complete metric spaces \citep[Lemma 1.3.2 in][]{van_der_vaart_wellner:weak_convergence}. Finally, by Lemma \ref{lem:concomitants_uniformly_integrable} and Eq.\@ \eqref{eq:indikator_funktionenklasse_F}, it holds that $\sup{n \in \mathbb{N}} \mathbb{E}\|\sqrt{n}(P_n - \mathbb{E}P_n)\|_\mathcal{F}^p < \infty$ for some sufficiently large $p > 1$.
\end{proof}

\begin{lemma}
  \label{lem:delta_n_0}

  Let $P_n$ be the empirical measure of $(Y_{n,1}', Y_{n,2}'), \ldots, (Y_{n,n-1}', Y_{n,n}')$, and let $P$ be the measure from Definition \ref{def:centring_P}. Then, for any $0 < c < \infty$, $\mathbb{E}P_n \to P$ in $\ell^\infty[\mathcal{HK}(c)]$.
\end{lemma}
\begin{proof}

  By Lemma \ref{lem:representation_lemma}, we may assume without loss of generality that $Y_{n,i}' = \tau(U_i, X_{n,i}')$ for any $n \in \mathbb{N}$ and $i = 1, \ldots, n-1$, where $U_i \sim \mathcal{U}[0,1]$ i.i.d.\@ and $\tau$ as defined in that lemma. Observe that, for any measurable function $h : [0,1]^2 \to \mathbb{R}$,
  \begin{align*}
     & \mathbb{E}\left[\frac{1}{n-1} \sum_{i=1}^{n-1} h\left\{\tau(U_i, X_{n,i}'), \tau(U_{i+1}, X_{n,i}')\right\}\right]   \\
     & = \frac{1}{n-1} \sum_{i=1}^{n-1} \mathbb{E}\left[h\left\{\tau(U_i, X_{n,i}'), \tau(U_{i+1}, X_{n,i}')\right\}\right] \\
     & = \frac{1}{n-1} \sum_{i=1}^{n-1} \mathbb{E}\left[h\left\{\tau(U_i, X_i), \tau(U_{i+1}, X_i)\right\}\right]           \\
     & = \mathbb{E}\left[h\left\{\tau(U_1, X_1), \tau(U_{2}, X_1)\right\}\right]                                            \\
     & = \int \mathbb{E}[h(Y_1, Y_2) ~|~ X_1 = X_2 = x] ~\mathrm{d}\mathbb{P}^X(x)                                          \\
     & = P(h).
  \end{align*}
  In the third inequality we have used that the $U_1, \ldots, U_n$ are independent of $X_1, \ldots, X_n$, and so applying a permutation to them which only depends on $X_1, \ldots, X_n$ does not change their joint distribution. Hence, by Theorem \ref{thm:koksma-hlawka_variant},
  \begin{align}
    \begin{split}
      \label{eq:bias_ineq}
       & \|\mathbb{E}P_n - P\|_{\mathcal{HK}(c)}                                                                                                                                                                           \\
       & \leq c \sup_{a \in [0,1]^2} \left|\mathbb{E}P_n(\textbf{1}_{[0,a]}) - P(\textbf{1}_{[0,a]})\right|                                                                                                                \\
       & = \sup_{a \in [0,1]^2} \left|\mathbb{E}\left[\frac{1}{n-1}\sum_{i=1}^{n-1} \textbf{1}_{[0,a]}(Y_{n,i}', Y_{n,i+1}') - \textbf{1}_{[0,a]}\left\{\tau(U_i, X_{n,i}'), \tau(U_{i+1}, X_{n,i}')\right\}\right]\right| \\
       & \leq \mathbb{E}^* \sup_{a \in [0,1]^2} \left|\frac{1}{n-1}\sum_{i=1}^{n-1} \textbf{1}_{[0,a]}(Y_{n,i}', Y_{n,i+1}') - \textbf{1}_{[0,a]}\left\{\tau(U_i, X_{n,i}'), \tau(U_{i+1}, X_{n,i}')\right\}\right|.
    \end{split}
  \end{align}
  The last inequality is an application of the following observation: If $f_i$, $i \in I$, is a family of measurable real-valued functions indexed in some arbitrary set $I$, then
  $$
    \sup_{i \in I} |\mathbb{E} f_i| \leq \sup_{i \in I} \mathbb{E}|f_i| = \sup_{i \in I} \mathbb{E}^*|f_i| \leq \sup_{i \in I} \mathbb{E}^* \sup_{j \in I} |f_j| = \mathbb{E}^* \sup_{j \in I} |f_j|
  $$
  by Jensen's inequality for the usual expectation and monotonicity of the outer expectation.

  For any $a = (a_1, a_2) \in [0,1]^2$, we define the functions $f_a(x) = \mathbb{P}(Y \leq a_1 ~|~ X = x)$ and $g_a(x) = \mathbb{P}(Y \leq a_2 ~|~ X = x)$. Then $\textbf{1}_{[0,a]}[\tau(x_1,u_1), \tau(x_2, u_2)]$ is equal to
  $$
    \textbf{1}\{\tau(x_1, u_1) \leq a_1\} \textbf{1}\{\tau(x_2, u_2) \leq a_2\} = \textbf{1}\{u_1 \leq f_a(x_1)\} \textbf{1}\{u_2 \leq g_a(x_2)\}.
  $$
  Hence, with $\mathcal{F}$ denoting the collection of all such functions $f_a$ and $g_a$,
  \begin{align}
    \begin{split}
      \label{eq:bias_ineq_2}
       & \sup_{a \in [0,1]^2} \left|\frac{1}{n-1}\sum_{i=1}^{n-1} \textbf{1}_{[0,a]}(Y_{n,i}', Y_{n,i+1}') - \textbf{1}_{[0,a]}\left\{\tau(U_i, X_{n,i}'), \tau(U_{i+1}, X_{n,i}')\right\}\right|                  \\
       & = \sup_{f,g \in \mathcal{F}} \left|\frac{1}{n-1} \sum_{i=1}^{n-1} \textbf{1}\{U_i \leq f(X_{n,i}')\}\left[\textbf{1}\{U_{i+1} \leq g(X_{n,i+1}')\} - \textbf{1}\{U_{i+1} \leq g(X_{n,i}')\}\right]\right| \\
       & \leq \sup_{g \in \mathcal{F}} \frac{1}{n-1} \sum_{i=1}^{n-1} \left|\textbf{1}\{U_{i+1} \leq g(X_{n,i+1}')\} - \textbf{1}\{U_{i+1} \leq g(X_{n,i}')\}\right|,
    \end{split}
  \end{align}
  and the last expression tends to $0$ outer almost surely by Lemma \ref{lem:dreifach_glivenko_cantelli}, Eq.\@ \eqref{eq:ordnungsstatistiken_3}. It is also bounded by $1$ for all $n \in \mathbb{N}$. Therefore,
  $$
    \mathbb{E}^* \sup_{a \in [0,1]^2} \left|\frac{1}{n-1}\sum_{i=1}^{n-1} \textbf{1}_{[0,a]}(Y_{n,i}', Y_{n,i+1}') - \textbf{1}_{[0,a]}\left\{\tau(U_i, X_{n,i}'), \tau(U_{i+1}, X_{n,i}')\right\}\right| \xrightarrow[n \to \infty]{} 0
  $$
  by the dominated convergence theorem for outer expectations \citep[Exercise 1.2.4 in][]{van_der_vaart_wellner:weak_convergence}, and so our claim follows by Eq.\@ \eqref{eq:bias_ineq}.
\end{proof}

\begin{proof}[Proof of Theorem \ref{thm:prozesskonvergenz_klammer}]

  By Corollary \ref{cor:convergence_sum_classes}, $\sqrt{n}(P_n - \mathbb{E}P_n) \rightsquigarrow G$ in $\ell^\infty[\mathcal{HK}(c)]$. $G$ is tight and has uniformly continuous sample paths with respect to the supremum norm. Furthermore, $\mathbb{E}P_n \to P$ in $\ell^\infty[\mathcal{HK}(c)]$ by Lemma \ref{lem:delta_n_0}, where $P$ is the measure from Definition \ref{def:centring_P}.

  Therefore, by Theorems \ref{thm:v_statistik_from_empirical_process} and \ref{thm:product_uniformly_integrable} and Lemma \ref{lem:concomitants_uniformly_integrable},
  $$
    \sqrt{n}\left(\tilde{P}_n^m - \mathbb{E}\left[P_n^m\right]\right) \rightsquigarrow G_P^{(m)}
  $$
  in $\ell^\infty[\mathcal{HK}(c)]$. To prove Eq.\@ \eqref{eq:concomitants_p_moments}, observe that
  \begin{align*}
     & \sup_{n \in \mathbb{N}}\mathbb{E}^* \left\|\sqrt{n}(P_n^m - \mathbb{E}P_n^m)\right\|_{\mathcal{HK}(c)_m}^p                                                                                                                                     \\
     & \leq 2^{p-1}\sup_{n \in \mathbb{N}}\mathbb{E}^* \left\|\sqrt{n}(P_n^m - (\mathbb{E}P_n)^m)\right\|_{\mathcal{HK}(c)_m}^p + 2^{p-1} \sup_{n \in \mathbb{N}} \left\|\mathbb{E}\sqrt{n}(P_n^m - (\mathbb{E}P_n)^m)\right\|_{\mathcal{HK}(c)_m}^p.
  \end{align*}
  The second line of this inequality is finite for sufficiently large $p$, again as a consequence of Theorems \ref{thm:v_statistik_from_empirical_process} and \ref{thm:product_uniformly_integrable} and Lemma \ref{lem:concomitants_uniformly_integrable}.
\end{proof}

\begin{proof}[Proof of Theorem \ref{thm:kovarianzfunktion_konkomitantenprozess}]
  This is a direct consequence of Corollary \ref{cor:weak_convergence_Wtilde} and Theorems \ref{thm:extension_covariance} and \ref{thm:v_statistik_from_empirical_process}, Eq.\@ \eqref{eq:V_kovarianz_statement}.
\end{proof}

Let us now prove the statement of Theorem \ref{thm:bias_conditions} concerning $\Delta_n$. The claim about $\delta_n$ will follow from this more general result. For this, we refer to Corollary \ref{cor:delta_n_bias}. It will be shown that the bias $\Delta_n$ is related to the behaviour of the order statistics $X_{n,1}', \ldots, X_{n,n}'$ in the sense that
$$
  \|\Delta_n\|_{\mathcal{HK}(c)} \leq \sup_{g \in \mathcal{F}} \mathbb{E}\left[\frac{1}{\sqrt{n}} \sum_{j=1}^{n-1} \left|g(X_{n,j+1}') - g(X_{n,j}')\right|\right] + o(1),
$$
and the conditions in Theorem \ref{thm:bias_conditions} are such that they make controlling the right-hand side easy.

\begin{proof}[Proof of Theorem \ref{thm:bias_conditions} for $\Delta_n$]

  Recall from the proof of Theorem \ref{thm:prozesskonvergenz_klammer} that Theorem \ref{thm:product_uniformly_integrable} is applicable. Therefore,
  $$
    \|\sqrt{n}\{\mathbb{E}P_n^m - (\mathbb{E}P_n)^m\}\|_{\mathcal{HK}(c)_m} \xrightarrow[n \to \infty]{} 0.
  $$
  This is Eq.\@ \eqref{eq:moment_convergence_conclusion_2} of Theorem \ref{thm:product_uniformly_integrable}. Since
  $$
    \|\Delta_n\|_{\mathcal{HK}(c)_m} \leq \|\sqrt{n}\{\mathbb{E}P_n^m - (\mathbb{E}P_n)^m\}\|_{\mathcal{HK}(c)_m} + \|\sqrt{n}\{(\mathbb{E}P_n)^m - P^m\}\|_{\mathcal{HK}(c)_m},
  $$
  it suffices to consider the second term on the right-hand side of this inequality. As in the derivation of Eq.\@ \eqref{eq:bias_bound}, one can show that
  \begin{align*}
     & \|\sqrt{n}\{(\mathbb{E}P_n)^m - P^m\}\|_{\mathcal{HK}(c)_m}                                                                         \\
     & \leq \|\sqrt{n}\{\mathbb{E}P_n - P\}\|_{\mathcal{HK}(c)} + \|\sqrt{n}\{(\mathbb{E}P_n)^{m-1} - P^{m-1}\}\|_{\mathcal{HK}(c)_{m-1}}.
  \end{align*}
  We can iterate this inequality to find that
  \begin{equation}
    \label{eq:neu_1}
    \|\sqrt{n}\{(\mathbb{E}P_n)^m - P^m\}\|_{\mathcal{HK}(c)_m} \leq m\|\sqrt{n}\{\mathbb{E}P_n - P\}\|_{\mathcal{HK}(c)}.
  \end{equation}
  Similar to Eqs.\@ \eqref{eq:bias_ineq} and \eqref{eq:bias_ineq_2}, we can show that
  \begin{align}
    \begin{split}
      \label{eq:neu_2}
       & \|\sqrt{n}\{\mathbb{E}P_n - P\}\|_{\mathcal{HK}(c)}                                                                                                                                           \\
       & \leq c \sup_{g \in \mathcal{F}} \mathbb{E} \left[\frac{\sqrt{n}}{n-1} \sum_{j=1}^{n-1} \left|\textbf{1}\{U_{j+1} \leq g(X_{n,j+1}')\} - \textbf{1}\{U_{j+1} \leq g(X_{n,j}')\}\right|\right],
    \end{split}
  \end{align}
  where $\mathcal{F} = \{f_a ~|~ a \in [0,1]\}$, $f_a(x) = \mathbb{P}(Y \leq a ~|~ X = x)$. The only difference compared to the procedure in Eqs.\@ \eqref{eq:bias_ineq} and \eqref{eq:bias_ineq_2} is that we don't pull the expectation outside of the supremum, and the extra factor of $\sqrt{n}$. Hence, by Eqs.\@ \eqref{eq:neu_1} and \eqref{eq:neu_2},

  \begin{align}
    \begin{split}
      \label{eq:delta_n_Sn}
       & \|\sqrt{n}\{(\mathbb{E}P_n)^m - P^m\}\|_{\mathcal{HK}(c)_m}                                                                                                                                 \\
       & \lesssim \sup_{g \in \mathcal{F}} \mathbb{E}\left[ \frac{1}{\sqrt{n}} \sum_{j=1}^{n-1}\left|\textbf{1}\{U_{j+1} \leq g(X_{n,j+1}')\} - \textbf{1}\{U_{j+1} \leq g(X_{n,j}')\}\right|\right] \\
       & \leq \sup_{g \in \mathcal{F}} \mathbb{E}\left[\frac{1}{\sqrt{n}} \sum_{j=1}^{n-1} \left|g(X_{n,j+1}') - g(X_{n,j}')\right|\right],
    \end{split}
  \end{align}
  where $\lesssim$ is hiding a constant depending only on $c$ and $m$.

  Under assumption (\ref{it:bounded_tv}), it holds that
  $$
    \frac{1}{\sqrt{n}} \sum_{j=1}^{n-1} \left|g(X_{n,j+1}') - g(X_{n,j}')\right| \leq \sup_{g \in \mathcal{F}} \|g\|_{\mathrm{TV}} n^{-1/2},
  $$
  since the order statistics $X_{n,1}' \leq \ldots \leq X_{n,n}'$ partition the unit interval. In all other cases, we can bound the right-hand side of Eq.\@ \eqref{eq:delta_n_Sn} in terms of the number of indices $j = 1, \ldots, n-1$ such that $X_{n,j}' \neq X_{n,j+1}'$. This number in turn cannot exceed the number $K_n$ of unique values in the sample $X_1, \ldots, X_n$; i.e.\@
  $$
    \frac{1}{\sqrt{n}} \sum_{j=1}^{n-1} \left|g(X_{n,j+1}') - g(X_{n,j}')\right| \leq K_n n^{-1/2},
  $$
  Under assumption (\ref{it:discrete_finite}), $\sup_{n \in \mathbb{N}} K_n$ can be bounded by some deterministic finite number $K_\infty$, and so
  $$
    \frac{1}{\sqrt{n}} \sum_{j=1}^{n-1} \mathbb{E}\left|g(X_{n,j+1}') - g(X_{n,j}')\right| \leq K_\infty n^{-1/2},
  $$
  Finally, under assumption (\ref{it:discrete_infinite}), it holds that
  $$
    \mathbb{E}\left[n^{-1/2} K_n\right] \leq C n^{\gamma - 1/2} L(n)
  $$
  for some constant $C$ depending only on the function $\alpha$. This is a consequence of Theorem 1' in \cite{karlin:1967}. Since $L$ is slowly varying, the right-hand side converges to $0$ for $n \to \infty$ by Lemma VIII.8.2 in \cite{feller:prob_theory_applications_vol2}. Thus, under any of the conditions of the theorem,
  $$
    \|\sqrt{n}\{(\mathbb{E}P_n)^m - P^m\}\|_{\mathcal{HK}(c)_m} \xrightarrow[n \to \infty]{} 0
  $$
  as a consequence of Eq.\@ \eqref{eq:delta_n_Sn}.
\end{proof}

Next, we prove Theorem \ref{thm:bias_conditions_large} for $\Delta_n$. The proof for $\delta_n$ will be given in Section \ref{sec:proofs_chatterjee}. We know from Theorem \ref{thm:bias_conditions} that the bias $\sqrt{n}\Delta_n$ is small if the functions $f_a : x \mapsto \mathbb{P}(Y \leq a ~|~ X = x)$ are sufficiently smooth in some sense. Our approach is to invert this idea, i.e.\@ to construct a joint distribution for which the functions $f_a$ are sufficiently non-smooth so as to make the bias non-negligible.

\begin{proof}[Proof of Theorem \ref{thm:bias_conditions_large} for $\Delta_n$]

  As in the proof of Theorem \ref{thm:bias_conditions}, observe that
  $$
    \|\sqrt{n}\{\mathbb{E}P_n^m - (\mathbb{E}P_n)^m\}\|_{\mathcal{HK}(c)_m} \xrightarrow[n \to \infty]{} 0.
  $$
  It therefore suffices to find a function $h \in \mathcal{HK}(c)_m$ such that
  $$
    \sqrt{n}\left\{(\mathbb{E}P_n)^m(h) - P^m(h)\right\} \xrightarrow[n \to \infty]{} \infty.
  $$
  We will do so for the case $m = 2$ and some sufficiently large $c > 0$, and then argue below how this can be extended to the general case. Consider the function
  $$
    h[(s_1, s_2), (t_1, t_2)] = \textbf{1}\{s_1 \land s_2 < t_1 \leq s_1 \lor s_2\}
  $$
  This choice of $h$ is not obvious. The main reason for choosing this function is that it will be important in the analysis of the bias $\delta_n$ for Chatterjee's rank correlation, so that by considering this function we are already doing some work for $\delta_n$.

  Let us now construct a specific joint distribution for $X$ and $Y$. Recall that we want the functions $f_a : x \mapsto \mathbb{P}(Y \leq a ~|~ X = x)$ to be sufficiently non-smooth in an appropriate sense. The construction is as follows:

  Fix some $1/2 < \beta < \alpha < 1$ and a countably infinite collection of points $0 < \omega_1 < \omega_2 < \ldots < 1$. Let $p_k' = c_0 k^{-\alpha/\beta}$, $k \in \mathbb{N}$, where $c_0 = (\sum_{k=1}^\infty k^{-\alpha/\beta})^{-1}$. From this, construct a new sequence of numbers $p_k$, $k \in \mathbb{N}$, such that $p_k = p_k'$ for all $k \geq K$, where $K \in \mathbb{N}$ is some arbitrary but fixed threshold, and $\sum_{k=1}^\infty p_{2k} = \sum_{k=1}^\infty p_{2k-1} = 1/2$. The specific way of how this is done is not relevant; what is important is that the even and odd indices have equal measure, and that $p_k = p_k'$ for almost all $k$. Let $X$ be distributed on $\{\omega_k ~|~ k \in \mathbb{N}\}$ with $\mathbb{P}(X = \omega_k) = p_k$, and define $E = \{\omega_{2k} ~|~ k \in \mathbb{N}\}$ and $f = \textbf{1}_E/2$. Define
  \begin{equation}
    \label{eq:def_Y}
    Y = U/2 + f(X),
  \end{equation}
  where $U \sim \mathcal{U}[0,1]$. Elementary calculations show that the cumulative distribution function of $Y$ is given by
  \begin{equation}
    \label{eq:cdf_von_Y}
    G(z) = \begin{cases}
      0, & \quad z < 0, \\  z, &\quad 0 \leq z < 1, \\  1, &\quad z \geq 1,
    \end{cases}
  \end{equation}
  and so the marginal distribution of $Y$ is $\mathcal{U}[0,1]$.  Let us now calculate the integrals $P^2(h)$ and $(\mathbb{E}P_n)^2 (h)$. Since the function $h$ depends on $(t_1, t_2)$ only through $t_1$, it holds that
  \begin{align*}
    P^2(h) & = \iint \textbf{1}\{s_1 \land s_2 < t_1 \leq s_1 \lor s_2\} ~\mathrm{d}P(t_1, t_2) ~\mathrm{d}P(s_1, s_2)       \\
           & = \iint \textbf{1}\{s_1 \land s_2 < t_1 \leq s_1 \lor s_2\} ~\mathrm{d}\mathbb{P}^Y(t_1) ~\mathrm{d}P(s_1, s_2) \\
           & = \int |s_1 - s_2| ~\mathrm{d}P(s_1, s_2)                                                                       \\
           & = \int \mathbb{E}[|Y_1 - Y_2| ~|~ X_1 = X_2 = x] ~\mathrm{d}\mathbb{P}^X(x)                                     \\
           & = \frac{1}{6}
  \end{align*}
  In the third equality we have used that $Y \sim \mathcal{U}[0,1]$. The last equality holds due to the construction of $Y$, Eq.\@ \eqref{eq:def_Y}, which ensures that the conditional expectation $\mathbb{E}[|Y_1 - Y_2| ~|~ X_1 = X_2 = x]$ is always equal to $\mathbb{E}|U_1 - U_2|/2 = 1/6$. To work out $(\mathbb{E}P_n)^2$, let $(\tilde{X}_1, \tilde{Y}_1), \ldots, (\tilde{X}_n, \tilde{Y}_n)$ be copies of $(X_1, Y_1), \ldots, (X_n, Y_n)$ independent of everything else. Then
  \begin{align*}
    (\mathbb{E}P_n)^2(h) & = \frac{1}{(n-1)^2} \sum_{i,j=1}^{n-1} \mathbb{E}\left[\textbf{1}\left\{Y_{n,i}' \land Y_{n,i+1}' < \tilde{Y}_{n,j}' \leq Y_{n,i}' \lor Y_{n,i+1}'\right\}\right] \\
                         & = \frac{1}{(n-1)^2} \sum_{i,j=1}^{n-1} \mathbb{E}\left[\textbf{1}\left\{Y_{n,i}' \land Y_{n,i+1}' < \tilde{Y}_j \leq Y_{n,i}' \lor Y_{n,i+1}'\right\}\right]      \\
                         & = \frac{1}{n-1} \sum_{i=1}^{n-1} \mathbb{E}\left|Y_{n,i}'- Y_{n,i+1}'\right|
  \end{align*}
  Once again, we have used the fact that $h$ depends on $(t_1,t_2)$ only through $t_1$ (second equality) and the fact that $Y$ has a uniform marginal distribution (third equality). Combining the identities for $P^2(h)$ and $(\mathbb{E}P_n)^2(h)$, we see that
  \begin{equation}
    \label{eq:diff_identity_1}
    \sqrt{n}\left\{(\mathbb{E}P_n)^2(h) - P^2(h)\right\} = \frac{\sqrt{n}}{n-1}\sum_{i=1}^{n-1} \mathbb{E}\left[\left|Y_{n,i}'- Y_{n,i+1}'\right| - \frac{1}{6}\right].
  \end{equation}

  Now use once more the construction of $Y$, Eq.\@ \eqref{eq:def_Y}, to see that
  $$
    |Y_{n,i}' - Y_{n,i+1}'| = \begin{cases}
      |U_{n,i}' - U_{n,i+1}'|/2     & \quad \textrm{if } X_{n,i}', X_{n,i+1}' \in E,          \\
      |U_{n,i}' - U_{n,i+1}'|/2     & \quad \textrm{if } X_{n,i}', X_{n,i+1}' \notin E,       \\
      |U_{n,i}' - U_{n,i+1}' + 1|/2 & \quad \textrm{if } X_{n,i}' \in E, X_{n,i+1}' \notin E, \\
      |U_{n,i}' - U_{n,i+1}' - 1|/2 & \quad \textrm{if } X_{n,i}' \notin E, X_{n,i+1}' \in E, \\
    \end{cases}
  $$
  where $U_{n,1}', \ldots, U_{n,n}'$ is a reordering of $U_1, \ldots, U_n$ according to the same permutation which transforms $X_1, \ldots, X_n$ into $X_{n,1}', \ldots, X_{n,n}'$. Since $U_1, \ldots, U_n$ is independent from $X_1, \ldots, X_n$, the joint distribution of $U_{n,i}'$ and $U_{n,i+1}'$ is the same as that of $U_1$ and $U_2$, i.e.\@ uniform on $[0,1]^2$. Hence,
  \begin{align*}
     & \mathbb{E}\left[\left|Y_{n,i}'- Y_{n,i+1}'\right| ~|~ X_1, \ldots, X_n\right]                                                                               \\
     & = \frac{1}{6} \textbf{1}\left\{X_{n,i}', X_{n,i+1}' \in E\right\} + \frac{1}{6} \textbf{1}\left\{X_{n,i}', X_{n,i+1}' \notin E\right\}                      \\
     & \quad + \frac{1}{2}\textbf{1}\left\{X_{n,i}' \in E, X_{n,i+1}' \notin E\right\}  + \frac{1}{2}\textbf{1}\left\{X_{n,i}' \notin E, X_{n,i+1}' \in E\right\}.
  \end{align*}
  Therefore, by Fubini's theorem and Eq.\@ \eqref{eq:diff_identity_1},
  \begin{align}
    \begin{split}
      \label{eq:diff_identity_2}
       & \sqrt{n}\left\{(\mathbb{E}P_n)(h) - P^2(h)\right\}                                                                                                                                             \\
       & = \frac{\sqrt{n}}{3(n-1)} \mathbb{E}\left[\sum_{i=1}^{n-1} \textbf{1}\left\{X_{n,i}' \in E, X_{n,i+1}' \notin E\right\}  + \textbf{1}\left\{X_{n,i}' \notin E, X_{n,i+1}' \in E\right\}\right]
    \end{split}
  \end{align}

  Recall the specific index $K$ from the construction of the numbers $p_k$ and define the events
  $$
    \Omega_{1,n} = \left\{\forall k = K, K+1, \ldots, \left\lceil n^\beta \right\rceil : \exists j = 1, \ldots, n : X_j = \omega_k\right\}.
  $$
  if $n$ is large enough that $\lceil n^\beta \rceil \geq K$, and $\Omega_{1,n} = \emptyset$ otherwise. Clearly, on the event $\Omega_{1,n}$, the sum in the last line of Eq.\@ \eqref{eq:diff_identity_2} is bounded from below by $n^\beta - K + 1$. The reason for considering $\omega_k$ only for $k \geq K$ is purely technical; namely, that for $k \geq K$, the associated probabilities $p_k$ are given by the nice formula $p_k = c_0 k^{-\alpha/\beta}$, so that the following arguments are easier. We will now work out the probability of $\Omega_{1,n}$ through its complement $\Omega_{1,n}^C$. If $n$ is sufficiently large that $K \leq \lceil n^\beta \rceil \leq 2^{\beta/\alpha} n^\beta$, then
  \begin{align}
    \begin{split}
      \label{eq:wsk_Omega_n}
      \mathbb{P}\left(\Omega_{1,n}^C\right) \leq \sum_{k=K}^{\left\lceil n^\beta \right\rceil} (1 - p_k)^n & \leq 2^{\beta/\alpha} n^\beta \left(1 - \min_{K \leq k \leq 2^{\beta/\alpha} n^\beta} p_k\right)^n \\
                                                                                                           & \leq 2^{\beta/\alpha} n^\beta \left[1 - \frac{c_0}{2} n^{-\alpha}\right]^n,
    \end{split}
  \end{align}
  since $p_k = c_0 k^{-\alpha/\beta}$ for $k \geq K$. Define the function $g : (1,\infty) \to \mathbb{R}$ by
  $$
    g(x) = \beta \log x + x \log \left[1 - \frac{c_0}{2} x^{-\alpha}\right].
  $$
  By Taylor's theorem, $\log [1-(c_0/2)z] = -(c_0/2) z + \mathcal{O}(z^2)$ for $z \downarrow 0$, and so
  $$
    g(x) = \beta \log x - \frac{c_0}{2} x^{1 - \alpha} + \mathcal{O}\left(x^{1 - 2\alpha}\right).
  $$
  Since $1/2 < \alpha < 1$, this means that $g(x) \to -\infty$ for $x \to \infty$. The right-hand side of Eq.\@ \eqref{eq:wsk_Omega_n} is equal to $2^{\beta/\alpha} \exp g(n)$, and so $\mathbb{P}(\Omega_{1,n}) \to 1$ for $n \to \infty$.  Therefore, for $n$ sufficiently large,
  \begin{align*}
     & \mathbb{E}\left[\sum_{i=1}^{n-1} \textbf{1}\left\{X_{n,i}' \in E, X_{n,i+1}' \notin E\right\}  + \textbf{1}\left\{X_{n,i}' \notin E, X_{n,i+1}' \in E\right\}\right] \\
     & \geq (n^\beta - K + 1) \mathbb{P}(\Omega_{1,n})                                                                                                                      \\
     & \geq n^\beta/2.
  \end{align*}
  By Eq.\@ \eqref{eq:diff_identity_2}, this means that
  $$
    \Delta_n(h) = \sqrt{n}\left\{(\mathbb{E}P_n)^2(h) - P^2(h)\right\} \geq \frac{n^{\beta + 1/2}}{6(n-1)}
  $$
  for sufficiently large $n$. Because $\beta > 1/2$, the right-hand side tends to infinity, and so $\|\Delta_n\|_{\mathcal{HK}(c)_2} \geq |\Delta_n(h)| \to \infty$ for some sufficiently large $c$. By scaling the function $h$ with some scalar parameter, we can prove the claim for arbitrary $c > 0$.

  We now sketch how one can prove the claim for arbitrary $m \in \mathbb{N}$. If $m > 2$, simply replace the function $h$ by $g^{(m)}$, defined by
  \begin{equation}
    \label{eq:gm}
    g^{(m)}[(s_{1,1}, s_{1,2}), \ldots, (s_{m,1}, s_{m,2})] = h[(s_{1,1}, s_{1,2}), (s_{2,1}, s_{2,2})],
  \end{equation}
  and proceed as above. For $m = 1$, replace $h$ by $g^{(1)}$, defined by
  $$
    g^{(1)}(s_1, s_2) = |s_1 - s_2|.
  $$
  Then $\mathbb{E}P_n(g^{(1)}) = (\mathbb{E}P_n)^2(h)$ and $P(g^{(1)}) = P^2(h)$, and the proof can be completed by the same arguments as before.
\end{proof}

\section{Proofs for Chatterjee's Rank Correlation}
\label{sec:proofs_chatterjee}
In this section we prove the results from Section \ref{sec:mains_chatterjee} which specifically concern Chatterjee's rank correlation. They are mainly simply consequences of the corresponding process results from that section.

Recall the kernel functions $h_1^*$ and $h_2^*$ from Eq.\@ \eqref{eq:h_stern_definition}. The following lemma ensures that the process results from Section \ref{sec:mains_chatterjee} can be used to deal with Chatterjee's rank correlation.

\begin{lemma}
  \label{lem:kerne_in_klammer}
  $h_1^*, h_2^* \in \mathcal{HK}(c)_3$ for any $c > 8$.
\end{lemma}
\begin{proof}

  In light of Theorem \ref{thm:characterisation_HKcm} it suffices to verify the coordinate-wise Hardy-Krause variations of $h_1^*$ and $h_2^*$. We can rewrite $h_1^*$ as
  \begin{align*}
    h_1^*\{(s_1, s_2), (t_1, t_2), (u_1, u_2)\} & = \textbf{1}_{[t_1, 1] \times [0,t_1)}(s_1, s_2) + \textbf{1}_{[0,t_1) \times  [t_1, 1]}(s_1, s_2)                                                  \\
                                                & = \textbf{1}_{(s_1 \land s_2, s_1 \lor s_2] \times [0,1]}(t_1, t_2)                                                                                 \\
                                                & = \left[\mathrm{sgn}(s_2 - s_1)\left\{\textbf{1}(t_1 \leq s_2) - \textbf{1}(t_1 \leq s_1)\right\}\right] \textbf{1}_{[0,1] \times [0,1]}(u_1, u_2).
  \end{align*}
  Indicator functions of axis parallel rectangles have Hardy-Krause variation of at most $4$. Therefore, the coordinate-wise Hardy-Krause variations of $h_1^*$ are bounded by $8$ as a consequence of the above representation. Similarly, we can write
  \begin{align*}
    h_2^*\{(s_1, s_2), (t_1, t_2), (u_1, u_2)\} & = \textbf{1}_{(u_1, t_1] \times [0,1]}(s_1, s_2)                         \\
                                                & = \textbf{1}\{u_1 < s_1\} \textbf{1}_{[s_1,1] \times [0,1]}(t_1, t_2)    \\
                                                & = \textbf{1}\{t_1 \geq s_1\} \textbf{1}_{[0,s_1) \times [0,1]}(u_1,u_2),
  \end{align*}
  which reveals that the coordinate-wise Hardy-Krause variations of $h_2^*$ are at most $4$.
\end{proof}

Next, a suitable representation of the Dette-Siburg-Stoimenov measure of dependence $\xi(X,Y)$. This follows from the definition of $\xi$ and some elementary calculations.

\begin{lemma}
  \label{lem:xi_alternative_darstellung}
  If $Y$ is not almost surely constant, then $\xi(X,Y) = 1 - \mu_1/(2\mu_2)$, where, with $P$ denoting the measure from Definition \ref{def:Fm_class},

  \begin{align*}
    \mu_{1} & = P(h_1^*) = \int \mathbb{P}\left[Y_1 \land Y_2 < Y_3 \leq Y_1 \lor Y_2 ~|~ (X_1, X_2) = (x,x)\right] ~\mathrm{d}\mathbb{P}^X(x), \\
    \mu_{2} & = P(h_2^*) = \mathbb{P}(Y_1 < Y_2 \leq Y_3).
  \end{align*}
\end{lemma}
\begin{proof}
  If $Y$ is not almost surely constant, then $\mu_2 > 0$. The identity $\xi = 1 - \mu_1/(2\mu_2)$ follows if we can show that
  $$
    \frac{\int \mathbb{E}\left[\mathrm{Var}(\textbf{1}_{[y,\infty)}(Y) ~|~ X)\right] ~\mathrm{d}\mathbb{P}^Y(y)}{\int \mathrm{Var}(\textbf{1}_{[y,\infty)}(Y)) ~\mathrm{d}\mathbb{P}^Y(y)} = \frac{\mu_1}{2 \mu_2},
  $$
  because $\mathrm{Var}(\mathbb{E}[U \,|\, V]) = \mathrm{Var}(U) - \mathbb{E}[\mathrm{Var}(U \,|\, V)]$ for any random variables $U$ and $V$. We only consider the numerator as the denominator is analogous. It holds that

  \begin{align*}
     & \mathrm{Var}(\textbf{1}_{[y,\infty)}(Y) ~|~ X = x)                                                                                                                                                                         \\
     & = \mathbb{P}(y \leq Y ~|~ X = x) - \mathbb{P}(y \leq Y ~|~ X = x)^2                                                                                                                                                        \\
     & = \mathbb{P}(y \leq Y_2 ~|~ X_2 = x) - \mathbb{P}(y \leq Y_1 ~|~ X_1 = x)\mathbb{P}(y \leq Y_2 ~|~ X_2 = x)                                                                                                                \\
     & = \mathbb{P}\left[y \leq Y_2 ~\Big|~ \begin{pmatrix}
                                                X_1 \\ X_2
                                              \end{pmatrix} = \begin{pmatrix}
                                                                x \\ x
                                                              \end{pmatrix}\right] - \mathbb{P}\left[y \leq Y_1 ~\Big|~ \begin{pmatrix}
                                                                                                                          X_1 \\ X_2
                                                                                                                        \end{pmatrix} = \begin{pmatrix}
                                                                                                                                          x \\ x
                                                                                                                                        \end{pmatrix}\right]\mathbb{P}\left[y \leq Y_2 ~\Big|~ \begin{pmatrix}
                                                                                                                                                                                                 X_1 \\ X_2
                                                                                                                                                                                               \end{pmatrix} = \begin{pmatrix}
                                                                                                                                                                                                                 x \\ x
                                                                                                                                                                                                               \end{pmatrix}\right] \\
     & = \mathbb{P}\left[y \leq Y_2 ~\Big|~ \begin{pmatrix}
                                                X_1 \\ X_2
                                              \end{pmatrix} = \begin{pmatrix}
                                                                x \\ x
                                                              \end{pmatrix}\right] - \mathbb{P}\left[y \leq Y_1 \land Y_2 ~\Big|~ \begin{pmatrix}
                                                                                                                                    X_1 \\ X_2
                                                                                                                                  \end{pmatrix} = \begin{pmatrix}
                                                                                                                                                    x \\ x
                                                                                                                                                  \end{pmatrix}\right]                                                              \\
     & = \mathbb{P}\left[Y_1 \land Y_2 < y \leq Y_2 ~\Big|~ \begin{pmatrix}
                                                                X_1 \\ X_2
                                                              \end{pmatrix} = \begin{pmatrix}
                                                                                x \\ x
                                                                              \end{pmatrix}\right].
  \end{align*}
  Integrating over $x$ with respect to $\mathbb{P}^X$ shows that
  \begin{align*}
     & \mathbb{E}\left[\mathrm{Var}(\textbf{1}_{[y,\infty)}(Y) ~|~ X)\right]                                      \\
     & = \int \mathbb{P}\left[Y_1 \land Y_2 < y \leq Y_2 ~|~ (X_1, X_2) = (x,x)\right] ~\mathrm{d}\mathbb{P}^X(x) \\
     & = \int \mathbb{P}\left[Y_1 < y \leq Y_2 ~|~ (X_1, X_2) = (x,x)\right] ~\mathrm{d}\mathbb{P}^X(x)
  \end{align*}
  since $\{Y_2 < Y_2\}$ is the empty set. By integrating over $y$, we finally obtain
  \begin{align*}
     & \int \mathbb{E}\left[\mathrm{Var}(\textbf{1}_{[y,\infty)}(Y) ~|~ X)\right] ~\mathrm{d}\mathbb{P}^Y(y)                        \\
     & = \iint \mathbb{P}\left[Y_1 < y \leq Y_2 ~|~ (X_1, X_2) = (x,x)\right] ~\mathrm{d}\mathbb{P}^X(x) ~\mathrm{d}\mathbb{P}^Y(y) \\
     & = \iint \mathbb{P}\left[Y_1 < y \leq Y_2 ~|~ (X_1, X_2) = (x,x)\right] ~\mathrm{d}\mathbb{P}^Y(y) ~\mathrm{d}\mathbb{P}^X(x) \\
     & = \int \mathbb{P}\left[Y_1 < Y_3 \leq Y_2 ~|~ (X_1, X_2) = (x,x)\right] ~\mathrm{d}\mathbb{P}^X(x)                           \\
     & = \frac{\mu_1}{2}.
  \end{align*}

  In a similar way, we can show that
  $$
    \int \mathrm{Var}(\textbf{1}_{[y,\infty)}(Y)) ~\mathrm{d}\mathbb{P}^Y(y) = \mu_2.
  $$
\end{proof}

We can now prove our main results concerning the asymptotic behaviour of Chatterjee's rank correlation. We will use the representation $\xi_n \approx 1 - V_1/(2V_2)$, and we can deduce from the already established results for the concomitant processes that $\sqrt{n}(V_1 - \mathbb{E}V_1)$ and $\sqrt{n}(V_2 - \mathbb{E}V_2)$ converge jointly to some two-dimensional normal distribution. From this, the step to the desired weak convergence of $\xi_n$ is taken with a usual Delta-method argument. The only challenge remaining is that we want our centring term to be $\mathbb{E}\xi_n$, whereas the Delta-method will give a centring term of $1 - \mathbb{E}V_1 / (2\mathbb{E}V_2)$. We will show that the difference between these centring terms is negligible using some standard arguments.

Following this proof are the proofs for the our results concerning the bias term $\delta_n$. These follow almost immediately from the more general results on the bias processes $\Delta_n$, which we have already proven.

\begin{proof}[Proof of Theorem \ref{thm:asymptotik} and Corollary \ref{cor:chatterjee_variance}]
  Define $\phi : \mathbb{R}^2 \to \mathbb{R}$ by
  $$
    \phi(x,y) = \begin{cases}
      1 - x/(2y) & \quad \textrm{if } y \neq 0, \\ 1 &\quad \textrm{otherwise.}
    \end{cases}
  $$

  Let $h_1^*$ and $h_2^*$ be the kernels from Eq.\@ \eqref{eq:h_stern_definition}, and write $V_1$ and $V_2$ for the V-statistics with kernel $h_1^*$ and $h_2^*$, respectively, based on the data $(Y_{n,1}', Y_{n,2}'), \ldots, (Y_{n,n-1}', Y_{n,n}')$. We have seen in Section \ref{sec:mains_chatterjee} that
  $$
    \xi_n = \phi(V_1, V_2) + \mathcal{O}\left(\frac{1}{n}\right),
  $$
  and by Lemma \ref{lem:xi_alternative_darstellung} it holds that $\xi = 1 - \mu_1/(2\mu_2)$, where $\mu_1 = P(h_1^*)$ and $\mu_2 = P(h_2^*)$. By Theorems \ref{thm:prozesskonvergenz_klammer} and \ref{thm:kovarianzfunktion_konkomitantenprozess} and Lemma \ref{lem:kerne_in_klammer}, it holds that
  $$
    \sqrt{n} \begin{pmatrix}
      V_1 - \mathbb{E}V_1 \\
      V_2 - \mathbb{E}V_2
    \end{pmatrix}
    \rightsquigarrow \mathcal{N}(0, \Sigma),
  $$
  where $\Sigma = (\sigma_{ij})_{1 \leq i,j \leq 2}$ and $\sigma_{ij} = \Lambda_P^{(3)}(h_i^*, h_j^*)$. $\phi$ is continuously differentiable on $\{(x,y) ~|~ y \neq 0\}$ with gradient $\nabla \phi(x,y) = (2y)^{-1} (-1, x/y)$. Since $Y$ is not almost surely constant, it holds that $\mu_2 > 0$. $\phi$ is therefore continuously differentiable in a neighbourhood of $\mu = (\mu_{1}, \mu_{2})$, and so by the Delta-method \citep[Theorem 3.8 in][]{vandervaart:asymptotic_statistics} it holds that
  \begin{equation}
    \label{eq:phi_convergence}
    \sqrt{n}\left[\phi(V_1, V_2) - \phi(\mathbb{E}V_1, \mathbb{E}V_2) \right] \rightsquigarrow \mathcal{N}\left(0, \sigma^2\right)
  \end{equation}
  for
  $$
    \sigma^2 = \nabla \phi(\mu_1, \mu_2)\Sigma \nabla \phi(\mu_1,\mu_2)^\top.
  $$
  Plugging in $\nabla\phi(\mu_1,\mu_2) = (2\mu_2)^{-1}(-1,\mu_1/\mu_2)$ establishes the formula for $\sigma^2$ given in Corollary \ref{cor:chatterjee_variance}.

  It remains to show that the difference between $\mathbb{E}\phi(V_1,V_2) = \mathbb{E}\xi_n(X,Y) + \mathcal{O}(1/n)$ and $\phi(\mathbb{E}V_1, \mathbb{E}V_2)$ is negligible. Let $U = \{(x_1, x_2) ~|~ x_2 > \mu_2/2\} \subseteq \mathbb{R}^2$. Observe that $V_2$ is actually just a V-statistic with a bounded kernel, based on i.i.d.\@ data. In particular, this implies that $\mathbb{E}V_2 = \mu_2 + \mathcal{O}(1/n)$, and so there exists some $n_0 \in \mathbb{N}$ such that $(\mathbb{E}V_1, \mathbb{E}V_2) \in U$ for all $n \geq n_0$. By Taylor's theorem, there exists for any $n \geq n_0$ and $u \in U$ some $\alpha_n = \alpha_n(u) \in U$ with
  \begin{align}
    \begin{split}
      \label{eq:phi_taylor}
      \phi(u_1, u_2) - \phi(\mathbb{E}V_1, \mathbb{E}V_2) & = \left\langle \nabla \phi(\mathbb{E}V_1, \mathbb{E}V_2), (u_1 - \mathbb{E}V_1, u_2 - \mathbb{E}V_2)\right\rangle_2                                  \\
                                                          & \quad + \frac{1}{2} \sum_{1 \leq k,l \leq 2} \frac{\partial^2}{\partial x_k \partial x_l} \phi(\alpha_n) (u_k - \mathbb{E}V_k)(u_l - \mathbb{E}V_l).
    \end{split}
  \end{align}
  By explicit calculation of the partial derivatives of $\phi$, we find that
  $$
    \left|\frac{\partial^2}{\partial x_k \partial x_l}\phi(\alpha_n)\right| \leq 8 \mu_2^{-3} < \infty
  $$
  for any $1 \leq k,l \leq 2$, where we use that $\alpha_n \in U$. Furthermore,
  $$
    \left\|\nabla \phi(\mathbb{E}V_1, \mathbb{E}V_2)\right\|_2 \leq \frac{1}{\mu_2} \sqrt{1 + 4\mu_1^2 / \mu_2^2} < \infty
  $$
  for all $n \geq n_0$. Eq.\@ \eqref{eq:phi_taylor} now gives us for all $n \geq n_0$
  \begin{align}
    \begin{split}
      \label{eq:phi_diff_E_1}
       & \mathbb{E}\left[\left|\sqrt{n}\{\phi(V_1, V_2) - \phi(\mathbb{E}V_1, \mathbb{E}V_2)\}\right|^p \textbf{1}\{(V_1, V_2) \in U\}\right]                                                                               \\
       & \lesssim \mathbb{E}\left|\left\langle \nabla \phi(\mathbb{E}V_1, \mathbb{E}V_2), \sqrt{n}(V_1 - \mathbb{E}V_1, V_2 - \mathbb{E}V_2)\right\rangle_2\right|^p                                                        \\
       & \quad + \sum_{1 \leq k,l \leq 2} \mathbb{E}\left|\sqrt{n}(V_k - \mathbb{E}V_k)(V_l - \mathbb{E}V_l)\right|^p                                                                                                       \\
       & \lesssim \mathbb{E}\left\|\sqrt{n}\begin{pmatrix}
                                             V_1 - \mathbb{E}V_1 \\ V_2 - \mathbb{E}V_2
                                           \end{pmatrix}\right\|_2^p + \sum_{1 \leq k,l \leq 2} \sqrt{\mathbb{E}\left|n^{1/4} (V_k - \mathbb{E}V_k)\right|^{2p} \mathbb{E}\left|n^{1/4} (V_l - \mathbb{E}V_l)\right|^{2p}},
    \end{split}
  \end{align}
  where $\lesssim$ is hiding constants depending only on $p$, $\mu_1$ and $\mu_2$. If $p$ is sufficiently large, then the last line is bounded uniformly in $n$ by Theorem \ref{thm:prozesskonvergenz_klammer}, Eq.\@ \eqref{eq:concomitants_p_moments}, and Lemma \ref{lem:kerne_in_klammer}. On the other hand, as we have already pointed out that $V_2$ is a V-statistic based on i.i.d.\@ data, standard concentration inequalities \citep[e.g.\@ the bounded difference inequality, Theorem 6.2 in][]{boucheron_etal:2013} imply that, for all $n$ sufficiently large that $|\mathbb{E}V_2 - \mu_2| \leq \mu_2/4$,
  \begin{align*}
    \mathbb{P}\{(V_1, V_2) \notin U\} & = \mathbb{P}(|V_2 - \mu_2| \geq \mu_2/2)                                      \\
                                      & \leq \mathbb{P}(|V_2 - \mathbb{E}V_2| \geq \mu_2/2 - |\mathbb{E}V_2 - \mu_2|) \\
                                      & \leq \mathbb{P}(|V_2 - \mathbb{E}V_2| \geq \mu_2/4)                           \\
                                      & \leq 2 \exp\left\{-n\mu_2^2/8\right\}.
  \end{align*}
  Finally, the minimum non-zero value which $V_2$ can take is $n^{-3}$, and so $|\phi(V_1, V_2)| \leq 1 + n^3/2 \leq n^3$ for $n > 1$. Therefore,
  \begin{align}
    \begin{split}
      \label{eq:phi_diff_E_2}
       & \mathbb{E}\left[\left|\sqrt{n}\{\phi(V_1, V_2) - \phi(\mathbb{E}V_1, \mathbb{E}V_2)\}\right|^p \textbf{1}\{(V_1, V_2) \notin U\}\right] \\
       & \leq 2^p n^{7p/2} \mathbb{P}\{(V_1, V_2) \notin U\} \xrightarrow[n \to \infty]{} 0.
    \end{split}
  \end{align}
  Eqs.\@ \eqref{eq:phi_diff_E_1} and \eqref{eq:phi_diff_E_2} together imply that the sequence $\sqrt{n}\{\phi(V_1, V_2) - \phi(\mathbb{E}V_1, \mathbb{E}V_2)\}$, $n \in \mathbb{N}$, is uniformly integrable (just pick some $p > 1$), and so by Eq.\@ \eqref{eq:phi_convergence},
  \begin{equation}
    \label{eq:chatterjee_phi_bias}
    \sqrt{n}\left\{\mathbb{E}\phi(V_1, V_2) - \phi(\mathbb{E}V_1, \mathbb{E}V_2)\right\} = \mathbb{E}\left[\sqrt{n}\left\{\phi(V_1, V_2) - \phi(\mathbb{E}V_1, \mathbb{E}V_2)\right\}\right] \xrightarrow[n \to \infty]{} 0.
  \end{equation}
  Therefore,
  $$
    \sqrt{n}\left[\phi(V_1, V_2) - \mathbb{E}\phi(V_1, V_2) \right] \rightsquigarrow \mathcal{N}\left(0, \sigma^2\right),
  $$
  again by Eq.\@ \eqref{eq:phi_convergence}, and our claim follows since $\xi_n(X,Y) = \phi(V_1, V_2) + \mathcal{O}(1/n)$.
\end{proof}

\begin{corollary}
  \label{cor:delta_n_bias}
  Under the assumptions of Theorem \ref{thm:bias_conditions}, $\sqrt{n}\delta_n \to 0$ as $n \to \infty$.
\end{corollary}
\begin{proof}

  Let $\phi, V_1, V_2, \mu_1, \mu_2$ be the objects from the proof of Theorem \ref{thm:asymptotik}, and recall that $\xi_n(X,Y) = \phi(V_1, V_2) + \mathcal{O}(1/n)$ and $\xi(X,Y) = \phi(\mu_1, \mu_2)$. Write
  \begin{align*}
    \sqrt{n}\delta_n & = \sqrt{n}\left\{\mathbb{E}\xi_n(X,Y) - \xi(X,Y)\right\}                                                                                                                                                    \\
                     & = \sqrt{n}\left\{\mathbb{E}\phi(V_1, V_2) - \phi(\mathbb{E}V_1, \mathbb{E}V_2)\right\} + \sqrt{n}\left\{\phi(\mathbb{E}V_1, \mathbb{E}V_2) - \phi(\mu_1, \mu_2)\right\} + \mathcal{O}\left(n^{-1/2}\right).
  \end{align*}
  The first term in the last line tends to $0$ by Eq.\@ \eqref{eq:chatterjee_phi_bias}, and it remains to consider the second term. Since $\mu_2 > 0$, $\phi$ is continuously differentiable and hence totally differentiable in a neighbourhood $U$ of $(\mu_1, \mu_2)$. Therefore,
  \begin{align*}
     & \sqrt{n}\left\{\phi(\mathbb{E}V_1, \mathbb{E}V_2) - \phi(\mu_1, \mu_2)\right\}                                                               \\
     & = \left\langle\nabla \phi(\mu_1, \mu_2), \sqrt{n}\begin{pmatrix}
                                                          \mathbb{E}V_1 - \mu_1 \\ \mathbb{E}V_2 - \mu_2
                                                        \end{pmatrix}\right\rangle_2 + R\left[\sqrt{n}\begin{pmatrix}
                                                                                                          \mathbb{E}V_1 - \mu_1 \\ \mathbb{E}V_2 - \mu_2
                                                                                                        \end{pmatrix}\right],
  \end{align*}
  where $R : U \to \mathbb{R}$ is a function satisfying $R(x)/\|x\|_2 \to 0$ as $x \to 0$. By the first part of Theorem \ref{thm:bias_conditions} (i.e.\@ the version for $\Delta_n$), which we have already proven, combined with Lemma \ref{lem:kerne_in_klammer}, it holds for any $c > 8$ that
  $$
    \left\|\sqrt{n}\begin{pmatrix}
      \mathbb{E}V_1 - \mu_1 \\ \mathbb{E}V_2 - \mu_2
    \end{pmatrix}\right\|_2 \leq \sqrt{2n} \|\Delta_n\|_{\mathcal{HK}(c)_3} \xrightarrow[n \to \infty]{} 0,
  $$
  and so $\delta_n \to 0$ as $n \to \infty$.
\end{proof}

\begin{proof}[Proof of Theorem \ref{thm:bias_conditions_large} for $\delta_n$]

  Let $\mu_1$ and $\mu_2$ be as in Lemma \ref{lem:xi_alternative_darstellung}, and $\phi$, $V_1$ and $V_2$ as in the proof of Theorem \ref{thm:asymptotik}. By Eq.\@ \eqref{eq:chatterjee_phi_bias}, we have
  $$
    \left|\sqrt{n}\left\{\mathbb{E}\xi_n(X,Y) - \phi(\mathbb{E}V_1, \mathbb{E}V_2)\right\}
    \right| \xrightarrow[n \to \infty]{} 0,
  $$
  and it therefore suffices to prove that
  $$
    \left|\sqrt{n}\left\{\phi(\mathbb{E}V_1, \mathbb{E}V_2) - \xi(X,Y)\right\}
    \right| \xrightarrow[n \to \infty]{} 0.
  $$
  Observe that $V_2$ is just a regular V-statistic of i.i.d.\@ data with a bounded kernel. Therefore, by standard arguments, $\mathbb{E}V_2 = \mu_2 + \mathcal{O}(1/n)$. In particular, $\mathbb{E}V_2 > 0$ for $n$ sufficiently large, and for such $n$ it holds that
  \begin{align}
    \begin{split}
      \label{eq:xi_bias_phi}
      \sqrt{n}\left\{\phi(\mathbb{E}V_1, \mathbb{E}V_2) - \xi(X,Y)\right\}
       & = \sqrt{n}\left\{\phi(\mathbb{E}V_1, \mathbb{E}V_2) - \phi(\mu_1,\mu_2)\right\}                                                 \\
       & = \sqrt{n}\left\{\frac{\mathbb{E}V_1}{2 \mathbb{E}V_2} - \frac{\mu_1}{2 \mu_2}\right\}                                          \\
       & = \frac{\sqrt{n}(\mu_2 \mathbb{E}V_1 - \mu_1 \mathbb{E}V_2)}{2\mathbb{E}V_2 \mu_2}                                              \\
       & =\frac{\sqrt{n}(\mathbb{E}V_1 - \mu_1)}{2 \mathbb{E}V_2} - \mu_1 \frac{\sqrt{n}(\mathbb{E}V_2 - \mu_2)}{2 \mathbb{E}V_2 \mu_2}.
    \end{split}
  \end{align}
  The second term in the last line tends to $0$ as $n \to \infty$ because $\mathbb{E}V_2 = \mu_2 + \mathcal{O}(1/n)$ and $\mu_2 > 0$. For the first term, observe that the function $h_1^*$ from Eq.\@ \eqref{eq:h_stern_definition} is equal to the function $g^{(3)}$ from Eq.\@ \eqref{eq:gm}. Therefore, by the proof of this Theorem for $\Delta_n$, which we have already completed, we get
  $$
    \sqrt{n}(\mathbb{E}V_1 - \mu_1) = \sqrt{n}\left\{(\mathbb{E}P_n)^3\left(g^{(3)}\right) - P^3\left(g^{(3)}\right)\right\} \xrightarrow[n \to \infty]{} \infty.
  $$
  Hence, the last line in Eq.\@ \eqref{eq:xi_bias_phi} tends to infinity as $n \to \infty$, which proves our claim.
\end{proof}
\section{Proofs for Mixing Sequences}
\label{sec:mixing_proofs}

The proofs in this section are very straightforward, since strongly mixing processes are well understood. In particular, we understand the behaviour of the empirical process based on strongly mixing observations. The proofs in this section are thus essentially just applications of our general results from Section \ref{sec:mains}, together with the Delta-method for Kendall's $\tau$.

\begin{proof}[Proof of Theorem \ref{thm:v_statistik_alpha_mixing}]

  Denote the $j$-th coordinate of $X_k$ by $X_{k,j}$ and write $F_j$ for the distribution function of $X_{1,j}$, $j = 1, \ldots d$. By a multivariate extension of Theorem 6.8 in \cite{bradley:mixing_vol1} -- see also Remark 6.11 (II) on p.\@ 204 of that reference --, there exists a strictly stationary sequence $U_k = (U_{k,1}, \ldots, U_{k,d})$ such that every $U_{k,j}$ follows a uniform distribution on the unit interval, $X_{k,j} = F_j^{-1}(U_{k,j})$ almost surely for every $j = 1, \ldots, d$ and $k \in \mathbb{N}$, and $\alpha(n) = \alpha_X(n) = \alpha_U(n)$, i.e.\@ the mixing coefficients of the sequences $(X_k)_{k \in \mathbb{N}}$ and $(U_k)_{k \in \mathbb{N}}$ are identical. Let $H_n = \sqrt{n}(Q_n - Q)$ be the empirical process based on $U_1, \ldots, U_n$ ($Q_n$ and $Q$ denote the empirical distribution function and true distribution function, respectively, of $U_1, \ldots, U_n$). By Theorem 7.3 in \cite{rio:2017} it holds that $H_n \rightsquigarrow H$ in $\ell^\infty([0,1]^d)$ for some tight mean-zero Gaussian process $H$. Let $P_n$ be the empirical measure of $X_1, \ldots, X_n$ and $G_n = \sqrt{n}(P_n - P)$ the corresponding empirical process. Since $X_k = (F_1^{-1}(U_{k,1}), \ldots, F_d^{-1}(U_{k,d}))$ almost surely for any $k \in \mathbb{N}$, Lemma 21.1 in \cite{van_der_vaart_wellner:weak_convergence} implies that $G_n(z) = (H_n \circ \phi)(z)$ for any $z \in \mathbb{R}^d$, where $\phi$ is defined by $\phi(z) = \phi(z_1, \ldots, z_d) = (F_1(z_1), \ldots, F(z_d))$. It is easily seen that the map
  \begin{align*}
    \ell^\infty\left([0,1]^d\right) & \to \ell^\infty\left(\mathbb{R}^d\right), \\
    X                               & \mapsto X \circ \phi,
  \end{align*}
  is Lipschitz-continuous and linear. By the continuous mapping theorem \citep[Theorem 1.3.6 in][]{van_der_vaart_wellner:weak_convergence}, we get
  $$
    G_n = H_n \circ \phi \rightsquigarrow  G = H \circ \phi,
  $$
  and $G$ is itself tight, mean-zero and Gaussian because $\phi$ is continuous and linear. The weak convergence of $V_n$ now follows from Theorems \ref{thm:extension_hk_simple} and \ref{thm:v_statistik_from_empirical_process}.

  Consider any fixed $h \in \mathcal{HK}(c)$. Following the argument in the proof of Lemma 5.7.3 in \cite{serfling:approximation_theorems}, we see that
  $$
    U_n(h) - V_n(h) = \mathcal{O}\left(n^{-1}\right) [U_n(h) - W_n(h)],
  $$
  where $W_n(h)$ is the average over all $h(X_{i_1}, \ldots, X_{i_m})$ with $i_k = i_l$ for at least one pair $k \neq l$, and the constant hidden in the Landau symbol only depends on $m$. Therefore,
  $$
    \sqrt{n} |U_n - V_n(h)| = \mathcal{O}\left(n^{-1/2}\right) \|h\|_\infty.
  $$
  We have seen in the proof of Theorem \ref{thm:characterisation_HKcm} that $\|h\|_\infty \leq c$ for any $h \in \mathcal{HK}(c)$. Therefore,
  $$
    \sup_{h \in \mathcal{HK}(c)} \sqrt{n} |U_n - V_n(h)| = \mathcal{O}\left(n^{-1/2}\right).
  $$
\end{proof}

\begin{proof}[Proof of Corollary \ref{cor:kendalls_tau}]
  Since Kendall's $\tau$ is a rank-based statistic, we can assume without loss of generality that $X,Y \in (0,1)$, since otherwise we can replace them by $w(X)$ and $w(Y)$, where $w : \mathbb{R} \to (0,1)$ is the strictly monotonic transformation given by $w(t) = 1/\{1 + \exp(-t)\}$. Let $P_n$ be the empirical measure of $(X_1, Y_1), \ldots (X_n, Y_n)$ and $P$ the distribution of $(X, Y)$. Define the kernel functions
  \begin{align*}
    f[(x_1, y_1), (x_2, y_2)] & = \textbf{1}\{(x_1 - x_2)(y_1 - y_2) > 0\}, \\
    g[(x_1, y_1), (x_2, y_2)] & = \textbf{1}\{(x_1 - x_2)(y_1 - y_2) < 0\}, \\
    u[(x_1, y_1), (x_2, y_2)] & = \textbf{1}\{x_1 = x_2\},                  \\
    v[(x_1, y_1), (x_2, y_2)] & = \textbf{1}\{y_1 = y_2\}.                  \\
  \end{align*}
  Then
  \begin{equation}
    \label{eq:kendalls_tau_terms}
    2 n^{-2} \begin{pmatrix}
      C_n \\
      D_n \\
      T_n \\
      U_n
    \end{pmatrix}
    =
    \begin{pmatrix}
      P_n^2(f)                                  \\
      P_n^2(g)                                  \\
      P_n^2(u) + \mathcal{O}\left(n^{-1}\right) \\
      P_n^2(v) + \mathcal{O}\left(n^{-1}\right)
    \end{pmatrix}.
  \end{equation}
  We claim that $f,g,u,v \in \mathcal{HK}(c)_2$ for some $c > 0$. We only verify this for $f$ and $u$, as the reasoning for $g$ and $v$ is analogous. For $f$, this follows from the identity
  $$
    f[(x_1, y_1), (x_2, y_2)] = \textbf{1}_{[0,x_2) \times [0,y_2)}(x_1, y_1) + \textbf{1}_{(x_2, 1] \times (y_2, 1]}(x_1, y_1),
  $$
  and the fact that the same identity holds if we switch the roles of $x_1$ and $x_2$, and those of $y_1$ and $y_2$. For $u$, it follows from
  $$
    u[(x_1, y_1), (x_2, y_2)] = \textbf{1}_{[0,x_2] \times [0,1]}(x_1, y_1) - \textbf{1}_{[0,x_2) \times [0,1]}(x_1, y_1),
  $$
  and the same switching argument of $x_1$ and $x_2$ as well as $y_1$ and $y_2$. By similar arguments for $g$ and $v$, we find that there is some sufficiently large $c > 0$ such that $f,g,u,v \in \mathcal{HK}(c)_2$. Hence, by Theorem \ref{thm:v_statistik_alpha_mixing},
  $$
    \sqrt{n}\left\{ P_n^2\begin{pmatrix}
      f \\ g \\ u \\ v
    \end{pmatrix}
    - P^2
    \begin{pmatrix}
      f \\ g \\ u \\ v
    \end{pmatrix}
    \right\} \rightsquigarrow \mathcal{N}(0, \Sigma)
  $$
  for some non-negative definite covariance matrix $\Sigma$, which also implies by Eq.\@ \eqref{eq:kendalls_tau_terms}
  \begin{equation}
    \label{eq:kendalls_tau_terme_in_prob}
    \sqrt{n}\left\{2n^{-2}\begin{pmatrix}
      C_n \\ D_n \\ T_n \\ U_n
    \end{pmatrix}
    -
    \begin{pmatrix}
      C \\ D \\ T \\ U
    \end{pmatrix}
    \right\} \rightsquigarrow \mathcal{N}(0, \Sigma),
  \end{equation}
  where, for an independent copy $(\tilde{X}, \tilde{Y})$ of $(X,Y)$,
  \begin{align*}
    C & = P^2(f) = \mathbb{P}\{(X - \tilde{X})(Y - \tilde{Y}) > 0\}, \\
    D & = P^2(g) = \mathbb{P}\{(X - \tilde{X})(Y - \tilde{Y}) < 0\}, \\
    T & = P^2(u) = \mathbb{P}(X = \tilde{X}),                        \\
    U & = P^2(v) = \mathbb{P}(Y = \tilde{Y}).
  \end{align*}
  Since $X$ and $Y$ are not almost surely constant, we have $U,T \in [0,1)$. Consider the map
  \begin{align*}
    \phi : \mathbb{R}^2 \times [0, 1)^2 & \to \mathbb{R},                                  \\
    (x_1, \ldots, x_4)                  & \mapsto \frac{x_1 - x_2}{\sqrt{(1-x_3)(1-x_4)}}.
  \end{align*}
  $\phi$ is differentiable everywhere with gradient
  $$
    \nabla \phi \begin{pmatrix}
      x_1 \\ x_2 \\ x_3 \\ x_4
    \end{pmatrix}
    =
    \begin{pmatrix}
      \{(1-x_3)(1-x_4)\}^{-1/2} \\ -\{(1-x_3)(1-x_4)\}^{-1/2} \\ \frac{x_1 - x_2}{2\sqrt{(1-x_3)^3 (1-x_4)}} \\ \frac{x_1 - x_2}{2\sqrt{(1-x_3)(1-x_4)^3}}
    \end{pmatrix}.
  $$
  Furthermore, since $[n(n-1)/2]^{-1} = 2 n^{-2} + \mathcal{O}\left(n^{-1}\right)$ and due to the fact that $C_n$, $D_n$, $T_n$ and $U_n$ are all absolutely bounded by $n^2$, $\hat{\tau}_b$ is equal to
  $$
    \frac{2 n^{-2} C_n - 2 n^{-2} D_n + \mathcal{O}\left(n^{-1}\right)}{\sqrt{\left[1 - 2 n^{-2} T_n + \mathcal{O}\left(n^{-1}\right)\right] \left[1 - 2 n^{-2} U_n + \mathcal{O}\left(n^{-1}\right)\right]}} = \phi\left[2n^{-2}\begin{pmatrix}
        C_n \\ D_n \\ T_n \\ U_n
      \end{pmatrix} + \mathcal{O}\left(n^{-1}\right)\right],
  $$
  where the $\mathcal{O}(n^{-1})$ term on the right-hand side is meant coordinate-wise. Since all partial derivatives of $\phi$ are continuous everywhere on $\mathbb{R}^2 \times [0, 1)^2$, $\phi$ is totally differentiable and we get
  $$
    \left|\hat{\tau}_b - \phi\left[2n^{-2}\begin{pmatrix}
        C_n \\ D_n \\ T_n \\ U_n
      \end{pmatrix} \right]\right|  =  \left\|\nabla\phi\left[2n^{-2}\begin{pmatrix}
        C_n \\ D_n \\ T_n \\ U_n
      \end{pmatrix} \right]\right\|_2 \mathcal{O}\left(n^{-1}\right) + o\left(n^{-1}\right) = \mathcal{O}_\mathbb{P}\left(n^{-1}\right).
  $$
  In the last equality we have used that the norm of the gradient is $\mathcal{O}_\mathbb{P}(1)$ due to Eq.\@ \eqref{eq:kendalls_tau_terme_in_prob} and the continuous mapping theorem. Furthermore, $\phi(C,D,T,U) = \tau_b$ as defined in the statement of the Lemma. The claim now follows by the Delta-method \citep[e.g.\@ Theorem 3.1 in][]{vandervaart:asymptotic_statistics}.
\end{proof}

\bibliographystyle{abbrvnat}
\bibliography{schwach_chatterjee}
\end{document}